\documentclass[a4paper, 12pt]{article}
\usepackage{amsmath,amssymb}
\usepackage{amsthm}
\usepackage[top=30truemm,bottom=30truemm,left=30truemm,right=30truemm]{geometry}
\usepackage{empheq}
\usepackage{color}
\usepackage{graphicx}
\usepackage{enumitem}
\usepackage{float}
\numberwithin{equation}{section}
\theoremstyle{plain}
\newtheorem{theorem}{Theorem}[section]
\newtheorem{lemma}[theorem]{Lemma}
\newtheorem{corollary}[theorem]{Corollary}
\newtheorem{proposition}[theorem]{Proposition}
\theoremstyle{definition}

\newtheorem{remark}[theorem]{Remark}

\newcommand{\Complex}{ \mathbb{C} }

\newcommand{\Fourier}{ \mathcal{F}}
\newcommand{\FourierInverse}{ \mathcal{F}^{ - 1 } }

\newcommand{\Integer}{ \mathbb{Z} }

\newcommand{\Real}{ \mathbb{R} }
\newcommand{\RealPart}{ \mathrm{Re} }

\title{Modeling and Mathematical Analysis of the Clogging Phenomenon in Filtration Filters Installed in Aquaria}
\author{Ken Furukawa\thanks{RIKEN Cluster for Pioneering Research (CPR), RIKEN, Hyogo 650-0047, Japan, \texttt{ken.furukawa@riken.jp}} \and Hiroyuki Kitahata\thanks{Graduate School of Science, Chiba University, Chiba 263-8522, Japan, \texttt{kitahata@chiba-u.jp}}}
\date{}

\begin{document}

\maketitle

\abstract{
    This paper proposes a mathematical model for replicating a simple dynamics in an aquarium with two components; bacteria and organic matter.
    The model is based on a system of partial differential equations (PDEs) with four components: the drift-diffusion equation, the dynamic boundary condition, the fourth boundary condition, and the prey-predator model.
    The system of PDEs is structured to represent typical dynamics, including the increase of organic matter in the aquarium due to the excretion of organisms ($e.g$. fish), its adsorption into the filtration filter, and the decomposition action of the organic matter both on the filtration filter and within the aquarium.
    In this paper, we prove the well-posedness of the system and show some results of numerical experiments.
    The numerical experiments provide a validity of the modeling and demonstrate filter clogging phenomena.
    We compare the feeding rate with the filtration performance of the filter.
    The model exhibits convergence to a bounded steady state when the feed rate is reasonable, and grow up to an unbounded solution when the feeding is excessively high.
    The latter corresponds to the clogging phenomenon of the filter.
}

%\begin{align}
%    \begin{array}{ccc}
%        a
%        & =
%        & \frac{b}{c} \\ [5mm]
%        a
%        & =
%        & \frac{b}{c}
%    \end{array}
%\end{align}
\section{Introduction}
\subsection{Background}

An aquarium is an artificial reproduction of an aquatic ecosystem within a confined space.
Within the water, living organisms such as fish reside, and the ecosystem is maintained by the balance between the metabolic waste produced by these organisms and bacteria that decompose and neutralize it as food.
To enhance this filtration process, aquariums are equipped with filtration filters.
These filters contain filter media, which serve as homes for bacteria, allowing more bacteria to thrive within them than in the open water.
Through the action of pumps installed in the filtration filter, water is drawn from the aquarium, passes through the filter, and then returns into the aquarium as clean water, thus accelerating the filtration process.
This process is crucial for maintaining the aquarium's ecosystem because the quantity of bacteria needed to decompose the metabolic waste produced by the living organisms is often insufficient in the aquarium water alone.

Besides an aquarium, the filtration to fluid such as water and air is used in various aspects of our daily lives, such as wastewater treatment and air purification.
In this paper, we consider a simple model based on partial differential equations (PDE), considering only the filtration for dust of organic matter in an aquarium and its decomposition by bacteria.
We present a mathematical analysis and some numerical examples of the model.

In the field of water filtration, the activated sludge model was proposed, see the book of Gujer $et$ $al.$ \cite{GujerHenzeLoosedrechtMino2006}.
This mathematical model finds application in wastewater treatment.
It describes the intricate decomposition processes of organic matter in water from both microbiological and chemical perspectives.
It is employed in various sewage treatment processes.
The model introduced in this paper, different from the activated sludge model, is based on a simple predator-prey relationship where microbes(predators) decompose organic matter(prey).
The filtration is characterized by using boundary conditions.
Complex interactions model between microbes and chemical substances are not utilized.
The boundary conditions we employ are dynamic boundary conditions and fourth boundary conditions, see the book of Schm\"{u}dge \cite{Schmudgen2012} Example 14.10.

Dynamic boundary conditions include time derivatives $\partial_t$ in the equations that the boundary unknowns satisfy.
The fourth boundary condition differs from the well-known the Dirichlet boundary condition (the first boundary condition), the Neumann boundary condition (the second boundary condition), and the Robin boundary condition (the third boundary condition).
The fourth boundary condition involves interactions between values at two boundaries.

\subsection{Model and problem}
The overview of the two-dimensional aquarium that we consider is described in Fig. \ref{fig_filter_abstracts}.
Note that, although the figure is two-dimensional for better understanding, we consider the one-dimensional aquarium in this paper for simplicity.
%Moreover, we do not consider the motion of fish in the aquarium.
Moreover, we do not take into account the volume of fish and the metabolic waste they produce.
There are two types of particles in the aquarium: dust of organic matter and bacteria.
The left and right sides of the one-dimensional aquarium are equipped with filters and pumps to circulate the water.
Water is sucked from the right filter, and an equal amount of water is expelled from the left side.
During this process, some of the dust and bacteria contained in the water passing through the left and right filters are absorbed into the filters.
The water that returns from the left side becomes clean.
It is reasonable to assume that the water flow through the pump decreases monotonically with the amount of dust in the filter is reasonable.
When there is no dust in the filter, the pump performance is at its maximum.
Conversely, if there is an infinite amount of dust in the filter, the water flow will be reduced to zero.
This case implies the clogging of the filter.
In this cycle, the water in the aquarium is filtered.
The system should {describe} simple mathematical and physical dynamics of filter clogging phenomena in a one-dimensional mathematically idealized aquarium where fish live.
The variables $u_1$ and $u_2$ correspond to dust or metabolite from fish and predator of dust or metabolite in the aquarium, respectively.
The variables $\rho_1$ and $\rho_2$ correspond to dust and predator of dust on the filtration filter which is identified as the boundary, respectively.

%In our PDE-based model, we consider the following actions and interactions:
%\begin{itemize}
%    \item Drift-diffusion equation for microbes and micro dust
%    \item Inflow of micro dust from external sources
%    \item Prey-Predator model within the interior of the aquarium and at the boundary of the filtration filter
%    \item Garbage does not reproduce on its own
%    \item There is a limit to the carrying capacity of microorganisms
%    \item Filtration at the boundary.
%\end{itemize}

\begin{figure*}[ht]
    \centering
    \includegraphics[width=0.8\textwidth]{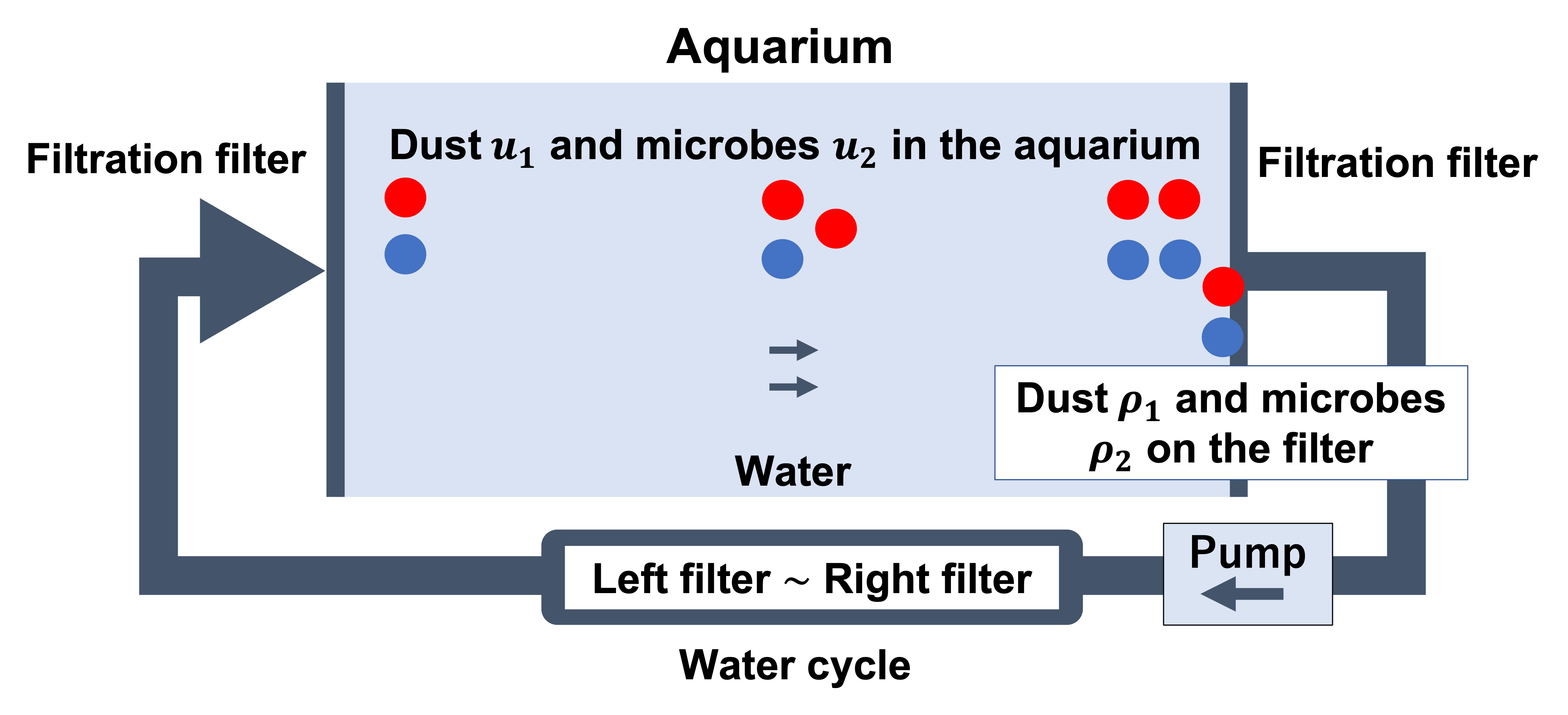}
    \caption{Outline of filtering system of an aquarium.
    There are dust and bacteria in the aquarium.
    Both dust and bacteria are taken in from the right boundary by the action of pumps within the filtration filters installed on both sides.
    A proportion in the filtration filter is absorbed, and the portion not absorbed is expelled from the left side.}
    \label{fig_filter_abstracts}
\end{figure*}

Let $I = (-1,1)$ and $\partial I = \{ x = \pm 1\}$, $\alpha \in [0, 1)$.
Let $\gamma_\pm \varphi = \varphi(\pm 1)$ be the boundary trace operator.
The model of the proposed system of PDE is given as below:
\begin{equation} \label{eq_filter_clogging}
    \begin{split}
        \begin{aligned}[t]
            &\partial_t u_1 - \nu_1 \partial_x^2 u_1 + c(\rho_1) \partial_x u_1
            = - \frac{R_1 u_1}{A + u_1} u_2
            + f,
            & x \in I,
            & \, t >0,\\
            &\partial_t u_2 - \nu_2 \partial_x^2 u_2 + c(\rho_1) \partial_x u_2
            = \left(
                \frac{R_2 u_1}{A + u_1} - \frac{u_2}{C_u}
            \right) u_2,
            & x \in I,
            & \, t >0,\\
            &B_1(u_1; F(\rho_1))
            = 0, \quad 
            B_2(u_1; F(\rho_1))
            =0,
            & x \in \partial I,
            & \, t >0,\\
            &B_1(u_2; F(\rho_1))
            = 0, \quad
            B_2(u_2; F(\rho_1))
            =0,
            & x \in \partial I,
            & \, t >0,\\
            &\frac{d\rho_1}{dt}
            = - \frac{S_1 \rho_1}{B + \rho_1} \rho_2
            + c(\rho_1) F(\rho_1) \gamma_+ u_1,
            & %x \in \partial I,
            & \, t >0,\\
            & \frac{d\rho_2}{dt}
            = \left(
                \frac{S_2 \rho_1}{B + \rho_1}
                - \frac{\rho_2}{C_\rho}
            \right) \rho_2
            + c(\rho_1) F(\rho_1) \gamma_+ u_2,
            & %x \in \partial I,
            & \, t >0,\\
            &c
            = P(\rho_1),
            &%x \in \partial I,
            & \, t >0,\\
        \end{aligned}
    \end{split}
\end{equation}
where $B_1, B_2$ are boundary operators defined by
\begin{align}
    B_1(u; \theta)
    &= (1 - \theta) \gamma_+ u - \gamma_- u, \label{eq_B1} \\
    B_2(u; \theta)
    &= \gamma_+ \partial_x u - (1 - \theta) \gamma_- \partial_x u. \label{eq_B2}
\end{align}
The prey-predator models in (\ref{eq_filter_clogging}), which correspond to first and second equations for the interior dynamics and fifth and sixth equations for the boundary dynamics in (\ref{eq_filter_clogging}), mean effects such that
\begin{itemize}
    \item Drift-diffusion for microbes and dust within the aquarium (for the first and second equations).
    \item Inflow of dust from external sources (for the first equation).
    \item The dust does not reproduce on its own (for the first and fifth equations).
    \item Prey-predator model within the interior of the aquarium and at the boundary of the filtration filter (for the first, second, fifth, and sixth equations).
    \item There is a limit to the carrying capacity of microbes (for the second and sixth equations).
    \item Filtration at the boundary (for the third and fourth equations).
\end{itemize}
The filtration function $F: \Real_{\geq 0} \rightarrow (0, 1]$ with the filter capacity constant $\beta$ is given by
\begin{equation}
    F(s)
    = \frac{1}{1 + \beta s}, \quad
    s \geq 0
\end{equation}
Note that the function $F$ chosen as above is not unique option, the inversely proportional relationship corresponds to the simplest case.
The function $F$ equals to $1$ when the filter is completely clean, however it means that the water is completely filtered, and if the filter is completely clogged, it means that no filtration occurs at all. 
We assume that the velocity $c$ is a function of $\rho_1$.
Furthermore, it is natural to assume that $c$ should tend to zero as $\rho_1 \rightarrow \infty$ and should attain maximum at $\rho_1=0$.
From these observations, the velocity $c$ is given by
\begin{align}
    c = c(\rho_1) = \Omega F(\rho_1) 
\end{align}
for some positive constant $\Omega$.
In this paper, to consider a simplified scenario, it is reasonable to assume that the filtration efficiency $F$ of the filter and the velocity $c$ of the flow are directly proportional.

We consider the boundary condition (\ref{eq_B1}).
The dusty water taken into the right boundary $\{ x = 1 \}$, and then the clean water is pushed out from the left boundary $\{ x = - 1 \}$.
The boundary condition $B_1(u; \theta) = 0$ is equivalent to
\begin{align*}
    \gamma_- u
    = \gamma_+ u - \theta \gamma_+ u.
\end{align*}
In this scenario the dust or predator is absorbed into the filter from the right with the rate $\theta>0$, and then the rest is pushed into the aquarium as cleaner water. 

The absorbed quantity accumulates in the filter and appears at the second term of the right-hand side of the fifth and sixth equations of (\ref{eq_filter_clogging}).
In order to understand the condition (\ref{eq_B2}) we begin by explanation for a simple case.
When $\theta=0$ the boundary conditions (\ref{eq_B1})-(\ref{eq_B2}) correspond to the periodic boundary condition, indicating that there is no filtration effect on the filter.
When $\theta=1$ the boundary conditions (\ref{eq_B1})-(\ref{eq_B2}) correspond to the Dirichlet-Neumann boundary condition for $x=-1$ and $x=1$, respectively, signifying that the filtration efficiency reaches its maximum, and the quantity absorbed from the boundary is entirely transferred into the filter.
Moreover, the pair of the boundary conditions (\ref{eq_B1})-(\ref{eq_B2}) serves as a mathematical requirement to make the Laplace operator $- \partial_x^2$ associated with the domain
\begin{align*}
    D(-\partial_x^2)
    = \{
            \varphi \in H^2(I)
        \,:\,
        B_1(\varphi; \theta)
        = B_2(\varphi; \theta)
        = 0
    \}
\end{align*}
as non-negative and self-adjoint.
It can be seen from the formula
\begin{align*}
    -\int_I 
        \partial_x^2 \varphi(x) \varphi(x)
    dx
    = \int_I 
        \vert
            \partial_x \varphi(x)
        \vert^2
    dx
    \geq 0.
\end{align*}
See also Lemma 3.2.1 in the book by Sohr \cite{Sohr2001}.
Therefore, it can be understood that the boundary conditions (\ref{eq_B1})-(\ref{eq_B2}) represent intermediate boundary conditions from the periodic boundary condition ($\theta=0$) to the Dirichlet-Neumann boundary condition ($\theta=1$) so that the Laplace operator $-\partial_x^2$ is non-negative self-adjoint.
A virtue of the model is replicating filtration by the combination of the forth boundary condition and the dynamic boundary condition.
This analogy should be applicable to a variety of other scientific and engineering disciplines.
For example, it may be applicable to the challenges of controlling substances in the blood by considering organs as boundaries.
It could also be applied to fluid purification problems, such as the targeted removal of specific substances such as toxins and salts from the air and water.
The dimensionless equations to (\ref{eq_filter_clogging}) are such that
\begin{equation} \label{eq_filter_clogging_nondimensional}
    \begin{aligned}
        &\partial_t v_1
        - \nu_1 \partial_x^2 v_1
        + \tilde{c}(\sigma_1) \partial_x v_1
        = - \frac{\tilde{R}_1}{1 + v_1} v_1 v_2 + \tilde{f}
        & x \in I,
        & \, t >0,\\
        &\partial_t v_2
        - \nu_2 \partial_x^2 v_2
        + \tilde{c}(\sigma_1) \partial_x v_2
        = \left(
            \frac{R_2 v_1}{1 + v_1}
            - v_2
        \right) v_2
        & x \in I,
        & \, t >0,\\
        &B_1(v_1; \tilde{F}(\sigma_1))
        = 0, \quad 
        B_2(v_1; \tilde{F}(\sigma_1))
        =0,
        & x \in \partial I,
        & \, t >0,\\
        &B_1(v_2; \tilde{F}(\sigma_1))
        = 0, \quad
        B_2(v_2; \tilde{F}(\sigma_1))
        =0,
        & x \in \partial I,
        & \, t >0,\\
        & \frac{d\sigma_1}{dt}
        = - \frac{\tilde{S}_1}{1 + \sigma_1} \sigma_1 \sigma_2
        + Q_1 \, \tilde{c}(\sigma_1) \tilde{F}(\sigma_1) \gamma_+ v_1,
        & %x \in \partial I,
        & \, t >0,\\
        & \frac{d\sigma_2}{dt}
        = \left(
            \frac{S_2 \sigma_1}{1 + \sigma_1}
            - \sigma_2
        \right) \sigma_2
        + Q_2 \tilde{c}(\sigma_1) \tilde{F}(\sigma_1) \gamma_+ v_2,
        & x \in \partial I,
        & \, t >0,\\
        &\tilde{c}
        = \Omega \tilde{F}(\sigma_1),
        & %x \in \partial I,
        & \, t >0,\\
    \end{aligned}
\end{equation}
where
\begin{align*}
    & u_1
    = A v_1, \quad
    u_2
    = C_u v_2, \\
    & \rho_1
    = B \sigma_1, \quad
    \rho_2
    = C_\rho \sigma_2, \\
    & \tilde{R}_1
    = \frac{R_1 C_u}{A}, \quad
    \tilde{S}_1
    = \frac{S_1 C_\rho}{B}, \quad
    Q_1 = \frac{A}{B}, \quad
    Q_2 = \frac{C_u}{C_\rho}\\
    &\tilde{F}(\sigma_1)
    = F(B \sigma_1)\\
    & \tilde{f}
    = \frac{f}{A}.
\end{align*}

Let $L^p(\Omega)$ be the Lebesgue space for $p \in [1, \infty]$ on $\Omega$ associated with the norm
\begin{align*}
    \Vert
        \varphi
    \Vert_{L^p(\Omega)}
    = \left(
        \int_\Omega
            \vert
                \varphi(x)
            \vert^p
        dx
    \right)^{1/p}.
\end{align*}
We use the standard modification for $p = \infty$.
Let $H^{m, p}$ be the $m$-th Sobolev space for $m \in \Integer_{\geq 0}$ associated with the norm
\begin{align*}
    \Vert
        \varphi
    \Vert_{H^{m,p}(\Omega)}
    = \sum_{\vert \alpha \vert \leq m}
    \Vert
        \partial^\alpha \varphi
    \Vert_{H^{m,p}(\Omega)},
\end{align*}
where $\alpha$ is a multi-index.
We denote by $BC^m(\Omega)$ the space of bounded continuous functions on $\Omega$ such that $m$-th derivatives are also bounded continuous.

There exist several results concerning well-posedness in the Hadamard sense for the Cauchy problem of partial differential equations under dynamic boundary conditions. Hintermann \cite{Hintermann1989} established the well-posedness of partial differential equations of elliptic, parabolic, and hyperbolic types in appropriate Sobolev spaces.
Denk, Pr\"{u}ss, and Zacher \cite{DenkPrussZacher2008} characterized the necessary and sufficient conditions for the well-posedness of $2m$-th order parabolic equations for where $m \in \Integer_{\geq 1}$ under dynamic boundary conditions in the $L^p$-$L^p$ maximal regularity framework.
The first author and Kajiwara \cite{FurukawaKajiwara2021} showed sufficient conditions for well-posedness in the $L^p$-$L^q$ maximal regularity setting.
In the book by Pr\"{u}ss and Simonett \cite{PrussSimonett2016}, various results related to the well-posedness of different types of partial differential equations of the parabolic type under dynamic boundary conditions in the $L^p$-$L^q$ maximal regularity settings are considered.
However, these results consider dynamic boundary conditions coupled with the Dirichlet, Neumann, and Robin boundary conditions.
The results are not applicable to our problem (\ref{eq_filter_clogging_nondimensional}), primarily because our boundary conditions include the fourth boundary condition.
The first main result of this paper is
\begin{theorem} \label{thm_main_theorem}
    Let $f \in C([0,T]; BC(I))$ be a non-negative function satisfying
    \begin{align*}
        f
        & \in H^2_tL^2_x(Q_T)
        \cap H^1_tH^2_x(Q_T)\\
        & \cap C([0,T); H^3(I))
        \cap C^1(0,T; H^1(I)).
    \end{align*}
    Let
    \begin{align*}
        v_{j,0} \in H^3(I),
        v_{j,0} \geq 0, \quad
        \sigma_{j,0}
        >0.
    \end{align*}
    Then there exists a unique solution
    \begin{align*}
        v_j
        \in C^1(0, T; C(\Omega)) \cap C(0,T; C^2(\Omega)), \quad
        \sigma_j
        \in C^1(0,T)
    \end{align*}
    to (\ref{eq_filter_clogging_nondimensional}) such that
    \begin{align} \label{eq_estimate_for_v_j_in_main_theorem}
        \begin{split}
            \sum_{j=1,2}(
                \Vert
                    v_j(t)
                \Vert_{H^2_tL^2_x(Q_T)}^2
                + \Vert
                    v_j(t)
                \Vert_{H^1_tH^2_x(Q_T)}^2
                + \Vert
                    v_j(t)
                \Vert_{L^2_tH^4_x(Q_T)}^2\\
                + \sup_{0<t<T} \Vert
                    v_j(t)
                \Vert_{H^3(I)}^2
                + \sup_{0<t<T} \Vert
                    \partial_t v_j(t)
                \Vert_{H^1(I)}^2
            )
            \leq C
        \end{split}
    \end{align}
    and 
    \begin{align} \label{eq_estimate_for_sigma_j_in_main_theorem}
        \sum_{j=1,2}(
            \sup_{0<t<T} \vert
                \sigma_j(t)
            \vert
            + \sup_{0<t<T} t^{\frac{1}{4} - \delta} \vert
                \sigma_j^\prime(t)
            \vert
            \sup_{0<t<T} t^{\frac{1}{2} -\delta} \vert
                \sigma_j^{\prime\prime}(t)
            \vert
        )
        \leq C
    \end{align}
    for some constant $C = C(f, v_{j,0}, \sigma_{j,0}, T)>0$.
\end{theorem}

In order to prove Theorem \ref{thm_main_theorem}, we begin by construction to the evolution operator for the equation
\begin{equation} \label{eq_linearized_equation_for_explanation}
    \begin{split}
        \begin{aligned}
            \partial_t u - \partial_x^2 u
            & = f
            & x \in I, \, t > 0,\\
            B_1(u; \theta)
            & = 0,
            & x = \pm 1, \, t > 0,\\
            B_2(u; \theta)
            & = 0,
            & x = \pm 1, \, t > 0,\\
            u
            & = u_0,
            & x \in I, \, t = 0,
        \end{aligned}
    \end{split}
\end{equation}
where $f \in C(0,T; L^q(\Omega))$, $u_0 \in L^q(\Omega)$ $\theta \in C^1(0,T)$.
We define the evolution operator for (\ref{eq_linearized_equation_for_explanation}) such that
\begin{align*}
    & u(t)
    = T(\theta; t, s) u(s)
    + \int_s^t
        T(\theta; t, \tau) f(\tau)
    d\tau, 
    & t > s > 0.
\end{align*}
We show resolvent estimate for the resolvent problem to (\ref{eq_linearized_equation_for_explanation}).
Using the result by Kato and Tanabe \cite{KatoTanabe1962}, we construct $T(t,s)$.
In construction of $T(t,s)$ we use the assumption that $\theta \in C^1(0,T)$.

We next establish a priori estimates for (\ref{eq_linearized_equation_for_explanation}) in $L^2$-framework.
In order to obtain the higher order a priori estimate, we show the existence of a extension operator $E(\theta; u)$ in Lemma \ref{lem_extension_theorem}.
In usual scenario, when we estimate $\sup_{0<t<T} \Vert \partial_t u (t) \Vert_{L^2(I)}$, we apply $\partial_t$ to the boundary condition.
However, the boundary conditions we use are time-dependent.
Since, $\partial_t u$ satisfies
\begin{align*}
    B_1(\theta; \partial_t u)
    & = \theta^\prime \gamma_+ u,\\
    B_2(\theta; \partial_t u)
    & = - \theta^\prime \gamma_- \partial_x u,
\end{align*}
the term on the right-hand side poses an obstacle when performing integration by parts.
we eliminate the reminder terms by the associated extension operator to enable to apply integration by parts, see Lemma \ref{lem_extension_theorem}.
The a priori estimates for $u$ and $\partial_t u$ in $C^1(0,T; H^1(I))$ are used to apply maximal principle to ensure the positivity of the solution.

For the nonlinear problems we construct the solution by the Leray-Schauder principle for $v$ and the Banach fixed point theorem for $\sigma$.
We first construct the solution $(v_1, v_2)$ to the first and second equations in (\ref{eq_filter_clogging_nondimensional}) for a given $\sigma_1 \in C^1(0,T)$ satisfying (\ref{eq_estimate_for_sigma_j_in_main_theorem}).
We consider the slightly generalized case such as (\ref{eq_abstract_filter_clogging_equation_in_the_interior}) to simplify notation.
We show the continuity, compactness, and the boundedness for the solution $v$ to
\begin{align} \label{eq_v_j=lambdaS}
    v_j
    = \lambda \mathcal{S}_{v,j}(v_1, v_2; \sigma_1),
\end{align}
where $\mathcal{S}_{v,j}$ is the solution operator to (\ref{eq_abstract_filter_clogging_equation_in_the_interior}) and $\lambda \in [0,1]$.
In order to show these properties of $v_j$ satisfying (\ref{eq_v_j=lambdaS}), we use a priori estimates for the linear case.
Since we assume sufficient regularity for $v_j$ as in (\ref{eq_estimate_for_v_j_in_main_theorem}), nonlinearity can be treated easily.

We show the existence of the solution to the fifth and sixth equations in (\ref{eq_filter_clogging_nondimensional}) by the Banach fixed point theorem.
In this scenario, $v_1$ and $v_2$ are functions of $\sigma_1$, namely $v_j = v_j(\sigma_1)$.
We prove the solution map is a self-mapping and contractive in the weighted Banach space equipped with the norm as in the left-hand side of (\ref{eq_estimate_for_sigma_j_in_main_theorem}).
The self-mapping property can be seen easily for small $T>0$ from the boundedness of $v_j, \partial_t v_j$ in $BC(0,T; H^1(I))$.
Since the dependence of $\sigma_1$ for $v_j(\sigma_1)$ is complicated, the proof of the contractivity of solution operator for small $T$ is also complicated.
We show this by a direct calculations, namely we estimate the difference
\begin{align*}
    \mathcal{S}_{v,j}(\psi_1, \psi_2; \tau_1)
    - \mathcal{S}_{v,j}(\psi_3, \psi_4; \tau_3)
\end{align*}
by integration by parts for some $\psi_j$ ($j=1,2,3,4$) satisfying (\ref{eq_estimate_for_v_j_in_main_theorem}) and $\tau_1, \tau_3$ satisfying (\ref{eq_estimate_for_sigma_j_in_main_theorem}).
Various methods are known for constructing solutions to partial differential equations. We have some comments regarding potential pitfalls when constructing the solution using certain methods that we do not employ:
\begin{remark}
    \begin{itemize}
        \item There are comprehensive studies of solvability for linear parabolic equations, such as \cite{DenkPrussZacher2008} and \cite{PrussSimonett2016}, but they do not include the fourth boundary conditions case treated in this paper, so their results cannot be used.
        \item We refrain from using the Galerkin approximation to construct the solution $(v_1, v_2)$ to the linearized problem in (\ref{eq_filter_clogging_nondimensional}) because it is difficult to elucidate the dependence on $t$ and $\sigma_j$ for the eigenvalues and eigenfunctions of the one-dimensional Laplace operator associated with the boundary conditions (\ref{eq_B1})-(\ref{eq_B2}).
        In the usual scenario for constructing the solution using the Galerkin approximation, the approximated solution is expanded by the eigenfunctions, and the problem is reduced to the existence of the solution to an ordinary differential equation.
        However, when the boundary conditions are time-dependent, the eigenvalues and eigenfunctions are also time-dependent.
        Therefore, this scenario is not straightforward.
        \item We avoid applying the Banach fixed-point theorem for the construction of $(v_1, v_2)$ because the evolution operator $U(t,s)$ in Lemma \ref{lem_existence_of_linearized_equation} depends on $\sigma_j$ and is derived using an abstract method, making it difficult to clarify the dependence on $\sigma_j$.
        The dependence of $\sigma_j$ is needed to construct the solution to the equation of $\sigma_j$.
    \end{itemize}
\end{remark}

The second main result of this paper is numerical simulations for the equations (\ref{eq_filter_clogging_nondimensional}).
We have performed some numerical experiments in Section \ref{sec_numerical_experiments} to show the behavior of the solution to the equations (\ref{eq_filter_clogging_nondimensional}).
The purpose of this section is to demonstrate numerically that the solutions of the equations exhibit an expected behavior of the aquarium ecosystem model (\ref{eq_filter_clogging_nondimensional}).
We first present the discretization scheme.
This discretization method is quadratic in accuracy with respect to time and spatial variables.
Note that this chapter is not intended for rigorous numerical analysis.
We note that our proposed discretization scheme will be novel for the numerical solution of partial differential equations with dynamic boundary conditions.
This is because, to the best of the authors' knowledge, there have been no numerical computations for this type of problem.
In the numerical experiments we assume that $f$ is a constant function with respect to $x$ and $t$.
We focus on the behavior of the solution for each $C_\rho$ and $f$.
This addresses the question of what is the appropriate feeding $f$ for the filter media $C_\rho$.
First, we see that the solution converges to a bounded steady state when the feeding rate is appropriate for the amount of filter media, and grows to an unbounded solution when the feeding rate is too high.
We then search for the boundary between this convergence and growth for multiple combinations of parameters.

We introduce notation.
For a Banach space $X$, we denote by $\Vert \cdot \Vert_{X \rightarrow X}$ the operator norm.
We write $\Fourier \varphi = \int_\Real e^{-i \xi x} \varphi(x) dx$ and $\FourierInverse \varphi = \frac{1}{2\pi} \int_\Real e^{i \xi x} \varphi(x) dx$ to denote the Fourier transform and the Fourier inverse transform.
We denote by $B(X,Y)$ the space of linear continues operators for Banach spaces $X, Y$ associated with the uniform norm.
We also denote $B(X) = B(X,X)$.
We denote by $L^p_t L^q_x(Q_T)$ the space-time Lebesgue space associated with the norm
\begin{align*}
    \Vert
        \varphi
    \Vert_{L^p_t L^q_x(Q_T)}
    := \left(
        \int_0^T
            \Vert
                \varphi(t, \cdot)
            \Vert_{L^q(\Omega)}^p
        dt
    \right)^{1/p}.
\end{align*}
We analogously define $H^{m,p}_tH^{n,q}_x(Q_T)$ for $m, n \in \Integer_{\geq 0}$ and $p,q \in (1,\infty)$.

In Section \ref{sec_linear_analysis} we show the existence of the solution to \ref{eq_filter_clogging_nondimensional}.
In Section \ref{sec_preliminary}, we introduce some known abstract results for sufficient condition to the existence of the evolution operator by Kato and Tanabe \cite{KatoTanabe1962}.
Using this result, we show existence of the solution to the linearized problems in Section \ref{section_existence_of_the_solution}.
We establish energy estimates for the linear part of the first equation of (\ref{eq_filter_clogging_nondimensional}) with fourth boundary condition.
%We next apply the energy estimates to obtain the solution to the nonlinear problem.
In Section \ref{sec_non_linear_problem} we consider the non-linear problem.
In this section, we consider a slightly generalized problem for simplicity of notation.
We first prove the existence of the solution to the first and second equations in (\ref{eq_filter_clogging_nondimensional}) by the Leray-Schauder principle.
We next prove the existence to the fifth and sixth equations in (\ref{eq_abstract_filter_clogging_equation}) by Banach's fixed point theorem.
In Section \ref{sec_numerical_experiments} we show some numerical result to demonstrate the model imitates the synthetic ecosystem correspond in an aquarium.

%%%%%%%%%%%%%%%%%%%%%%%%%%%%%%%%%%%%%%%%%%%%%%%%%%%%%%%%%%
%%%%%%%%%%%%%%%%%%%%%%%%%%%%%%%%%%%%%%%%%%%%%%%%%%%%%%%%%%
%-----------------------Section1.5-----------------------%
%%%%%%%%%%%%%%%%%%%%%%%%%%%%%%%%%%%%%%%%%%%%%%%%%%%%%%%%%%
%%%%%%%%%%%%%%%%%%%%%%%%%%%%%%%%%%%%%%%%%%%%%%%%%%%%%%%%%%
\section{Analysis for Linearized Problem} \label{sec_linear_analysis}
\subsection{Preliminary} \label{sec_preliminary}
We consider the abstract evolution equation associated with a closed operator $A(t)$ such that
\begin{align*}
    \begin{aligned}
        \partial_t u + A(t)u
        & = f,
        & t>0,\\
        B_1(u; \theta)
        & = h_1,
        & t>0,\\
        B_2(u; \theta)
        & = h_2,
        & t>0.
    \end{aligned}
\end{align*}
in a Hilbert space $X$.
The theory for the evolution operator is developed by Kato and Tanabe \cite{KatoTanabe1962}.
See also the book by Tanabe \cite{Tanabebook} for extensive research on evolution operators.
We construct the evolution operator in the same scenario as Theorem 4.1 in \cite{KatoTanabe1962}.
The domain of $A(t)$ associated with the graph norm is defined by
\begin{gather*}
    D(A(t))
    = \{
        \varphi \in X
        \, ; \,
        A(t) \varphi \in X
    \},\\
    \Vert
        \varphi
    \Vert_{D(A(t))}
    = \Vert
        \varphi
    \Vert_{X}
    + \Vert
        A(t) \varphi
    \Vert_{X}.
\end{gather*}
We construct the evolution operator such that
\begin{align*}
    U(t, s)
    = e^{-(t - s) A(t)}
    + \int_s^t
        e^{- (t - \tau)A(t)} R(\tau, s)
    d \tau
\end{align*}
and the operator $R$ is give as the solution to the integral equation
\begin{align*}
    R(t, s)
    = R_1(t, s)
    + \int_s^t
        R_1(t, \tau)
        R(\tau, s)
    d\tau,
\end{align*}
where
\begin{align*}
    R_1(t, s)
    = -\frac{1}{2 \pi i} \int_{\Gamma}
        e^{- (t - s) \zeta} \partial_t(
            \zeta
            - A(t)
        )^{-1}
    d\zeta
\end{align*}
and $\Gamma$ is a boundary of the sector $\{ z \in \Complex; \vert \arg z \vert < \theta\}$.
Note that if $z=0$ does not belong to the resolvent set of $A(t)$, we apply a standard modification to $\Gamma$ and integrate over $\Gamma_{\phi, \varepsilon}$ such that
\begin{align*}
    \Gamma_{\phi, \varepsilon}
    & = \Gamma_{-, \varepsilon, \phi}
    \cup \Gamma_{\text{circle}, \varepsilon, \phi}
    \cup \Gamma_{+, \varepsilon, \phi},\\
    \Gamma_{-, \varepsilon, \phi}
    & := \{ e^{i\phi}x \in \Complex
        ;
        - \infty < r < - \varepsilon
    \},\\
    \Gamma_{\text{circle}, \varepsilon, \phi}
    & := \{ \varepsilon e^{i \psi}\in \Complex
        ;
        - \pi < \psi < - \phi, \phi < \psi \leq \pi
    \},\\
    \Gamma_{-, \varepsilon, \phi}
    & := \{ e^{i\phi}r \in \Complex
        ;
        \varepsilon < r < \infty
    \},
\end{align*}
for $\varepsilon>0$ and $\phi \in (0,\pi/2)$.
In construction of the evolution operator we use
\begin{lemma}[Theorem 4.1-4.2 in \cite{KatoTanabe1962}] \label{lem_abstract_theorem_for_evolution_operator_by_Kato_Tanabe}
    Let $T>0$ and $X$ be a Banach space.
    Let $\mathcal{A}(t)$ the densely defined closed operator $\mathcal{A}(t)$ associated with the domain $D(\mathcal{A}(t)) \subset X$.
    Assume that
    \begin{enumerate}
        \item The operator $- \mathcal{A}(t)$ is a generator of an analytic semigroup and the resolvent set includes $\Sigma = \{z \in \Complex ; \phi < \arg z \leq 2\pi - \phi\}$ for $\phi \in (0, \pi/2)$ such that
        \begin{align*}
            \Vert
                (\lambda - \mathcal{A}(t))^{-1}
            \Vert_{B(X)}
            \leq \frac{C}{1 + \vert \lambda \vert}
        \end{align*}
        for $t$-independent constant $C>0$.
        \item The inverse operator $\mathcal{A}(t)^{-1}$ is differentiable in $B(X)$ for $t \in [0, T]$.
        \item The derivative $\frac{d\mathcal{A}(t)}{dt}$ is $\alpha$-H\"{o}lder continuous in $B(X)$ for $\alpha \in (0,1)$.
        \item There exist $C>0$, $\phi \in (0,\pi/2)$, and $\beta \leq 1$ such that
        \begin{align*}
            \Vert
                \partial_t (\lambda - \mathcal{A}(t))^{-1}
            \Vert_{B(X)}
            \leq \frac{C}{\vert \lambda \vert^{\beta}}
        \end{align*}
        for $\lambda \in \Sigma$ and $t \in [0,T]$.
    \end{enumerate}
    Then there exists a differentiable evolution operator $\mathcal{T}(t, s) \in B(X)$ for $0 < s \leq t \leq T$ in $B(X)$ such that
    \begin{align*}
        & R(\mathcal{T}(t,s)) \subset D(\mathcal{A}(t))\\
        & \partial_t \mathcal{T}(t,s)
        + \mathcal{A}(t) \mathcal{T}(t,s)
        = 0 \quad
        \text{in B(X) for $0 < s \leq t < T$},
    \end{align*}
    and $\mathcal{T}(s,s)$ is the identify in $B(X)$ satisfying
    \begin{align*}
        \Vert
            \partial_s \mathcal{T}(t,s)
        \Vert_{B(X)}
        & \leq C (t - s)^{-1},\\
        \Vert
            \partial_t \mathcal{T}(t,s)
        \Vert_{B(X)}
        + \Vert
            \mathcal{A}(t) \mathcal{T}(t,s)
        \Vert_{B(X)}
        & \leq C (t - s)^{-1}
    \end{align*}
    for some constant $C>0$.
    Moreover, for $f \in C^\gamma(0,T; X)$ for $\gamma \in (0,1)$ and $u_0 \in X$, there exists a unique solution to
    \begin{align*}
        \frac{du(t)}{dt}
        + \mathcal{A}(t) u
        = f(t) \quad
        \text{in $X$}, \quad
        u(0) = u_0
    \end{align*}
    of the form $u(t) = U(t,0)u_0 + \int_0^t U(t,s) f(s) ds$ in $X$.
\end{lemma}
\begin{remark}
    \begin{itemize}
        \item If $\mathcal{A}(t)$ is not invertible, we can use Lemma \ref{lem_abstract_theorem_for_evolution_operator_by_Kato_Tanabe} by shifting $\mathcal{A}(t)$ as $\mu + \mathcal{A}(t)$ for some $\mu>0$ or by modifying the integral curve to avoid around zero to define $\mathcal{T}(t,s)$.
        \item The condition (\ref{eq_estimate_for_sigma_j_in_main_theorem}) implies that $\sigma_j$ is in $H^2(0,T)$.
        Therefore, the Sobolev embedding implies that $d\sigma_j/dt \in C^{\alpha}[0,T]$ for $\alpha \in (0,1)$.
        This fact is used for the third assumption of $\mathcal{A}(t)$ in Lemma \ref{lem_abstract_theorem_for_evolution_operator_by_Kato_Tanabe} to construct the evolution operator for the linearized problem to (\ref{eq_filter_clogging_nondimensional}).
    \end{itemize}
\end{remark}

\subsection{Existence of the solution} \label{section_existence_of_the_solution}
We first construct an evolution operator for
\begin{equation} \label{eq_filter_clogging_linearized}
    \begin{split}
        \begin{aligned}[t]
            \partial_t u - \nu \partial_x^2 u
            & = f,
            & t \in (0, T),
            & \, x \in I,\\
            B_1(u; \theta(t))
            & = 0,
            & t \in (0, T),
            & \,\\
             B_2(u; \theta(t))
            & =0,
            & t \in (0, T).
            & \,
        \end{aligned}
    \end{split}
\end{equation}
For $\theta_0 \in (0, 1]$. and functions $\varphi, \psi \in H^2(I)$ satisfying $B_1(\cdot , \theta_0) =  B_2(\cdot , \theta_0) = 0$, we see from by integration by parts that 
\begin{align*}
    & \int_I
        \partial_x^2 \varphi(x) \psi(x)
    dx \\
    & = \partial_x \varphi(l) \psi(l)
    - \partial_x \varphi(-l) \psi(-l)
    + \int_I
        \partial_x \varphi(x) \partial_x \psi(x)
    dx \\
    &= \partial_x \varphi(l) (1 - \theta_0) \psi(-l)
    - (1 - \theta_0) \partial_x \varphi(l) \psi(-l)
    + \int_I
        \partial_x \varphi(x) \partial_x \psi(x)
    dx\\
    & = \int_I
        \partial_x \varphi(x) \partial_x \psi(x)
    dx\\
    & = \int_I
        \varphi(x) \partial_x^2 \psi(x)
    dx.
\end{align*}
Let $B : C^\infty(I) \times C^\infty(I) \rightarrow \Real$ be a bilinear operator such that
\begin{align*}
    & B(\varphi, \psi)
    = \int_I
        \partial_x \varphi(x) \partial_x \psi(x)
    dx
\end{align*}
for $\varphi, \phi \in C^\infty(I)$ satisfying $B_1(\cdot , \theta_0) =  B_2(\cdot , \theta_0) = 0$.
Since the bilinear $B(\cdot , \cdot)$ is a positive and symmetric, Lemma 3.2.1 in the book by Sohr \cite{Sohr2001} implies that the associated one-dimensional Laplace the operator $- \partial_x^2$ with the domain 
\begin{align*}
    D (- \partial_x^2)
    = \left \{  \varphi \in H^2(I) \,:\, B_1(\varphi, \theta_0) =  B_2(\varphi, \theta_0) = 0 \right \}
\end{align*}
is positive self-adjoint.
We denote by $A(t)$ such operator $- \partial_x^2$ with $\theta_0$-dependent domain $D(-\partial_x^2)$.
When the choice of $\theta_0$ is clear, we simply denote it by $- \partial_x^2$.
Note that when $\theta_0 = 0, 1$, the boundary conditions are equivalent to the periodic boundary condition and the Dirichlet-Neumann boundary condition for the left and right boundaries, respectively.
Zero is a spectrum for the case $\theta_0 = 0$.

\begin{lemma} \label{lem_existence_of_linearized_equation}
    Let $T>0$ and $\theta \in C^{1, \alpha}[0,T]$.
    Then there exists a unique evolution operator $U(t, s) \in B(L^2(I))$ for $0 \leq s < t \leq T$,
    \begin{align*}
        u(t)
        = U(t,0) u_0
        + \int_0^t
            U(t,s) f(s)
        ds
    \end{align*}
    is the solution to (\ref{eq_filter_clogging_linearized}) with initial data $u_0 \in L^2(I)$ such that
    \begin{align} \label{eq_estimate_for_u_in_lem_existence_of_linearized_equation}
        \Vert
            \partial_t u(t)
        \Vert_{L^2(I)}
        + \Vert
            \partial_t^2 u(t)
        \Vert_{L^2(I)}
        \leq C t^{-1}
    \end{align}
    for some constant $C>0$.
\end{lemma}
\begin{proof}
    The equation has time-dependent term in the boundary condition.
    We cannot use the semigroup theory directly.
    We consider the resolvent problem such that
\begin{equation} \label{eq_linearized_resolvent}
    \begin{split}
        \begin{aligned}[t]
            \lambda v - \nu \partial_x^2 v
            & = g,
            & t \in (0, T),
            & \, x \in I,\\
            B_1(v; \theta(t))
            & = 0,
            & t \in (0, T),
            & \,\\
             B_2(v; \theta(t))
            & =0,
            & t \in (0, T),
            & \,
        \end{aligned}
    \end{split}
\end{equation}
to construct the evolution operator for $\lambda \in \Complex \setminus \Real_-$ for $\theta \in C^{1, \alpha}[0,T]$ for $\alpha \in (0, 1/2)$satisfying
\begin{align*}
    0 < \theta(t) \leq 1, \\
    \inf_{0<t<T} \theta(t) > 0.
\end{align*}
We seek the solution formula for the resolvent problem.
Let $\tilde{g}$ be the zero extension of $g$ with respect to $x$ from $(-1, 1)$ into $\Real$.
We divide $v$ into two parts:
\begin{equation} \label{eq_linearized_resolvent_v1}
    \begin{split}
        \begin{aligned}[t]
            \lambda v_1 - \nu \partial_x^2 v_1
            & = \tilde{g},
            & t \in (0, T), \quad
            & x \in \Real,
        \end{aligned}
    \end{split}
\end{equation}
%where $\tilde{g}$ is an extension from $I$ to $\Real$ $cf.$ Triebel [Theory of function spaces, II, Section 1. 10. 2] with cut-off with in $[-1 - \varepsilon, -1]$ and $[1 - \varepsilon, 1]$, and
and
\begin{equation} \label{eq_linearized_resolvent_v2}
    \begin{split}
        \begin{aligned}[t]
            \lambda v_2 - \nu \partial_x^2 v_2
            & = 0,
            & t \in (0, T), \quad
            & x \in I,\\
            B_1(v_2; \theta(t))
            & = - B_1(v_1; \theta(t)),
            & t \in (0, T),
            & \,\\
             B_2(v_2; \theta(t))
            & = - B_2(v_1; \theta(t)),
            & t \in (0, T).
            & \,
        \end{aligned}
    \end{split}
\end{equation}
By the definition of $v_1, v_2$, we have $v = v_1 + v_2$.
To simplify the notation, we assume that $\nu = 1$.
Since
\begin{align*}
    \Fourier v_1
    = \frac{\Fourier \tilde{g}}{\lambda + \xi^2},
\end{align*}
the Plancherel theorem implies
\begin{align} \label{eq_L2_estimate_for_v1}
    \vert \lambda \vert \Vert
        v_1
    \Vert_{L^2(\Real)}
    + \vert \lambda \vert^{1/2} \Vert
        \partial_x v_1
    \Vert_{L^2(\Real)}
    + \Vert
        \partial_x^2 v_1
    \Vert_{L^2(\Real)}
    \leq C \Vert
        g
    \Vert_{L^2(I)}
\end{align}
for all $t \geq 0$.
We construct the solution formula to $v_2$.
We observe that $v_2$ is of the form
\begin{align} \label{eq_formula_for_v2}
    v_2
    = C_1 e^{+ \lambda^{1/2} x}
    + C_2 e^{- \lambda^{1/2} x}.
\end{align}
We determine $C_1, C_2$.
Since
\begin{align*}
    B_1(e^{\pm \lambda^{1/2}x}; \theta(t))
    &= (1- \theta(t)) e^{\pm \lambda^{1/2}}
    - e^{\mp \lambda^{1/2}},\\
    B_2(e^{\pm \lambda^{1/2}x}; \theta(t))
    &= \lambda^{1/2} \left(
        \pm (1- \theta(t)) e^{\pm \lambda^{1/2}}
        \mp e^{\mp \lambda^{1/2}}
    \right),
\end{align*}
inserting $v_2$ to the boundary conditions, we find that
\begin{align}\label{eq_formula_for_constants_of_v2}
    \begin{split}
        & \left (
            \begin{array}{c}
                C_1 \\
                C_2
            \end{array}
        \right )
        = \frac{
            1
        }{
            M(\lambda, \theta(t))
        }\\
        & \times \left (
            \begin{array}{c}
                - B_2(e^{- \lambda^{1/2} x}; \theta(t)) B_1(v_1; \theta(t))
                + B_1(e^{- \lambda^{1/2} x}; \theta(t)) B_2(v_1; \theta(t)) \\
                B_2(e^{+ \lambda^{1/2} x}; \theta(t)) B_1(v_1; \theta(t))
                + B_1(e^{+ \lambda^{1/2} x}; \theta(t)) B_2(v_1; \theta(t))
            \end{array}
        \right ),\\
        & M(\lambda, \theta(t))
        := B_1(e^{+ \lambda^{1/2} x}; \theta(t)) B_2(e^{- \lambda^{1/2} x}; \theta(t))\\
        & \quad\quad\quad\quad\quad
        - B_1(e^{- \lambda^{1/2} x}; \theta(t)) B_2(e^{+ \lambda^{1/2} x}; \theta(t)).
    \end{split}
\end{align}
By elementary calculations, we observe that
\begin{align*}
    M(\lambda, \theta(t))
    & = \lambda^{1/2} \left(
        (1 - \theta(t)) e^{+ \lambda^{1/2}}
        - e^{- \lambda^{1/2}}
    \right)^2\\
    & + \lambda^{1/2} \left(
        e^{+ \lambda^{1/2}}
        - (1 - \theta(t)) e^{- \lambda^{1/2}}
    \right)^2.
\end{align*}
%We invoke the two formulas such that
We invoke the formula such that
\begin{align*}
    M(\lambda, \theta(t))
    & = \lambda^{1/2} 
    \left[
        \{
            (1 + i) - \theta(t)
        \} e^{+\lambda^{1/2}}
        + \{
            (-1 - i)
            + i \theta(t)
        \} e^{- \lambda^{1/2}}
    \right]\\
    & \times
    \left[
        \{
            (1 - i) - \theta(t)
        \} e^{+\lambda^{1/2}}
        + \{
            (-1 + i)
            - i \theta(t)
        \} e^{- \lambda^{1/2}}
    \right].%,\\
    %M(\lambda, \theta(t))
    %& = \lambda^{1/2} 
    %\left[
    %    (1 + i)
    %    (e^{+\lambda^{1/2}} - e^{- \lambda^{1/2}})
    %    + \theta(t) (- e^{+\lambda^{1/2}} + i e^{- \lambda^{1/2}})
    %\right]\\
    %& \times
    %\left[
    %    (1 - i)
    %    (e^{+\lambda^{1/2}} - e^{- \lambda^{1/2}})
    %    - \theta(t) (
    %        e^{+ \lambda^{1/2}}
    %        + i e^{- \lambda^{1/2}}
    %    ) 
    %\right].
\end{align*}
By the first formula of $M(\lambda, \theta(t))$, the non-zero polars are such that
\begin{align*}
    e^{2\lambda^{1/2}}
    = \frac{
        1
        + (1 - \theta(t)) i
    }{
        (1 - \theta(t))
        + i
    },
    \frac{
        1
        - (1 - \theta(t)) i
    }{
        (1 - \theta(t))
        - i
    }
    =: p_{\theta}^+, p_{\theta}^-.
\end{align*}
The terms in the right-hand side are on the unit circle of $\Complex$.
Therefore, the polars attain on $\RealPart \lambda^{1/2} = 0$, $i.e.$ $\lambda \in \Real_-$ and $1/M(\lambda, \theta(t))$ is bounded for $\lambda \in \Sigma_{\phi}$ for $\pi/2 < \psi < \pi/2$.

We estimate $v_1$ and $v_2$.
Since the interval $I$ is bounded, it is clear that
\begin{align*}
    \Vert
        e^{\pm \lambda^{1/2} x}
    \Vert_{L^2(I)}
    \leq C e^{\vert \lambda \vert^{1/2}}.
\end{align*}
The kernel $k_\lambda (x)$ for $(\lambda + \xi^2)^{-1}$ is given by
\begin{align*}
    k_\lambda (x)
    & = \frac{e^{- \lambda^{1/2} \vert x \vert}}{2 \lambda^{1/2}}
\end{align*}
and satisfies
\begin{align*}
    \frac{d}{dx} k_\lambda (x)
    & = - \lambda^{1/2} (\mathrm{sgn} x ) k_\lambda (x)
    = \frac{(- \mathrm{sgn} x) e^{- \lambda^{1/2} \vert x \vert}}{2}.
\end{align*}
Therefore, we have
\begin{align*}
    & \gamma_\pm v_1(x)
    = \int_\Real
        \frac{
            e^{- \lambda^{1/2} \vert
                \pm 1 - y
            \vert
        }}{
            2 \lambda^{1/2}
        }
        \tilde{g}(y)
    dy,\\
    & \gamma_\pm \partial_x v_1(x)
    = \int_\Real
        (- \mathrm{sgn} (\pm 1 - y))\frac{
            e^{- \lambda^{1/2} \vert
                \pm 1 - y
            \vert}
        }{
            2
        }
        \tilde{g}(y)
    dy.
\end{align*}
By (\ref{eq_formula_for_v2}) and (\ref{eq_formula_for_constants_of_v2}), we obtain the formula for $v_2$ such that
\begin{align} \label{eq_formula_of_v2}
    \begin{split}
        & v_2(x) \\
        & = \frac{1}{2 \lambda^{1/2}} \frac{
            1
        }{
            \left(
                (1 - \theta(t)) e^{+ \lambda^{1/2}}
                - e^{- \lambda^{1/2}}
            \right)^2
            + \left(
                e^{+ \lambda^{1/2}}
                - (1 - \theta(t)) e^{- \lambda^{1/2}}
            \right)^2
        }\\
        & \times \left[
            - (
                e^{- \lambda^{1/2}}
                + (1 - \theta(t)) e^{+ \lambda^{1/2}}
            )
            \left(
                (1- \theta(t))
                \int_\Real
                    e^{
                        - \lambda^{1/2} \vert + 1 - y \vert
                        + \lambda^{1/2}x
                    }
                    \tilde{g}(y)
            \right.
            dy
        \right.\\
        & \left.
            \quad\quad\quad\quad\quad\quad\quad\quad\quad\quad\quad\quad
                - \int_\Real
                    e^{
                        - \lambda^{1/2}\vert - 1 -y \vert
                        + \lambda^{1/2}x
                    }
                    \tilde{g}(y)
                dy
            \right)\\
        & \quad + (
            (1 - \theta(t)) e^{- \lambda^{1/2}}
            - e^{+ \lambda^{1/2}}
        )
        \left(
            \int_\Real
                - \mathrm{sgn} (+ 1 - y) e^{
                    - \lambda^{1/2} \vert + 1 - y \vert
                    + \lambda^{1/2}x
                }
                \tilde{g}(y)
            dy
        \right.\\
        & \left.
            \quad\quad\quad\quad\quad\quad\quad\quad\quad
            + (1 + \theta(t)) \int_\Real
                ( - \mathrm{sgn} (- 1 - y)) e^{
                    - \lambda^{1/2} \vert - 1 - y \vert
                    + \lambda^{1/2}x
                }
                \tilde{g}(y)
            dy
        \right)\\
        & \quad + (
            e^{- \lambda^{1/2}}
            - (1 - \theta(t)) e^{- \lambda^{1/2}}
        )
        \left(
            (1- \theta(t))
            \int_\Real
                e^{
                    - \lambda^{1/2} \vert + 1 - y \vert
                    - \lambda^{1/2}x
                }
                \tilde{g}(y)
            dy
        \right.\\
        & \left.
            \quad\quad\quad\quad\quad\quad\quad\quad\quad\quad\quad\quad
            - \int_\Real
                e^{
                    - \lambda^{1/2} \vert - 1 - y \vert
                    - \lambda^{1/2}x
                }
                \tilde{g}(y)
            dy
        \right)\\
        & %\left.
            \quad + (
            (1 - \theta(t)) e^{+ \lambda^{1/2}}
            - e^{- \lambda^{1/2}}
        )
        \left(
            \int_\Real
                ( - \mathrm{sgn} (+1 - y)) e^{
                    - \lambda^{1/2} \vert + 1 - y \vert
                    - \lambda^{1/2}x
                }
                \tilde{g}(y)
            dy
        \right.\\
        & \left.
            \quad\quad\quad\quad\quad\quad\quad\quad\quad
            - (1- \theta(t))
            \int_\Real
                ( - \mathrm{sgn} (-1 - y)) e^{
                    - \lambda^{1/2} \vert - 1 - y \vert
                    - \lambda^{1/2}x
                }
                \tilde{g}(y)
            dy
        \right)\\
        & =: l_1(\lambda, x)
        + l_2(\lambda, x)
        + l_3(\lambda, x)
        + l_4(\lambda, x).
    \end{split}
\end{align}
Since $x \in (-1, 1)$ and $\mathrm{spt} \tilde{g} = [-1 ,1]$, the sign of $1 \pm x$ and $\pm 1 - y$ is consistent.
Therefore, we can take the supremum with respect to $\lambda$ and use the boundedness of the Hilbert transform to get
\begin{align*}
    & \left \Vert
        \int_\Real
            \lambda^{1/2} e^{\lambda^{1/2}(1 \pm x)} e^{\lambda^{1/2} \vert
                \pm 1 - y
            \vert}
            \tilde{g}(y)
        dy
    \right \Vert_{L^2(I)}
    \leq C \Vert
        \tilde{g}
    \Vert_{L^2(I)}
    \leq C \Vert
        g
    \Vert_{L^2(I)},
\end{align*}
and similarly
\begin{align*}
    \left \Vert
        \partial_x^2 \int_\Real
            \lambda^{-1/2} e^{\lambda^{1/2}(1 \pm x)} e^{\lambda^{1/2} \vert
                \pm 1 - y
            \vert}
            \tilde{g}(y)
        dy
    \right \Vert_{L^2(I)}
    \leq C \Vert
        g
    \Vert_{L^2(I)}
\end{align*}
for some constant $C>0$, which is independent of $\lambda$.
The reader is also referred to Lemma 3.3 in \cite{Abels2002}.
%\begin{align*}
%    & \left \Vert
%            \lambda
%        \frac{
%            B_2(e^{\lambda^{1/2} x}; \theta(t)) B_1(v_1; \theta(t))
%        }{
%            M(\lambda, \theta(t))
%        }
%        e^{\lambda^{\pm 1/2}x}
%    \right \Vert_{L^2(I)}\\
%    & + \left \Vert
%            \partial_x^2
%        \frac{
%            B_2(e^{\lambda^{1/2} x}; \theta(t)) B_1(v_1; \theta(t))
%        }{
%            M(\lambda, \theta(t))
%        }
%        e^{\lambda^{\pm 1/2}x}
%    \right \Vert_{L^2(I)}\\
%    & \leq C \sup_{\lambda \in \Sigma_\phi} \left|
%        \frac{
%            \left(
%                (1- \theta(t)) e^{+ \lambda^{1/2}}
%                - e^{- \lambda^{1/2}}
%            \right)
%            \vert
%            (
%                \vert 1 - \theta(t) \vert
%                + 1
%            )
%        }{
%            \left(
%                (1 - \theta(t)) e^{+ \lambda^{1/2}}
%                - e^{- \lambda^{1/2}}
%            \right)^2
%            + \left(
%                e^{+ \lambda^{1/2}}
%                - (1 - \theta(t)) e^{- \lambda^{1/2}}
%            \right)^2
%        }
%    \right| \Vert
%        g
%    \Vert_{L^2(I)}
%\end{align*}
%for $\phi = \arg \lambda \in (0, \pi)$ and some constant $C>0$.
\begin{align*}
    & \left \Vert
        \lambda l_1(\lambda, \cdot)
    \right \Vert_{L^2(I)}
    + \left \Vert
        \partial_x^2 l_1(\lambda, \cdot)
    \right \Vert_{L^2(I)}\\
    & \leq C \sup_{\lambda \in \Sigma_\phi} \left|
        \frac{
            \left(
                (1- \theta(t)) e^{+ \lambda^{1/2}}
                - e^{- \lambda^{1/2}}
            \right)
            \vert
            (
                \vert 1 - \theta(t) \vert
                + 1
            )
        }{
            \left(
                (1 - \theta(t)) e^{+ \lambda^{1/2}}
                - e^{- \lambda^{1/2}}
            \right)^2
            + \left(
                e^{+ \lambda^{1/2}}
                - (1 - \theta(t)) e^{- \lambda^{1/2}}
            \right)^2
        }
    \right| \Vert
        g
    \Vert_{L^2(I)}\\
    & \leq C \Vert
        g
    \Vert_{L^2(I)}
\end{align*}
for $\phi = \arg \lambda \in (0, \pi)$ and some constant $C>0$.
We can similarly estimate $l_2$, $l_3$, and $l_4$ to get
\begin{align*}
    \left \Vert
        \lambda l_j(\lambda, \cdot)
    \right \Vert_{L^2(I)}
    + \left \Vert
        \partial_x^2 l_j(\lambda, \cdot)
    \right \Vert_{L^2(I)}
    \leq C \Vert
        g
    \Vert_{L^2(I)}
\end{align*}
for all $j=2,3,4$.
Therefore, combining this with the interpolation inequality, we conclude that
\begin{align*}
    \vert \lambda \vert \Vert
        v_2 (t)
    \Vert_{L^2(\Real)}
    + \vert \lambda \vert^{1/2} \Vert
        \partial_x v_2 (t)
    \Vert_{L^2(\Real)}
    + \Vert
        \partial_x^2 v_2 (t)
    \Vert_{L^2(\Real)}
    \leq C \Vert
        g
    \Vert_{L^2(I)},
\end{align*}
and conclude that
\begin{align} \label{eq_estimate_for_v}
    \vert \lambda \vert \Vert
        v (t)
    \Vert_{L^2(\Real)}
    + \vert \lambda \vert^{1/2} \Vert
        \partial_x v (t)
    \Vert_{L^2(\Real)}
    + \Vert
        \partial_x^2 v (t)
    \Vert_{L^2(\Real)}
    \leq C \Vert
        g
    \Vert_{L^2(I)}
\end{align}
for some constant $C>0$ and all $0 \leq t \leq T$.
Since the $v_1$ is independent of $t$.
The integrands in (\ref{eq_formula_of_v2}) are independent of $t$ and the coefficients are $C^{1, \alpha}[0,T]$ by the assumption for $\theta$.
Moreover, the derivative of $1/M(\lambda, \theta)$ and $B_2(e^{\pm \lambda x}; \theta)$ is $\alpha$-H\"{o}lder continuous by the formula (\ref{eq_formula_for_v2}) and (\ref{eq_formula_for_constants_of_v2}).
In the same way as (\ref{eq_estimate_for_v}), we deduce that
\begin{align} \label{eq_estimate_for_partial_t_v}
    \begin{split}
        & \vert \lambda \vert \Vert
            \partial_t v (t)
        \Vert_{L^2(\Real)}
        + \vert \lambda \vert^{1/2} \Vert
            \partial_x \partial_t v (t)
        \Vert_{L^2(\Real)}
        + \Vert
            \partial_x^2 \partial_t v (t)
        \Vert_{L^2(\Real)}\\
        & \leq C \Vert
            g
        \Vert_{L^2(I)}
    \end{split}
\end{align}
for some constant $C>0$ and $C$ is uniform on $\lambda$ and $t$.
Therefore, Lemma \ref{lem_abstract_theorem_for_evolution_operator_by_Kato_Tanabe} implies the existence of the solution and the estimate (\ref{eq_estimate_for_u_in_lem_existence_of_linearized_equation}).
\end{proof}

%\subsection{Higher order regularity in linear cases}
\subsection{$L^2$-estimates}
We establish a priori estimates.
Throughout this section we assume that $\nu = 1$ and $\sigma \in C^{1, \alpha}[0,T]$.
We consider the equations
\begin{equation} \label{eq_linear_heat}
    \begin{split}
        \begin{aligned}
            &\partial_t v - \partial_x^2 v
            = f,
            & t \in (0, T),
            & \, x \in I,\\
            &B_1(v; F(\sigma))
            = 0,
            & t \in (0, T),
            & \,\\
            &B_2(v; F(\sigma))
            =0,
            & t \in (0, T),
            & \,\\
            &\frac{d\sigma}{dt}
            = g
            + a \gamma_+ v,
            & t \in (0, T),
            & \,\\
            &v
            =v_0,
            \quad \sigma
            = \sigma_0
            & t=0,
            &\,x \in I
        \end{aligned}
    \end{split}
\end{equation}
for some positive initial data $v_0 \in H^1(I)$ and $\sigma_0 \in \Real_+$ and some non-negative functions $f \in L^2_tL^2_x(Q_T)$, $g \in L^\infty(0,T) \hookrightarrow L^2(0,T)$, $a \in L^\infty(0,T)$, and $F: \Real_+ \rightarrow (0, 1) \in BC^\infty(\Real_+)$ satisfying
\begin{align*}
    \lim_{s\rightarrow 0}F(s)
    =1, \quad
    \lim_{s \rightarrow \infty} F(s)
    = 0.
\end{align*}
We first establish the $H^1$- a priori estimates
\begin{gather*}
    \int_0^t
        \Vert
            \partial_t v(s)
        \Vert_{L^2(I)}^2
    ds
    + \Vert
        \partial_x v(t)
    \Vert_{L^2(I)}^2
    + \int_0^t
        \Vert
            \partial_x^2 v(s)
        \Vert_{L^2(I)}^2
    ds
    < \infty
\end{gather*}
for $0 < t < T$.
\begin{proposition} \label{prop_L2_estimate_for_linear_heat_eq}
    It holds that
    \begin{align} \label{eq_L2_estimate_for_linear_heat_eq}
        \Vert
            v(t)
        \Vert_{L^2(I)}^2
        + \int_0^t
            \Vert
                \partial_x v(s)
            \Vert_{L^2(I)}^2
        ds
        \leq Ce^{t} \left[
            \Vert
                v_0
            \Vert_{L^2(I)}^2
            + \int_0^t
                \Vert
                    f(s)
                \Vert_{L^2(I)}^2
            ds
        \right]
    \end{align}
    for some constant $C>0$.
\end{proposition}
\begin{proof}
    By integration by parts to the first equation of (\ref{eq_linear_heat}), we have
    \begin{align} \label{eq_differential_inequality_L2_v}
        \begin{split}
            \frac{\partial_t}{2} \int_I
                v(x, t)^2
            dx
            + \int_I
                |\partial_x v(x, t)|^2
            dx
            & \leq \left \vert
                \int_I
                    f(x, t) v(x, t)
                dx
            \right \vert\\
            & \leq \frac{1}{2} \Vert
                f(t)
            \Vert_{L^2(I)}^2
            + \frac{1}{2} \Vert
                v(t)
            \Vert_{L^2(I)}^2.
        \end{split}
    \end{align}
    The Gronwall inequality yields
    \begin{align*}
        \Vert
            v(t)
        \Vert_{L^2(I)}^2
        & \leq e^{t} \Vert
            v_0
        \Vert_{L^2(I)}^2
        + \int_0^t
            e^{t - s}
            \Vert
                f(s)
            \Vert_{L^2(I)}^2
        ds\\
        & \leq e^{t} \left[
            \Vert
                v_0
            \Vert_{L^2(I)}^2
            + \int_0^t
                \Vert
                    f(s)
                \Vert_{L^2(I)}^2
            ds
        \right]
    \end{align*}
    for some constant $C>0$.
    Integrating (\ref{eq_differential_inequality_L2_v}) over $(0,t)$ and using the above inequality, we have (\ref{eq_L2_estimate_for_linear_heat_eq}).
\end{proof}

\begin{proposition} \label{prop_H1_estimate_for_sigma}
    It holds that
    \begin{align*}
        & \vert
            \sigma(t)
        \vert
        \leq \sigma_0
        + t^{1/2} \Vert
            g
        \Vert_{L^2(0, T)}
        + C t^{1/2} \Vert
            a
        \Vert_{L^\infty(0, T)}
        \Vert
            u
        \Vert_{L^2_t H^1_x(Q_T)},\\
        & \left \Vert    
            \frac{d\sigma}{dt}
        \right \Vert_{L^2(0,T)}
        \leq \Vert
            g
        \Vert_{L^2(0,T)}
        + C \Vert
            a
        \Vert_{L^\infty(0,T)}
        \Vert
            u
        \Vert_{L^2_tH^1_x(Q_T)},
    \end{align*}
    for some constant $C>0$.
\end{proposition}
\begin{proof}
    Since $\sigma(t) = \sigma_0 + \int_0^t f(s) + a(s) \gamma_+ u(s) ds$, we deduce from the Schwarz inequality and the trace theorem that
    \begin{align*}
        \vert
            \sigma(t)
        \vert
        & \leq \sigma_0
        + \left \vert
            \int_0^t
                f(s)
                + a(s) \gamma_+ u(s)
            ds
        \right \vert\\
        & \leq \sigma_0
        + t^{1/2} \Vert
            g
        \Vert_{L^2(0, T)}
        + C t^{1/2} \Vert
            a
        \Vert_{L^\infty(0, T)}
        \Vert
            u
        \Vert_{L^2_t H^1_x(Q_T)}
    \end{align*}
    for some constant $C>0$.
    Similarly, we deduce the second estimate by the fourth equation of (\ref{eq_linear_heat}) and the trace theorem.
\end{proof}

\begin{proposition} \label{prop_L_infty_H1_estimate_for_v}
    It holds that
    \begin{align} \label{eq_diff_eq_L_infty_H1_estimate_for_v}
        \begin{split}
            & \Vert
                \partial_t v(t)
            \Vert_{L^2(I)}^2
            + \partial_t \Vert
                \partial_x v(t)
            \Vert_{L^2(I)}^2
            + \Vert
                \partial_x^2 v(t)
            \Vert_{L^2(I)}^2\\
            & \leq C \left[
                1
                + \left \vert
                    \frac{dF(\sigma(t))}{dt}
                \right \vert^2
            \right]
            \Vert
                \partial_x v(t)
            \Vert_{L^2(I)}^2\\
            & + C \left[
                1
                + \left \vert
                    \frac{dF(\sigma(t))}{dt}
                \right \vert^2
            \right]
            \Vert
                v(t)
            \Vert_{L^2(I)}^2
            + C \Vert
                f(t)
            \Vert_{L^2(I)}^2.
        \end{split}
    \end{align}
    for some $C>0$, which is independent of $F$ and $T$.
    Moreover,
    \begin{align} \label{eq_L_infty_H1_estimate_for_v}
        \begin{split}
            & \int_0^t
                \Vert
                    \partial_s v(s)
                \Vert_{L^2(I)}^2
            ds
            + \Vert
                \partial_x v(t)
            \Vert_{L^2(I)}^2
            + \int_0^t
                \Vert
                    \partial_x^2 v(s)
                \Vert_{L^2(I)}^2
            ds\\
            & \leq C^\prime e^{
                \int_0^t
                    C \left(
                        1
                        + \vert
                            \sigma^\prime(r)
                        \vert^2
                    \right)
                dr
            }\\
            & \times \left(
                \Vert
                        \partial_x v_0
                    \Vert_{L^2(I)}
                + \sup_{0<t<T} \Vert
                        v(s)
                    \Vert_{L^2(I)}^2
                \int_0^t
                    \left(
                        1
                        + \vert
                            \sigma^\prime(s)
                        \vert^2
                    \right)
                ds
                + \int_0^t
                    \Vert
                        f(s)
                    \Vert_{L^2(I)}^2
                ds
            \right)
        \end{split}
    \end{align}
    for some constants $C, C^\prime>0$ and $0 \leq t \leq T$.
\end{proposition}
\begin{remark}
    We do not use the fourth inequality in (\ref{eq_linear_heat}) to obtain the estimates (\ref{eq_diff_eq_L_infty_H1_estimate_for_v}) and (\ref{eq_L_infty_H1_estimate_for_v}).
\end{remark}
\begin{proof}
    Multiplying $\partial_t v$ to (\ref{eq_linear_heat}) from both sides and integrating by parts, we have the formula
    \begin{align*}
        \int_I
            |\partial_t v(x, t)|^2
        dx
        +\int_I
            \partial_x^2 v(x, t) \partial_t v(x, t)
        dx
        = \int_I
            f(x, t) \partial_t v(x, t)
        dx.
    \end{align*}
    Applying $\partial_t$ to $B_1(F; v)=0$, we have
    \begin{align*}
        (1 - F(\sigma(t)) )\gamma_+ \partial_t v
        - \gamma_- \partial_t v
        - \frac{d}{dt}[F(\sigma(t))] \gamma_+ v
        = 0.
    \end{align*}
    Since
    \begin{align*}
        &\int_I
            \partial_x^2 v(x, t) \partial_t v(x, t)
        dx\\
        & = (\gamma_+ \partial_x v \gamma_+ \partial_t v
        - \gamma_- \partial_x v \gamma_- \partial_t v)
        + \int_I
            \partial_x v(x, t) \partial_x \partial_t v(x, t) 
        dx \\
        & = \left[
            (1 - F(\sigma(t))) \gamma_- \partial_x v \gamma_+ \partial_t v
            - \gamma_- \partial_x v \left(
                (1 - F(\sigma(t)) )\gamma_+ \partial_t v
                - \frac{d}{dt}[F(\sigma(t))] \gamma_+ v
            \right)
        \right]\\
        & + \frac{\partial_t}{2}\int_I
            |\partial_x v(x, t)|^2
        dx \\
        & = \frac{d}{dt}[F(\sigma(t))] \gamma_- \partial_x v \gamma_+ v
        + \frac{\partial_t}{2}\int_I
            |\partial_x v(x, t)|^2
        dx,
    \end{align*}
    we deduce from the trace theorem that
    \begin{align}\label{eq_L2_estimate_for_dtv}
        \begin{split}
            & \int_I
                |\partial_t v(x, t)|^2
            dx
            + \frac{\partial_t}{2} \int_I
                |\partial_x v(x, t)|^2
            dx\\
            & \leq \left \vert
                \frac{d}{dt}[F(\sigma(t))] \gamma_- \partial_x v \gamma_+ v
            \right \vert
            + \frac{1}{2} \int_I
                |\partial_t v(x, t)|^2
            dx
            + \frac{1}{2} \int_I
                |f(x, t)|^2
            dx\\
            & \leq C \left \vert
                \frac{
                    d F(\sigma(t))
                }{
                    dt
                }
            \right \vert
            \Vert
                v (t)
            \Vert_{H^2(I)}
            \Vert
                v (t)
            \Vert_{H^1(I)}\\
            & + \frac{1}{2} \int_I
                |\partial_t v(x, t)|^2
            dx
            + \frac{1}{2} \int_I
                |f(x, t)|^2
            dx.
        \end{split}
    \end{align}
%    \begin{align} \label{eq_L2_estimate_for_dtv}
%        \begin{split}
%            & \int_I
%                |\partial_t v(x, t)|^2
%            dx
%            + \frac{\partial_t}{2} \int_I
%                |\partial_x v(x, t)|^2
%            dx \\
%            & \leq \left \vert
%                \frac{d}{dt}[F(\sigma(t))] \gamma_- \partial_x v \gamma_+ v
%            \right \vert
%            + \frac{1}{2} \int_I
%                |\partial_t v(x, t)|^2
%            dx
%            + \frac{1}{2} \int_I
%                |f(x, t)|^2
%            dx\\
%            & \leq \Vert
%                F^\prime
%            \Vert_{L^\infty(\Real_+)}
%            \left \vert
%                \sigma^\prime(t)
%            \right \vert
%            \Vert
%                v (t)
%            \Vert_{H^2(I)}
%            \Vert
%                v (t)
%            \Vert_{H^1(I)}\\
%            & + \frac{1}{2} \int_I
%                |\partial_t v(x, t)|^2
%            dx
%            + \frac{1}{2} \int_I
%                |f(x, t)|^2
%            dx.
%        \end{split}
%    \end{align}
    From the inequalities
    \begin{align*}
        x \leq 1 + x^m,
        \quad x\geq 0, \, m \geq 0,
    \end{align*}
    and
    \begin{align}\label{eq_bound_for_partial_x2_from_heat_eq}
        \Vert \partial_x^2 v \Vert_{L^2(I)}^2
        \leq 2 \Vert \partial_t v \Vert_{L^2(I)}^2 + 2 \Vert f \Vert_{L^2(I)}^2,
    \end{align}
    the Schwarz inequality, and the Young inequality, we observe that
    \begin{align*}
        & \left \vert
            \frac{
                d F(\sigma(t))
            }{
                dt
            }
        \right \vert
        \vert \Vert
            v (t)
        \Vert_{H^2(I)}
        \Vert
            v (t)
        \Vert_{H^1(I)}\\
        & = \left \vert
            \frac{
                d F(\sigma(t))
            }{
                dt
            }
        \right \vert
        \vert (
            \Vert
                v (t)
            \Vert_{H^1(I)}
            + \Vert
                \partial_x^2 v (t)
            \Vert_{L^2(I)}
        ) \Vert
            v (t)
        \Vert_{H^1(I)}\\
        & \leq C \left(
            1
            + \left \vert
                \frac{
                    d F(\sigma(t))
                }{
                    dt
                }
            \right \vert^2
        \right) \Vert
            v (t)
        \Vert_{H^1(I)}^2
        + C \left(
            1
            + \left \vert
                \frac{
                    d F(\sigma(t))
                }{
                    dt
                }
            \right \vert^2
        \right) \Vert
            v (t)
        \Vert_{L^2(I)}^2\\
        & + \frac{1}{2} \Vert
            \partial_t v (t)
        \Vert_{L^2(I)}^2
        + \frac{1}{2} \Vert
            f (t)
        \Vert_{L^2(I)}^2.
    \end{align*}
    Inserting this estimate to (\ref{eq_L2_estimate_for_dtv}) and using (\ref{eq_bound_for_partial_x2_from_heat_eq}) again, we have the first differential inequality (\ref{eq_diff_eq_L_infty_H1_estimate_for_v}).
    Since
    \begin{align*}
        \left \vert
            \frac{
                d F(\sigma(t))
            }{
                dt
            }
        \right \vert
        \leq \Vert
            F^\prime
        \Vert_{L^\infty(\Real_+)}
        \vert
            \sigma^\prime(t)
        \vert
    \end{align*}
    We conclude from (\ref{eq_bound_for_partial_x2_from_heat_eq}) and Gronwall's inequality that
    \begin{align*}
        %\begin{split}
            & \int_0^t
                \Vert
                    \partial_s v(s)
                \Vert_{L^2(I)}^2
            ds
            + \Vert
                \partial_x v(t)
            \Vert_{L^2(I)}^2
            + \int_0^t
                \Vert
                    \partial_x^2 v(s)
                \Vert_{L^2(I)}^2
            ds\\
            & \leq C^\prime e^{
                \int_0^t
                    C \left(
                        1
                        + \vert
                            \sigma^\prime(r)
                        \vert^2
                    \right)
                dr
            } \Vert
                \partial_x v_0
            \Vert_{L^2(I)}^2 \\
            & + C^\prime \int_0^t
                e^{\int_s^t
                    C \left(
                        1
                        + \vert
                            \sigma^\prime(r)
                        \vert^2
                    \right)
                dr} \left(
                    1
                    + \vert
                        \sigma^\prime(t)
                    \vert^2
                \right)
                \Vert
                    v(s)
                \Vert_{L^2(I)}^2
            ds\\
            & + C^\prime \int_0^t
                e^{\int_s^t
                    C \left(
                        1
                        + \vert
                            \sigma^\prime(r)
                        \vert^2
                    \right)
                dr} \Vert
                    f(s)
                \Vert_{L^2(I)}^2
            ds\\
            & \leq C^\prime e^{
                \int_0^t
                    C \left(
                        1
                        + \vert
                            \sigma^\prime(r)
                        \vert^2
                    \right)
                dr
            }\\
            & \times \left(
                \Vert
                    \partial_x v_0
                \Vert_{L^2(I)}^2
                + \sup_{0<t<T} \Vert
                    v(s)
                \Vert_{L^2(I)}^2
                \int_0^t
                    \left(
                        1
                        + \vert
                            \sigma^\prime(s)
                        \vert^2
                    \right)
                ds
                + \int_0^t
                    \Vert
                        f(s)
                    \Vert_{L^2(I)}^2
                ds
            \right)\\
            & < \infty,
        %\end{split}
    \end{align*}
    for some constants $C, C^\prime>0$ and $0 \leq t \leq T$.    
\end{proof}

\begin{proposition} \label{prop_estimate_for_dsigma_L_infty}
    For the solution $\sigma$ to (\ref{eq_linear_heat}), there exists a constant $C>0$ such that
    \begin{align} \label{eq_estimate_for_dsigma_L_infty}
        \left \Vert
            \frac{d \sigma}{dt}
        \right \Vert_{L^\infty(0,T)}
        \leq \Vert
            g
        \Vert_{L^\infty(0,T)}
        + \Vert
            a
        \Vert_{L^\infty(0,T)}
        \Vert
            v
        \Vert_{L^\infty_t H^1_x(Q_T)}
        < \infty.
    \end{align}
\end{proposition}
\begin{proof}
    This is a direct consequent of the fourth equation of (\ref{eq_linear_heat}) and the trace theorem.
\end{proof}
\subsection{Higher order regularity}
In order to derive the higher order regularity, we begin by proving the extension theorem.
\begin{lemma}\label{lem_extension_theorem}
    Let $m \in \Integer_{> 0}$, $\alpha \in \Real$, and $u \in H^m(I)$.
    Then there exists a function $v \in H^m(I)$ such that
    \begin{align} \label{eq_boundary_condition_and_estimate_of_extension}
        \begin{split}
            \alpha \gamma_+ v
            - \gamma_- v
            = \gamma_+ u,
            & \quad
            \gamma_+ \partial_x v
            - \alpha \gamma_- \partial_x v
            = - \gamma_- \partial_x u,\\
            \Vert
                v
            \Vert_{H^m(I)}
            & \leq C \Vert
                u
            \Vert_{H^m(I)}
        \end{split}
    \end{align}
    for some $\alpha$-independent constant $C = C(m,I)>0$.
\end{lemma}
\begin{proof}
    We construct the extension $v$.
    We assume that $v$ is such that
    \begin{align*}
        v(x)
        &= a_+ u(x)
        + b_+ u(-x),
        \quad
        x \in (3/4, 1),\\
        v(x)
        &= a_- u(x)
        + b_- u(-x),
        \quad
        x \in (-1, -3/4)
    \end{align*}
    for some constant $a_\pm, b_\pm$.
    We determine $a_\pm, b_\pm$.
    Since
    \begin{align*}
        & \partial_x v (x)
        = a_\pm \partial_x u(x)
        - b_\pm \partial_x u(-x)
    \end{align*}
    around $x = \pm 1$, we see that
    \begin{align*}
        & \gamma_+ v
        = a_+ \gamma_+ u 
        + b_+ \gamma_- u,\\ 
        & \gamma_- v
        = a_- \gamma_- u 
        + b_- \gamma_+ u,\\ 
        & \gamma_+ \partial_x v
        = a_+ \gamma_+ \partial_x u 
        - b_+ \gamma_- \partial_x u,\\ 
        & \gamma_- \partial_x v
        = a_- \gamma_- \partial_x u 
        - b_- \gamma_+ \partial_x u.
    \end{align*}
    Therefore,
    \begin{align*}
        \alpha \gamma_+ v
        - \gamma_- v
        & = \alpha (
            a_+ \gamma_+ u
            + b_+ \gamma_- u
        )
        - (
            a_- \gamma_- u
            + b_- \gamma_+ u
        )\\
        & = (
            \alpha a_+
            - b_-
        ) \gamma_+ u
        + (
            \alpha b_+
            - a_-
        ) \gamma_- u,\\
        \gamma_+ \partial_x v
        - \alpha \gamma_- \partial_x v
        & = (
            a_+ \gamma_+ \partial_x u
            - b_+ \gamma_- \partial_x u
        )
        - \alpha (
            a_- \gamma_- \partial_x u
            - b_- \gamma_+ \partial_x u
        )\\
        & = (
            a_+
            + \alpha b_-
        ) \gamma_+ \partial_x u
        + (
            - b_+
            - \alpha a_-
        ) \gamma_- \partial_x u.
    \end{align*}
    When $v$ satisfies the boundary conditions in (\ref{eq_boundary_condition_and_estimate_of_extension}), $a_\pm$ and $b_\pm$ must satisfy
    \begin{align*}
        \left[
            \begin{array}{cccc}
                1       &0       &0      &\alpha\\
                0       &1       &\alpha &0\\
                0       &-\alpha &1      &0\\
                -\alpha &0       &0      &1
            \end{array}
        \right]
        \left[
            \begin{array}{c}
                a_+\\
                b_+\\
                a_-\\
                b_-
            \end{array}
        \right]
        = \left[
            \begin{array}{c}
                0\\
                1\\
                0\\
                -1
            \end{array}
        \right].
    \end{align*}
    We solve this equation to obtain
    \begin{align*}
        \left[
            \begin{array}{c}
                a_+\\
                b_+\\
                a_-\\
                b_-
            \end{array}
        \right]
        = \left[
            \begin{array}{c}
                \alpha/(1 + \alpha^2)\\
                1/(1 + \alpha^2)\\
                \alpha/(1 + \alpha^2)\\
                -1/(1 + \alpha^2)
            \end{array}
        \right].
    \end{align*}
    Let $\varphi_\pm \in [0, 1]$ be two smooth cut-off functions such that
    \begin{align*}
        & \varphi_+(x)
        = \left\{
            \begin{array}{ll}
                1,  & x \in [3/4, 1],\\
                0, & x \in [-1, 1/2],
            \end{array}
        \right.\\
        & \varphi_-(x)
        = \left\{
            \begin{array}{ll}
                1,  & x \in [-1, -3/4],\\
                0, & x \in [-1/2, 1].
            \end{array}
        \right.
    \end{align*}
    Therefore, it is reasonable to define the extension function $v$ as
    \begin{align} \label{eq_form_of_the_extension_operator_E}
        \begin{split}
            v(x)
            & = \varphi_+ (x) \left(
                \frac{
                    \alpha
                }{
                    1 + \alpha^2
                }u(x)
                + \frac{
                    1
                }{
                    1 + \alpha^2
                }u(-x)
            \right)\\
            & + \varphi_- (x) \left(
                \frac{
                    \alpha
                }{
                    1 + \alpha^2
                }u(x)
                + \frac{
                    -1
                }{
                    1 + \alpha^2
                }u(-x)
            \right).
        \end{split}
    \end{align}
    By the definition, $v$ satisfies
    \begin{align*}
        \Vert
            v
        \Vert_{H^m(I)}
        \leq C \Vert
            u
        \Vert_{H^m(I)}
    \end{align*}
    for some $\alpha$-independent constant $C = C(m, I)>0$.
\end{proof}
For $u \in H^m(I)$ and $\alpha \in \Real$, we denote by the extension operator $E$ by
\begin{align*}
    E(u; \alpha)
    := v,
\end{align*}
where $v$ is the extension function given by (\ref{eq_form_of_the_extension_operator_E}).
\begin{proposition} \label{prop_estimate_for_dtE}
    Let $T>0$, $m \in \Integer_{>0}$, and $\theta \in C^1(0,T)$.
    Then there exist a $\theta$-independent constant $C>0$ such that
    \begin{align*}
        & \Vert
            \partial_t E(u(\cdot, t); \theta(t))
        \Vert_{H^m(I)}\\
        & \leq C \Vert
            \partial_t u(\cdot, t)
        \Vert_{H^m(I)}
        + C \vert
            \theta^\prime(t)
        \vert
        \Vert
            u(\cdot, t)
        \Vert_{H^m(I)},
        \quad
        a.a. \, t \in (0, T),
    \end{align*}
    and 
    \begin{align*}
        & \Vert
            \partial_t^2 E(u(\cdot, t); \theta(t))
        \Vert_{H^m(I)}\\
        & \leq C \Vert
            \partial_t^2 u(\cdot, t)
        \Vert_{H^m(I)}
        + C \vert
            \theta^\prime(t)
        \vert
        \Vert
            \partial_t u(\cdot, t)
        \Vert_{H^m(I)}\\
        & + C \vert
            \theta^{\prime}(t)
        \vert^2
        \Vert
            u(\cdot, t)
        \Vert_{H^m(I)}
        + C \vert
            \theta^{\prime\prime}(t)
        \vert
        \Vert
            u(\cdot, t)
        \Vert_{H^m(I)},
        \quad
        a.a. \, t \in (0, T),
    \end{align*}
    for all $u \in H^1_t H^m_x(Q_T)$.
\end{proposition}
\begin{proof}
    Since
    \begin{align} \label{eq_dE}
        \begin{split}
            & \partial_t E(u(x, t); \theta(t))\\
            & = \varphi_+ (x) \partial_t \left(
                \frac{
                    \theta(t)
                }{
                    1 + \theta(t)^2
                }u(x, t)
                + \frac{
                    1
                }{
                    1 + \theta(t)^2
                }u(-x, t)
            \right)\\
            & + \varphi_- (x) \partial_t \left(
                \frac{
                    \theta(t)
                }{
                    1 + \theta(t)^2
                }u(x, t)
                + \frac{
                    -1
                }{
                    1 + \theta(t)^2
                }u(-x, t)
            \right)\\
            & = \varphi_+ (x) \left(
                \frac{
                    \theta(t)
                }{
                    1 + \theta(t)^2
                }\partial_t u(x, t)
                + \frac{
                    1
                }{
                    1 + \theta(t)^2
                }\partial_t u(-x, t)
            \right)\\
            & + \varphi_- (x) \left(
                \frac{
                    \theta(t)
                }{
                    1 + \theta(t)^2
                }\partial_t u(x, t)
                + \frac{
                    -1
                }{
                    1 + \theta(t)^2
                }\partial_t u(-x, t)
            \right)\\
            %%%%%%%%%%%%%%%%%%%%%%%%%%%%%%%%%
            & + \varphi_+ (x) \left(
                \frac{
                    \theta^\prime(t)
                }{
                    1 + \theta(t)^2
                    }
                - \frac{
                    2 \theta(t)^2 \theta^\prime(t)
                }{
                    (1 + \theta(t)^2)^2
                    }
                \right) u(x, t)\\
            & + \varphi_+ (x)
            \frac{
                - 2 \theta^\prime(t) \theta(t)
            }{
                (1 + \theta(t)^2)^2
            } u(-x, t)\\
            & + \varphi_- (x) \left(
                \frac{
                    \theta^\prime(t)
                }{
                    1 + \theta(t)^2
                    }
                - \frac{
                    2 \theta(t)^2 \theta^\prime(t)
                }{
                    (1 + \theta(t)^2)^2
                }
            \right) u(x, t)\\
            & + \varphi_- (x)
            \frac{
                2 \theta^\prime(t) \theta(t)
            }{
                1 + \theta(t)^2
            } u(-x, t),
        \end{split}
    \end{align}
    we have
    \begin{align*}
        & \Vert
            \partial_t E(u(\cdot, t); \theta(t))
        \Vert_{H^m(I)}\\
        & \leq C_1 \Vert
            \partial_t u(\cdot, t)
        \Vert_{H^m(I)}
        + C_2 \vert
            \theta^\prime(t)
        \vert
        \Vert
            u(\cdot, t)
        \Vert_{H^m(I)}
    \end{align*}
    for some $\theta$-independent constants $C_1, C_2>0$.
    The second assertion is derived by applying $\partial_t$ to (\ref{eq_dE}) and the elementary calculations.
    Since the way is same as (\ref{eq_dE}), we omit the detail here.
\end{proof}
We next observe the higher order regularity.
We assume that
\begin{align*}
    & v_0
    \in H^3(I), \\
    & f
    \in H^1_tL^2_x(Q_T)
    \cap L^2_tH^2_x(Q_T)
    \hookrightarrow C ([0,T); H^1(I)),\\
    & a, g
    \in C^1[0,T]
    \hookrightarrow H^1(0,T)
    \hookrightarrow L^\infty(0,T).
\end{align*}
Applying $\partial_t$ to the first equation of (\ref{eq_linear_heat}) from the both sides, we obtain the equation
\begin{equation} \label{eq_linear_heat_time_derivative}
    \begin{split}
        \begin{aligned}[t]
            \partial_t \tilde{v} - \partial_x^2 \tilde{v}
            & = \tilde{f},
            & t \in (0, T),
            & \, x \in I,\\
            B_1(\tilde{v}; F(\sigma))
            & = \frac{
                dF(\sigma(t))
            }{
                dt
            } \gamma_+ v,
            & t \in (0, T),
            & \,\\
            B_2(\tilde{v}; F(\sigma))
            & = -  \frac{
                dF(\sigma(t))
            }{
                dt
            } \gamma_- \partial_x v,
            & t \in (0, T),
            & \,\\
            \tilde{v}
            &=\tilde{v}_0,
            & t=0,
            &\,x \in I
        \end{aligned}
    \end{split}
\end{equation}
where
\begin{align*}
    & \tilde{v}
    := \partial_t v,\,
    \tilde{\sigma}
    := \frac{d\sigma}{dt}, \,
    \tilde{f}
    := \partial_t f,\\
    & \tilde{v}_0
    := \partial_x^2 v_0 + f|_{t=0}.
\end{align*}
Note that
\begin{align*}
    \lim_{t \rightarrow 0}\tilde{v}(t)
    = \lim_{t \rightarrow 0}\partial_t v(t)
    = \lim_{t \rightarrow 0}\partial_x^2 v(t)
    + \lim_{t \rightarrow 0}f(t)
    = \partial_x^2 v_0
    + f|_{t=0}
    \in H^1(I).
\end{align*}
We consider the generalized equation from (\ref{eq_linear_heat_time_derivative}) such that
\begin{equation} \label{eq_linear_heat_time_derivative_generalized}
    \begin{split}
        \begin{aligned}[t]
            \partial_t \tilde{v} - \partial_x^2 \tilde{v}
            & = \tilde{f},
            & t \in (0, T),
            & \,x \in I,\\
            B_1(\tilde{v}; F(\sigma))
            & = \gamma_+ G w
            & t \in (0, T),
            & \,\\
            B_2(\tilde{v}; F(\sigma))
            & = - \gamma_- \partial_x G w,
            & t \in (0, T),
            & \,\\
            v
            &=v_0,
            & t=0,
            &\,x \in I,
        \end{aligned}
    \end{split}
\end{equation}
for smooth functions $G: [0, T) \rightarrow \Real$ and $w : Q_T \rightarrow \Real$.
When $G=dF/dt$ and $v=w$, the equation (\ref{eq_linear_heat_time_derivative_generalized}) is equivalent to (\ref{eq_linear_heat_time_derivative}).
We set
\begin{align} \label{eq_def_Hsigma_tildev1}
    \begin{split}
        H(\sigma)
        & : = 1 - F(\sigma),\\
        \tilde{v}_1
        & : = E(G w; H(\sigma)),
    \end{split}
\end{align}
where $E$ is the extension operator given by Lemma \ref{lem_extension_theorem}.
For the solution $\tilde{v}$ to (\ref{eq_linear_heat_time_derivative_generalized}), we set
\begin{align} \label{eq_definition_v2}
    \tilde{v}_2
    : = \tilde{v} - \tilde{v}_1.
\end{align}
Then $\tilde{v}_2$ satisfies the heat equation under the homogeneous boundary condition such that
\begin{equation} \label{eq_linear_heat_time_derivative_tildev2}
    \begin{split}
        \begin{aligned}[t]
            \partial_t \tilde{v}_2 - \partial_x^2 \tilde{v}_2
            & = \tilde{f}
            - (\partial_t \tilde{v}_1 - \partial_x^2 \tilde{v}_1),
            & t \in (0, T),
            & \, x \in I,\\
            B_1(\tilde{v}_2; F(\sigma))
            & = 0,
            & t \in (0, T),
            & \,\\
            B_2(\tilde{v}_2; F(\sigma))
            & = 0,
            & t \in (0, T).
            & \,
        \end{aligned}
    \end{split}
\end{equation}
Differentiating the fourth equation (\ref{eq_linear_heat}) from both sides, we observe that
\begin{align} \label{eq_d_tilde_sigma}
    \frac{d \tilde{\sigma}}{dt}
    = g^\prime
    + a^\prime \gamma_+ v
    + a \gamma_+ \tilde{v}.
\end{align}
\begin{proposition}\label{prop_L2_a_priori_estimate_for_tildev1}
    Let $\sigma \in H^1(0,T)$ and $G\in W^{1,\infty}(0,T)$.
    Let $w \in L^\infty_t L^2_x(Q_T) \cap L^2_t H^1(Q_T)$.
    For $\tilde{v}_1$ defined in (\ref{eq_def_Hsigma_tildev1}), there exists a constant $C>0$ such that
    \begin{align} \label{eq_estimate_for_tildev1_by_w}
        \begin{split}
            & \Vert
                \tilde{v}_1(t)
            \Vert_{L^2(I)}^2
            + \int_0^t
                \Vert
                    \partial_x \tilde{v}_1(s)
                \Vert_{L^2(I)}^2
            ds\\
            & \leq C \sup_{0<t<T}
            \left \vert
                G(\sigma(t))
            \right \vert^2
            \left[
                \Vert
                    w(t)
                \Vert_{L^2(I)}^2
                + \int_0^t
                    \Vert
                        w(s)
                    \Vert_{H^1(I)}^2
                ds
            \right].
        \end{split}
    \end{align}
    Moreover, the higher order estimates
    \begin{align} \label{eq_L2L2_estimate_for_dtw_dx2w}
        \begin{split}
            & \Vert
                \tilde{v}_1(t)
            \Vert_{H^1(I)}^2
            \leq C \vert
                G(\sigma(t))
            \vert^2
            \Vert
                w(t)
            \Vert_{H^1(I)}^2,\\
            & \Vert
                \partial_x^2 \tilde{v}_1(t)
            \Vert_{L^2(I)}^2
            + \Vert
                \partial_t \tilde{v}_1(t)
            \Vert_{L^2(I)}^2\\
            & \leq 
            C \vert
                G(\sigma(t))
            \vert^2
            (
                \Vert
                    w(t)
                \Vert_{H^2(I)}^2
                + \Vert
                    \partial_t w(t)
                \Vert_{L^2(I)}^2
            )\\
            & + C \left[
                \left \vert
                    \frac{
                        dG(\sigma(t))
                    }{
                        dt
                    }
                \right \vert^2
                + \left \vert
                    \frac{
                        d H(\sigma(t))
                    }{
                        dt
                    }
                \right \vert^2
                \vert
                    G(\sigma(t))
                \vert^2
            \right]
            \Vert
                w(t)
            \Vert_{L^2(I)}^2
        \end{split}
    \end{align}
    hold for $a.a.$ $t>0$.
\end{proposition}
\begin{proof}
    By Lemma \ref{lem_extension_theorem}, we have
    \begin{align}\label{eq_Hm_estimate_for_w}
        \begin{split}
            & \Vert
                \tilde{v}_1(t)
            \Vert_{H^m(I)}
            \leq C
            \vert
                G(\sigma(t))
            \vert
            \Vert
                w(t)
            \Vert_{H^m(I)}
        \end{split}
    \end{align}
    for some constant $C = C(F^\prime, I)>0$ and $m \in \Integer_{\geq 0}$.
    The integration over $(0, T)$ implies (\ref{eq_estimate_for_tildev1_by_w}).
    We deduce from Proposition \ref{prop_estimate_for_dtE} that
    \begin{align}\label{eq_L2L2_estimate_for_dw}
        \begin{split}
            & \Vert
                \partial_t \tilde{v}_1 (t)
            \Vert_{L^2(I)}\\
            & \leq C \Vert
                \partial_t (
                    G(\sigma(t))w(t)
                )
            \Vert_{L^2(I)}
            + C \left \vert
                \frac{
                    dH(\sigma(t))
                }{
                    dt
                }
            \right \vert
            \vert
                G(\sigma(t))
            \vert
            \Vert
                v(t)
            \Vert_{L^2(I)}\\
            & \leq C
            \vert
                    G(\sigma(t))
                \vert^2
            \Vert
                \partial_t w(t)
            \Vert_{L^2(I)}
            + C
            \left \vert
                \frac{
                    dG(\sigma(t))
                }{
                    dt
                }
            \right \vert
            \Vert
                w(t)
            \Vert_{L^2(I)}\\
            & + C \left \vert
                \frac{
                    d H(\sigma(t))
                }{
                    dt
                }
            \right \vert
            \vert
                G(\sigma(t))
            \vert
            \Vert
                w(t)
            \Vert_{L^2(I)}.
        \end{split}
    \end{align}
    The estimates (\ref{eq_Hm_estimate_for_w}) and (\ref{eq_L2L2_estimate_for_dw}) imply (\ref{eq_L2L2_estimate_for_dtw_dx2w}).
\end{proof}

\begin{corollary} \label{cor_differential_ineq_L2_v2}
    Let $\tilde{f} \in L^2_t L^2_x(Q_T)$.
    Under the same assumption as Proposition \ref{prop_L2_a_priori_estimate_for_tildev1} for $w$, $\sigma$, and $G$, the solution $\tilde{v}_2$ to (\ref{eq_linear_heat_time_derivative_tildev2}) satisfies
    \begin{align}\label{eq_differential_ineq_L2_v2}
        \begin{split}
            & \frac{\partial_t}{2} \Vert
                \tilde{v}_2(t)
            \Vert_{L^2(I)}^2
            + \Vert
                \partial_x \tilde{v}_2(t)
            \Vert_{L^2(I)}^2\\
            & \leq \frac{1}{2}
            \Vert
                \tilde{v}_2(t)
            \Vert_{L^2(I)}^2
            + C \Vert
                \tilde{f}(t)
            \Vert_{L^2(I)}^2\\
            & + C \vert
                G(\sigma(t))
            \vert^2
            \left[
                \Vert
                    w(t)
                \Vert_{H^2(I)}^2
                + \Vert
                    \partial_t w(t)
                \Vert_{L^2(I)}^2
            \right]\\
            & + C \left[
                \left \vert
                    \frac{
                        dG(\sigma(t))
                    }{
                        dt
                    }
                \right \vert^2
                + \left \vert
                    \frac{
                        d H(\sigma(t))
                    }{
                        dt
                    }
                \right \vert^2
                \vert
                    G(\sigma(t))
                \vert^2
            \right]
            \Vert
                w(t)
            \Vert_{L^2(I)}^2.
        \end{split}
    \end{align}
    for some constant $C>0$.
\end{corollary}
\begin{proof}
    Similar to the proof of Proposition \ref{prop_L2_estimate_for_linear_heat_eq}, the differential inequality (\ref{eq_differential_ineq_L2_v2}) is derived through integration by parts and Proposition \ref{prop_L2_a_priori_estimate_for_tildev1}.
\end{proof}
\begin{proposition} \label{prop_L2_a_priori_estimate_for_tildev2}
    %Let
    %\begin{align*}
    %    \tilde{\sigma}
    %    & = \frac{d\sigma}{dt} \in H^1(0,T) \hookrightarrow L^\infty(0,T),\\
    %    \tilde{f}
    %    & = \partial_t f \in L^2_t L^2_x(Q_T).
    %\end{align*}
    Assume that
    \begin{align*}
        G(\sigma(t))
        = \frac{
            dF(\sigma(t))
        }{
            dt
        },\quad
        w
        = v. \quad
        %H(\sigma(t))
        %= 1 - F(\sigma(t))
    \end{align*}
    Then the solution $\tilde{v}_2$ to (\ref{eq_linear_heat_time_derivative_tildev2}) satisfies
    \begin{align}\label{eq_L2_a_priori_estimate_for_tildev2}
        \begin{split}
            & \Vert
                \tilde{v}_2(t)
            \Vert_{L^2(I)}^2
            + \int_0^t
                \Vert
                    \partial_x \tilde{v}_2(s)
                \Vert_{L^2(I)}^2
            ds\\
            %& \leq C \Vert
            %    \tilde{\sigma}
            %\Vert_{L^\infty(0,T)}^2
            %\left[
            %    \Vert
            %        v(t)
            %    \Vert_{L^2(I)}^2
            %    + \int_0^t
            %        \Vert
            %            \partial_x v(s)
            %        \Vert_{L^2(I)}^2
            %    ds
            %\right]\\
            & \leq C e^t \left[
                \Vert
                    \tilde{v}_0
                \Vert_{L^2(I)}^2
                + \int_0^t
                    \Vert
                        \tilde{f}(s)
                    \Vert_{L^2(I)}^2
                ds
            \right.\\
            & \left.
                \quad\quad
                + \Vert
                    \tilde{\sigma}
                \Vert_{L^\infty(0,T)}^2
                \int_0^t
                    \Vert
                        v(s)
                    \Vert_{H^2(I)}^2
                    + \Vert
                        \partial_t v(s)
                    \Vert_{L^2(I)}^2
                ds
            \right.\\
            & \left.
                \quad\quad
                + 
                \int_0^t
                    \vert
                        \tilde{\sigma}(t)
                    \vert^4
                    \Vert
                        v(s)
                    \Vert_{L^2(I)}^2
                ds
                + \sup_{0<t<T}\Vert
                    v(t)
                \Vert_{L^2(I)}^2
                \int_0^t
                    \left \vert
                        \frac{
                            d\tilde{\sigma}(s)
                        }{
                            ds
                        }
                    \right \vert^2
                ds
            \right]
        \end{split}
    \end{align}
    for some constant $C>0$.
    Moreover, if $\tilde{\sigma}$ is the solution to (\ref{eq_d_tilde_sigma}), then
    \begin{align*}
        \begin{split}
            & \Vert
                \tilde{v}_2(t)
            \Vert_{L^2(I)}^2
            + \int_0^t
                \Vert
                    \partial_x \tilde{v}_2(s)
                \Vert_{L^2(I)}^2
            ds\\
            %& \leq C \Vert
            %    \tilde{\sigma}
            %\Vert_{L^\infty(0,T)}^2
            %\left[
            %    \Vert
            %        v(t)
            %    \Vert_{L^2(I)}^2
            %    + \int_0^t
            %        \Vert
            %            \partial_x v(s)
            %        \Vert_{L^2(I)}^2
            %    ds
            %\right]\\
            %& + C \left(
            & \leq C \left(
                1
                + e^{
                    \int_0^t
                        \Phi_2(s)
                    ds
                }
                + e^{
                    \int_0^t
                        \Phi_2(s)
                    ds
                }
                \int_0^t
                    \Phi_2(s)
                ds
            \right)\\
            & \quad\quad\quad
            \times \left(
                \Vert
                    \tilde{v}_2(0)
                \Vert_{L^2(I)}^2
                + \int_0^t
                    \Phi_1(s)
                ds
            \right),
        \end{split}
    \end{align*}
    where $\Phi_1, \Phi_2 \in L^1(0,T)$ are define as
    \begin{align*}
        \Phi_1(t)
        & = C \left[
            \Vert
                \tilde{f}(t)
            \Vert_{L^2(I)}^2
            + \Vert
                \tilde{\sigma}
            \Vert_{L^\infty(0,T)}^2 \left(
                \Vert
                    \partial_t v(t)
                \Vert_{L^2(I)}^2
                + \Vert
                    v(t)
                \Vert_{H^2(I)}^2
            \right)
        \right.\\
        &\left.
            \quad\quad\quad
            + \Vert
                v(t)
            \Vert_{L^2(I)}^2
            (
                \Vert
                    \tilde{\sigma}
                \Vert_{L^\infty(0,T)}^4
                + \vert
                    g^\prime(t)
                \vert^2
        \right .\\
        & \left .
                \quad\quad\quad
                + \vert
                    a^\prime(t)
                \vert^2
                \Vert
                    v(t)
                \Vert_{H^1(I)}^2
                + \Vert
                    a
                \Vert_{L^\infty(0,T)}^2
                \Vert
                    \tilde{\sigma}
                \Vert_{L^\infty(0,T)}^2
                \Vert
                    v(t)
                \Vert_{H^1(I)}^2
            )
        \right],\\
        \Phi_2(t)
        & = C \left[
            1
            + \left(
                \Vert
                    a
                \Vert_{L^\infty(0,T)}^2
                + \Vert
                    a
                \Vert_{L^\infty(0,T)}^8
            \right)
            \Vert
                v(t)
            \Vert_{L^2(I)}^2
        \right],
    \end{align*}
    for some constant $C>0$.
\end{proposition}
\begin{proof}
    We estimate the right-hand side of (\ref{eq_differential_ineq_L2_v2}).
    Since
    \begin{align*}
        - \frac{
            dH(\sigma(t))
        }{
            dt
        }, \, 
        G(\sigma(t))
        & = F^\prime(\sigma(t)) \sigma^\prime(t)
        = F^\prime(\sigma(t)) \tilde{\sigma}(t),\\
        \frac{
            dG(\sigma(t))
        }{
            dt
        }
        & = F^{\prime\prime}(\sigma(t))
        \vert
            \tilde{\sigma}(t)
        \vert^2 
        + F^\prime(\sigma(t)) \tilde{\sigma}^\prime(t),
    \end{align*}
    we observe that
    \begin{align} \label{eq_estimates_dH_G_dG}
        \begin{split}
            & \left\vert
                \frac{
                    dH(\sigma(t))
                }{
                    dt
                }
            \right\vert^2,
            \vert
                G(\sigma(t))
            \vert
            \leq \Vert
                F^\prime
            \Vert_{L^\infty(\Real_+)}^2
            \vert
                \tilde{\sigma}(t)
            \vert^2,\\
            & \left \vert
                \frac{
                    dG(\sigma(t))
                }{
                    dt
                }
            \right \vert^2
            \leq \Vert
                F^{\prime\prime}
            \Vert_{L^\infty(\Real_+)}^2
            \vert
                \tilde{\sigma}(t)
            \vert^4
            + \Vert
                F^{\prime}
            \Vert_{L^\infty(\Real_+)}^2
            \left \vert
                \frac{
                    d\tilde{\sigma}(t)
                }{
                    dt
                }
            \right \vert^2.
        \end{split}
    \end{align}
    The definition (\ref{eq_definition_v2}) of $\tilde{v}_2$ , Proposition \ref{prop_L2_a_priori_estimate_for_tildev1}, Corollary \ref{cor_differential_ineq_L2_v2}, and the Gronwall inequality imply (\ref{eq_L2_a_priori_estimate_for_tildev2}).
    We next estimate the term $\vert d \tilde{\sigma}/dt \vert^2$ in (\ref{eq_estimates_dH_G_dG}).
    We invoke the interpolation inequality that
    \begin{align}\label{eq_interpolation_inequality_for_trace_operator}
        \begin{split}
            \vert
                \gamma_\pm \varphi(t)
            \vert
            & \leq C \Vert
                \varphi(t)
            \Vert_{H^{1/2+\delta}(I)}\\
            & \leq C \Vert
                \varphi(t)
            \Vert_{L^2(I)}^{1/2 - \delta}
            \Vert
                \varphi(t)
            \Vert_{H^1(I)}^{1/2 + \delta}\\
            & \leq C \left(
                \Vert
                    \varphi(t)
                \Vert_{L^2(I)}
                + \Vert
                    \varphi(t)
                \Vert_{L^2(I)}^{1/2 - \delta}
                \Vert
                    \partial_x \varphi(t)
                \Vert_{L^2(I)}^{1/2 + \delta}
            \right)
        \end{split}
    \end{align}
    for small $\delta>0$ and all $\varphi \in H^1(I)$.
    We find from the triangle inequality and this interpolation inequality that
    \begin{align} \label{eq_estimate_dsigma}
        \begin{split}
            & \left \vert
                \frac{d\tilde{\sigma}(t)}{dt}
            \right \vert^2\\
            & \leq C \vert
                g^\prime(t)
            \vert^2
            + C \vert
                a^\prime(t)
            \vert^2
            \vert
                \gamma_+ v(t)
            \vert^2
            + C \Vert
                a
            \Vert_{L^\infty(0,T)}^2
            \vert
                \gamma_+ \tilde{v}(t)
            \vert^2\\
            & \leq C \vert
                g^\prime(t)
            \vert^2
            + C \vert
                a^\prime(t)
            \vert^2
            \vert
                \gamma_+ v(t)
            \vert^2\\
            & + C \Vert
                a
            \Vert_{L^\infty(0,T)}^2
            \vert
                \gamma_+ \tilde{v}_1(t)
            \vert^2
            + C \Vert
                a
            \Vert_{L^\infty(0,T)}^2
            \vert
                \gamma_+ \tilde{v}_2(t)
            \vert^2\\
            & \leq C \vert
                g^\prime(t)
            \vert^2
            + C \vert
                a^\prime(t)
            \vert^2
            \Vert
                v(t)
            \Vert_{H^1(I)}^2
            + C \Vert
                a
            \Vert_{L^\infty(0,T)}^2
            \Vert
                \tilde{v}_1(t)
            \Vert_{H^1(I)}^2\\
            & + C \Vert
                a
            \Vert_{L^\infty(0,T)}^2
            \left(
                \Vert
                    \tilde{v}_2(t)
                \Vert_{L^2(I)}^2
                + \Vert
                    \tilde{v}_2(t)
                \Vert_{L^2(I)}^{1 - 2 \delta}
                \Vert
                    \partial_x \tilde{v}_2(t)
                \Vert_{L^2(I)}^{1 + 2\delta}
            \right)\\
            & \leq C \vert
                g^\prime(t)
            \vert^2
            + C \vert
                a^\prime(t)
            \vert^2
            \Vert
                v(t)
            \Vert_{H^1(I)}^2\\
            & + C \Vert
                a
            \Vert_{L^\infty(0,T)}^2
            \Vert
                \tilde{\sigma}
            \Vert_{L^\infty(0,T)}^2
            \Vert
                \tilde{v}(t)
            \Vert_{H^1(I)}^2\\
            & + \left(
                C \Vert
                    a
                \Vert_{L^\infty(0,T)}^2
                + C_\varepsilon \Vert
                    a
                \Vert_{L^\infty(0,T)}^{4/(1 - 2\delta)}
            \right)
            \Vert
                \tilde{v}_2(t)
            \Vert_{L^2(I)}^2
            + \varepsilon
            \Vert
                \partial_x \tilde{v}_2(t)
            \Vert_{L^2(I)}^2
        \end{split}
    \end{align}
    for small $\varepsilon > 0$.
    Therefore, taking $\delta=1/4,\, \varepsilon=1/2$, multiplying $\tilde{v}_2$ from both sides of (\ref{eq_estimate_dsigma}), integrating over $I$, and using Corollary \ref{cor_differential_ineq_L2_v2} and the estimates (\ref{eq_estimates_dH_G_dG}), we deduce that
    \begin{align*}
        & \partial_t\Vert
            \tilde{v}_2(t)
        \Vert_{L^2(I)}^2
        + \Vert
            \partial_x \tilde{v}_2(t)
        \Vert_{L^2(I)}^2\\
        & \leq C \left[
            \Vert
                \tilde{f}(t)
            \Vert_{L^2(I)}^2
            + \Vert
                \tilde{\sigma}
            \Vert_{L^\infty(0,T)}^2
            (
                \Vert
                    \partial_t v(t)
                \Vert_{L^2(I)}^2
                + \Vert
                    v(t)
                \Vert_{H^2(I)}^2
            )
        \right.\\
        &\left.
            \quad\quad\quad
            + \Vert
                v(t)
            \Vert_{L^2(I)}^2
            \left(
                \Vert
                    \tilde{\sigma}
                \Vert_{L^\infty(0,T)}^4
                + \vert
                    g^\prime(t)
                \vert^2
            \right.
        \right.\\
        & \left.
            \left.
                \quad\quad\quad
                + \vert
                    a^\prime(t)
                \vert^2
                \Vert
                    v(t)
                \Vert_{H^1(I)}^2
                + \Vert
                    a
                \Vert_{L^\infty(0,T)}^2
                \Vert
                    \tilde{\sigma}
                \Vert_{L^\infty(0,T)}^2
                \Vert
                    v(t)
                \Vert_{H^1(I)}^2
            \right)
        \right]\\
        & + C
        \left[
            1
            + (
                \Vert
                    a
                \Vert_{L^\infty(0,T)}^2
                + \Vert
                    a
                \Vert_{L^\infty(0,T)}^8
            )
            \Vert
                v(t)
            \Vert_{L^2(I)}^2
        \right]
        \Vert
            \tilde{v}_2(t)
        \Vert_{L^2(I)}^2\\
        &  + \frac{1}{2} \Vert
            \partial_x \tilde{v}_2(t)
        \Vert_{L^2(I)}^2\\
        & =: \Phi_1(t)
        + \Phi_2(t) \Vert
            \tilde{v}_2(t)
        \Vert_{L^2(I)}^2
        + \frac{1}{2} \Vert
            \partial_x \tilde{v}_2(t)
        \Vert_{L^2(I)}^2.
    \end{align*}
    We obtain by the Gronwall inequality that
    \begin{align*}
        \Vert
            \tilde{v}_2(t)
        \Vert_{L^2(I)}^2
        \leq e^{
            \int_0^t
                \Phi_2(s)
            ds
        }
        \left(
            \Vert
                \tilde{v}_2(0)
            \Vert_{L^2(I)}^2
            + \int_0^t
                \Phi_1(s)
            ds
        \right)
        \in L^\infty(0,T).
    \end{align*}
    Therefore, we conclude that
    \begin{align*}
        & \Vert
            \tilde{v}_2(t)
        \Vert_{L^2(I)}^2
        + \int_0^t
            \Vert
                \partial_x \tilde{v}_2(s)
            \Vert_{L^2(I)}^2
        ds\\
        & \leq (
            1
            + e^{
                \int_0^t
                    \Phi_2(s)
                ds
            }
        )
        \left(
            \Vert
                \tilde{v}_2(0)
            \Vert_{L^2(I)}^2
            + \int_0^t
                \Phi_1(s)
            ds
        \right)\\
        & + \sup_{0<s<t} \Vert
            \tilde{v}_2(s)
        \Vert_{L^2(I)}
        \int_0^t
            \Phi_2(s)
        ds\\
        & \leq \left(
            1
            + e^{
                \int_0^t
                    \Phi_2(s)
                ds
            }
            + e^{
                \int_0^t
                    \Phi_2(s)
                ds
            }
            \int_0^t
                \Phi_2(s)
            ds
        \right)
        \left(
            \Vert
                \tilde{v}_2(0)
            \Vert_{L^2(I)}^2
            + \int_0^t
                \Phi_1(s)
            ds
        \right). 
    \end{align*}
\end{proof}
\begin{corollary}\label{cor_L2_a_priori_estimate_for_tildev}
    %Assume that
    %\begin{align*}
    %    G(t)
    %    = \frac{
    %        dF(\sigma(t))
    %    }{
    %        dt
    %    },\quad
    %    w
    %    = v. \quad
    %    %H(\sigma(t))
    %    %= 1 - F(\sigma(t))
    %\end{align*}
    Under the same assumption for $G, w$ as in Proposition \ref{prop_L2_a_priori_estimate_for_tildev2}, the solution $\tilde{v}$ to (\ref{eq_linear_heat_time_derivative_generalized}) satisfies
    \begin{align}\label{eq_L2_a_priori_estimate_for_tildev}
        \begin{split}
            & \Vert
                \tilde{v}(t)
            \Vert_{L^2(I)}^2
            + \int_0^t
                \Vert
                    \partial_x \tilde{v}(s)
                \Vert_{L^2(I)}^2
            ds\\
            & \leq C \Vert
                \tilde{\sigma}
            \Vert_{L^\infty(0,T)}^2
            \left[
                \Vert
                    v(t)
                \Vert_{L^2(I)}^2
                + \int_0^t
                    \Vert
                        \partial_x v(s)
                    \Vert_{L^2(I)}^2
                ds
            \right]\\
            & + C e^t \left[
                \Vert
                    \tilde{v}_0
                \Vert_{L^2(I)}^2
                + \int_0^t
                    \Vert
                        \tilde{f}(s)
                    \Vert_{L^2(I)}^2
                ds
            \right.\\
            & \left.
                \quad\quad
                + \Vert
                    \tilde{\sigma}
                \Vert_{L^\infty(0,T)}^2
                \int_0^t
                    \Vert
                        v(s)
                    \Vert_{H^2(I)}^2
                    + \Vert
                        \partial_t v(s)
                    \Vert_{L^2(I)}^2
                ds
            \right.\\
            & \left.
                \quad\quad
                + 
                \int_0^t
                    \vert
                        \tilde{\sigma}(t)
                    \vert^4
                    \Vert
                        v(s)
                    \Vert_{L^2(I)}^2
                ds
                + \sup_{0<t<T}\Vert
                    v(t)
                \Vert_{L^2(I)}^2
                \int_0^t
                    \left \vert
                        \frac{
                            d\tilde{\sigma}(s)
                        }{
                            ds
                        }
                    \right \vert^2
                ds
            \right]\\
            & < \infty
        \end{split}
    \end{align}
    for all $\tilde{\sigma} \in H^1(0,T) \hookrightarrow L^\infty(0,T)$.
    Moreover, if $\tilde{\sigma}$ is the solution to (\ref{eq_d_tilde_sigma}), then
    \begin{align*}
        \begin{split}
            & \Vert
                \tilde{v}(t)
            \Vert_{L^2(I)}^2
            + \int_0^t
                \Vert
                    \partial_x \tilde{v}(s)
                \Vert_{L^2(I)}^2
            ds\\
            & \leq C \Vert
                \tilde{\sigma}
            \Vert_{L^\infty(0,T)}^2
            \left[
                \Vert
                    v(t)
                \Vert_{L^2(I)}^2
                + \int_0^t
                    \Vert
                        \partial_x v(s)
                    \Vert_{L^2(I)}^2
                ds
            \right]\\
            & + C \left(
                1
                + e^{
                    \int_0^t
                        \Phi_2(s)
                    ds
                }
                + e^{
                    \int_0^t
                        \Phi_2(s)
                    ds
                }
                \int_0^t
                    \Phi_2(s)
                ds
            \right)
            \left(
                \Vert
                    \tilde{v}_2(0)
                \Vert_{L^2(I)}^2
                + \int_0^t
                    \Phi_1(s)
                ds
            \right)\\
            & < \infty,
        \end{split}
    \end{align*}
    where $\Phi_1, \Phi_2 \in L^2(0,T)$ are the same function as in Proposition \ref{prop_L2_a_priori_estimate_for_tildev2}.
\end{corollary}
\begin{proof}
    This directly follows from Propositions \ref{prop_L2_a_priori_estimate_for_tildev1} and \ref{prop_L2_a_priori_estimate_for_tildev2}.
\end{proof}

\begin{proposition}\label{prop_H1_estimates_for_tilde_sigma}
    Let $g \in H^1(0,T)$, $a \in H^1(0,T) \hookrightarrow L^\infty(0,T)$, $v \in L^\infty_t H^1_x(Q_T)$, $\tilde{v} = \partial_t v \in L^2_t H^1_x(Q_T)$ for $T>0$.
    Then the solution $\tilde{\sigma}$ to (\ref{eq_d_tilde_sigma}) satisfies
    \begin{align*}
        \left \Vert
            \frac{
                d \tilde{\sigma}
            }{
                dt
            }
        \right \Vert_{L^2(0,T)}^2
        & \leq C \Vert
            g^\prime
        \Vert_{L^2(0,T)}^2
        + C \Vert
            a^\prime
        \Vert_{L^2(0,T)}^2
        \sup_{0<t<T} \Vert
            v(t)
        \Vert_{H^1(I)}^2\\
        & + C  \Vert
            a
        \Vert_{L^\infty(0,T)}^2
        \int_0^T
            \Vert
                \tilde{v}(s)
            \Vert_{H^1(I)}^2
        ds
    \end{align*}
    for some constant $C>0$.
\end{proposition}
\begin{proof}
    Taking the $L^2(0,T)$-norm of both sides of (\ref{eq_d_tilde_sigma}), we obtain the estimate.
\end{proof}
%Now we have found that $d^2\sigma/dt^2 = d\tilde{\sigma}/dt$ makes sense in $L^2(0,T)$.
%Therefore, we obtain
\begin{proposition} \label{prop_H1_estimate_for_tilde_v2}
    \begin{enumerate}[label=(\roman*)]
        \item 
    For the solution $\tilde{v}_2$ to (\ref{eq_linear_heat_time_derivative_generalized}), it follows that
    \begin{align}%\label{eq_differential_ineq_L2_v2}
        \begin{split}
            & \frac{\partial_t}{2} \Vert
                \partial_x \tilde{v}_2(t)
            \Vert_{L^2(I)}^2
            + \Vert
                \partial_t \tilde{v}_2(t)
            \Vert_{L^2(I)}^2
            + \Vert
                \partial_x^2 \tilde{v}_2(t)
            \Vert_{L^2(I)}^2\\
            & \leq C \left[
                1
                + \left \vert
                    \frac{
                        dF(\sigma(t))
                    }{
                        dt
                    }
                \right \vert^2
            \right]
            \Vert
                \partial_x \tilde{v}_2(t)
            \Vert_{L^2(I)}^2\\
            & + C \left[
                1
                + \left \vert
                    \frac{
                        dF(\sigma(t))
                    }{
                        dt
                    }
                \right \vert^2
            \right]
            \Vert
                \tilde{v}_2(t)
            \Vert_{L^2(I)}^2\\
            & + C \vert
                G(\sigma(t))
            \vert^2
            \left[
                \Vert
                    w(t)
                \Vert_{H^2(I)}^2
                + \Vert
                    \partial_t w(t)
                \Vert_{L^2(I)}^2
            \right]\\
            & + C \left[
                \left \vert
                    \frac{
                        dG(\sigma(t))
                    }{
                        dt
                    }
                \right \vert^2
                + \left \vert
                    \frac{
                        d H(\sigma(t))
                    }{
                        dt
                    }
                \right \vert^2
                \vert
                    G(\sigma(t))
                \vert^2
            \right]
            \Vert
                w(t)
            \Vert_{L^2(I)}^2\\
            & + C \Vert
                \tilde{f}(t)
            \Vert_{L^2(I)}^2,
        \end{split}
    \end{align}
    for some constant $C>0$.
    \item
    Assume that
    \begin{align*}
        G(\sigma(t))
        = \frac{
            dF(\sigma(t))
        }{
            dt
        },\quad
        w
        = v,
        %H(\sigma(t))
        %= 1 - F(\sigma(t))
    \end{align*}
    and set
    \begin{align*}
        %& \tilde{v}_1 
        %:= E(F^\prime(\sigma) \tilde{\sigma} v; 1 - F(\sigma))\\
        %& v_2
        %:= \tilde{v} - \tilde{v}_1 \\
        \tilde{f}_\ast
        := \tilde{f}
        - (\partial_t \tilde{v}_1  - \partial_x^2 \tilde{v}_1 )
        \in L^2_tL^2_x(Q_T).
    \end{align*}
    Then it follows that
    \begin{align} \label{eq_H1_a_priori_estimate_for_tildev2}
        \begin{split}
            & \int_0^t
                \Vert
                    \partial_s \tilde{v}_2(s)
                \Vert_{L^2(I)}^2
            ds
            + \Vert
                \partial_x \tilde{v}_2(t)
            \Vert_{L^2(I)}^2
            + \int_0^t
                \Vert
                    \partial_x^2 \tilde{v}_2(s)
                \Vert_{L^2(I)}^2
            ds\\
            & \leq e^{
                \int_0^t
                    C \left(
                        1
                        + \vert
                            \sigma^\prime(r)
                        \vert^2
                    \right)
                dr
            }
            \left[
                \Vert
                    \partial_x \tilde{v}_2(0)
                \Vert_{L^2(I)}
            \right.\\
            & \left.
                \quad\quad\quad
                + \sup_{0<t<T} \Vert
                        \tilde{v}_2(s)
                    \Vert_{L^2(I)}
                \int_0^t
                    \left(
                        1
                        + \vert
                            \sigma^\prime(s)
                        \vert^2
                    \right)
                ds
                + \int_0^t
                    \Vert
                        \tilde{f}_\ast(s)
                    \Vert_{L^2(I)}^2
                ds
            \right]\\
            & < \infty
        \end{split}
    \end{align}
    for $a.a.$ $t \in (0,T)$.
    \end{enumerate}
\end{proposition}

\begin{proof}
    Propositions \ref{prop_L_infty_H1_estimate_for_v} and \ref{prop_L2_a_priori_estimate_for_tildev1} yield the first differential inequality.
    Corollary \ref{cor_L2_a_priori_estimate_for_tildev} and Proposition \ref{prop_H1_estimates_for_tilde_sigma} imply that
    \begin{align*}
        \frac{
            d \tilde{\sigma}
        }{
            dt
        }
        \in L^2(0,T).
    \end{align*}
    Therefore, the estimates (\ref{eq_estimates_dH_G_dG}) and Propositions \ref{prop_L2_a_priori_estimate_for_tildev1} and \ref{prop_H1_estimates_for_tilde_sigma} imply $f_\ast \in L^2_t L^2_x(Q_T)$.
    We deduce (\ref{eq_H1_a_priori_estimate_for_tildev2}) from Proposition \ref{prop_L_infty_H1_estimate_for_v}.
\end{proof}
We establish $H^3$-a priori estimate for the solution $v$ to (\ref{eq_linear_heat}).
\begin{corollary} \label{cor_H3_estimate_for_v}
    Let
    \begin{align*}
        & f
        \in H^1_t L^2_2(Q_T)
        \cap L^2_t H^2_x(Q_T)
        \cap C([0,T); H^1(I)),\\
        & v_0 \in H^3(I), \sigma_0 > 0,\\
        & a, g \in H^1(0,T) \hookrightarrow C[0,T].
    \end{align*}
    Let $(v, \sigma)$ be the solution to (\ref{eq_linear_heat}) and
    \begin{align*}
        \tilde{v}
        := \partial_t v,\,
        \tilde{f}
        := \partial_t f.
    \end{align*}
    Then $\sigma \in H^2(0,T)$ and
    \begin{align} \label{eq_H1_a_pripori_estimate_for_tilde_v_no_v1}
        \begin{split}
            & \int_0^t
                \Vert
                    \partial_s \tilde{v}(s)
                \Vert_{L^2(I)}^2
            ds
            + \Vert
                \partial_x \tilde{v}(t)
            \Vert_{L^2(I)}^2
            + \int_0^t
                \Vert
                    \partial_x^2 \tilde{v}(s)
                \Vert_{L^2(I)}^2
            ds\\
            & \leq C (
                1
                + e^{
                    \int_0^t
                        C \left(
                            1
                            + \vert
                                \sigma^\prime(r)
                            \vert^2
                        \right)
                    dr
                }
            )\\
            & \quad
            \times
            \left[
                \Vert
                    \sigma^\prime
                \Vert_{L^\infty(0,T)}^2
                \int_0^t
                    (
                        \Vert
                            v(s)
                        \Vert_{H^2(I)}^2
                        + \Vert
                            \partial_t v(s)
                        \Vert_{L^2(I)}^2
                    )
                ds
            \right.\\
            & \left.
                \quad\quad\quad
                +
                \int_0^t
                        \vert
                            \sigma^\prime(s)
                        \vert^4
                        + \vert
                            \sigma^{\prime\prime}(s)
                        \vert^2
                ds
                \sup_{0<t<T} \Vert
                    v(t)
                \Vert_{L^2(I)}^2
            \right]\\
            & + C
            \Vert
                \sigma^\prime
            \Vert_{L^\infty(0,T)}^2
            \Vert
                v(t)
            \Vert_{H^1(I)}^2\\
            & + C e^{
                \int_0^t
                    C \left(
                        1
                        + \vert
                            \sigma^\prime(r)
                        \vert^2
                    \right)
                dr
            }\\
            & \times
            \left[
                \Vert
                    v_0
                \Vert_{H^3(I)}
                + \sup_{0<t<T} \Vert
                        \tilde{v}(s)
                    \Vert_{L^2(I)}
                \int_0^t
                    \left(
                        1
                        + \vert
                            \sigma^\prime(s)
                        \vert^2
                    \right)
                ds
                + \int_0^t
                    \Vert
                        \tilde{f}(s)
                    \Vert_{L^2(I)}^2
                ds
            \right]\\
            & < \infty.
        \end{split}
    \end{align}
    Moreover, $v \in C([0,T]; H^3(I)) \cap L^2_tH^4_x(Q_T)$ such that
    \begin{gather}\label{eq_H3_estimate_for_v}
        \begin{split}
            \Vert
                \partial_x^3 v(t)
            \Vert_{L^2(I)}^2
            & \leq C(
                \Vert
                    \partial_x \tilde{v} (t)
                \Vert_{L^2(I)}^2
                + \Vert
                    \partial_x f(t)
                \Vert_{L^2(I)}^2
            )
            < \infty,\\
            \int_0^t
                \Vert
                    \partial_x^4 v(s)
                \Vert_{L^2(I)}^2
            ds
            & \leq C \int_0^t
                \Vert
                    \partial_x^2 \tilde{v} (s)
                \Vert_{L^2(I)}^2
                + \Vert
                    \partial_x^2 f(s)
                \Vert_{L^2(I)}^2
            ds
            < \infty,
        \end{split}
    \end{gather}
    for $a.a. \, t \geq 0$ and some constant $C>0$.
    Furthermore, if $a, g \in C^1[0,T]$, then $\sigma \in C^2[0,T]$.
\end{corollary}
\begin{proof}
    Set
    \begin{align*}
        & \tilde{v}_1 
        := E \left(
            \frac{
                dF(\sigma(t))
            }{
                dt
            } v; 1 - F(\sigma)
        \right)
        \in H^1_tL^2_x(Q_T)
        \cap L^2_t H^2_x(Q_T),\\
        %& v_2
        %:= \tilde{v} - \tilde{v}_1 \\
        & \tilde{f}_\ast
        := \tilde{f}
        - (\partial_t \tilde{v}_1  - \partial_x^2 \tilde{v}_1 )
        \in L^2_tL^2_x(Q_T).
    \end{align*}
    Proposition \ref{prop_H1_estimates_for_tilde_sigma} implies $\sigma \in H^2(0,T)$.
    Proposition \ref{prop_H1_estimate_for_tilde_v2}(ii) and the triangle inequality imply
    \begin{align} \label{eq_H1_a_pripori_estimate_for_tilde_v}
        \begin{split}
            & \int_0^t
                \Vert
                    \partial_s \tilde{v}(s)
                \Vert_{L^2(I)}^2
            ds
            + \Vert
                \partial_x \tilde{v}(t)
            \Vert_{L^2(I)}^2
            + \int_0^t
                \Vert
                    \partial_x^2 \tilde{v}(s)
                \Vert_{L^2(I)}^2
            ds\\
            & \leq C \left(
                \int_0^t
                    \Vert
                        \partial_s \tilde{v}_1 (s)
                    \Vert_{L^2(I)}^2
                ds
                + \Vert
                    \partial_x \tilde{v}_1 (t)
                \Vert_{L^2(I)}^2
                + \int_0^t
                    \Vert
                        \partial_x^2 \tilde{v}_1 (s)
                    \Vert_{L^2(I)}^2
                ds
            \right)\\
            & + C e^{
                \int_0^t
                    C \left(
                        1
                        + \vert
                            \sigma^\prime(r)
                        \vert^2
                    \right)
                dr
            }\\
            & \times \left(
                \Vert
                    v_0
                \Vert_{H^3(I)}
                + \sup_{0<t<T} \Vert
                    \tilde{v}(s)
                \Vert_{L^2(I)}
                \int_0^t
                    \left(
                        1
                        + \vert
                            \sigma^\prime(s)
                        \vert^2
                    \right)
                ds
                + \int_0^t
                    \Vert
                        \tilde{f}_\ast(s)
                    \Vert_{L^2(I)}^2
                ds
            \right)\\
            & =: I_1 + I_2.
        \end{split}
    \end{align}
    We find from Proposition \ref{prop_L2_a_priori_estimate_for_tildev1} and the estimates (\ref{eq_estimates_dH_G_dG}) that
    \begin{align*}
        \begin{split}
            I_1
            & \leq C \int_0^t
                \vert
                    G(\sigma(s))
                \vert^2
                (
                    \Vert
                        v(s)
                    \Vert_{H^2(I)}^2
                    + \Vert
                        \partial_t v(s)
                    \Vert_{L^2(I)}^2
                )\\
            & \quad\quad\quad
                + \left[
                    \left \vert
                        \frac{
                            dG(\sigma(s))
                        }{
                            dt
                        }
                    \right \vert^2
                    + \left \vert
                        \frac{
                            d H(\sigma(s))
                        }{
                            dt
                        }
                    \right \vert^2
                    \vert
                        G(\sigma(s))
                    \vert^2
                \right]
                \Vert
                    v(s)
                \Vert_{L^2(I)}^2
            ds\\
            & + C \vert
                G(\sigma(t))
            \vert^2
            \Vert
                v( t)
            \Vert_{H^1(I)}^2\\
            & \leq C \int_0^t
                \vert
                    \sigma^\prime(s)
                \vert^2
                (
                    \Vert
                        v(s)
                    \Vert_{H^2(I)}^2
                    + \Vert
                        \partial_t v(s)
                    \Vert_{L^2(I)}^2
                )\\
            & \quad\quad\quad
                + (
                    \vert
                        \sigma^\prime(s)
                    \vert^4
                    + \vert
                        \sigma^{\prime\prime}(s)
                    \vert^2
                )
                \Vert
                    v(s)
                \Vert_{L^2(I)}^2
            ds\\
            & + C \vert
                \sigma^\prime(t)
            \vert^2
            \Vert
                v(t)
            \Vert_{H^1(I)}^2
        \end{split}
    \end{align*}
    for some constant $C>0$.
    Similarly, we deduce that
    \begin{align*}
        \begin{split}
            & \int_0^t
                \Vert
                    \tilde{f}_\ast(s)
                \Vert_{L^2(I)}^2
            ds\\
            & \leq C \int_0^t
                \Vert
                    \partial_t f(s)
                \Vert_{L^2(I)}^2
            ds\\
            & + C \int_0^t
                \vert
                    \sigma^\prime(t)
                \vert^2
                (
                    \Vert
                        v(s)
                    \Vert_{H^2(I)}^2
                    + \Vert
                        \partial_t v(s)
                    \Vert_{L^2(I)}^2
                )
            ds\\
            &+ C \int_0^t
                (
                    \vert
                        \sigma^\prime(s)
                    \vert^4
                    + \vert
                        \sigma^{\prime\prime}(s)
                    \vert^2
                )
                \Vert
                    v(s)
                \Vert_{L^2(I)}^2
            ds
        \end{split}
    \end{align*}
    for some constant $C>0$.
    The third terms in the right-hand side makes sense because of the embedding $H^2(0,T) \hookrightarrow H^{1,4}(0,T)$.
    Then we obtain (\ref{eq_H1_a_pripori_estimate_for_tilde_v_no_v1}).
    Since $\partial_t v - \partial_x^2 v = f$, we have
    \begin{align*}
        - \partial_x^2 v
        & = -\partial_t v
        + f\\
        & = - \tilde{v}
        + f
        \in L^\infty_t H^1_x(Q_T),\\
        - \partial_x^4 v
        & = - \partial_x^2 \partial_t v
        + \partial_x^2 f\\
        & = - \partial_x^2 \tilde{v}
        + \partial_x^2 f
        \in L^2_tL^2_x(Q_T).
    \end{align*}
    Therefore, we find that
    \begin{align*}
        v
        \in C([0,T); H^3(I))
        \cap L^2_tH^4_x(Q_T)
    \end{align*}
    and $v$ satisfies (\ref{eq_H3_estimate_for_v}).
    We deduce that
    \begin{align*}
        \left \Vert
            \frac{
                d^2\sigma
            }{
                dt
            }
        \right\Vert_{L^\infty(0,T)}
        & \leq \Vert
            g^\prime
        \Vert_{L^\infty(0,T)}
        + C \Vert
            a^\prime
        \Vert_{L^\infty(0,T)}
        \Vert
            v
        \Vert_{L^\infty_t H^1(Q_T)}\\
        & + C \Vert
            a
        \Vert_{L^\infty(0,T)}
        \Vert
            \partial_t v
        \Vert_{L^\infty_t H^1(Q_T)}\\
        & < \infty.
    \end{align*}
    We have proved Corollary \ref{cor_H3_estimate_for_v}.
\end{proof}

Similarly, by repeating the above procedures, we obtain
\begin{lemma} \label{lemma_H2m_a_priori_estimates}
    Let $T>0$, $m \in \Integer_{\geq 0}$ and
    \begin{align}\label{eq_conditions_v0_f_g_a_in_higher_regularity_theorem}
        \begin{split}
            & v_0
            \in H^{2m+1}(I), \\
            & f
            \in H^m_tL^2_x(Q_T)
            \cap L^2_tH^{2m}_x(Q_T)
            \hookrightarrow C ([0,T); H^{2m+1}(I)),\\
            & g
            \in H^m(0,T),\\
            & a
            \in H^m(0,T).
        \end{split}
    \end{align}
    Then the solution $v, \sigma$ to (\ref{eq_linear_heat}) satisfy
    \begin{align*}
        & v
        \in H^{m+1}_tL^2_x(Q_T)
        \cap BC([0,T);H^{2m+1}(I))
        \cap L^2_tH^{2m+2}_x(Q_T),\\
        & \sigma
        \in H^{m+1}(0,T),
    \end{align*}
    such that
    \begin{align} \label{eq_H2m1_estimate_for_v}
        \begin{split}
            & \Vert
                v(t)
            \Vert_{H^{2m+1}(I)}^2
            + \int_0^t
                \sum_{0\leq j \leq m+1}
                    \Vert
                        \partial_s^{j} \partial_x^{2(m+1-j)} v(s)
                    \Vert_{L^2(I)}^2
            ds
            %& + \int_0^t
            %    \sum_{0\leq j \leq m}
            %        \Vert
            %            \partial_x^{2j+2}v(s)
            %        \Vert_{L^2(I)}^2
            %ds
            \leq C_1(t),\\
            & \sum_{0 \leq j \leq m+1}
                \left \Vert
                    \frac{
                        d^{j+1} \sigma
                    }{
                        dt^{j+1}
                    }
                \right \Vert_{L^2(0,T)}
            \leq C_2(t),
        \end{split}
    \end{align}
    for $t>0$ and some bounded functions $C_1, C_2$ in $[0,T]$ which depend continuously on the norms of the functions in (\ref{eq_conditions_v0_f_g_a_in_higher_regularity_theorem}) to which they belong.
    Furthermore, if $a, g \in C^m[0,T]$
    \begin{align*}
        \sum_{0 \leq j \leq m+1}
            \left \Vert
                \frac{
                    d^{j} \sigma
                }{
                    dt^{j}
                }
            \right \Vert_{L^\infty(0,T)}
            < \infty.
    \end{align*}
\end{lemma}
\begin{proof}
    We prove the lemma by induction.
    Let $k \in \Integer_{\geq 1}$.
    Assume that (\ref{eq_H2m1_estimate_for_v}) holds for $m=k$ and (\ref{eq_conditions_v0_f_g_a_in_higher_regularity_theorem}) holds for $m=k+1$.
    Applying $\partial_t^{k+1}$ to the boundary conditions, we observe that
    \begin{align*}
        & B_1(\partial_t^{k+1} v; F(\sigma(t)))
        = \gamma_+ \left(
            \frac{
                d^{k+1} F(\sigma(t))
            }{
                dt^{k+1}
            }
        \right) v
        + \gamma_+ \left(
            \sum_{j=1}^k
                \frac{d^{k-j}}{dt^{k-j}} F(\sigma(t))
                \partial_t^j v
        \right),\\
        & B_2(\partial_t^{k+1} v; F(\sigma(t)))
        = \gamma_- \partial_x \left(
            \frac{
                d^{k+1} F(\sigma(t))
            }{
                dt^{k+1}
            }
        \right) v
        + \gamma_- \partial_x \left(
            \sum_{j=1}^k
                \frac{d^{k-j}}{dt^{k-j}} F(\sigma(t))
                \partial_t^j v
        \right).
    \end{align*}
    We find from the assumptions and Lemma \ref{lem_extension_theorem} that
    \begin{align*}
        E \left(
            \frac{d^{k-j}}{dt^{k-j}} F(\sigma(t)) \partial_t^j v
            ; 1 - F(\sigma(t))
        \right)
        \in H^1_tL^2_x(Q_T)
        \cap L^2_tH^2_x(Q_T),
    \end{align*}
    then
    \begin{align*}
        f_{k+1}
        & := \partial_t^{k+1} f - (
            \partial_t
            - \partial_x^2
        ) E \left(
            \sum_{j=1}^k
                \frac{d^{k-j}}{dt^{k-j}} F(\sigma(t))
                \partial_t^j v;
                1 - F(\sigma(t))
        \right)\\
        & \in L^2_t L^2_x(Q_T).
    \end{align*}
    Therefore, by repeating the procedure of proving the $H^2$-a priori estimate for the solution to (\ref{eq_linear_heat}), we deduce that the bounds in Proposition \ref{prop_L2_estimate_for_linear_heat_eq} and Corollary \ref{cor_H3_estimate_for_v} hold for $\partial_t^{k+1} v$.
    Furthermore, from the equation for $\sigma$ and the trace theorem, we also deduce that the $L^2$-estimate holds for $d^{k+2}\sigma/dt^{k+2}$, and if $a, g \in C^{k+1}(0,T)$, the $L^\infty$-estimate holds for $d^{k+2}\sigma/dt^{k+2}$.
\end{proof}

%%%%%%%%%%%%%%%%%%%%%%%%%%%%%%%%%%%%%%%%%%%%%%%%%%%%%%%%%%
%%%%%%%%%%%%%%%%%%%%%%%%%%%%%%%%%%%%%%%%%%%%%%%%%%%%%%%%%%
%------------------------Section3------------------------%
%%%%%%%%%%%%%%%%%%%%%%%%%%%%%%%%%%%%%%%%%%%%%%%%%%%%%%%%%%
%%%%%%%%%%%%%%%%%%%%%%%%%%%%%%%%%%%%%%%%%%%%%%%%%%%%%%%%%%
\section{Non-linear Problems} \label{sec_non_linear_problem}
To simplify the notation we assume that $\nu_1 = \nu_2 = 1$ without loss of generality.
We consider a slightly generalized equations than the original equations (\ref{eq_filter_clogging_nondimensional}) to simplify notation such that
\begin{equation} \label{eq_abstract_filter_clogging_equation}
    \begin{aligned}
        &\partial_t v_1
        - \partial_x^2 v_1
        + c(\sigma_1) \partial_x v_1
        = - N_{v,1}(v_1) v_2 + f
        & t > 0,
        & \, x \in I,\\
        &\partial_t v_2
        - \partial_x^2 v_2
        + c(\sigma_1) \partial_x v_2
        = N_{v,2}(v_1, v_2) v_2
        & t > 0,
        & \, x \in I,\\
        &B_1(v_1; F(\sigma_1))
        = 0, \quad 
        B_2(v_1; F(\sigma_1))
        =0,
        & t>0,
        & \,\\
        &B_1(v_2; F(\sigma_1))
        = 0, \quad
        B_2(v_2; F(\sigma_1))
        =0,
        & t>0,
        & \,\\
        &\frac{d\sigma_1}{dt}
        = - N_{\sigma, 1}(\sigma_1) \sigma_2
        + Q_1 \, d(\sigma_1) \, \gamma_+ v_1,
        & t>0,
        & \,\\
        & \frac{d\sigma_2}{dt}
        = N_{\sigma, 2}(\sigma_1, \sigma_2) \sigma_2
        + Q_2 \, d(\sigma_1) \, \gamma_+ v_2,
        & t>0,
        & \, \\
        & c
        = \Omega F(\sigma_1).
        & t>0,
        &\,\\
        & v_1
        = v_{1,0}>0, \quad
        v_2
        = v_{2,0}>0
        & t = 0,
        & \, x \in I,\\
        & \sigma_1
        = \sigma_{1,0}>0, \quad
        \sigma_2
        = \sigma_{2,0}>0
        & t = 0.
        & \,
    \end{aligned}
\end{equation}
We assume the following assumptions:

\underline{Assumption A}
\begin{itemize}
    \item The functions $d$, $F$, and $c$ are a monotonically decreasing positive $C^\infty$- function such that
    \begin{align*}
        d(+\infty)
        = F(+\infty)
        = c(+\infty)
        = 0.
    \end{align*}
    \item The functions $N_{v, 1}, N_{\sigma, 1} \in BC^\infty(\Real_+)$ is a non-negative monotonically increasing function such that $N_{v, 1}(0), N_{\sigma, 1}(0) = 0$.
    \item The functions
    \begin{align*}
        N_{v, 2}(v_1,v_2)
        = \alpha_1 N_{v, 1}(v_1) - \alpha_2 v_2
    \end{align*}
    for some $\alpha_1, \alpha_2 > 0$.
    $N_{\sigma, 2}(\sigma_1, \sigma_2) = \beta_1 N_{\sigma, 1}(\sigma_1) - \beta_2 \sigma_2$ for some $\beta_1, \beta_2 >0$.
    %\item The function $d \in BC^\infty(\Real_+)$ is non-negative monotonically decreasing such that $d(0) = 1$ and $d(+\infty)=0$.
\end{itemize}

We denote the function space for external forces and the solution by
\begin{align*}
    X_{f,T}
    & : = L^2_t H^2_x (Q_T)
    \cap H^1_t L^2_x (Q_T)
    \hookrightarrow BC(Q_T)\\
    X_{v,T}
    & : = L^2_t H^4_x (Q_T)
    \cap H^2_t L^2_x (Q_T)
    \cap C([0, T); H^3(I))
    \cap C^1([0, T); H^1(I)),\\
    %X_{\sigma, T}
    %& := H^2(0,T)
    X_{\sigma,T}
    & := \{
        \sigma \in L^\infty(0,T) \cap W^{2,\infty}_{loc}(0,T)
        :
        \Vert
            \sigma
        \Vert_{X_{\sigma, T}}
        < \infty
    \},
\end{align*}
and denote their norm by
\begin{align*}
    \Vert
        \cdot
    \Vert_{X_{v,T}}
    & =: \Vert
        \cdot
    \Vert_{L^2_t H^4_x (Q_T)}
    + \Vert
        \cdot
    \Vert_{H^2_t L^2_x (Q_T)}
    + \Vert
        \cdot
    \Vert_{L^\infty_t H^3_x(Q_T)}
    + \Vert
        \cdot
    \Vert_{L^\infty_t H^3_x(Q_T)}.%,\\
    %\Vert
    %    \cdot
    %\Vert_{X_{\sigma,T}}
    %& := \Vert
    %    \cdot
    %\Vert_{H^2(0,T)}.
\end{align*}
and
\begin{align*}
    \Vert
        \cdot
    \Vert_{X_{\sigma, T}}
    & := \sup_{0<t<T} \vert
        \sigma(t)
    \vert
    + \sup_{0<t<T} t^{\frac{1}{4} - \delta} \vert
        \sigma^\prime(t)
    \vert
    + \sup_{0<t<T} t^{\frac{1}{2} - \delta} \vert
        \sigma^{\prime\prime}(t)
    \vert
\end{align*}
for some small $\delta>0$.
For $T>0$, the external force $f$, and initial data $v_{0, j}$, $\sigma_{j,0}$, we consider the linear equations such that
\begin{equation} \label{eq_for_iteration_of_full_filter_clogging_equations}
    \begin{aligned}
        &\partial_t v_1
        - \partial_x^2 v_1
        = - c(\tau_1) \partial_x \psi_1
        - N_{v,1}(\psi_1) \psi_2
        + f
        & t>0,\,
        & x \in I,\\
        &\partial_t v_2
        - \partial_x^2 v_2
        = - c(\tau_1) \partial_x \psi_2
        + N_{v,2}(\psi_1, \psi_2) \psi_2
        & t>0,\,
        & x \in I,\\
        &B_1(v_1; F(\tau_1))
        = 0, \quad 
        B_2(v_1; F(\tau_1))
        =0,
        & t>0,
        & \,\\
        &B_1(v_2; F(\tau_1))
        = 0, \quad
        B_2(v_2; F(\tau_1))
        =0,
        & t>0,
        & \,\\
        &\frac{d\sigma_1}{dt}
        = - N_{\sigma, 1}(\tau_1) \tau_2
        + d(\tau_1) \, \gamma_+ \psi_1,
        & t>0,
        & \,\\
        & \frac{d\sigma_2}{dt}
        = N_{\sigma, 2}(\tau_1, \tau_2) \tau_2
        + d(\tau_1) \, \gamma_+ \psi_2,
        & t>0,
        & \, \\
        & c
        = \Omega F(\tau_1),
        & t>0,
        & \, \\
        & v_1
        =v_{1,0}, \quad
        v_2
        = v_{2,0},
        & t=0,
        & \, x \in I,\\
        & \sigma_1
        = \sigma_{0,1}, \quad
        \sigma_2
        = \sigma_{0,2},
        & t=0.
        & \,
    \end{aligned}
\end{equation}
The solutions to (\ref{eq_abstract_filter_clogging_equation}) are given by the fixed point to (\ref{eq_definition_v1_v2_sigma1_sigma2}).
%The pair of functions $(v_1, v_2, \sigma_1, \sigma_2)$ define by (\ref{eq_definition_v1_v2_sigma1_sigma2}) is the solution to the linear equations
We denote the solution operators to (\ref{eq_for_iteration_of_full_filter_clogging_equations}) by
\begin{align} \label{eq_definition_v1_v2_sigma1_sigma2}
    \begin{split}
        & v_1
        = \mathcal{S}_{v,1}(\psi_1, \psi_2, \tau_1),\\
        & v_2
        = \mathcal{S}_{v,2}(\psi_1, \psi_2, \tau_1),\\
        & \sigma_1
        = \mathcal{S}_{\sigma,1}(\psi_1, \tau_1, \tau_2),\\
        & \sigma_2
        = \mathcal{S}_{\sigma,2}(\psi_2, \tau_1, \tau_2).
    \end{split}
\end{align}
Although $\mathcal{S}_{v,j}$ and $\mathcal{S}_{\sigma, j}$ are depend on $T, f, v_{j,0}, \sigma_{j,0}$ and constants in the equations, we omit them for simplicity of notation.
We first construct the solution $(v_1, v_2)$ for given function $\sigma_1 \in X_{\sigma, T}$ by the Leray-Schauder principle, then we construct the local-in-time solution $(\sigma_1, \sigma_2)$ by Banach's fixed point theorem assuming that $v_j$ is a non-linear function of $\sigma_1$.
We extend the existence time of $\sigma_j$ up to $T$ by a priori estimates.
%For given functions
%\begin{align*}
%    (\psi_1, \psi_2, \tau_1, \tau_2)
%    \in X_{v,T}^2 \times X_{\sigma,T}^2.
%\end{align*}
We also consider the equation of the difference for fixed point arguments.
For $(\psi_3, \psi_4, \tau_3, \tau_4)$, we set
\begin{align*}
    & v_3
    = \mathcal{S}_{v,1}(\psi_3, \psi_4, \tau_3),\\
    & v_4
    = \mathcal{S}_{v,2}(\psi_3, \psi_4, \tau_3),\\
    & \sigma_3
    = \mathcal{S}_{\sigma,1}(\psi_3 \tau_3, \tau_4),\\
    & \sigma_4
    = \mathcal{S}_{\sigma,2}(\psi_4, \tau_3, \tau_4),
\end{align*}
We also set the differences by
\begin{align*}
    & V_1
    := v_1
    - v_3, \quad
    V_2
    := v_2
    - v_4,\\
    & \Sigma_1
    := \sigma_1
    - \sigma_3, \quad
    \Sigma_2
    := \sigma_2
    - \sigma_4.
\end{align*}
We see from the formulas
\begin{align*}
    &B_1(v_1; F(\tau_1))
    - B_1(v_3; F(\tau_3))\\
    & = \gamma_+ V
    - \gamma_- V
    - (
        F(\tau_1) \gamma_+ v_1
        - F(\tau_3) \gamma_+ v_3
    )\\
    & = \gamma_+ V
    - \gamma_- V
    - \frac{
        F(\tau_1)
        - F(\tau_3)
    }{
        2
    }\gamma_+(
        v_1
        + v_3
    )
    - \frac{
        F(\tau_1)
        + F(\tau_3)
    }{
        2
    } \gamma_+ (
        v_1
        - v_3
    )\\
    &= B_1 \left(
        V; \frac{
                F(\tau_1)
                + F(\tau_3)
            }{
                2
            }
    \right)
    - \frac{
        F(\tau_1)
        - F(\tau_3)
    }{
        2
    }\gamma_+(
        v_1
        + v_3
    )\\
    &= 0,
\end{align*}
and
\begin{align*}
    &B_2(v_1; F(\tau_1))
    - B_2(v_3; F(\tau_3))\\
    &= B_2 \left(
        V; \frac{
            F(\tau_1)
            + F(\tau_3)
        }{
            2
        }
    \right)
    + \frac{
        F(\tau_1)
        - F(\tau_2)
    }{
        2
    }\gamma_- \partial_x (
        v_1
        + v_3
    )\\
    &= 0,
\end{align*}
that the pair of the differences $(V_1, V_2, \Sigma_1, \Sigma_2)$ satisfies
\begin{equation}\label{eq_difference_equation_for_iteration_of_full_filter_clogging_equations}
    \begin{aligned}
        &\partial_t V_1
        - \partial_x^2 V_1
        = - \mathcal{N}_{V,0}(\psi_1, \psi_3, \tau_1, \tau_3)
        - \mathcal{N}_{V,1}(\psi_1, \psi_2, \psi_3, \psi_4),
        & t>0, \,
        & x \in I,\\
        &\partial_t V_2
        - \partial_x^2 V_2
        = - \mathcal{N}_{V,0}(\psi_2, \psi_4, \tau_1, \tau_3)
        + \mathcal{N}_{V,2}(\psi_1, \psi_2, \psi_3, \psi_4),
        & t>0, \,
        & x \in I,\\
        &B_1 \left(
            V_1;
            \frac{
                F(\tau_1)
                + F(\tau_3)
            }{
                2
            }
        \right)
        = \gamma_+ \frac{
            F(\tau_1)
            - F(\tau_3)
        }{
            2
        }
        (
            v_1
            + v_3
        ),
        & t>0,
        & \,\\
        & B_2\left(
            V_1;
            \frac{
                F(\tau_1)
                + F(\tau_3)
            }{
                2
            }
        \right)
        = - \gamma_- \partial_x \frac{
            F(\tau_1)
            - F(\tau_3)
        }{
            2
        }
        (
            v_1
            + v_3
        ),
        & t>0,
        & \,\\
        &B_1 \left(
            V_2;
            \frac{
                F(\tau_1)
                + F(\tau_3)
            }{
                2
            }
        \right)
        = \gamma_+ \frac{
            F(\tau_1)
            - F(\tau_3)
        }{
            2
        }
        (
            v_2
            + v_4
        ),
        & t>0,
        & \,\\
        & B_2\left(
            V_2;
            \frac{
                F(\tau_1)
                + F(\tau_3)
            }{
                2
            }
        \right)
        = - \gamma_- \partial_x \frac{
            F(\tau_1)
            - F(\tau_3)
        }{
            2
        }
        (
            v_2
            + v_4
        ),
        & t>0,
        & \,\\
        &\frac{d\Sigma_1}{dt}
        = - \mathcal{N}_{\Sigma,1}(\tau_1, \tau_2, \tau_3, \tau_4)
        + Q_1 \mathcal{N}_{\Sigma,0}( \psi_1, \psi_3, \tau_1, \tau_3),
        & t>0,
        & \,\\
        & \frac{d\Sigma_2}{dt}
        = \mathcal{N}_{\Sigma,2}(\tau_1, \tau_2, \tau_3, \tau_4)
        + Q_2 \mathcal{N}_{\Sigma,0}( \psi_1, \psi_3, \tau_2, \tau_4),
        & t>0,
        & \, \\
        & V_1
        =0, \quad
        V_2
        = 0,
        & t=0,
        & \, x \in I,\\
        & \Sigma_1
        = 0, \quad
        \Sigma_2
        = 0,
        & t=0,
        & \,
    \end{aligned}
\end{equation}
where
\begin{align*}
    \mathcal{N}_{V,0}(\psi_1, \psi_3, \tau_1, \tau_3)
    &:= c(\tau_1) \partial_x \psi_1
    - c(\tau_3) \partial_x \psi_3,\\
    \mathcal{N}_{V,1}(\psi_1, \psi_2, \psi_3, \psi_4)
    &:= N_{v,1}(\psi_1) \psi_2
    - N_{v,1}(\psi_3) \psi_4,\\
    \mathcal{N}_{V,2}(\psi_1, \psi_2, \psi_3, \psi_4)
    &:= N_{v,2}(\psi_1, \psi_2) \psi_2
    - N_{v,2}(\psi_3, \psi_4) \psi_4,
\end{align*}
and
\begin{align*}
    \mathcal{N}_{\Sigma,0}(\psi_1, \psi_3, \tau_1, \tau_3)
    &:= d(\tau_1) \gamma_+ \psi_1
    - d(\tau_3) \gamma_+ \psi_3\\
    \mathcal{N}_{\Sigma,1}(\tau_1, \tau_2, \tau_3, \tau_4)
    &:= N_{\sigma,1}(\tau_1) \tau_2
    - N_{\sigma,1}(\tau_3) \tau_4,\\
    \mathcal{N}_{\Sigma,2}(\tau_1, \tau_2, \tau_3, \tau_4)
    &:= N_{\sigma,2}(\tau_1, \tau_2) \tau_2
    - N_{\sigma,2}(\tau_3, \tau_4) \tau_4.
\end{align*}
We observe that
\begin{align*}
    \mathcal{N}_{\Sigma,0}(\psi_1, \psi_3, \tau_1, \tau_3)
    &= (
        d(\tau_1) - d(\tau_3)
    )\gamma_+ \psi_1
    - d(\tau_3) \gamma_+ 
    (
        \psi_1
        - \psi_3
    )\\
    \mathcal{N}_{\Sigma,1}(\tau_1, \tau_2, \tau_3, \tau_4)
    & = 
    (
        N_{\sigma,1}(\tau_1)
        - N_{\sigma,1}(\tau_3)
    )\tau_2
    - N_{\sigma,1}(\tau_3)
    (
        \tau_2
        - \tau_4
    ),\\
    \mathcal{N}_{\Sigma,2}(\tau_1, \tau_2, \tau_3, \tau_4)
    & =
    (
        N_{\sigma,2}(\tau_1, \tau_2)
        - N_{\sigma,2}(\tau_3, \tau_4)
    ) \tau_2\\
    & \quad - N_{\sigma,2}(\tau_3, \tau_4)
    (
        \tau_2
        - \tau_4
    ).
\end{align*}

\begin{theorem} \label{thm_global_well_posedness_for_eq_abstract_filter_clogging_equation}
    Let $T>0$ and $\delta \in (0,1/8)$.
    Assume that $v_{j,0} \in H^3(I)$, $\sigma_{0,j} \in \Real$ are positive and $f \in X_{f,T}$ is non-negative.
    Then there exists a unique solution $(v_1, v_2, \sigma_1, \sigma_2)$ to (\ref{eq_abstract_filter_clogging_equation}) such that
    \begin{align*}
        \Vert
            (v_1, v_2)
        \Vert_{X_{v,T}^2}
        + \Vert
            (\sigma_1, \sigma_2)
        \Vert_{X_{\sigma,T}^2}
        < \infty.
    \end{align*}
\end{theorem}    

\subsection{Non-linear estimates}
We begin by establishing some nonlinear estimates.
%We first invoke the following estimates from elementary calculus:
\begin{proposition} \label{prop_nonlinear_estimates_for_Nv}
    \begin{enumerate}[label=(\roman*)]
        %\item
        %%By the definition of $N_{v, j}$ and $N_{\sigma,j}$ we have
        %It follows that
        %\begin{align} \label{eq_pointwis_estimates_for_nonlinear_terms}
        %    \begin{aligned}
        %        & - N_{v,1} (v_1) v_2 v_1
        %        \leq 0,
        %        && - N_{\sigma,1} (\sigma_1) \sigma_2 \sigma_1 
        %        \leq 0,\\
        %        & N_{v,1} (v_1) v_2 \varphi
        %        \leq R_1 v_2 \varphi,
        %        && N_{v,2} (v_1, v_2) v_2 \varphi
        %        \leq R_1 \alpha_1 v_2 \varphi
        %        - \alpha_2 v_2 \varphi,\\
        %        & N_{\sigma,1} (v_1) \sigma_2
        %        \leq R_1 \sigma_2,
        %        && N_{\sigma,2} (v_1, \sigma_2) \sigma_2
        %        \leq R_1 \beta_1 v_2
        %        - \beta_2 v_2.
        %    \end{aligned}
        %\end{align}
        \item It follows that
        \begin{align} \label{eq_estimates_for_Nvj_dNvj_ddNvj}
            \begin{split}
                \vert
                    N_{v, 1} (v_1) v_2
                \vert
                & \leq C \vert
                    v_2
                \vert\\
                \vert
                    N_{v, 2} (v_1, v_2) v_2
                \vert
                & \leq C (
                    \vert
                        v_2
                    \vert
                    + \vert
                        v_2
                    \vert^2
                )\\
                \vert
                    \partial (
                        N_{v, 1} (v_1) v_2
                    )
                \vert
                & \leq C
                (
                    \vert
                        \partial v_1
                    \vert
                    \vert
                        v_2
                    \vert
                    + \vert
                        \partial v_2
                    \vert
                ),\\
                %%%%%%%%%%%%%%%%%%%%%%%%%%%%%%%%%%%%%
                \vert
                    \partial (
                        N_{v,2} (v_1, v_2) v_2
                    )
                \vert
                & \leq C
                \vert
                    \partial (
                        N_{v, 1} (v_1) v_2
                    )
                \vert
                + C \vert
                    \partial v_2
                \vert
                \vert
                    v_2
                \vert,\\
                %%%%%%%%%%%%%%%%%%%%%%%%%%%%%%%%%%%%%
                \vert
                    \partial^2 (
                        N_{v, 1} (v_1) v_2
                    )
                \vert
                & \leq C
                (
                    \vert
                        \partial v_1
                    \vert^2
                    \vert
                        v_2
                    \vert
                    + \vert
                        \partial^2 v_1
                    \vert
                    \vert
                        v_2
                    \vert
                    + \vert
                        \partial v_1
                    \vert
                    \vert
                        \partial v_2
                    \vert
                    + \vert
                        v_1
                    \vert
                    \vert
                        \partial^2 v_2
                    \vert
                ),\\
                %%%%%%%%%%%%%%%%%%%%%%%%%%%%%%%%%%%%%
                \vert
                    \partial^2 (
                        N_{v,2} (v_1) v_2
                    )
                \vert
                & \leq C
                \vert
                    \partial^2 (
                        N_{v, 1} (v_1) v_2
                    )
                \vert
                + C \vert
                    \partial v_2
                \vert^2
                + C \vert
                    v_2
                \vert
                \vert
                    \partial^2 v_2
                \vert,
            \end{split}
        \end{align}
        for $\partial \in \{\partial_t, \partial_x\}$, positive smooth functions $v_j$ $(j=1,2)$, and a constant $C>0$.
        %Moreover, the same estimates holds for $N_{\sigma, j}$ $(j=1,2)$.
        \item There exists a constant $C>0$ such that
        \begin{align}\label{eq_L2_estimates_for_Nvj}
            \begin{split}
                & \Vert
                    N_{v, 1} (v_1) v_2
                \Vert_{L^2(I)}
                \leq C
                \Vert
                    v_2
                \Vert_{L^2(I)},\\
                & \Vert
                    N_{v, 2} (v_1, v_2) v_2
                \Vert_{L^2(I)}
                \leq C (
                    \Vert
                        v_2
                    \Vert_{L^2(I)}
                    + \Vert
                        v_2
                    \Vert_{L^2(I)}
                    \Vert
                        v_2
                    \Vert_{H^1(I)}
                ),\\
                %%%%%%%%%%%%%%%%%%%%%%%%%%%%%%%%%%%%%
            \end{split}     
        \end{align}
        for all positive functions $v_j \in H^1(I)$ $(j=1,2)$.
        \item There exists a constant $C>0$ such that
        \begin{align}\label{eq_L2_estimates_for_dNvj}
            \begin{split}
                & \Vert
                    \partial (
                        N_{v, 1} (v_1) v_2
                    )
                \Vert_{L^2(I)}
                \leq C
                (
                    \Vert
                        \partial v_1
                    \Vert_{L^2(I)}
                    \Vert
                        v_2
                    \Vert_{H^1(I)}
                    + \Vert
                        \partial v_2
                    \Vert_{L^2(I)}
                ),\\
                %%%%%%%%%%%%%%%%%%%%%%%%%%%%%%%%%%%%%
                & \Vert
                    \partial (
                        N_{v,2} (v_1, v_2) v_2
                    )
                \Vert_{L^2(I)}\\
                & \leq C (
                    \Vert
                        \partial (
                            N_{v, 1} (v_1) v_2
                        )
                    \Vert_{L^2(I)}
                    + \Vert
                            \partial v_2
                        \Vert_{L^2(I)}
                        \Vert
                            v_2
                        \Vert_{H^1(I)}
                ),
                %%%%%%%%%%%%%%%%%%%%%%%%%%%%%%%%%%%%%
            \end{split}     
        \end{align}
        for $a.a. \, t>0$, $\partial \in \{\partial_t, \partial_x\}$, and all positive functions $v_j \in L^2_t H^1_x(Q_T)$ $(j=1,2)$ satisfying $\partial v_j \in L^2_t H^1_x(Q_T)$.
        \item There exists a constant $C>0$ such that
        \begin{align}\label{eq_L2_estimates_for_ddNvj}
            \begin{split}
                & \Vert
                    \partial^2 (
                        N_{v, 1} (v_1) v_2
                    )
                \Vert_{L^2(I)}\\
                & \leq C
                (
                    \Vert
                        \partial v_1
                    \Vert_{L^2(I)}
                    \Vert
                        \partial v_1
                    \Vert_{H^1(I)}
                    \Vert
                        v_2
                    \Vert_{H^1(I)}
                    + \Vert
                        \partial^2 v_1
                    \Vert_{L^2(I)}
                    \Vert
                        v_2
                    \Vert_{H^1(I)}\\
                & \quad\quad
                    + \Vert
                        \partial v_1
                    \Vert_{L^2(I)}
                    \Vert
                        \partial v_2
                    \Vert_{H^1(I)}
                    + \Vert
                        v_1
                    \Vert_{H^1(I)}
                    \Vert
                        \partial^2 v_2
                    \Vert_{L^2(I)}
                ),\\
                %%%%%%%%%%%%%%%%%%%%%%%%%%%%%%%%%%%%%
                & \Vert
                    \partial^2 (
                        N_{v,2} (v_1) v_2
                    )
                \Vert_{L^2(I)}\\
                & \leq C (
                \Vert
                    \partial^2 (
                        N_{v, 1} (v_1) v_2
                    )
                \Vert_{L^2(I)}\\
                & \quad\quad + \Vert
                        \partial v_2
                    \Vert_{L^2(I)}
                    \Vert
                        \partial v_2
                    \Vert_{H^1(I)}
                    + \Vert
                        v_2
                    \Vert_{H^1(I)}
                    \Vert
                        \partial^2 v_2
                    \Vert_{L^2(I)},
            \end{split}
        \end{align}
        for $a.a. \, t>0$, $\partial \in \{\partial_t, \partial_x\}$, and all positive functions $v_j \in L^2_t H^1_x(Q_T)$ $(j=1,2)$ satisfying $\partial v_j \in L^2_t H^1_x(Q_T)$ and $\partial^2 v_j \in L^2_t L^2_x(Q_T)$.
    \end{enumerate}
\end{proposition}
\begin{remark}
    The same estimates as (\ref{eq_estimates_for_Nvj_dNvj_ddNvj}) hold for $N_{\sigma, 1}(\sigma_1) \sigma_2$ and $N_{\sigma, 2}(\sigma_1, \sigma_2) \sigma_2$.
\end{remark}
\begin{proof}
    %The bounds (\ref{eq_pointwis_estimates_for_nonlinear_terms}) are from the definition of the nonlinear terms.
    The estimates in (\ref{eq_estimates_for_Nvj_dNvj_ddNvj}) are due to direct calculations.
    The estimates (ii)-(iv) are from (\ref{eq_estimates_for_Nvj_dNvj_ddNvj}) and the embedding $H^1(I) \hookrightarrow L^\infty(I)$.
\end{proof}

\begin{proposition} \label{prop_estimates_for_difference_between_nonlinear_terms}
    \begin{enumerate}[label=(\roman*)]
        \item There exists a constant $C>0$ such that
        \begin{align*}
            & \Vert
                N_{v, 1}(\varphi_1) \psi_1
                - N_{v, 1}(\varphi_2) \psi_2
            \Vert_{L^2(I)}\\
            & \leq (
                1
                + \Vert
                    \psi_1
                \Vert_{H^1(I)}
            )
            (
                \Vert
                    \varphi_1
                    - \varphi_2
                \Vert_{L^2(I)}
                + \Vert
                    \psi_1
                    - \psi_2
                \Vert_{L^2(I)}
            )
        \end{align*}
        for all positive functions $\varphi_j, \psi_j \in H^1(I)$ $(j=1,2)$.
        \item There exists a constant $C>0$ such that
        \begin{align*}
            & \Vert
                N_{v, 2}(\varphi_1) \psi_1
                - N_{v, 2}(\varphi_2) \psi_2
            \Vert_{L^2(I)}\\
            & \leq C
            (
                1
                + \Vert
                    \psi_1
                \Vert_{H^1(I)}
                + \Vert
                    \psi_2
                \Vert_{H^1(I)}
            )
            (
                \Vert
                    \varphi_1
                    - \varphi_2
                \Vert_{L^2(I)}
                + \Vert
                    \psi_1
                    - \psi_2
                \Vert_{L^2(I)}
            )
        \end{align*}
        for all positive functions $\varphi_j, \psi_j \in H^1(I)$ $(j=1,2)$.
        \item There exists a constant $C>0$ such that
        \begin{align*}
            \begin{split}
                & \Vert
                    \partial (
                        N_{v, 1} (\varphi_1) \psi_1
                        - N_{v, 1} (\varphi_2) \psi_2
                    )
                \Vert_{L^2(I)}\\
                &\leq C \left(
                    \sum_{j=1,2}
                        \Vert
                            \partial \varphi_j
                        \Vert_{L^2(I)}
                    \Vert
                        \psi_1
                    \Vert_{H^1(I)}
                    + \Vert
                        \partial \psi_1
                    \Vert_{L^2(I)}
                \right)
                \Vert
                    \varphi_1
                    - \varphi_2
                \Vert_{H^1(I)}\\
                & + C \Vert
                    \psi_1
                \Vert_{H^1(I)}
                \Vert
                    \partial \varphi_1
                    - \partial \varphi_2
                \Vert_{L^2(I)}\\
                & + C \Vert
                    \partial \varphi_2
                \Vert_{L^2(I)}
                \Vert
                    \psi_1
                    -  \psi_2
                \Vert_{H^1(I)}\\
                & + C
                \Vert
                    \partial \psi_1
                    - \partial \psi_2
                \Vert_{L^2(I)}
            \end{split}
        \end{align*}
        for $a.a. \, t>0$, $\partial \in \{\partial_t, \partial_x\}$, and all positive functions $\varphi_j, \psi_j \in L^2_t H^1_x(Q_T)$ $(j=1,2)$ satisfying $\partial \varphi_j, \partial \psi_j \in L^2_t H^1_x(Q_T)$.
        \item There exists a constant $C>0$ such that
        \begin{align*}
            \begin{split}
                & \Vert
                    \partial (
                        N_{v, 2} (\varphi_1, \psi_1) \psi_1
                        - N_{v, 2} (\varphi_2, \psi_2) \psi_2
                    )
                \Vert_{L^2(I)}\\
                & \leq \Vert
                    \partial (
                        N_{v, 1} (\varphi_1) \psi_1
                        - N_{v, 1} (\varphi_2) \psi_2
                    )
                \Vert_{L^2(I)}\\
                & + C \sum_{j=1,2} \left(
                    \Vert
                        \partial \psi_j
                    \Vert_{L^2(I)}
                    \Vert
                        \psi_1
                        - \psi_2
                    \Vert_{H^1(I)}
                \right.\\
                & \left.
                    \quad\quad\quad\quad\quad
                    + \Vert
                            \psi_j
                        \Vert_{H^1(I)}
                    \Vert
                        \partial \psi_1
                        - \partial \psi_2
                    \Vert_{L^2(I)}
                \right)
            \end{split}
        \end{align*}
        for $a.a. \, t>0$ and all positive functions $\varphi_j, \psi_j \in L^2_t H^1_x(Q_T)$ $(j=1,2)$ satisfying $\partial \varphi_j, \partial \psi_j \in L^2_t H^1_x(Q_T)$.
        \item There exists a constant $C>0$ such that
        \begin{align*}
            \begin{split}
                & \Vert
                    \partial^2_x (
                        N_{v, 1} (\varphi_1) \psi_1
                        - N_{v, 1} (\varphi_2) \psi_2
                    )
                \Vert_{L^2(I)}\\
                & \leq C
                \sum_{j,k=1,2}
                    (
                        \Vert
                            \varphi_j
                        \Vert_{H^1(I)}
                        \Vert
                            \varphi_j
                        \Vert_{H^2(I)}
                        \Vert
                            \psi_k
                        \Vert_{H^1(I)}
                        + \Vert
                            \varphi_j
                        \Vert_{H^2(I)}
                        \Vert
                            \psi_k
                        \Vert_{H^1(I)}\\
                    & \quad\quad\quad
                        + \Vert
                            \varphi_j
                        \Vert_{H^1(I)}
                        \Vert
                            \psi_k
                        \Vert_{H^2(I)}
                        + \Vert
                            \psi_k
                        \Vert_{H^2(I)}
                    )
                \Vert
                    \varphi_1
                    -  \varphi_2
                \Vert_{H^1(I)}\\
                & + \sum_{j=1,2}
                    (
                        1
                        + \Vert
                            \varphi_j
                        \Vert_{H^1(I)}
                    )
                    \Vert
                        \varphi_j
                    \Vert_{H^2(I)}
                \Vert
                    \psi_1
                    - \psi_2
                \Vert_{H^1(I)}\\
                & + C
                \Vert
                    \psi_1
                \Vert_{H^1(I)}
                \Vert
                    \varphi_1
                    - \varphi_2
                \Vert_{H^2(I)}
                + C
                \Vert
                    \psi_1
                    -  \psi_2
                \Vert_{H^2(I)}
            \end{split}
        \end{align*}
        for $a.a. \, t>0$, $\partial \in \{\partial_t, \partial_x\}$, and all positive functions $\varphi_j, \psi_j \in L^2_t H^1_x(Q_T)$ $(j=1,2)$ satisfying $\partial \varphi_j, \psi_j \in L^2_t H^1_x(Q_T)$ and $\partial^2 \varphi_j, \partial^2 \psi_j \in L^2_t L^2_x(Q_T)$.
        \item There exists a constant $C>0$ such that
        \begin{align} \label{eq_difference_d2_Nv2_psi1_Nv1_psi2}
            \begin{split}
                & \Vert
                    \partial_x^2 (
                        N_{v, 2} (\varphi_1, \psi_1) \psi_1
                        - N_{v, 2} (\varphi_2, \psi_2) \psi_2
                    )
                \Vert_{L^2(I)}\\
                & \leq \Vert
                    \partial_x^2 (
                        N_{v, 1} (\varphi_1) \psi_1
                        - N_{v, 1} (\varphi_2) \psi_2
                    )
                \Vert_{L^2(I)}\\
                & + C \sum_{j=1,2} \left(
                    \Vert
                        \partial \psi_j
                    \Vert_{H^2(I)}
                    \Vert
                        \psi_1
                        - \psi_2
                    \Vert_{H^1(I)}
                \right.\\
                & \left.
                    \quad\quad\quad\quad\quad
                    + \Vert
                            \psi_j
                        \Vert_{H^1(I)}
                    \Vert
                        \partial \psi_1
                        - \partial \psi_2
                    \Vert_{H^2(I)}
                \right)
            \end{split}
        \end{align} 
        for $a.a. \, t>0$ and all positive functions $\varphi_j, \psi_j \in L^2_t H^1_x(Q_T)$ $(j=1,2)$ satisfying $\partial \varphi_j, \psi_j \in L^2_t H^1_x(Q_T)$ and $\partial^2 \varphi_j, \partial^2 \psi_j \in L^2_t L^2_x(Q_T)$.
    \end{enumerate}
\end{proposition}

\begin{proof}
    Since
    \begin{align*}
        N_{v,1}(\varphi_1) \psi_1
        - N_{v,1}(\varphi_2) \psi_2
        = (
            N_{v,1}(\varphi_1)
            - N_{v,1}(\varphi_2)
        ) \psi_1
        + N_{v,1}(\varphi_2)
        (
            \psi_1
            - \psi_2
        ),
    \end{align*}
    combining with the embedding $H^1(I) \hookrightarrow L^\infty(I)$, we deduce that
    \begin{align*}
        & \Vert
            N_{v, 1} (\varphi_1) \psi_1
            - N_{v, 1} (\varphi_2) \psi_2
        \Vert_{L^2(I)}\\
        & \leq C
        (
            1
            + \Vert
                \psi_1
            \Vert_{H^1(I)}
        )
        (
            \Vert
                \varphi_1
                - \varphi_2
            \Vert_{L^2(I)}
            + \Vert
                \psi_1
                - \psi_2
            \Vert_{L^2(I)}
        ).
    \end{align*}
    We have proved (i).
    Similarly, we find from the definition of $N_{v, 2}$ and the formula
    \begin{align*}
        \psi_1^2 - \psi_2^2
        = (
            \psi_1
            + \psi_2
        )
        (
            \psi_1
            - \psi_2
        )
    \end{align*}
    that
    \begin{align*}
        & \Vert
            N_{v, 2} (\varphi_1, \psi_1) \psi_1
            - N_{v, 2} (\varphi_2, \psi_2) \psi_2
        \Vert_{L^2(I)}\\
        & \leq \Vert
            N_{v, 1} (\varphi_1) \psi_1
            - N_{v, 1} (\varphi_2) \psi_2
        \Vert_{L^2(I)}
        + C \Vert
            \psi_1^2
            - \psi_2^2
        \Vert_{L^2(I)}\\
        & \leq C
        (
            1
            + \Vert
                \psi_1
            \Vert_{H^1(I)}
            + \Vert
                \psi_2
            \Vert_{H^1(I)}
        )
        (
            \Vert
                \varphi_1
                - \varphi_2
            \Vert_{L^2(I)}
            + \Vert
                \psi_1
                - \psi_2
            \Vert_{L^2(I)}
        ).
    \end{align*}
    We have proved (ii).
    We invoke the formula
    \begin{align} \label{eq_diff_Nv1_phi1_phi2}
        \begin{split}
            N_{v,1}(\varphi_1)
            - N_{v,1}(\varphi_2)
            & = \int_0^1
                N_{v,1}^\prime(\varphi_\theta)
            d\theta
            (
                \varphi_1
                - \varphi_2
            )\\
            & =: \tilde{N}_{v,1}(\varphi_1, \varphi_2)(
                \varphi_1
                - \varphi_2
            ),\\
            \varphi_\theta
            & := \theta \varphi_1 + (1-\theta)\varphi_2.
        \end{split}
    \end{align}
    %%%%%%%%%%%%%%%%%%%%%%%%%%%%%%%%%%%%%%%%%%%%%%%
    Using this formula, we deduce the estimates:
    \begin{itemize}
        \item We see find from the above formula that
        \begin{align*}
            \begin{split}
                & \vert
                    \partial (
                        N_{v, 1} (\varphi_1) \psi_1
                        - N_{v, 1} (\varphi_2) \psi_2
                    )
                \vert\\
                &\leq \vert
                    \partial [
                        (
                            N_{v, 1} (\varphi_1)
                            - N_{v, 1} (\varphi_2)
                        )\psi_1
                    ]
                \vert
                + \vert
                    \partial [
                        N_{v, 1} (\varphi_2) (
                            \psi_1
                            -  \psi_2
                        )
                    ]
                \vert\\
                &\leq \vert
                    \partial (
                        \tilde{N}_{v,1}(\varphi_1, \varphi_2) \psi_1
                    )
                \vert
                \vert
                    \varphi_1
                    - \varphi_2
                \vert
                + \vert
                    \tilde{N}_{v,1}(\varphi_1, \varphi_2) \psi_1
                \vert
                \vert
                    \partial \varphi_1
                    - \partial \varphi_2
                \vert\\
                & + \vert
                    \partial N_{v, 1}(\varphi_2)
                \vert
                \vert
                    \psi_1
                    -  \psi_2
                \vert
                + \vert
                    N_{v, 1}(\varphi_2)
                \vert
                \vert
                    \partial \psi_1
                    - \partial \psi_2
                \vert.
            \end{split}
        \end{align*}
        Therefore, we have
        \begin{align*}
            \begin{split}
                & \Vert
                    \partial (
                        N_{v, 1} (\varphi_1) \psi_1
                        - N_{v, 1} (\varphi_2) \psi_2
                    )
                \Vert_{L^2(I)}\\
                &\leq C \left(
                    \sum_{j=1,2}
                        \Vert
                            \partial \varphi_j
                        \Vert_{L^2(I)}
                    \Vert
                        \psi_1
                    \Vert_{H^1(I)}
                    + \Vert
                        \partial \psi_1
                    \Vert_{L^2(I)}
                \right)
                \Vert
                    \varphi_1
                    - \varphi_2
                \Vert_{H^1(I)}\\
                & + C \Vert
                    \psi_1
                \Vert_{H^1(I)}
                \Vert
                    \partial \varphi_1
                    - \partial \varphi_2
                \Vert_{L^2(I)}\\
                & + C \Vert
                    \partial \varphi_2
                \Vert_{L^2(I)}
                \Vert
                    \psi_1
                    -  \psi_2
                \Vert_{H^1(I)}\\
                & + C
                \Vert
                    \partial \psi_1
                    - \partial \psi_2
                \Vert_{L^2(I)}.
            \end{split}
        \end{align*}
        \item We see from the definition of $N_{v,2}$ that
        \begin{align*}
            \begin{split}
                & \Vert
                    \partial (
                        N_{v, 2} (\varphi_1, \psi_1) \psi_1
                        - N_{v, 2} (\varphi_2, \psi_2) \psi_2
                    )
                \Vert_{L^2(I)}\\
                & \leq \Vert
                    \partial (
                        N_{v, 1} (\varphi_1) \psi_1
                        - N_{v, 1} (\varphi_2) \psi_2
                    )
                \Vert_{L^2(I)}\\
                & + C \sum_{j=1,2} \left(
                    \Vert
                        \partial \psi_j
                    \Vert_{L^2(I)}
                    \Vert
                        \psi_1
                        - \psi_2
                    \Vert_{H^1(I)}
                \right.\\
                & \left.
                    \quad\quad\quad\quad\quad
                    + \Vert
                            \psi_j
                        \Vert_{H^1(I)}
                    \Vert
                        \partial \psi_1
                        - \partial \psi_2
                    \Vert_{L^2(I)}
                \right).
            \end{split}
        \end{align*}
        \item We find from (\ref{eq_diff_Nv1_phi1_phi2}) that
        \begin{align*}
            \begin{split}
                & \vert
                    \partial^2 (
                        N_{v, 1} (\varphi_1) \psi_1
                        - N_{v, 1} (\varphi_2) \psi_2
                    )
                \vert\\
                &\leq \vert
                    \partial^2 (
                        \tilde{N}_{v,1}(\varphi_1, \varphi_2) \psi_1
                    )
                \vert
                \vert
                    \varphi_1
                    - \varphi_2
                \vert\\
                & + \vert
                    \partial (
                        \tilde{N}_{v,1}(\varphi_1, \varphi_2) \psi_1
                    )
                \vert
                \vert
                    \partial \varphi_1
                    - \partial \varphi_2
                \vert\\
                & + \vert
                    \tilde{N}_{v,1}(\varphi_1, \varphi_2) \psi_1
                \vert
                \vert
                    \partial^2 \varphi_1
                    - \partial^2 \varphi_2
                \vert\\
                %%%%%%%%%%%%%%%%%%%%%%%%%%%%%%%%%%%%%%%%%%%
                & + \vert
                    \partial^2 N_{v, 1}(\varphi_2)
                \vert
                \vert
                    \psi_1
                    - \psi_2
                \vert
                + \vert
                    \partial N_{v, 1}(\varphi_2)
                \vert
                \vert
                    \partial \psi_1
                    - \partial \psi_2
                \vert\\
                & + \vert
                    N_{v, 1}(\varphi_2)
                \vert
                \vert
                    \partial^2 \psi_1
                    - \partial^2 \psi_2
                \vert\\
                & =: I_1 + I_2 + I_3 + I_4 + I_5 + I_6.
            \end{split}
        \end{align*}
        \underline{Estimate for $I_1$:}
        We see from the definition of $\tilde{N}_{v,1}$ that
        \begin{align*}
            & \vert
                \partial^2 (
                    \tilde{N}_{v,1}(\varphi_1, \varphi_2) \psi_1
                )
            \vert\\
            & \leq
            C \vert
                \nabla^2 \tilde{N}_{v,1}(\varphi_1, \varphi_2)
            \vert
            \sum_{j=1,2}\vert
                    \partial \varphi_j
                \vert^2
            \vert
                \psi_1
            \vert
            + C \vert
                \nabla \tilde{N}_{v,1}(\varphi_1, \varphi_2)
            \vert
            \sum_{j=1,2} \vert
                    \partial^2 \varphi_j
                \vert
            \vert
                \psi_1
            \vert\\
            & + C \vert
                \tilde{N}_{v,1}(\varphi_1, \varphi_2)
            \vert
            \sum_{j=1,2}\vert
                    \partial \varphi_j
                \vert
            \vert
                \partial \psi_1
            \vert
            + C \vert
                \tilde{N}_{v,1}(\varphi_1, \varphi_2)
            \vert
            \sum_{j=1,2} \vert
                    \varphi_j
                \vert
            \vert
                \partial^2 \psi_1
            \vert.
        \end{align*}
        Therefore, we have
        \begin{align*}
            & \Vert
               I_1
            \Vert_{L^2(I)}\\
            & \leq C \Vert
                \partial^2 (
                    \tilde{N}_{v,1}(\varphi_1, \varphi_2) \psi_1
                )
            \Vert_{L^2(I)}
            \Vert
                \varphi_1
                - \varphi_2
            \Vert_{H^1(I)}\\
            & \leq
            C \sum_{j=1,2} \left(
                \Vert
                    \partial \varphi_j
                \Vert_{L^2(I)}
                \Vert
                    \partial \varphi_j
                \Vert_{H^1(I)}
                \Vert
                    \psi_1
                \Vert_{H^1(I)}
                + \Vert
                    \partial^2 \varphi_j
                \Vert_{L^2(I)}
                \Vert
                    \psi_1
                \Vert_{H^1(I)}
            \right.\\
            & \left.
                \quad\quad\quad
                + \Vert
                    \partial \varphi_j
                \Vert_{H^1(I)}
                \Vert
                    \partial \psi_1
                \Vert_{L^2(I)}
                + \Vert
                    \varphi_j
                \Vert_{H^1(I)}
                \Vert
                    \partial^2 \psi_1
                \Vert_{L^2(I)}
            \right)
            \Vert
                \varphi_1
                - \varphi_2
            \Vert_{H^1(I)}.
        \end{align*}

        \underline{Estimates for $I_2$ and $I_4$:}
        We see from the definition of $\tilde{N}_{v,1}$ that
        \begin{align*}
            & \Vert
               I_2
            \Vert_{L^2(I)}
            + \Vert
               I_4
            \Vert_{L^2(I)}\\
            & \leq
            C \left(
                \sum_{j=1,2}
                    \Vert
                        \partial \varphi_j
                    \Vert_{H^1(I)}
                    \Vert
                        \psi_1
                    \Vert_{H^1(I)}
                + \Vert
                    \partial \psi_1
                \Vert_{H^1(I)}
            \right)
            \Vert
                \partial \varphi_1
                - \partial \varphi_2
            \Vert_{L^2(I)}\\
            & + C \Vert
                \partial \varphi_2
            \Vert_{H^1(I)}
            \Vert
                \partial \psi_1
                - \partial \psi_2
            \Vert_{L^2(I)}.
        \end{align*}

        \underline{Estimates for $I_3$ and $I_6$:}
        Applying the embedding $H^1(I) \hookrightarrow L^\infty(I)$, we have
        \begin{align*}
            & \Vert
               I_3
            \Vert_{L^2(I)}
            + \Vert
               I_6
            \Vert_{L^2(I)}\\
            & \leq
            C (
                \Vert
                    \psi_1
                \Vert_{H^1(I)}
                \Vert
                    \partial^2 \varphi_1
                    - \partial^2 \varphi_2
                \Vert_{L^2(I)}
                + \Vert
                    \partial^2 \psi_1
                    - \partial^2 \psi_2
                \Vert_{L^2(I)}
            ).
        \end{align*}

        \underline{Estimates for $I_4$:}
        Since
        \begin{align*}
            \vert
                \partial^2 N_{v,1}(\varphi_2)
            \vert
            \leq \vert
                N_{v,1}^{\prime\prime}(\varphi_2)
            \vert
            \vert
                \partial \varphi_2
            \vert^2
            +  \vert
                N_{v,1}^\prime(\varphi_2)
            \vert
            \vert
                \partial^2 \varphi_2
            \vert,
        \end{align*}
        we have
        \begin{align*}
            \Vert
               I_4
            \Vert_{L^2(I)}
            \leq
            C (
                \Vert
                    \partial \varphi_2
                \Vert_{L^2(I)}
                \Vert
                    \partial \varphi_2
                \Vert_{H^1(I)}
                + \Vert
                    \partial^2 \varphi_2
                \Vert_{L^2(I)}
            )
            \Vert
                \psi_1
                - \psi_2
            \Vert_{H^1(I)}.
        \end{align*}
        %%%%%%%%%%%%%%%%%%%%%%%%%%%%%%%%%%%%%%%%%%%%%%
        %%%%%%%%%%%%%%%%%%%%%%%%%%%%%%%%%%%%%%%%%%%%%%
        %%%%%%%%%%%%%%%%%%%%%%%%%%%%%%%%%%%%%%%%%%%%%%
        \underline{Summary:}
        We obtain
        \begin{align} \label{eq_estimate_partial2_for_N_v_1}
            \begin{split}
                & \Vert
                    \partial^2_x (
                        N_{v, 1} (\varphi_1) \psi_1
                        - N_{v, 1} (\varphi_2) \psi_2
                    )
                \Vert_{L^2(I)}\\
                & \leq C
                \sum_{j,k=1,2}
                    (
                        \Vert
                            \varphi_j
                        \Vert_{H^1(I)}
                        \Vert
                            \varphi_j
                        \Vert_{H^2(I)}
                        \Vert
                            \psi_k
                        \Vert_{H^1(I)}
                    \\
                & \quad\quad\quad
                        + \Vert
                            \varphi_j
                        \Vert_{H^2(I)}
                        \Vert
                            \psi_k
                        \Vert_{H^1(I)}
                        + \Vert
                            \varphi_j
                        \Vert_{H^1(I)}
                        \Vert
                            \psi_k
                        \Vert_{H^2(I)}\\
                & \quad\quad\quad
                        + \Vert
                            \psi_k
                        \Vert_{H^2(I)}
                    )
                \Vert
                    \varphi_1
                    -  \varphi_2
                \Vert_{H^1(I)}\\
                & + C \sum_{j=1,2}
                    (
                        1
                        + \Vert
                            \varphi_j
                        \Vert_{H^1(I)}
                    )
                    \Vert
                        \varphi_j
                    \Vert_{H^2(I)}
                \Vert
                    \psi_1
                    - \psi_2
                \Vert_{H^1(I)}\\
                & + C
                \Vert
                    \psi_1
                \Vert_{H^1(I)}
                \Vert
                    \varphi_1
                    - \varphi_2
                \Vert_{H^2(I)}
                + C
                \Vert
                    \psi_1
                    -  \psi_2
                \Vert_{H^2(I)}.
            \end{split}
        \end{align}
        \item
        Since
        \begin{align*}
            \partial^2 (
                \psi_1^2
                - \psi_2^2
            )
            & =
            (
                \partial^2 \psi_1
                + \partial^2 \psi_2
            )
            (
                \psi_1
                - \psi_2
            )\\
            & + 2(
                \partial \psi_1
                + \partial \psi_2
            )
            (
                \partial \psi_1
                - \partial \psi_2
            )
            + (
                \psi_1
                + \psi_2
            )
            (
                \partial^2\psi_1
                - \partial^2 \psi_2
            ),
        \end{align*}
        we deduce that
        \begin{align*}
            \begin{split}
                & \Vert
                    \partial^2 (
                        \psi_1^2
                        - \psi_2^2
                    )
                \Vert_{L^2(I)}\\
                & \leq C \sum_{j=1,2} (
                    \Vert
                        \partial \psi_j
                    \Vert_{H^2(I)}
                    \Vert
                        \psi_1
                        - \psi_2
                    \Vert_{H^1(I)}\\
                & \quad\quad\quad\quad\quad
                    + \Vert
                            \psi_j
                        \Vert_{H^1(I)}
                    \Vert
                        \partial \psi_1
                        - \partial \psi_2
                    \Vert_{H^2(I)}
                ).
            \end{split}
        \end{align*}
        This estimate and (\ref{eq_estimate_partial2_for_N_v_1}) imply (\ref{eq_difference_d2_Nv2_psi1_Nv1_psi2}).
    \end{itemize}
\end{proof}

We next estimate the advection terms.
\begin{proposition} \label{prop_nonlinear_estimates_for_cdv}
    \begin{enumerate}[label=(\roman*)]
        \item There exists a constant $C>0$ such that
        \begin{align*}
            & \Vert
                \partial_x^m (
                    c(\tau) \partial_x \varphi
                )
            \Vert_{L^2(I)}
            \leq C \Vert
                \partial_x^{m+1} \varphi
            \Vert_{L^2(I)}
        \end{align*}
        for $\varphi \in H^m(I)$ and $\tau \in L^\infty(0,T)$.
        \item There exists a constant $C>0$ such that
        \begin{align*}
            \Vert
                \partial_t (c(\tau) \partial_x \varphi)
            \Vert_{L^2(I)}
            & \leq C (
                \vert
                    \tau^\prime
                \vert
                \Vert
                    \partial_x \varphi
                \Vert_{L^2(I)}
                + \Vert
                    \partial_t \partial_x \varphi
                \Vert_{L^2(I)}
            )            
        \end{align*}
        for $a.a. \, t>0$ and $\varphi \in H^1_tH^1_x(I)$ and $\tau \in H^1(0,T)$.
        \item
        \begin{align*}
            \begin{split}
                \Vert
                    \partial_t^2 (c(\tau) \varphi)
                \Vert_{L^2(I)}
                & \leq C (
                    \vert
                        \tau^\prime
                    \vert^2
                    \Vert
                        \partial_x \varphi
                    \Vert_{L^2(I)}
                    + \vert
                        \tau^{\prime\prime}
                    \vert
                    \Vert
                        \partial_x \varphi
                    \Vert_{L^2(I)}\\
                & \quad\quad\quad
                    + \vert
                        \tau^\prime
                    \vert
                    \Vert
                        \partial_t \partial_x \varphi
                    \Vert_{L^2(I)}
                    + \Vert
                        \partial_t^2 \partial_x \varphi
                    \Vert_{L^2(I)}
                )
            \end{split}
        \end{align*}
        for $a.a. \, t>0$ and $\varphi \in H^2_tH^1_x(I)$ and $\tau \in H^2(0,T)$.
    \end{enumerate}
\end{proposition}
\begin{proof}
    The first estimate is trivial.
    Since
    \begin{align*}
        & \partial_t(
            c(\tau) \varphi
        )
        = c^\prime(\tau) \tau^\prime \partial_x \varphi
        + c(\tau) \partial_t \partial_x \varphi,\\
        & \partial_t^2(
            c(\tau) \varphi
        )
        = c^{\prime\prime}(\tau)
        \vert
            \tau^\prime
        \vert^2
        \partial_x \varphi
        + c^\prime(\tau) \tau^{\prime\prime} \partial_x \varphi
        + 2 c^\prime(\tau) \tau^{\prime} \partial_t \partial_x \varphi
        + c(\tau) \partial_t^2 \partial_x \varphi,
    \end{align*}
    we obtain the estimates in (ii) and (iii).
\end{proof}

\begin{proposition}\label{prop_estimates_for_difference_between_convection_terms}
    %The following estimates hold:
    \begin{enumerate}[label=(\roman*)]
        \item There exists a constant $C>0$ such that
        \begin{align*}
            \begin{split}
                & \Vert
                    c(\tau_1) \partial_x \varphi_1
                    - c(\tau_2) \partial_x \varphi_2
                \Vert_{H^m(I)}\\
                & \leq C (
                    \vert
                        \tau_1
                        - \tau_2
                    \vert
                    \sum_{j=1,2}
                        \Vert
                           \varphi_j
                        \Vert_{H^{m+1}(I)}
                    +
                    \Vert
                        \varphi_1
                        -\varphi_2
                    \Vert_{H^{m+1}(I)}
                )
            \end{split}
        \end{align*}
        for $\varphi_j \in H^m(I)$ and $\tau_j \in L^\infty(0,T)$ for $j=1,2$.
        \item There exists a constant $C>0$ such that
        \begin{align*}
            \begin{split}
                & \Vert   
                    \partial_t (
                        c(\tau_1) \partial_x \varphi_1
                        - c(\tau_2) \partial_x \varphi_2
                    )
                \Vert_{H^m(I)}\\
                & \leq C (
                    \sum_{j=1,2}
                        \vert
                            \tau_j^\prime
                        \vert
                    \vert
                        \tau_1
                        - \tau_2
                    \vert
                    + \vert
                    \tau_1^\prime
                    - \tau_2^\prime
                    \vert
                )
                \sum_{k=1,2}
                    \Vert
                       \varphi_k
                    \Vert_{H^{m+1}(I)}\\
                & + C \vert
                    \tau_1
                    - \tau_2
                \vert
                \sum_{k=1,2}
                    \Vert
                       \partial_t \varphi_k
                    \Vert_{H^{m+1}(I)}\\
                & + \sum_{j=1,2}
                    \vert
                        \tau_j^\prime
                    \vert
                \Vert
                    \varphi_1
                    -\varphi_2
                \Vert_{H^{m+1}(I)}
                +
                    \Vert
                        \partial_t \varphi_1
                        - \partial_t \varphi_2
                    \Vert_{H^{m+1}(I)}
                )
            \end{split}
        \end{align*}
        for $a.a. \, t>0$ and $\varphi_j \in H^1_tH^{m+1}_x(I)$ and $\tau_j \in H^1(0,T)$ for $j=1,2$.
    \end{enumerate}
\end{proposition}
\begin{proof}
    We invoke the formulas such that
    \begin{align*}
        c(\tau_1)
        - c(\tau_2)
        & = \int_0^1
            c^\prime(
                \theta \tau_1
                + (1-\theta) \tau_2
            )
        d\theta
        (
            \tau_1
            - \tau_2
        )\\
        & =: \tilde{c}(\tau_1, \tau_2)
        (
            \tau_1
            - \tau_2
        )
    \end{align*}
    and
    \begin{align} \label{eq_formula_for_difference_to_c_dphi}
        c(\tau_1) \partial_x \varphi_1
        - c(\tau_2) \partial_x \varphi_2
        = \tilde{c}(\tau_1, \tau_2) (
            \tau_1
            - \tau_2
        )
        \partial_x \varphi_1
        + c(\tau_2) \partial_x (
            \varphi_1
            - \varphi_2
        ).
    \end{align}
    Taking $L^2(I)$-norm, we have (i).
    Then we take the time-derivatives to (\ref{eq_formula_for_difference_to_c_dphi}) to obtain the estimate in (ii).
\end{proof}

We next establish some estimates for extension functions.
\begin{proposition} \label{prop_estimates_for_difference_between_boundary_extension_terms}
    Let $T>0$, $\tau_1, \tau_2 \in L^2(0,T)$, and $v \in L^2_t H^m_x(Q_T)$.
    We set
    \begin{align*}
        \mathcal{V}
        & := E \left(
            \frac{
                F(\tau_1)
                - F(\tau_2)
            }{
                2
            }
            v
            ;
            1 - \frac{
                F(\tau_1)
                + F(\tau_2)
            }{
                2
            }
        \right).
    \end{align*}
    (i) There exists a constant $C>0$ such that
    \begin{align} \label{eq_Hm_estimate_for_V1}
        \Vert
            \mathcal{V}(t)
        \Vert_{H^m(I)}
        \leq C
        \vert
            \tau_1(t)
            - \tau_2(t)
        \vert
        \Vert
            v(t)
        \Vert_{H^m(I)} \quad
        \text{for $a.a.$ $t>0$.}
    \end{align}
    (ii) There exists a constant $C>0$ such that for $v$ satisfying $v, \partial_t v \in L^2_tH^m_x((Q_T))$ and $\tau_j \in H^1(0,T)$, it holds that
    \begin{align} \label{eq_bound_dt_V1}
        \begin{split}
            & \Vert
                \partial_t \mathcal{V}(t)
            \Vert_{H^m(I)}\\
            & \leq C \left(
                \sum_{j=1,2}
                    \vert
                        \tau_j^\prime(t)
                    \vert
                \vert
                    \tau_1(t)
                    - \tau_2(t)
                \vert
                +
                \vert
                    \tau_1^\prime(t)
                    - \tau_2^\prime(t)
                \vert
            \right)
            \Vert
                v(t)
            \Vert_{H^m(I)}\\
            & + C
            \vert
                \tau_1(t)
                - \tau_2(t)
            \vert
            \Vert
                \partial_t v(t)
            \Vert_{H^m(I)} \quad
            \text{for $a.a. t>0$.}
        \end{split}
    \end{align}
    (iii) There exists a constant $C>0$ such that for $v$ satisfying $v, \partial_t v, \partial_t^2 v \in L^2_tH^m_x(Q_T)$ and $\tau_j \in H^2(0,T)$, it holds that
    \begin{align*}
        \begin{split}
            & \Vert
                \partial_t^2 \mathcal{V}(t)
            \Vert_{H^m(I)}\\
            & \leq C (
                \sum_{j=1,2}
                    \vert
                        \tau_j^\prime(t)
                    \vert^2
                \vert
                    \tau_1(t)
                    - \tau_2(t)
                \vert
                + 
                    \vert
                        \tau_j^{\prime\prime}(t)
                    \vert
                \vert
                    \tau_1(t)
                    - \tau_2(t)
                \vert
            \\
            &
                \quad\quad\quad\quad\quad\quad
                \vert
                        \tau_j^\prime(t)
                    \vert
                \vert
                    \tau_1^\prime(t)
                    - \tau_2^\prime(t)
                \vert
                +
                \vert
                    \tau_1^{\prime\prime}(t)
                    - \tau_2^{\prime\prime}(t)
                \vert
            )
            \Vert
                v(t)
            \Vert_{H^m(I)}\\
            & + C \left(
                \sum_{j=1,2}
                    \vert
                        \tau_j^\prime(t)
                    \vert
                \vert
                    \tau_1(t)
                    - \tau_2(t)
                \vert
                +
                \vert
                    \tau_1^\prime(t)
                    - \tau_2^\prime(t)
                \vert
            \right)
            \Vert
                \partial_t v(t)
            \Vert_{H^m(I)}\\
            & + C
            \vert
                \tau_1(t)
                - \tau_2(t)
            \vert
            \Vert
                \partial_t^2 v(t)
            \Vert_{H^m(I)} \quad
            \text{for $a.a. \, t \in (0,T)$.}
        \end{split}
    \end{align*}
\end{proposition}
\begin{proof}
    The estimate (\ref{eq_Hm_estimate_for_V1}) is due to Lemma \ref{lem_extension_theorem}.
    % and Proposition \ref{prop_estimate_for_dtE}.
    We invoke the formula
    \begin{align*}
        & F(\tau_1(t))
        - F(\tau_2(t))\\
        & = \int_0^1
            F^\prime(
                s\tau_1(t)
                + (1-s) \tau_2(t)
            )
        ds
        (
            \tau_1(t)
            - \tau_2(t)
        )
        =: \tilde{F}(\tau_1, \tau_2) (\tau_1 - \tau_1).
    \end{align*}
    Then we deduce that
    \begin{align*}
        & \vert
            F(\tau_1(t))
            - F(\tau_2(t))
        \vert
        \leq C \vert
            \tau_1
            - \tau_2
        \vert,\\
        %%%%%%%%%%%%%%%%%%%%%%%%%%%%%%%%%%%%%%%%%%%
        & \left \vert
            \frac{d}{dt}(
                F(\tau_1(t))
                - F(\tau_2(t))
            )
        \right \vert
        \leq C
        \sum_{j=1,2}
            \vert
                \tau_j^\prime
            \vert
        \vert
            \tau_1(t)
            - \tau_2(t)
        \vert
        + C
        \vert
            \tau_1^\prime(t)
            - \tau_2^\prime(t)
        \vert,\\
        %%%%%%%%%%%%%%%%%%%%%%%%%%%%%%%%%%%%%%%%%%%
        & \left \vert
            \frac{d^2}{dt^2}(
                F(\tau_1(t))
                - F(\tau_2(t))
            )
        \right \vert\\
        & \leq C
        \sum_{j=1,2} \vert
            \tau_j^\prime(t)
        \vert^2
        \vert
            \tau_1(t)
            - \tau_2(t)
        \vert
        + C \sum_{j=1,2} \vert
            \tau_j^{\prime\prime}(t)
        \vert
        \vert
            \tau_1(t)
            - \tau_2(t)
        \vert\\
        & + C 
        \sum_{j=1,2} \vert
            \tau_j^{\prime}(t)
        \vert
        \vert
            \tau_1^\prime(t)
            - \tau_2^\prime(t)
        \vert
        + C
        \vert
            \tau_1^{\prime\prime}(t)
            - \tau_2^{\prime\prime}(t)
        \vert,
    \end{align*}
    for some constant $C>0$ depending on $F$ and its higher order derivatives.
    We find from Proposition \ref{prop_estimate_for_dtE} that
    \begin{align*}
        \begin{split}
            & \Vert
                \partial_t \mathcal{V}(t)
            \Vert_{H^m(I)}\\
            & \leq C \left \Vert
                \partial_t [
                    (
                        F(\tau_1(t))
                        - F(\tau_2(t))
                    )
                    v(t)
                ]
            \right \Vert_{H^m(I)}\\
            & + C \left \vert
                \partial_t [
                    F(\tau_1(t))
                    + F(\tau_2(t))
                ]
            \right \vert
            \left \Vert
                [
                    F(\tau_1(t))
                    - F(\tau_2(t))
                ]
                v(t)
            \right \Vert_{H^m(I)}\\
            %%%%%%%%%%%%%%%%%%%%%%%%%%%%%%%%%%%%%%%%%%%
            %%%%%%%%%%%%%%%%%%%%%%%%%%%%%%%%%%%%%%%%%%%
            & \leq C \left(
                \sum_{j=1,2}
                    \vert
                        \tau_j^\prime(t)
                    \vert
                \vert
                    \tau_1(t)
                    - \tau_2(t)
                \vert
                +
                \vert
                    \tau_1^\prime(t)
                    - \tau_2^\prime(t)
                \vert
            \right)
            \Vert
                v(t)
            \Vert_{H^m(I)}\\
            & + C
            \vert
                \tau_1(t)
                - \tau_2(t)
            \vert
            \Vert
                \partial_t v(t)
            \Vert_{H^m(I)},
        \end{split}
    \end{align*}
    and
    \begin{align} \label{eq_bound_dt2_V1}
        \begin{split}
            & \Vert
                \partial_t^2 \mathcal{V}(t)
            \Vert_{H^m(I)}\\
            & \leq C \left \Vert
                \partial_t^2 [
                    (
                        F(\tau_1(t))
                        - F(\tau_2(t))
                    )
                    v(t)
                ]
            \right \Vert_{H^m(I)}\\
            & + C \left \vert
                \frac{d}{dt}
                (
                    F(\tau_1(t))
                    + F(\tau_2(t))
                )
            \right \vert
            \left \Vert
                \partial_t [
                    (
                        F(\tau_1(t))
                        - F(\tau_2(t))
                    )
                    v(t)
                ]
            \right \Vert_{H^m(I)}\\
            & + C \left \vert
                \frac{d}{dt}(
                    F(\tau_1(t))
                    + F(\tau_2(t))
                )
            \right \vert^2
            \left \Vert
                [
                    F(\tau_1(t))
                    - F(\tau_2(t))
                ]
                v(t)
            \right \Vert_{H^m(I)}\\
            & + C \left \vert
                \frac{d^2}{dt^2}(
                    F(\tau_1(t))
                    + F(\tau_2(t))
                )
            \right \vert
            \left \Vert
                [
                    F(\tau_1(t))
                    - F(\tau_2(t))
                ]
                v(t)
            \right \Vert_{H^m(I)}\\
            & =: I_1 + I_2 + I_3 + I_4.
        \end{split}
    \end{align}
    We have proved (ii).
    Elementary calculations lead to
    \begin{align*}
        \begin{split}
            I_1
            & \leq C (
                \sum_{j=1,2}
                    \vert
                        \tau_j^\prime(t)
                    \vert^2
                \vert
                    \tau_1(t)
                    - \tau_2(t)
                \vert
                + 
                    \vert
                        \tau_j^{\prime\prime}(t)
                    \vert
                \vert
                    \tau_1(t)
                    - \tau_2(t)
                \vert
            \\
            &
                \quad\quad\quad\quad
                +
                \vert
                        \tau_j^\prime(t)
                    \vert
                \vert
                    \tau_1^\prime(t)
                    - \tau_2^\prime(t)
                \vert
                +
                \vert
                    \tau_1^{\prime\prime}(t)
                    - \tau_2^{\prime\prime}(t)
                \vert
            )
            \Vert
                v(t)
            \Vert_{H^m(I)}\\
            & + C (
                \sum_{j=1,2}
                    \vert
                        \tau_j^\prime(t)
                    \vert
                \vert
                    \tau_1(t)
                    - \tau_2(t)
                \vert
                +
                \vert
                    \tau_1^\prime(t)
                    - \tau_2^\prime(t)
                \vert
            )
            \Vert
                \partial_t v(t)
            \Vert_{H^m(I)}\\
            & + C
            \vert
                \tau_1(t)
                - \tau_2(t)
            \vert
            \Vert
                \partial_t^2 v(t)
            \Vert_{H^m(I)},
        \end{split}
    \end{align*}
    \begin{align*}
        \begin{split}
            I_2
            & \leq C
            %\sum_{j=1,2}
            %    \vert
            %        \tau_j^\prime(t)
            %    \vert
            (
                \sum_{j=1,2}
                    \vert
                        \tau_j^\prime(t)
                    \vert^2
                    \vert
                        \tau_1(t)
                        - \tau_2(t)
                    \vert
                    +
                    \vert
                        \tau_j^\prime(t)
                    \vert
                    \vert
                        \tau_1^\prime(t)
                        - \tau_2^\prime(t)
                    \vert
            )
            \Vert
                v(t)
            \Vert_{H^m(I)}\\
            & + C
            \sum_{j=1,2}
                \vert
                    \tau_j^\prime(t)
                \vert
            \vert
                \tau_1(t)
                - \tau_2(t)
            \vert
            \Vert
                \partial_t v(t)
            \Vert_{H^m(I)},
        \end{split}
    \end{align*}
    and
    \begin{align*}
        \begin{split}
            I_3
            + I_4
            & \leq C (
                \sum_{j=1,2}
                    \vert
                        \tau_j^\prime(t)
                    \vert^2
                +
                    \vert
                        \tau_j^{\prime\prime}(t)
                    \vert
            )
            \vert
                \tau_1(t)
                - \tau_2(t)
            \vert
            \Vert
                v(t)
            \Vert_{H^m(I)}.
        \end{split}
    \end{align*}
    We have proved (iii).
\end{proof}

\subsection{Construction of the solution $v_1$ and $v_2$ for a given $\tau_1$}

    We seek the solution $(v_1, v_2)$ to
    \begin{equation} \label{eq_abstract_filter_clogging_equation_in_the_interior}
        \begin{aligned}
            &\partial_t v_1
            - \partial_x^2 v_1
            = - c(\tau_1) \partial_x v_1
            - N_{v,1}(v_1) v_2
            + f
            & t > 0,
            & \, x \in I,\\
            &\partial_t v_2
            - \partial_x^2 v_2
            = -  c(\tau_1) \partial_x v_2
            + N_{v,2}(v_1, v_2) v_2
            & t > 0,
            & \, x \in I,\\
            &B_1(v_1; F(\tau_1))
            = 0, \quad 
            B_2(v_1; F(\tau_1))
            =0,
            & t>0,
            & \,\\
            &B_1(v_2; F(\tau_1))
            = 0, \quad
            B_2(v_2; F(\tau_1))
            =0,
            & t>0,
            & \,\\
            &v_1
            = v_{1,0}, \quad
            v_2
            = v_{2,0},
            & t=0,
            & \, x \in I,
        \end{aligned}
    \end{equation}
    for a given function $\tau_1 \in X_{\sigma, T} \hookrightarrow L^\infty(0,T) \cap H^2(0,T)$.
    % by means of the Leray-Schauder principle.
    The non-linear external forces are not so singular.
    The nonlinearity from the boundary conditions are more involved.
    We assume that the functions $\psi_1, \psi_2$ are such that
    \begin{align*}
        \psi_1, \psi_2
        &\in H^1_tH^1_x(Q_T) \cap L^2_t H^3_x(Q_T)
        \hookrightarrow H^{\theta}_t H^{1 + 2(1-\theta)}_x(Q_T).
    \end{align*}
    The embedding is derived from using the extension to a function on $Q_T$ to $\Real^2$ and the interpolation inequality.
    To simplify notation, we use the following symbols:
    \begin{align*}
        \begin{split}
            & A_v(\psi_j)
            := (
                \Vert
                    \psi_j
                \Vert_{H^1_tL^2_x(Q_T)}^2
                + \Vert
                    \psi_j
                \Vert_{L^2_t H^3_x(Q_T)}^2
            )^{1/2}, \quad
            j=1,2,\\
            &A_\tau(\tau_1)
            := (
                \Vert
                    \tau_1
                \Vert_{H^2(0,T)}^2
            )^{1/2},\\
            &B_{0,v}
            := (
                \sum_{j=1,2}
                    \Vert
                        v_{j,0}
                    \Vert_{H^3(I)}^2
            )^{1/2},
            B_{0,\sigma}
            := (
                \sum_{j=1,2}
                    \vert
                        \sigma_{j,0}
                    \vert^2
            )^{1/2},\\
            &B_f(f)
            := (
                \Vert
                    f
                \Vert_{H^1_tL^2_x(Q_T)}^2
                + \Vert
                    f
                \Vert_{L^2_tH^2_x(Q_T)}^2
            )^{1/2}.
        \end{split}
    \end{align*}
    We denote the external forces in the left-hand side of (\ref{eq_abstract_filter_clogging_equation_in_the_interior}) by
    \begin{align*}
        \tilde{N}_{v,1}(\psi_1, \psi_2; \tau_1)
        : = - c(\tau_1) \partial_x \psi_1
        - N_{v,1}(\psi_1) \psi_2
        + f,\\
        \tilde{N}_{v,2}(\psi_1, \psi_2; \tau_1)
        : = - c(\tau_1) \partial_x \psi_2
        + N_{v,2}(\psi_1, \psi_2) \psi_2.
    \end{align*}
    We denote the solution operator for the equation
    \begin{equation} \label{eq_abstract_filter_clogging_equation_in_the_interior_for_fixed_point_theorem}
        \begin{aligned}
            &\partial_t v_1
            - \partial_x^2 v_1
            = \tilde{N}_{v,1}(\psi_1, \psi_2; \tau_1)
            & t > 0,
            & \, x \in I,\\
            &\partial_t v_2
            - \partial_x^2 v_2
            = \tilde{N}_{v,2}(\psi_1, \psi_2; \tau_1)
            & t > 0,
            & \, x \in I,\\
            &B_1(v_1; F(\tau_1))
            = 0, \quad 
            B_2(v_1; F(\tau_1))
            =0,
            & t>0,
            & \,\\
            &B_1(v_2; F(\tau_1))
            = 0, \quad
            B_2(v_2; F(\tau_1))
            =0,
            & t>0,
            & \,
        \end{aligned}
    \end{equation}
    as
    \begin{align*}
        v_1
        = \mathcal{S}_{v,1}(\psi_1, \psi_2; \tau_1), \quad
        v_2
        = \mathcal{S}_{v,2}(\psi_1, \psi_2; \tau_1).
    \end{align*}
    \begin{theorem} \label{thm_existence_of_v_1_v_2_for_given_tau_1}
        Let $T>0$.
        Let the functions $v_{j,0} \in H^3(I)$ and $\tau_1 \in X_{\sigma,T}$ be positive and $f \in H^1_tL^2_x(Q_T) \cap L^2_tH^2_x(Q_T)$ be non-negative.
        Then there exist a unique non-negative solution
        \begin{align*}
            v_1, v_2 \in H^2_tL^2_x(Q_T) \cap L^4_tH^4_x(Q_T)
            \hookrightarrow H^{2\theta}_t H^{4(1-\theta)}_x(Q_T), \quad
            \theta \in [0,1],
        \end{align*}
        to (\ref{eq_abstract_filter_clogging_equation_in_the_interior_for_fixed_point_theorem}) such that
        \begin{align*}
            \sum_{j=1,2} (
                \Vert
                    v_j
                \Vert_{H^2_tL^2_x(Q_T)}
                + \Vert
                    v_j
                \Vert_{L^2_tH^4_x(Q_T)}
            )
            \leq C
        \end{align*}
        for some $C>0$, which depends continuously on $T$, $B_{0,v}$, $B_f(f)$, and $\Vert \tau_1 \Vert_{X_{\sigma,T}}$.
    \end{theorem}
    \begin{proof}
        We find from Propositions \ref{prop_nonlinear_estimates_for_Nv} and \ref{prop_nonlinear_estimates_for_cdv} that
        \begin{align} \label{eq_L2_estimate_for_Nvj}
            \begin{split}
                & \Vert
                    \tilde{N}_{v,1}(\psi_1, \psi_2; \tau_1)
                \Vert_{L^2_tL^2_x(Q_T)}^2
                + \Vert
                    \tilde{N}_{v,2}(\psi_1, \psi_2; \tau_1)
                \Vert_{L^2_tL^2_x(Q_T)}^2\\
                & \leq C
                \sum_{j=1,2}
                    \Vert
                        \psi_j
                    \Vert_{L^2_t H^1_x(Q_T)}^2\\
                & + C
                \Vert
                    \psi_2
                \Vert_{L^2_t L^2_x(Q_T)}^2
                + C \Vert
                    \psi_2
                \Vert_{L^\infty_t L^2_x(Q_T)}^2
                \Vert
                    \psi_2
                \Vert_{L^2_t H^1_x(Q_T)}^2\\
                & + C B_f(f)^2.
            \end{split}
        \end{align}
        Similarly, we also deduce that
        \begin{align} \label{eq_L2_estimate_for_dt_Nvj}
            \begin{split}
                & \Vert
                    \partial_t \tilde{N}_{v,1}(\psi_1, \psi_2; \tau_1)
                \Vert_{L^2_tL^2_x(Q_T)}^2
                + \Vert
                    \partial_t \tilde{N}_{v,2}(\psi_1, \psi_2; \tau_1)
                \Vert_{L^2_tL^2_x(Q_T)}^2\\
                & \leq C
                \sum_{j=1,2}
                    (
                        \Vert
                            \tau_1^\prime
                        \Vert_{L^2(0,T)}^2
                        \Vert
                            \psi_j
                        \Vert_{L^\infty_t H^1_x(Q_T)}^2
                        + \Vert
                            \psi_j
                        \Vert_{H^1_t H^1_x(Q_T)}^2
                    )\\
                & + C
                    (
                        \Vert
                            \psi_1
                        \Vert_{H^1_t H^1_x(Q_T)}^2
                        \Vert
                            \psi_2
                        \Vert_{L^\infty_t H^1_x(Q_T)}^2\\
                & \quad\quad
                        + \Vert
                            \psi_2
                        \Vert_{H^1_t H^1_x(Q_T)}^2
                        + \Vert
                            \psi_2
                        \Vert_{H^1_t L^2_x(Q_T)}^2
                        \Vert
                            \psi_2
                        \Vert_{L^\infty_t H^1_x(Q_T)}^2
                    )\\
                & + B_f(f)^2,
            \end{split}
        \end{align}
        and
        \begin{align}\label{eq_L2_estimate_for_dx2_Nvj}
            \begin{split}
                & \Vert
                    \partial_x^2 \tilde{N}_{v,1}(\psi_1, \psi_2; \tau_1)
                \Vert_{L^2_tL^2_x(Q_T)}^2
                + \Vert
                    \partial_x^2 \tilde{N}_{v,2}(\psi_1, \psi_2; \tau_1)
                \Vert_{L^2_tL^2_x(Q_T)}^2\\
                & \leq C
                \sum_{j=1,2}
                    \Vert
                        \psi_j
                    \Vert_{L^2_t H^3_x(Q_T)}^2\\
                & + C
                    (
                        \Vert
                            \psi_1
                        \Vert_{L^\infty_t H^1_x(Q_T)}^2
                        \Vert
                            \psi_1
                        \Vert_{L^2_t H^2_x(Q_T)}^2
                        \Vert
                            \psi_2
                        \Vert_{L^\infty_t H^1_x(Q_T)}^2\\
                    & \quad\quad\quad 
                        + \Vert
                            \psi_1
                        \Vert_{H^1_t H^2_x(Q_T)}^2
                        \Vert
                            \psi_2
                        \Vert_{L^\infty_t H^1_x(Q_T)}^2\\
                    & \quad\quad\quad
                        + \Vert
                            \psi_1
                        \Vert_{L^\infty_t H^1_x(Q_T)}^2
                        \Vert
                            \psi_2
                        \Vert_{L^2_t H^2_x(Q_T)}^2
                        %+ \Vert
                        %    \psi_1
                        %\Vert_{H^1_t H^1_x(Q_T)}^2
                        %\Vert
                        %    \psi_2
                        %\Vert_{L^2_t H^2_x(Q_T)}^2
                    )\\
                & + C
                        \Vert
                            \psi_2
                        \Vert_{L^\infty_t H^1_x(Q_T)}^2
                        \Vert
                            \psi_2
                        \Vert_{L^2_t H^2_x(Q_T)}^2
                        %+ \Vert
                        %    \psi_2
                        %\Vert_{H^1_t H^1_x(Q_T)}^2
                        %\Vert
                        %    \psi_2
                        %\Vert_{L^2_t H^2_x(Q_T)}^2
                    \\
                & + C B_f(f)^2.
            \end{split}
        \end{align}
        Therefore, we find from Lemma \ref{lemma_H2m_a_priori_estimates} and the embedding $H^1(0,T) \hookrightarrow L^\infty(0,T)$ that the solution $(v_1, v_2)$ satisfy
        \begin{align*}
            v_1, v_2
            \in H^2_t L^2_x(Q_T)
            \cap H^1_t H^2_x(Q_T)
            \cap L^2_t H^4_x(Q_T)
        \end{align*}
        and bounded.
        Moreover, we see that the solution operators $\mathcal{S}_{v, j}(\cdot, \cdot; \tau_1)$ are relatively compact operators in $H^1_tH^1_x(Q_T) \cap L^2_t H^3_x(Q_T)$.
        We next prove that $\mathcal{S}_{v, j}(\cdot, \cdot; \tau_1)$ are continuous.
        Let
        \begin{align*}
            \psi_3, \psi_4
            \in H^1_tH^1_x(Q_T) \cap L^2_t H^3_x(Q_T).
        \end{align*}
        We denote the differences by
        \begin{align*}
            & \tilde{\mathcal{N}}_{v,1}(\psi_1, \psi_2; \psi_3, \psi_4; \tau_1)\\
            & \quad := - c(\tau_1) (\psi_1 - \psi_3)
            + (
                \tilde{N}_{v,1}(\psi_1, \psi_2; \tau_1)
                - \tilde{N}_{v,1}(\psi_3, \psi_4; \tau_1)
            ),\\
            & \tilde{\mathcal{N}}_{v,2}(\psi_1, \psi_2; \psi_3, \psi_4; \tau_1)\\
            & \quad := - c(\tau_1) (\psi_2 - \psi_4)
            - (
                \tilde{N}_{v,2}(\psi_1, \psi_2; \tau_1)
                + \tilde{N}_{v,2}(\psi_3, \psi_4; \tau_1)
            ).
        \end{align*}
        We apply Propositions \ref{prop_estimates_for_difference_between_nonlinear_terms} and \ref{prop_estimates_for_difference_between_convection_terms} to $\tilde{\mathcal{N}}_{v,1}$ and $\tilde{\mathcal{N}}_{v,2}$ as $\tau_1=\tau_2$ to see that
        \begin{align} \label{eq_L2_estimate_for_difference_Nvj}
            \begin{split}
                & \Vert
                    \tilde{\mathcal{N}}_{v,1}(\psi_1, \psi_2; \psi_3, \psi_4; \tau_1)
                \Vert_{L^2_tL^2_x(Q_T)}^2
                + \Vert
                    \tilde{\mathcal{N}}_{v,2}(\psi_1, \psi_2; \psi_3, \psi_4; \tau_1)
                \Vert_{L^2_tL^2_x(Q_T)}^2\\
                & \leq C
                \sum_{j=1,2}
                    \Vert
                        \psi_j - \psi_{j+2}
                    \Vert_{L^2_t H^1_x(Q_T)}^2\\
                & + C
                (
                    1
                    + \sum_{j=1,2,3,4}
                        \Vert
                            \psi_j
                        \Vert_{L^2_tH^1_x(Q_T)}^2
                )
                \sum_{j=1,2}
                    \Vert
                        \psi_j
                        - \psi_{j+2}
                    \Vert_{L^\infty_tL^2_x(Q_T)}^2.
            \end{split}
        \end{align}
        Similarly, we also deduce that
        \begin{align}\label{eq_L2_estimate_for_difference_dt_Nvj}
            \begin{split}
                & \Vert
                    \partial_t \tilde{\mathcal{N}}_{v,1}(\psi_1, \psi_2; \psi_3, \psi_4; \tau_1)
                \Vert_{L^2_tL^2_x(Q_T)}^2\\
                & + \Vert
                    \partial_t \tilde{\mathcal{N}}_{v,2}(\psi_1, \psi_2; \psi_3, \psi_4; \tau_1)
                \Vert_{L^2_tL^2_x(Q_T)}^2\\
                & \leq C
                (
                    \Vert
                        \tau_1^\prime
                    \Vert_{L^2(0,T)}^2
                    \sum_{j=1,2}
                        \Vert
                            \psi_j
                            - \psi_{j+2}
                        \Vert_{L^\infty_t H^1_x(Q_T)}^2\\
                & \quad\quad\quad
                    + \sum_{j=1,2}
                        \Vert
                            \psi_j
                            - \psi_{j+2}
                        \Vert_{H^1_t H^1_x(Q_T)}^2
                )\\
                & + C[
                    (
                        \sum_{j=1,3}
                            \Vert
                                \psi_j
                            \Vert_{H^1_t L^2_x(Q_T)}^2
                        \Vert
                            \psi_2
                        \Vert_{L^\infty_t H^1_x(Q_T)}^2\\
                & \quad\quad\quad\quad\quad\quad
                        + \Vert
                            \psi_2
                        \Vert_{H^1_t H^1_x(Q_T)}^2
                    )
                    \Vert
                        \psi_1
                        - \psi_3
                    \Vert_{L^\infty_t H^1_x(Q_T)}^2\\
                & \quad\quad\quad
                    + \Vert
                        \psi_2
                    \Vert_{L^\infty_t H^1_x(Q_T)}^2
                    \Vert
                        \psi_1
                        - \psi_3
                    \Vert_{H^1_t L^2_x(Q_T)}^2\\
                & \quad\quad\quad
                    + \Vert
                        \psi_3
                    \Vert_{H^1_t L^2_x(Q_T)}^2
                    \Vert
                        \psi_1
                        - \psi_3
                    \Vert_{L^\infty_t H^1_x(Q_T)}^2
                    + \Vert
                        \psi_2
                        - \psi_4
                    \Vert_{H^1_t L^2_x(Q_T)}^2
                ]\\
                & + C \sum_{j=2,4}
                    [
                        \Vert
                            \psi_j
                        \Vert_{H^1_t L^2_x(Q_T)}^2
                        \Vert
                            \psi_2
                            - \psi_4
                        \Vert_{L^\infty_t H^1_x(Q_T)}^2\\
                & \quad\quad\quad
                        + \Vert
                                \psi_j
                            \Vert_{L^\infty_t H^1_x(Q_T)}^2
                        \Vert
                            \psi_2
                            - \psi_4
                        \Vert_{H^1_t L^2_x(Q_T)}^2
                    ],
            \end{split}
        \end{align}
        and
        \begin{align}\label{eq_L2_estimate_for_difference_dx2_Nvj}
            \begin{split}
                & \Vert
                    \partial_x^2 \tilde{\mathcal{N}}_{v,1}(\psi_1, \psi_2; \psi_3, \psi_4; \tau_1)
                \Vert_{L^2_tL^2_x(Q_T)}^2\\
                & + \Vert
                    \partial_x^2 \tilde{\mathcal{N}}_{v,2}(\psi_1, \psi_2; \psi_3, \psi_4; \tau_1)
                \Vert_{L^2_tL^2_x(Q_T)}^2\\
                & \leq C
                \sum_{j,k=1,3}
                    (
                        \Vert
                            \psi_j
                        \Vert_{L^\infty_tH^1_x(Q_T)}^2
                        \Vert
                            \psi_j
                        \Vert_{L^2_tH^2_x(Q_T)}^2
                        \Vert
                            \psi_{k+1}
                        \Vert_{L^\infty_tH^1_x(Q_T)}^2
                    \\
                & \quad\quad\quad
                        + \Vert
                            \psi_j
                        \Vert_{L^2_tH^2_x(Q_T)}^2
                        \Vert
                            \psi_{k+1}
                        \Vert_{L^\infty_tH^1_x(Q_T)}^2\\
                & \quad\quad\quad
                        + \Vert
                            \psi_j
                        \Vert_{L^\infty_tH^1_x(Q_T)}^2
                        \Vert
                            \psi_{k+1}
                        \Vert_{L^t_tH^2_x(Q_T)}^2\\
                & \quad\quad\quad
                        + \Vert
                            \psi_{k+1}
                        \Vert_{L^2_tH^2_x(Q_T)}^2
                )
                \Vert
                    \psi_1
                    -  \psi_3
                \Vert_{L^\infty_tH^1_x(Q_T)}^2\\
                & + C \sum_{j=1,3}
                    (
                        1
                        + \Vert
                            \psi_j
                        \Vert_{L^\infty_tH^1_x(Q_T)}^2
                    )
                    \Vert
                        \psi_j
                    \Vert_{L^2_tH^2_x(Q_T)}^2
                \Vert
                    \psi_2
                    - \psi_4
                \Vert_{L^\infty_tH^1_x(Q_T)}^2\\
                & + C
                \Vert
                    \psi_2
                \Vert_{L^\infty_tH^1_x(Q_T)}^2
                \Vert
                    \psi_1
                    - \psi_2
                \Vert_{L^2_tH^2_x(Q_T)}^2
                + C
                \Vert
                    \psi_2
                    -  \psi_4
                \Vert_{L^2_tH^2_x(Q_T)}^2.
            \end{split}
        \end{align}
        Using Proposition \ref{prop_L2_estimate_for_linear_heat_eq} with external forces $\tilde{\mathcal{N}}_{v,j}(\psi_1, \psi_2; \psi_3, \psi_4;\tau_1)$ ($j=1,2$) and initial data set to zero, as well as (\ref{eq_L2_estimate_for_difference_Nvj}), we observe that the $L^\infty_tL^2_x$- and $L^2_tH^1_x$-norms for $v_j$ ($j=1,2$) depend continuously on $\psi_k$ ($k=1,2,3,4$).
        Therefore, we see from Proposition \ref{prop_L_infty_H1_estimate_for_v} that $L^\infty_tH^1_x$-, $H^1_tL^2_x$- and $L^2_tH^2_x$-norms for $v_j$ ($j=1,2$) depend continuously on $\psi_k$ ($k=1,2,3,4$).
        We similarly deduce by applying Corollaries \ref{cor_L2_a_priori_estimate_for_tildev} and \ref{cor_H3_estimate_for_v} to $v_j$ ($j=1,2$) that $W^{1,\infty}_tH^1_x$-, $L^\infty_tH^3_x$-, $H^2_tL^2_x$-, and $L^2_tH^4_x$-norms for $v_j$ ($j=1,2$) depend continuously on $\psi_k$ ($k=1,2,3,4$).
        We conclude that the solution operators $\mathcal{S}_{v, j}(\cdot, \cdot; \tau_1)$ are continue in $H^1_tH^1_x(Q_T) \cap L^2_t H^3_x(Q_T)$.

        We assume that the positive functions $v_j$ ($j=1,2$) are such that
        \begin{align*}
            v_j
            = \lambda \mathcal{S}_{v, j}(v_1, v_2; \tau_1)
        \end{align*}
        for some $\lambda \in [0,1]$.
        Then the pair $(v_1, v_2)$ is the solution to
        \begin{equation} \label{eq_abstract_filter_clogging_equation_in_the_interior_for_fixed_point_theorem_with_lambda}
            \begin{aligned}
                &\partial_t v_1
                - \partial_x^2 v_1
                = \lambda \tilde{N}_{v,1}(v_1, v_2; \tau_1)
                & t > 0,
                & \, x \in I,\\
                &\partial_t v_2
                - \partial_x^2 v_2
                = \lambda \tilde{N}_{v,2}(v_1, v_2; \tau_1)
                & t > 0,
                & \, x \in I,\\
                &B_1(v_1; F(\tau_1))
                = 0, \quad 
                B_2(v_1; F(\tau_1))
                =0,
                & t>0,
                & \,\\
                &B_1(v_2; F(\tau_1))
                = 0, \quad
                B_2(v_2; F(\tau_1))
                =0,
                & t>0,
                & \,
            \end{aligned}
        \end{equation}
        We prove the boundedness for $v_j$ in $H^1_tH^1_x(Q_T) \cap L^2_t H^3_x(Q_T)$ and, moreover $v_j \in H^2_tL^2_x(Q_T) \cap L^2_t H^4_x(Q_T) \cap C_t H^1_x(Q_T) \cap C^1_tH^1_x(Q_T)$ by a priori estimates in Section \ref{sec_linear_analysis}.
        However, we will also use the estimates for $v_j$ to prove the existence for the boundary equations in (\ref{eq_abstract_filter_clogging_equation}) by contraction mapping principle.
        Therefore, we have to identify the dependence of $\tau_1$ for energy bounds.
        The dependency is complicated.

        Note that the initial data $v_{1,0}, v_{2,0}$ are positive, the external force $f$ is non-negative, the solutions to (\ref{eq_abstract_filter_clogging_equation_in_the_interior_for_fixed_point_theorem_with_lambda}), and satisfy $v_1, v_2 \in BC(\overline{Q_T})$.
        Suppose that there exists a point $(x_0, t_0) \in \overline{Q_T}$ such that $v_1(\cdot, t), v_2(\cdot, t)>0$ for $0\leq t < t_0$ and $v_1(x_0, t_0)=0$ or $v_2(x_0, t_0)=0$.

        We first consider the case that $v_1(x_0, t_0)=0$.
        By the assumption $f \geq 0$, the bound from below such that
        \begin{align*}
            \partial_t v_1 - \partial_x^2 v_1  + \lambda c \partial_x v_1
            = - \lambda  N_{v,1}(v_1) v_2
            + \lambda f
            \geq \lambda f
            \geq 0,
        \end{align*}
        and the maximal principle, we see that $v_1(\cdot, t)$ for $t\leq t_0$ is a constant.
        By the boundary condition we deduce that $v_1(\cdot, t)\equiv 0$ for $t \leq t_0$.
        Therefore, $v_1$ is non-negative for $t \geq 0$.
        Since
        \begin{align*}
            \partial_t v_2 - \partial_x^2 v_2  + \lambda c \partial_x v_2 + \lambda \alpha_2 v_2^2
            = \lambda \alpha_1 N_{v,1}(v_1) v_2
            \geq 0,
        \end{align*}
        we can apply the similar argument for the case that $v_2(x_0, t_0)=0$ and $v_1(\cdot, t) > 0$ for $t \leq t_0$ to find that $v_2$ is non-negative.
        Therefore, we find that $v_1, v_2 \geq 0$ in $Q_T$.

        We estimate the equation (\ref{eq_abstract_filter_clogging_equation_in_the_interior_for_fixed_point_theorem_with_lambda}).
        We begin by $L^2$-energy estimates for $v_1$ and $v_2$.
        We invoke the bounds
        \begin{align*}
            - \int_I N_{v,1}(v_1) v_2 v_1dx
            \leq 0, \quad
            - \int_I v_2^2dx
            \leq 0,
        \end{align*}
        by the non-negativity of $v_1, v_2$.
        We use integration by parts as in Proposition \ref{prop_L2_estimate_for_linear_heat_eq} to see that
        \begin{align*}% \label{eq_L2_estimate_for_v1_to_Leray_Schauder}
            \Vert
                v_1(t)
            \Vert_{L^2(I)}^2
            + \int_0^t
                \Vert
                    \partial_x v_1(s)
                \Vert_{L^2(I)}^2
            ds
            \leq
            C e^t \left[
                \Vert
                    v_{1,0}
                \Vert_{L^2(I)}^2
                + \int_0^t
                    \Vert
                        f(s)
                    \Vert_{L^2(I)}^2
                ds
            \right]
        \end{align*}
        for some $\lambda$-independent constant $C>0$.
        We also have
        \begin{align*}% \label{eq_L2_estimate_for_v2_to_Leray_Schauder}
            \Vert
                v_2(t)
            \Vert_{L^2(I)}^2
            + \int_0^t
                \Vert
                    \partial_x v_2(s)
                \Vert_{L^2(I)}^2
            ds
            \leq
            C e^{2t}
            \Vert
                v_{2,0}
            \Vert_{L^2(I)}^2
        \end{align*}
        for all $\lambda \in [0,1]$ and some constant $C>0$.
        These estimates together with (\ref{eq_L2_estimate_for_Nvj}) implies that $N_{v,j}(v_1, v_2)$ ($j=1,2$) are bounded.
        Then we have
        \begin{align} \label{eq_L2_estimate_for_vj_to_Leray_Schauder}
            \begin{split}
                & \sum_{j=1,2}
                    \left[
                        \Vert
                            v_j(t)
                        \Vert_{L^2(I)}^2
                        + \int_0^t
                            \Vert
                                \partial_x v_j(s)
                            \Vert_{L^2(I)}^2
                        ds
                    \right]\\
                & \leq C e^{2t} (
                    B_{v,0}^2
                    + B_f(f)^2
                )
                : = C_{L^\infty_tL^2_x}(t).
            \end{split}
        \end{align}
        The positive function $C_{L^\infty_tL^2_x}(t)>0$ is independent of $\tau_1$.
        Let
        \begin{align*}
            \Psi^\ast(\tau_1;t)
            := \int_0^t
                \left(
                    1
                    + \vert
                        \tau_1^\prime(r)
                    \vert^2
                \right)
            dr.
        \end{align*}
        We find from Proposition \ref{prop_L_infty_H1_estimate_for_v} and (\ref{eq_L2_estimate_for_Nvj}) that
        \begin{align} \label{eq_H1_estimate_for_vj_to_Leray_Schauder}
            \begin{split}
                & \sum_{j=1,2}
                    \left[
                        \int_0^t
                            \Vert
                                \partial_s v_j(s)
                            \Vert_{L^2(I)}^2
                        ds
                        + \Vert
                            \partial_x v_j(t)
                        \Vert_{L^2(I)}^2
                        + \int_0^t
                            \Vert
                                \partial_x^2 v_j(s)
                            \Vert_{L^2(I)}^2
                        ds
                    \right]\\
                & \leq C \sum_{j=1,2} e^{
                        C \Psi^\ast(\tau_1;t)
                    }
                    \left[
                        \Vert
                                \partial_x v_{j,0}
                            \Vert_{L^2(I)}
                    \right.\\
                    & \quad\quad\quad
                        + \sup_{0<t<T} \Vert
                                v_j(s)
                            \Vert_{L^2(I)}^2
                        \Psi^\ast(\tau_1;t)
                        + \int_0^t
                            \Vert
                                f(s)
                            \Vert_{L^2(I)}^2
                        ds\\
                    & \left.
                        \quad\quad\quad
                        + \int_0^T
                            \Vert
                                v_j(s)
                            \Vert_{H^1(I)}^2
                        ds
                        + \sup_{0<t<T} \Vert
                            v_j(s)
                        \Vert_{L^2(I)}^2
                        \int_0^T
                            \Vert
                                v_j(s)
                            \Vert_{H^1(I)}^2
                        ds
                    \right]\\
                & \leq C e^{
                        C \Psi^\ast(\tau_1;t)
                    }
                    \left[
                        B_{v,0}^2
                        + C_{L^\infty_tL^2_x}(t) \Psi^\ast(\tau_1;t)
                    \right.\\
                    & \left.
                        \quad\quad\quad\quad\quad\quad\quad\quad\quad
                        + B_f(f)^2
                        + C_{L^\infty_tL^2_x}(t)
                        + C_{L^\infty_tL^2_x}(t)^2
                    \right]\\
                & =: C_{L^\infty_tH^1_x}(t) e^{
                        C \Psi^\ast(\tau_1;t)
                    }\left[
                        1
                        + \Psi^\ast(\tau_1;t)
                    \right]
            \end{split}
        \end{align}
        for all $\lambda \in [0,1]$, and some bounded positive and exponentially increasing functions $C_{L^\infty_tH^1_x}(t)>0$ on $[0,T]$ such that
        \begin{align*}
            C_{L^\infty_tH^1_x}(t)
            \leq C (1 + e^{4t})
        \end{align*}
        for some constant $C>0$.
        Combining Corollary \ref{cor_L2_a_priori_estimate_for_tildev} with this estimate and the bounds (\ref{eq_L2_estimate_for_dt_Nvj}) and (\ref{eq_L2_estimate_for_dx2_Nvj}), we find that
        \begin{align*}
            & \sum_{j=1,2}
                \left[
                    \Vert
                        \partial_t v_j (t)
                    \Vert_{L^2(I)}^2
                    + \int_0^t
                        \Vert
                            \partial_x \partial_s v_j(s)
                        \Vert_{L^2(I)}^2
                    ds
                \right]\\
            & \leq C
            \Vert
                \tau_1^\prime
            \Vert_{L^\infty(0,T)}^2
            \sum_{j=1,2}
                \left[
                    \Vert
                        v_j(t)
                    \Vert_{L^2(I)}^2
                    + \int_0^t
                        \Vert
                            \partial_x v(s)
                        \Vert_{L^2(I)}^2
                    ds
                \right]\\
            & + C e^t \sum_{j=1,2}
            \left[
                \Vert
                    v_{j,0}
                \Vert_{H^2(I)}^2
                + \Vert
                    \tau_1^\prime
                \Vert_{L^\infty(0,T)}^2
                \int_0^t
                    \Vert
                        v_j(s)
                    \Vert_{H^2(I)}^2
                    + \Vert
                        \partial_t v_j(s)
                    \Vert_{L^2(I)}^2
                ds
            \right.\\
            &
                \quad\quad\quad\quad
                +
                \sup_{0<t<T}\Vert
                    v_j(t)
                \Vert_{L^2(I)}^2
                (
                    \Vert
                        \tau_1^\prime
                    \Vert_{L^4(0,T)}^4
                    + \Vert
                        \tau_1^{\prime\prime}
                    \Vert_{L^2(0,T)}^2
                )
            \\
            & \left.
                \quad\quad\quad\quad
                + \int_0^t
                    \Vert
                        \partial_t \tilde{N}_{v,j}(v_1(s), v_2(s))
                    \Vert_{L^2(I)}^2
                ds 
            \right]\\
            & =: I_1 + e^t[I_2 + I_3 + I_4 + I_5],
        \end{align*}
        for all $\lambda \in [0,1]$.
        Since $\tau_1 \in X_{\sigma,T} \hookrightarrow H^2(0,T) \cap W^{1, \infty}(0,T) \cap H^{1,4}(0,T)$, the right-hand side makes sense.
        We observe the following estimates:
        \begin{itemize}
            \item
            \begin{align*}
                I_1
                \leq \Vert
                    \tau_1^\prime
                \Vert_{L^\infty(0,T)}^2
                C_{L^\infty_tL^2_x}(t).
            \end{align*}
            \item
            \begin{align*}
                I_2
                \leq B_{v,0}^2.
            \end{align*}
            \item
            \begin{align*}
                I_3
                \leq C C_{L^\infty_tH^1_x}(t)
                \Vert
                    \tau_1^\prime
                \Vert_{L^\infty(0,T)}^2
                e^{
                    C\Psi^\ast(\tau_1;t)
                }\left[
                    1
                    + \Psi^\ast(\tau_1;t)
                \right].
            \end{align*}
            \item
            \begin{align*}
                I_4
                & \leq C C_{L^\infty_tL^2_x}(t)
                (
                    \Vert
                        \tau_1^{\prime}
                    \Vert_{L^4(0,T)}^4
                    + \Vert
                        \tau_1^{\prime\prime}
                    \Vert_{L^2(0,T)}^2
                ).
            \end{align*}
            \item The estimates (\ref{eq_L2_estimate_for_dt_Nvj}), (\ref{eq_L2_estimate_for_vj_to_Leray_Schauder}), and (\ref{eq_H1_estimate_for_vj_to_Leray_Schauder}) lead to
            \begin{align*}
                I_5
                & \leq C
                \sum_{j=1,2}
                    (
                        \Vert
                            \tau_1^\prime
                        \Vert_{L^2(0,T)}^2
                        \Vert
                            v_j
                        \Vert_{L^\infty_t H^1_x(Q_T)}^2
                        + \Vert
                            v_j
                        \Vert_{L^2_t H^2_x(Q_T)}^2
                    )\\
                & + C
                    (
                        \Vert
                            v_1
                        \Vert_{L^\infty_t H^1_x(Q_T)}^2
                        \Vert
                            v_2
                        \Vert_{H^1_t L^2_x(Q_T)}^2
                        + \Vert
                            v_2
                        \Vert_{L^2_t H^1_x(Q_T)}^2\\
                & \quad\quad\quad\quad\quad\quad
                        + \Vert
                            v_2
                        \Vert_{H^1_t L^2_x(Q_T)}^2
                        \Vert
                            v_2
                        \Vert_{L^\infty_t H^1_x(Q_T)}^2
                    )\\
                & + B_f(f)^2\\
                & \leq C C_{L^\infty_tH^1_x}(t) \Vert
                    \tau_1^\prime
                \Vert_{L^2(0,T)}^2
                e^{
                    C\Psi^\ast(\tau_1;t)
                }\left[
                    1
                    + \Psi^\ast(\tau_1;t)
                \right]\\
                & + C C_{L^\infty_tH^1_x}(t) e^{
                    C\Psi^\ast(\tau_1;t)
                }\left[
                    1
                    + \Psi^\ast(\tau_1;t)
                \right]\\
                & + C C_{L^\infty_tH^1_x}(t) e^{
                    C\Psi^\ast(\tau_1;t)
                }\left[
                    1
                    + \Psi^\ast(\tau_1;t)
                \right]^2
                + B_f(f)^2.
            \end{align*}
        \end{itemize}
        Therefore, using the inequality
        \begin{align} \label{eq_trivial_inequality}
            x^m
            \leq 1+ x^n, \quad
            \text{for $x>0$, $0< m \leq n$,}
        \end{align}
        we conclude that
        \begin{align} \label{eq_H2_estimate_for_vj_to_Leray_Schauder}
            \begin{split}
                & \sum_{j=1,2}
                    \left[
                        \Vert
                            \partial_t v_j (t)
                        \Vert_{L^2(I)}^2
                        + \int_0^t
                            \Vert
                                \partial_x \partial_s v_j(s)
                            \Vert_{L^2(I)}^2
                        ds
                    \right]\\
                & \leq C_{W^{1,\infty}_tL^2_x}(t)
                \left(
                    1
                    + \Vert
                        \tau_1^\prime
                    \Vert_{L^\infty(0,T)}^2
                    + \Vert
                        \tau_1^{\prime}
                    \Vert_{L^4(0,T)}^4
                    + \Vert
                        \tau_1^{\prime\prime}
                    \Vert_{L^2(0,T)}^2
                \right.\\
                & 
                \left.
                    \quad\quad\quad\quad\quad\quad\quad
                    + e^{
                        C \Psi^\ast(\tau_1;t)
                    }
                    (
                        1
                        + \Vert
                            \tau_1^\prime
                        \Vert_{L^2(0,T)}^2
                    )
                    (
                        1
                        + \Psi^\ast(\tau_1;t)^2
                    )
                \right)
            \end{split}
        \end{align}
        for some positive functions $C_{W^{1,\infty}_tL^2_x}(t)$ which is independent of $\tau_1$ and bounded as
        \begin{align} \label{eq_bound_for_expoinentiall_increasing_function}
            C_{W^{1,\infty}_tL^2_x}(t)
            \leq C_1
            (
                1
                + e^{C_2t}
            )
        \end{align}
        for some constants $C_1, C_2>0$.
        Similarly, we see from Corollary \ref{cor_H3_estimate_for_v} that
        \begin{align*}% \label{eq_H3_estimate_for_vj_to_Leray_Schauder}
            \begin{split}
                &\sum_{j=1,2}
                    \int_0^t
                        \Vert
                            \partial_s^2 v_j(s)
                        \Vert_{L^2(I)}^2
                    ds
                    + \Vert
                        \partial_x \partial_t v_j(t)
                    \Vert_{L^2(I)}^2
                    + \int_0^t
                        \Vert
                            \partial_x^2 \partial_s v_j(s)
                        \Vert_{L^2(I)}^2
                    ds\\
                & \leq C (
                    1
                    + e^{
                        C \Psi^\ast(\tau_1;t)
                    }
                )\\
                & \quad
                \times
                \left[
                    \Vert
                        \tau_1^\prime
                    \Vert_{L^\infty(0,T)}^2
                    \int_0^t
                        (
                            \Vert
                                v_j(s)
                            \Vert_{H^2(I)}^2
                            + \Vert
                                \partial_t v_j(s)
                            \Vert_{L^2(I)}^2
                        )
                    ds
                \right.\\
                & \left.
                    \quad\quad\quad
                    +
                    \int_0^t
                            \vert
                                \tau_1^\prime(s)
                            \vert^4
                            + \vert
                                \tau_1^{\prime\prime}(s)
                            \vert^2
                    ds
                    \sup_{0<s<t} \Vert
                        v_j(s)
                    \Vert_{L^2(I)}^2
                \right]\\
                & + C
                \Vert
                    \tau_1^\prime
                \Vert_{L^\infty(0,T)}^2
                \Vert
                    v_j(t)
                \Vert_{H^1(I)}^2\\
                & + C e^{
                    C\Psi^\ast(\tau_1;t)
                }
                \left[
                    \Vert
                        v_{j,0}
                    \Vert_{H^3(I)}
                    + \sup_{0<s<t} \Vert
                            \partial_t v(s)
                        \Vert_{L^2(I)}
                    \Psi^\ast(\tau_1;t)
                \right.\\
                & \left.
                    + \int_0^t
                        \Vert
                            \partial_s \tilde{N}_{v,j}(v_1,v_2)(s)
                        \Vert_{L^2(I)}^2
                    ds
                \right]\\
                & =:
                (
                    1
                    + e^{
                        C \Psi^\ast(\tau_1;t)
                    }
                )
                (
                    I_1 + I_2
                )
                + I_3
                + e^{
                    C \Psi^\ast(\tau_1;t)
                }
                (
                    I_4
                    + I_5
                    + I_6
                ).
            \end{split}
        \end{align*}
        We observe the following estimates:
        \begin{itemize}
            \item
            \begin{align*}
                I_1
                & \leq C
                \Vert
                    \tau_1^\prime
                \Vert_{L^\infty(0,T)}^2
                \int_0^t
                    (
                        \Vert
                            v_j(s)
                        \Vert_{H^2(I)}^2
                        + \Vert
                            \partial_t v_j(s)
                        \Vert_{L^2(I)}^2
                    )
                ds\\
                & \leq C C_{L^\infty_tH^1_x}(t) e^{
                    C\Psi^\ast(\tau_1;t)
                }\left(
                    1
                    + \Psi^\ast(\tau_1;t)
                \right).
            \end{align*}
            \item
            \begin{align*}
                I_2
                & \leq
                C (
                    \Vert
                        \tau_1^\prime
                    \Vert_{L^4(0,T)}^4
                    + \Vert
                        \tau_1^\prime
                    \Vert_{L^2(0,T)}^2
                )
                C_{L^\infty_tL^2_x}(t).
            \end{align*}
            \item 
            \begin{align*}
                I_3
                \leq C 
                \Vert
                    \tau_1^\prime
                \Vert_{L^\infty(0,T)}^2
                C_{L^\infty_tH^1_x}(t)
                e^{
                    C\Psi^\ast(\tau_1;t)
                }\left(
                    1
                    + \Psi^\ast(\tau_1;t)
                \right).
            \end{align*}
            \item 
            \begin{align*}
                I_4
                \leq C B_{v,0}^2
            \end{align*}
            \item 
            \begin{align*}
                I_5
                & \leq C \Psi^\ast(\tau_1;t) C_{W^{1,\infty}_tL^2_x}(t)
                \left(
                    1
                    + \Vert
                        \tau_1^\prime
                    \Vert_{L^\infty(0,T)}^2
                    + \Vert
                        \tau_1^{\prime}
                    \Vert_{L^4(0,T)}^4
                    + \Vert
                        \tau_1^{\prime\prime}
                    \Vert_{L^2(0,T)}^2
                \right.\\
                & 
                \left.
                    \quad\quad\quad\quad\quad\quad\quad
                    + e^{
                        C \Psi^\ast(\tau_1;t)
                    }
                    (
                        1
                        + \Vert
                            \tau_1^\prime
                        \Vert_{L^2(0,T)}^2
                    )
                    (
                        1
                        + \Psi^\ast(\tau_1;t)^2
                    )
                \right)
            \end{align*}
            \item We find from (\ref{eq_L2_estimate_for_dt_Nvj}) that
            \begin{align*}
                I_6
                & \leq C C_{L^\infty_tH^1_x}(t) \Vert
                    \tau_1^\prime
                \Vert_{L^2(0,T)}^2
                e^{
                    C\Psi^\ast(\tau_1;t)
                }\left[
                    1
                    + \Psi^\ast(\tau_1;t)
                \right]\\
                & + C C_{W^{1,\infty}_tL^2_x}(t)
                \left(
                    1
                    + \Vert
                        \tau_1^\prime
                    \Vert_{L^\infty(0,T)}^2
                    + \Vert
                        \tau_1^{\prime}
                    \Vert_{L^4(0,T)}^4
                    + \Vert
                        \tau_1^{\prime\prime}
                    \Vert_{L^2(0,T)}^2
                \right.\\
                & 
                \left.
                    \quad\quad\quad\quad\quad\quad\quad
                    + e^{
                        C \Psi^\ast(\tau_1;t)
                    }
                    (
                        1
                        + \Vert
                            \tau_1^\prime
                        \Vert_{L^2(0,T)}^2
                    )
                    (
                        1
                        + \Psi^\ast(\tau_1;t)^2
                    )
                \right)\\
                & \quad\quad\quad\quad\quad\quad\quad \times
                \left(
                    1
                    + C_{L^\infty_tH^1_x}(t) e^{
                            C \Psi^\ast(\tau_1;t)
                        }
                        (
                            1
                            + \Psi^\ast(\tau_1;t)
                        )
                \right)\\
                & + B_f(f)^2.
            \end{align*}
        \end{itemize}
        Therefore, using (\ref{eq_trivial_inequality}), we conclude that
        \begin{align}\label{eq_H3_estimate_for_vj_to_Leray_Schauder}
            \begin{split}
                &\sum_{j=1,2}
                    \int_0^t
                        \Vert
                            \partial_s^2 v_j(s)
                        \Vert_{L^2(I)}^2
                    ds
                    + \Vert
                        \partial_x \partial_t v_j(t)
                    \Vert_{L^2(I)}^2
                    + \int_0^t
                        \Vert
                            \partial_x^2 \partial_s v_j(s)
                        \Vert_{L^2(I)}^2
                    ds\\
                & \leq C_{L^\infty_tH^3_x}(t)
                (
                    1
                    + \Psi^\ast(\tau_1;t)
                )
                (
                    1
                    + e^{C\Psi^\ast(\tau_1;t)}
                )
                \\
                & \quad
                \times
                \left[
                    (
                        1
                        +\Vert
                            \tau_1^\prime
                        \Vert_{L^\infty(0,T)}^2
                        + \Vert
                            \tau_1^{\prime}
                        \Vert_{L^4(0,T)}^4
                        + \Vert
                            \tau_1^{\prime\prime}
                        \Vert_{L^2(0,T)}^2
                    )
                \right.\\
                & \left.
                    \quad\quad\quad
                    +
                    (
                        1
                        + \Vert
                            \tau_1^\prime
                        \Vert_{L^2(0,T)}^2
                    )
                    \Psi^\ast(\tau_1;t)^2
                \right].
            \end{split}
        \end{align}
        for some positive functions $C_{L^\infty_tH^3_x}(t)>0$ which is independent of $\tau_1$ and bounded as (\ref{eq_bound_for_expoinentiall_increasing_function}).
        We use Corollary \ref{cor_H3_estimate_for_v} again to obtain
        \begin{align*}
            &\sum_{j=1,2}
                \left[
                    \Vert
                        \partial_x^3 v_j(t)
                    \Vert_{L^2(I)}^2
                    + \int_0^t
                        \Vert
                            \partial_x^4 v_j(s)
                        \Vert_{L^2(I)}^2
                    ds
                \right]
            < \infty
        \end{align*}
        Therefore, we find from the Leray-Schauder fixed point theorem that there exists a positive solution to (\ref{eq_abstract_filter_clogging_equation_in_the_interior}) such that
        \begin{align*}
            v_j
            \in H^1_tH^1_x(Q_T) \cap L^2_t H^3_x(Q_T).
        \end{align*}
        %The positivity is due to the maximal principle.
        Moreover, $v_j$ satisfy
        \begin{align*}
            \sum_{j=1,2}
                \left[
                    \Vert
                        v_j(s)
                    \Vert_{H^2_tL^2_x(Q_T)}^2
                    + \Vert
                        \partial_x^3 v_j(t)
                    \Vert_{L^\infty_t H^3_x(Q_T)}^2
                    + \Vert
                        v_j(s)
                    \Vert_{L^2_tH^4_x(Q_T)}^2
                \right]
            < \infty.
        \end{align*}

        We finally prove the uniqueness of the solution.
        Suppose that there exist another solution $(\tilde{v}_1, \tilde{v}_2) \in (H^1_tH^1_x(Q_T) \cap L^2_t H^3_x(Q_T))^2$ to (\ref{eq_abstract_filter_clogging_equation_in_the_interior}) associated with same initial data.
        Similar to (\ref{eq_L2_estimate_for_difference_Nvj}), using Propositions \ref{prop_estimates_for_difference_between_nonlinear_terms} and \ref{prop_estimates_for_difference_between_convection_terms}, we observe that
        \begin{align*}
            \begin{split}
                & \Vert
                    \tilde{\mathcal{N}}_{v,1}(v_1, v_2; \tilde{v}_1, \tilde{v}_2; \tau_1)
                \Vert_{L^2(I)}^2
                + \Vert
                    \tilde{\mathcal{N}}_{v,2}(v_1, v_2; \tilde{v}_1, \tilde{v}_2; \tau_1)
                \Vert_{L^2(I)}^2\\
                & \leq C
                \sum_{j=1,2}
                    \Vert
                        v_j - \tilde{v}_{j}
                    \Vert_{H^1(I)}^2\\
                & + C
                (
                    1
                    + \sum_{j=1,2}
                        \Vert
                            v_j
                        \Vert_{H^1(I)}^2
                    + \sum_{j=1,2}
                        \Vert
                            \tilde{v}_j
                        \Vert_{H^1(I)}^2
                )
                \sum_{j=1,2}
                    \Vert
                        v_j
                        - \tilde{v}_j
                    \Vert_{L^2(I)}^2.
            \end{split}
        \end{align*}
        Using same arguments as Propositions \ref{prop_L2_estimate_for_linear_heat_eq} and \ref{prop_L_infty_H1_estimate_for_v} and the estimate
        \begin{align*}
            \left \vert
                \frac{dF(\sigma(t))}{dt}
            \right \vert
            \leq C \vert
                \tau^\prime
            \vert,
        \end{align*}
        we deduce that
        \begin{align*}
            & \partial_t \left(
                \sum_{j=1,2}
                    \Vert
                            v_j(t)
                            - \tilde{v}_j(t)
                    \Vert_{L^2(I)}^2
                    + \Vert
                        \partial_x (
                            v_j(t)
                            - \tilde{v}_j(t)
                        )
                    \Vert_{L^2(I)}^2
            \right)\\
            & \leq C
            \left(
                1
                + \vert
                    \tau_1^\prime(t)
                \vert^2
                + \sum_{j=1,2}
                    \Vert
                        v_j
                    \Vert_{H^1(I)}^2
                + \sum_{j=1,2}
                    \Vert
                        \tilde{v}_j
                    \Vert_{H^1(I)}^2
            \right)
            \sum_{j=1,2}
                \Vert
                    v_j
                    - \tilde{v}_j
                \Vert_{L^2(I)}^2\\
            & + (
                1
                + \vert
                    \tau_1^\prime(t)
                \vert^2
            )
            \sum_{j=1,2}
                \Vert
                    \partial_x (
                        v_j(t)
                        - \tilde{v}_{j}(t)
                    )
                \Vert_{H^1(I)}^2.
        \end{align*}
        Therefore, by the Gronwall inequality and $v_j(0) = \tilde{v}_j(0)$, we obtain
        \begin{align*}
            \sum_{j=1,2}
                (
                    \Vert
                            v_j(t)
                            - \tilde{v}_j(t)
                    \Vert_{L^2(I)}^2
                    + \Vert
                        \partial_x (
                            v_j(t)
                            - \tilde{v}_j(t)
                        )
                    \Vert_{L^2(I)}^2
                )
            \leq 0, \quad
            t \in (0,T).
        \end{align*}
        We conclude that $v_j = \tilde{v}_j$, thereby establishing the uniqueness.
    \end{proof}

    \subsection{Construction of the solution $\tau_1$ and $\tau_2$}
    Let $\delta \in (0,1/8)$.
    Let $\tau_1, \tau_2 \in X_{\sigma,T}$.
    We consider the ordinary differential equations
    \begin{equation} \label{eq_abstract_filter_clogging_equation_in_the_boundary}
        \begin{aligned}
            &\frac{d\sigma_1}{dt}
            = - N_{\sigma, 1}(\tau_1) \tau_2
            + Q_1 d(\tau_1) \, \gamma_+ v_1,
            & t>0,
            & \,\\
            & \frac{d\sigma_2}{dt}
            = N_{\sigma, 2}(\tau_1, \tau_2) \tau_2
            + Q_2 d(\tau_1) \, \gamma_+ v_2,
            & t>0,
            & \, \\
            & \sigma_1
            = \sigma_{0,1}>0, \quad
            \sigma_2
            = \sigma_{0,2}>0,
            & t=0.
            & \,
        \end{aligned}
    \end{equation}
    where $(v_1, v_2)$ is the solution to (\ref{eq_abstract_filter_clogging_equation_in_the_interior}) for given $\tau_1$.
    We denote $v_j$ by
    \begin{align*}
        v_j
        =: \mathcal{S}_{v,j}(\tau_1), \quad
        j=1,2,
    \end{align*}
    and the solution map to (\ref{eq_abstract_filter_clogging_equation_in_the_boundary}) as
    \begin{align*}
        (\sigma_1, \sigma_2)
        = \mathcal{S}_{\sigma}(\tau_1, \tau_2)
        = (
            \mathcal{S}_{\sigma,1}(\tau_1, \tau_2),
            \mathcal{S}_{\sigma,2}(\tau_1, \tau_2)
        ).
    \end{align*}
    We construct the solution to (\ref{eq_abstract_filter_clogging_equation_in_the_boundary}) by Banach's fixed point theorem in $X_{\sigma,T}$.
    To this end we bound the right-hand side of (\ref{eq_abstract_filter_clogging_equation_in_the_boundary}) as small for small $T>0$.
    The difficult point is the estimate for
    \begin{align*}
        \sup_{0<t<T} t^{1/2 - \delta} \vert
            \sigma_j^{\prime\prime}
        \vert.
    \end{align*}
    To estimate this term we have to estimate $\partial_t v_j$ in $L^\infty_tH^1_x(Q_T)$.
    However, $\Vert \partial_t v_j \Vert_{L^\infty_tH^1_x(Q_T)}$ itself can be large for small $T>0$.
    To escape this difficultly we employ the $t$-weighted norm similarly to the Fujita-Kato approach.
    The quantity $T^{1/2 - \delta} \Vert \partial_t v_j \Vert_{L^\infty_tH^1_x(Q_T)}$ can be bounded as small for small $T$.
    \begin{proposition} \label{prop_estimate_for_tau_prime_L_infty}
        Let $0 \leq \eta_1 \leq \eta_2<1$ and $T>0$.
        Then
        \begin{align*}
            \sup_{0<t<T}
            \vert
                \varphi(t)
            \vert
            \leq T^{- \eta_1}
            \sup_{0<t<T}t^{\eta_1}
            \vert
                \varphi(t)
            \vert
            + T^{1 - \eta_2}
            \sup_{0<t<T}t^{\eta_2}
            \vert
                \varphi^\prime(t)
            \vert
        \end{align*}
        for all $\varphi \in C^1(0,T)$.
    \end{proposition}
    \begin{proof}
        Let $0<t<T$.
        By an elementary calculus, we observe that
        \begin{align*}
            \varphi(t)
            & = \varphi(T)
            + (t - T) \int_0^1
                \varphi^\prime(\theta t + (1 - \theta) T)
            d\theta,
        \end{align*}
        and then
        \begin{align*}
            & \sup_{0<t<T}
            \vert
                \varphi(t)
            \vert\\
            & \leq T^{- \eta_1}
            \sup_{0<t<T}t^{\eta_1}
            \vert
                \varphi(t)
            \vert
            +
            \vert
                t - T
            \vert
            \int_0^1
                \vert
                    \theta t
                    + (1-\theta) T
                \vert^{-\eta_2}
            d\theta \,
            \sup_{0<t<T}t^{\eta_2}
            \vert
                \varphi^\prime(t)
            \vert\\
            & \leq T^{- \eta_1}
            \sup_{0<t<T}t^{\eta_1}
            \vert
                \varphi(t)
            \vert
            + T^{1 - \eta_2}
            \sup_{0<t<T}t^{\eta_2}
            \vert
                \varphi^\prime(t)
            \vert.
        \end{align*}
    \end{proof}
    We find from Proposition \ref{prop_estimate_for_tau_prime_L_infty} and the definition for $X_{\sigma,T}$ that
    \begin{align} \label{eq_estimates_for_W1_infty_H1_H14_H2_by_X_sigma_T}
        \begin{split}
            \Vert
                \tau_1^\prime
            \Vert_{L^\infty(0,T)}
            & \leq C (
                T^{- \frac{1}{4} + \delta}
                \Vert
                    \tau_1
                \Vert_{X_{\sigma,T}}
                + T^{\frac{1}{2} + \delta}
                \Vert
                    \tau_1
                \Vert_{X_{\sigma,T}}
            ),\\
            \Vert
                \tau_1^\prime
            \Vert_{L^2(0,T)}
            & \leq C T^{\frac{1}{4} + \delta} \Vert
                \tau_1
            \Vert_{X_{\sigma,T}},\\
            \Vert
                \tau_1^\prime
            \Vert_{L^4(0,T)}
            & \leq C T^{\delta} \Vert
                \tau_1
            \Vert_{X_{\sigma,T}},\\
            \Vert
                \tau_1^{\prime\prime}
            \Vert_{L^2(0,T)}
            & \leq C T^{\delta} \Vert
                \tau_1
            \Vert_{X_{\sigma,T}}.
        \end{split}
    \end{align}
    In particular,
    \begin{align*}
        \Vert
            \tau_1^\prime
        \Vert_{L^2(0,T)}
        + \Vert
            \tau_1^\prime
        \Vert_{L^4(0,T)}
        + \Vert
            \tau_1^{\prime\prime}
        \Vert_{L^2(0,T)}
        \leq C \Vert
            \tau_1
        \Vert_{X_{\sigma,T}}
    \end{align*}
    for $T \in [0,1]$.
    Note that the right-hand side of the first inequality of (\ref{eq_estimates_for_W1_infty_H1_H14_H2_by_X_sigma_T}) can be large for small $T>0$.
    Then we have
    \begin{align} \label{eq_estimate_for_Psi_ast_by_T_and_tau_1_X_sigma_T}
        \Psi^\ast(\tau_1;T)
        = \int_0^t
            \left(
                1
                + \vert
                    \tau_1^\prime(r)
                \vert^2
            \right)
        dr
        \leq T
        + T^{\frac{1}{2} + 2 \delta} \Vert
            \tau_1
        \Vert_{X_{\sigma,T}}^2.
    \end{align}
    We set
    \begin{align*}
        \Vert
            (\tau_1, \tau_2)
        \Vert_{X_{\sigma,T}^2}
        : = \Vert
            \tau_1
        \Vert_{X_{\sigma,T}}
        + \Vert
            \tau_2
        \Vert_{X_{\sigma,T}}.
    \end{align*}
    We take $T$ so small again that
    \begin{align} \label{eq_assumption_for_tau1_tau2_for_exponential_deal_with_term}
        0
        \leq T
        \leq 1, \quad 
        0
        \leq T^{\frac{1}{2} + 2 \delta}
        \Vert
            (\tau_1, \tau_2)
        \Vert_{X_{\sigma,T}^2}
        \leq 1.
    \end{align}
    Under these assumptions, we have
    \begin{align} \label{eq_estimate_for_Psi_ast_for_small_t}
        \Psi^\ast(\tau_1;T)
        \leq 2.
    \end{align}
    %We prove that the solution map
    %\begin{align*}
    %    \mathcal{S}_{\sigma}:
    %    \{
    %        (\tau_1, \tau_2) \in X_{\sigma,T}^2
    %        ;
    %        \Vert
    %            (\tau_1, \tau_2)
    %        \Vert_{X_{\sigma,T}^2} \leq R
    %    \}
    %\end{align*}
    %is a contraction mapping for some $R>0$ if we take $T>0$ sufficiently small.
    %If we take $T>0$ so 
    \begin{theorem} \label{thm_local_in_time_existence_for_sigma1_sigma2}
        Let $T>0$.
        Assume the same assumptions for $v_{j,0}$ and $f$ as Theorem \ref{thm_existence_of_v_1_v_2_for_given_tau_1}.
        Let $\sigma_{j,0}>0$.
        There exists a unique solution $(\sigma_1, \sigma_2) \in C[0,T]$ to (\ref{eq_abstract_filter_clogging_equation_in_the_boundary}) such that
        \begin{align*}
            \sum_{j=1,2}
                \Vert
                    \sigma_j
                \Vert_{X_{\sigma, T}}
                \leq C
        \end{align*}
        for some constant $C>0$, which depends continuously on $v_{j,0}, f, \sigma_{j,0}$.
    \end{theorem}
%    \begin{proof}
    To prove this theorem we need several steps:

    \subsubsection*{Step1: boundedness}
        We take $0<T\leq1$ so small that (\ref{eq_assumption_for_tau1_tau2_for_exponential_deal_with_term}) and (\ref{eq_estimate_for_Psi_ast_for_small_t}) hold.
        We deduce from the estimates (\ref{eq_L2_estimate_for_vj_to_Leray_Schauder}), (\ref{eq_H1_estimate_for_vj_to_Leray_Schauder}), (\ref{eq_H2_estimate_for_vj_to_Leray_Schauder}), and (\ref{eq_H3_estimate_for_vj_to_Leray_Schauder}) that
        \begin{align} \label{eq_L^2_estimate_for_vj_for_contraction_mapping_principle}
            \sum_{j=1,2}
                \left[
                    \int_0^t
                        \Vert
                            \partial_x v_j(s)
                        \Vert_{L^2(I)}^2
                    ds
                    + \Vert
                        v_j(t)
                    \Vert_{(I)}^2
                \right]
            \leq C,
        \end{align}
        \begin{align} \label{eq_H1_estimate_for_vj_for_contraction_mapping_principle}
            \begin{split}
                & \sum_{j=1,2}
                    \left[
                        \int_0^t
                            \Vert
                                \partial_s v_j(s)
                            \Vert_{L^2(I)}^2
                        ds
                        + \Vert
                            v_j(t)
                        \Vert_{H^1(I)}^2
                        + \int_0^t
                            \Vert
                                \partial_x^2 v_j(s)
                            \Vert_{L^2(I)}^2
                        ds
                    \right]\\
                & \leq C_{L^\infty_tH^1_x}(t) e^{
                    C \Psi^\ast(\tau_1;t)
                }\left[
                    1
                    + \Psi^\ast(\tau_1;t)
                \right]\\
                &\leq C(
                    1
                    + \sum_{j=1,3}
                        \Vert
                            \tau_j
                        \Vert_{X_{\sigma,T}}^2
                ),
            \end{split}
        \end{align}
        and
        \begin{align}% \label{eq_H2_estimate_for_vj_to_Leray_Schauder}
            \begin{split}
                & \sum_{j=1,2} [
                    \Vert
                        \partial_t v_j(t)
                    \Vert_{H^1(I)}^2
                    + \Vert
                        \partial_x^2 v_j(t)
                    \Vert_{H^1(I)}^2\\
                & \quad
                    + \int_0^t
                        \Vert
                            \partial_s^2 v_j(s)
                        \Vert_{L^2(I)}^2
                        + \Vert
                            \partial_x^2 \partial_s v_j(s)
                        \Vert_{L^2(I)}^2\\
                & \quad\quad\quad\quad
                        + \Vert
                            \partial_x \partial_s v_j(s)
                        \Vert_{L^2(I)}^2
                        + \Vert
                            \partial_x^4 v_j(s)
                        \Vert_{L^2(I)}^2
                    ds
                ]\\
                & \leq C_{L^\infty_tH^3_x}(t)
                (
                    1
                    + \Psi^\ast(\tau_1;t)
                )
                (
                    1
                    + e^{C\Psi^\ast(\tau_1;t)}
                )
                \\
                & \quad
                \times
                \left[
                    (
                        1
                        +\Vert
                            \tau_1^\prime
                        \Vert_{L^\infty(0,T)}^2
                        + \Vert
                            \tau_1^{\prime}
                        \Vert_{L^4(0,T)}^4
                        + \Vert
                            \tau_1^{\prime\prime}
                        \Vert_{L^2(0,T)}^2
                    )
                \right.\\
                & \left.
                    \quad\quad\quad
                    +
                    (
                        1
                        + \Vert
                            \tau_1^\prime
                        \Vert_{L^2(0,T)}^2
                    )
                    \Psi^\ast(\tau_1;t)^2
                \right]\\
                & \leq C \sum_{j=1,3}
                    [
                        1
                        + (
                            1
                            + T^{-\frac{1}{2}+2\delta}
                        )
                        \Vert
                            \sigma_j
                        \Vert_{X_{\sigma,T}}^2
                        + \Vert
                            \sigma_j
                        \Vert_{X_{\sigma,T}}^4
                        + \Vert
                            \sigma_j
                        \Vert_{X_{\sigma,T}}^6
                    ]\\
                & \leq C \sum_{j=1,3}
                    [
                        1
                        + T^{-\frac{1}{2}+2\delta}
                        \Vert
                            \sigma_j
                        \Vert_{X_{\sigma,T}}^2
                        + \Vert
                            \sigma_j
                        \Vert_{X_{\sigma,T}}^6
                    ]
            \end{split}
        \end{align}
        for $T \in [0,1]$, where we used (\ref{eq_estimates_for_W1_infty_H1_H14_H2_by_X_sigma_T}) and (\ref{eq_estimate_for_Psi_ast_for_small_t}), the estimates $e^{\Psi^\ast(\tau, t)}\leq C$.
        %\begin{align*}
        %    \Psi^\ast(\tau, t) \Vert
        %        \tau_j
        %    \Vert_{L^\infty(0,T)}^2
        %    & \leq C
        %    (
        %        T
        %        + T^{\frac{1}{2} + 2 \delta} \Vert
        %            \tau_1
        %        \Vert_{X_{\sigma,T}}^2
        %    )
        %    (
        %        T^{- \frac{1}{2} + 2\delta}
        %        + T^{1 + 2\delta}
        %    )\Vert
        %        \tau_j
        %    \Vert_{X_{\sigma,T}}^2\\
        %    & \leq C (
        %        1
        %        + T^{4\delta} \Vert
        %        \tau_j
        %    \Vert_{X_{\sigma,T}}^2
        %    )\\
        %    & \leq C (
        %        1
        %        + \Vert
        %        \tau_j
        %    \Vert_{X_{\sigma,T}}^2
        %    )
        %\end{align*}
        %for $T \in [0,1]$ and some constant $C>0$.

        Using the integral formulations for (\ref{eq_abstract_filter_clogging_equation_in_the_boundary}), the above estimates, and the trace theorem, we deduce that
        \begin{align*}
            \begin{split}
                \vert
                    \sigma_1(t)
                \vert
                & \leq \sigma_{1,0}
                + C T \Vert
                    \tau_2
                \Vert_{L^\infty(0,T)}
                + C T \Vert
                    v_1
                \Vert _{L^\infty_tH^1_x(Q_T)}\\
                &\leq C B_{0, \sigma}
                + C T \Vert
                    \tau_2
                \Vert_{X_{\sigma,T}}
                + C T e^{
                    C\Psi^\ast(\tau_1;T)
                }\left[
                    1
                    + \Psi^\ast(\tau_1;T)
                \right]^{\frac{1}{2}},\\
                & \leq C
                + C T \Vert
                    \tau_2
                \Vert_{L^\infty(0,T)}
                + C T
                (
                    1
                    + \sum_{j=1,3}
                        \Vert
                            \tau_j
                        \Vert_{X_{\sigma,T}}^2
                )^{\frac{1}{2}}\\
                & \leq C
                + C T
                \Vert
                    (\tau_1, \tau_2)
                \Vert_{X_{\sigma,T}^2},
                %+ T^{\frac{5}{4} + \delta}
                %\Vert
                %    (\tau_1, \tau_2)
                %\Vert_{X_{\sigma,T}^2}
            \end{split}
        \end{align*}
        and
        \begin{align*}
            \begin{split}
                \vert
                    \sigma_2(t)
                \vert
                & \leq \sigma_{2,0}
                + C T \Vert
                    \tau_2
                \Vert_{L^\infty(0,T)}
                + C T \left(
                    \Vert
                        \tau_2
                    \Vert_{L^\infty(0,T)}
                    + \Vert
                        v_2
                    \Vert_{L^\infty_tH^1_x(Q_T)}
                \right)\\
                & \leq C B_{0, \sigma}
                + C T
                \Vert
                    \tau_2
                \Vert_{X_{\sigma, T}}
                + C T e^{
                    C\Psi^\ast(\tau_1;T)
                }\left[
                    1
                    + \Psi^\ast(\tau_1;T)
                \right]^{\frac{1}{2}}\\
                & \leq C
                + C T
                \Vert
                    (\tau_1, \tau_2)
                \Vert_{X_{\sigma, T}}.
            \end{split}
        \end{align*}
        Similarly we also deduce that
        \begin{align} \label{eq_estimate_for_fixed_point_L2_estimate_for_sigma}
            \begin{split}
                \sup_{0<t<T}
                    t^{\frac{1}{4}-\delta}
                    \left \vert
                        \frac{d\sigma_1(t)}{dt}
                    \right \vert
                & \leq C T^{\frac{1}{4}-\delta} \left(
                    \Vert
                        \tau_2
                    \Vert_{X_{\sigma,T}}
                    + \Vert
                        v_1
                    \Vert_{L^\infty_tH^1_x(Q_T)}
                \right)\\
                & \leq C T^{\frac{1}{4}-\delta} 
                \Vert
                    \tau_2
                \Vert_{X_{\sigma,T}}\\
                & + C T^{\frac{1}{4}-\delta} e^{
                    C\Psi^\ast(\tau_1;T)
                }\left[
                    1
                    + \Psi^\ast(\tau_1;T)
                \right]^{\frac{1}{2}}\\
                & \leq C T^{\frac{1}{4}-\delta} (
                    1
                    + \Vert
                        (\tau_1, \tau_2)
                    \Vert_{X_{\sigma, T}}
                ),
            \end{split} 
        \end{align}
        and
        \begin{align*}
            \begin{split}
                \sup_{0<t<T} t^{\frac{1}{4}-\delta}
                    \left \vert
                        \frac{d\sigma_2(t)}{dt}
                    \right \vert
                & \leq C T^{\frac{1}{4}-\delta} \left(
                    \Vert
                        \tau_2
                    \Vert_{L^\infty(0,T)}
                    + \Vert
                        \tau_2
                    \Vert_{L^\infty(0,T)}^2
                    + \Vert
                        v_2
                    \Vert_{L^\infty_tH^1_x(Q_T)}
                \right)\\
                & \leq C T^{\frac{1}{4}-\delta}
                + C T^{\frac{1}{4}-\delta} 
                (
                    \Vert
                        \tau_2
                    \Vert_{X_{\sigma,T}}
                    + \Vert
                        \tau_2
                    \Vert_{X_{\sigma,T}}^2
                )\\
                & + C T^{\frac{1}{4}-\delta} e^{
                    C\Psi^\ast(\tau_1;T)
                }\left[
                    1
                    + \Psi^\ast(\tau_1;T)
                \right]^{\frac{1}{2}}\\
                & \leq C T^{\frac{1}{4}-\delta}
                (
                    1
                    + \Vert
                        (\tau_1, \tau_2)
                    \Vert_{X_{\sigma, T}}
                    + \Vert
                        (\tau_1, \tau_2)
                    \Vert_{X_{\sigma, T}}^2
                )
            \end{split}
        \end{align*}
        for some constant $C>0$.
        We observe that
        \begin{align*}
            &\frac{d^2\sigma_1}{dt^2}
            = - \frac{d}{dt}
            \left[
                N_{\sigma, 1}(\tau_1) \tau_2
            \right] 
            + Q_1 d^\prime(\tau_1) \tau^\prime_1 \gamma_+ v_1
            + Q_1 d(\tau_1) \gamma_+ \partial_t v_1,\\
            & \frac{d^2\sigma_2}{dt^2}
            = \frac{d}{dt} \left[
                N_{\sigma, 2}(\tau_1, \tau_2) \tau_2
            \right]
            + Q_2 d^\prime(\tau_1) \tau^\prime_1 \gamma_+ v_2
            + Q_2 d(\tau_1) \gamma_+ \partial_t v_2.
        \end{align*}
        Therefore, we deduce from Proposition \ref{prop_nonlinear_estimates_for_Nv} (i) and the estimates for $v_j$ (\ref{eq_H1_estimate_for_vj_to_Leray_Schauder}) and (\ref{eq_H3_estimate_for_vj_to_Leray_Schauder}) that
        \begin{align*}
            & \sup_{0<t<T} t^{\frac{1}{2} - \delta}
                \left \vert
                    \frac{d^2 \sigma_1(t)}{dt}
                \right \vert\\
            & \leq C \sup_{0<t<T} t^{\frac{1}{2} - \delta}
            \left(
                \vert
                    \tau^\prime_1(t)
                \vert
                \vert
                    \tau_2(t)
                \vert
                + \vert
                    \tau^\prime_2(t)
                \vert
            \right)\\
            & + C \sup_{0<t<T} t^{\frac{1}{2} - \delta}
                \vert
                    \tau^\prime_1(t)
                \vert
                \Vert
                    v_1(t)
                \Vert_{H^1(I)}
            + C \sup_{0<t<T} t^{\frac{1}{2} - \delta}
                \Vert
                    \partial_t v_1(t)
                \Vert_{H^1(I)}\\
            & =: I_1 + I_2 + I_3.
        \end{align*}
        We observe that
        \begin{align*}
            I_1
            \leq C T^{\frac{1}{4}}
            \left(
                \Vert
                    \tau_1(t)
                \Vert_{X_{\sigma, T}}
                \Vert
                    \tau_2(t)
                \Vert_{X_{\sigma, T}}
                + \Vert
                    \tau_2(t)
                \Vert_{X_{\sigma, T}}
            \right),
        \end{align*}
        \begin{align*}
            I_2
            & \leq C T^{\frac{1}{4}} \Vert
                    \tau_1(t)
                \Vert_{X_{\sigma, T}}
            e^{
                C\Psi^\ast(\tau_1;T)
            }\left[
                1
                + \Psi^\ast(\tau_1;T)
            \right]^{\frac{1}{2}}\\
            & \leq C T^{\frac{1}{4}} \Vert
                    \tau_1(t)
                \Vert_{X_{\sigma, T}}
            [
                1
                + \Vert
                    \tau_1
                \Vert_{X_{\sigma,T}}^2
            ]^{\frac{1}{2}},
        \end{align*}
        and
        \begin{align*}
            I_3
            & \leq C T^{\frac{1}{2} - \delta}
                (
                    1
                    + \Psi^\ast(\tau_1;t)
                )^{\frac{1}{2}}
                (
                    1
                    + e^{C\Psi^\ast(\tau_1;t)}
                )^{\frac{1}{2}}
                \\
                & \quad
                \times
                \left[
                    (
                        1
                        +\Vert
                            \tau_1^\prime
                        \Vert_{L^\infty(0,T)}
                        + \Vert
                            \tau_1^{\prime}
                        \Vert_{L^4(0,T)}^2
                        + \Vert
                            \tau_1^{\prime\prime}
                        \Vert_{L^2(0,T)}
                    )
                \right.\\
                & \left.
                    \quad\quad\quad\quad\quad\quad
                    +
                    (
                        1
                        + \Vert
                            \tau_1^\prime
                        \Vert_{L^2(0,T)}
                    )
                    \Psi^\ast(\tau_1;t)
                \right]\\
            & \leq C T^{\frac{1}{2} - \delta}
                (
                    1
                    + T^{\frac{1}{2} + 2 \delta} \Vert
                        \tau_1
                    \Vert_{X_{\sigma,T}}^2
                )^{\frac{1}{2}}
                \\
                & \quad
                \times
                \left[
                    1
                    + T^{- \frac{1}{4} + \delta}
                    \Vert
                        \tau_1
                    \Vert_{X_{\sigma,T}}
                    + T^{\frac{1}{2} + \delta}
                    \Vert
                        \tau_1
                    \Vert_{X_{\sigma,T}}
                    + T^{2\delta}
                    \Vert
                        \tau_1
                    \Vert_{X_{\sigma,T}}^2
                    + T^{\delta}\Vert
                            \tau_1
                        \Vert_{X_{\sigma,T}}
                \right.\\
                & \left.
                    \quad\quad\quad\quad
                    +
                    (
                        1
                        + T^{\delta} \Vert
                            \tau_1
                        \Vert_{X_{\sigma,T}}
                    )
                    (
                        T
                        + T^{\frac{1}{2} + 2\delta} \Vert
                            \tau_1
                        \Vert_{X_{\sigma,T}}^2
                    )
                \right]\\
            & \leq C T^{\frac{1}{4}}
                (
                    1
                    + \Vert
                        \tau_1
                    \Vert_{X_{\sigma,T}}^4
                )
        \end{align*}
        We find (\ref{eq_trivial_inequality}) that
        \begin{align*}
            \begin{split}
                \sup_{0<t<T} t^{\frac{1}{2} - \delta}
                    \left \vert
                        \frac{d^2 \sigma_1(t)}{dt}
                    \right \vert
                \leq C T^{\frac{1}{4}}
                + C T^{\frac{1}{4}}
                \Vert
                    (\tau_1, \tau_2)
                \Vert_{X_{\sigma,T}^2}^4.
            \end{split}
        \end{align*}
        Similarly, we observe that
        \begin{align*}
            \begin{split}
                & \sup_{0<t<T} t^{\frac{1}{2} - \delta}
                    \left \vert
                        \frac{d^2 \sigma_2(t)}{dt}
                    \right \vert\\
                & \leq \sup_{0<t<T} t^{\frac{1}{2} - \delta}
                \left(
                    \vert
                        \tau^\prime_1(t)
                    \vert
                    \vert
                        \tau_2(t)
                    \vert
                    + \vert
                        \tau^\prime_2(t)
                    \vert
                    + \vert
                        \tau^\prime_2(t)
                    \vert
                    \vert
                        \tau_2(t)
                    \vert
                \right)\\
                & + C \sup_{0<t<T} t^{\frac{1}{2} - \delta}
                    \vert
                        \tau^\prime_1(t)
                    \vert
                    \Vert
                        v_2(t)
                    \Vert_{H^1(I)}\\
                & + C \sup_{0<t<T} t^{\frac{1}{2} - \delta}
                    \Vert
                        \partial_t v_2(t)
                    \Vert_{H^1(I)}\\
                & \leq C T^{\frac{1}{4}}
                + C T^{\frac{1}{4}}
                \Vert
                    (\tau_1, \tau_2)
                \Vert_{X_{\sigma,T}^2}^4.
            \end{split}
        \end{align*}
        We conclude that
        \begin{align}\label{eq_quadratic_estimate_for_tau1_tau2_in_X_sigma_T}
            \Vert
                (\sigma_1, \sigma_2)
            \Vert_{X_{\sigma,T}^2}
            \leq C_1
            + C_2 T^{\frac{1}{4}}
            \Vert
                (\tau_1, \tau_2)
            \Vert_{X_{\sigma,T}^2}^4.
        \end{align}
        Let $y^\ast_1, y^\ast_2$ be the positive solution of the algebraic equation
        \begin{align*}
            y
            = C_1
            + C_2 T^{\frac{1}{4}} y^4.
        \end{align*}
        Note that since $C_2 T^{\frac{1}{4} - \delta} y^4$ is convex and zero when $y=0$, if we take $T$ sufficiently small there exist two positive solution $0<y^\ast_1 < y^\ast_2$.
        For $(\tau_1, \tau_2) \in X_{\sigma,T}^2$ such that (\ref{eq_assumption_for_tau1_tau2_for_exponential_deal_with_term}) and $\Vert (\tau_1, \tau_2)\Vert_{X_{\sigma,T}^2} \leq y^\ast_1$, we see from (\ref{eq_quadratic_estimate_for_tau1_tau2_in_X_sigma_T}) that for small $T>0$
        \begin{align*}
            \Vert
                (\sigma_1, \sigma_2)
            \Vert_{X_{\sigma,T}^2}
            & \leq C_1
            + C T^{\frac{1}{4}}
            \Vert
                (\tau_1, \tau_2)
            \Vert_{X_{\sigma,T}^2}^4
            \leq C_1
            + C T^{\frac{1}{4}}
            (y^\ast_1)^4\\
            & = y^\ast_1
        \end{align*}
        in the second inequality, we used convexity of $C_1 + C_2 T^{\frac{1}{4} - \delta} y^4$.
        Note that $y^\ast_1$ does not growth as $T \rightarrow 0$.
        Therefore if we take $T$ sufficiently small, the assumption (\ref{eq_estimate_for_Psi_ast_for_small_t}) is automatically satisfied on a ball with radius $y^\ast_1$
        \begin{align*}
            B_{X_{\sigma, T}^2, y^\ast_1}
            = \{
                (\tau_1, \tau_2) \in X_{\sigma, T}^2
                ;
                \Vert
                    ((\tau_1, \tau_2))
                \Vert_{X_{\sigma, T}^2}
                \leq y^\ast_1
            \}.
        \end{align*}
        Therefore, we find that the solution mapping $\mathcal{S}_{\sigma}$ is a self-mapping on $B_{X_{\sigma, T}^2, y^\ast_1}$.

        \subsubsection*{Step2: contraction mapping}
        We next prove connectivity of $B_{X_{\sigma, T}^2, y^\ast_1}$ for small $T \in (0,1]$.
        We can assume that (\ref{eq_estimate_for_Psi_ast_for_small_t}) is satisfied on $B_{X_{\sigma, T}^2, y^\ast_1}$.
        For $(\tau_3, \tau_4) \in X_{\sigma,T}^2$, we denote
        \begin{align*}
            v_3
            := \mathcal{S}_{v,1}(\tau_3), \quad
            v_4
            := \mathcal{S}_{v,2}(\tau_3).
        \end{align*}
        We consider the equations (\ref{eq_difference_equation_for_iteration_of_full_filter_clogging_equations}) under $\psi_j = v_j$ ($j=1,2,3,4$).
        We set
        \begin{align*}
            \mathcal{V}_1
            & := E \left(
                \frac{
                    F(\tau_1)
                    - F(\tau_3)
                }{
                    2
                }
                (
                    v_1
                    + v_3
                )
                ;
                1 - \frac{
                    F(\tau_1)
                    + F(\tau_3)
                }{
                    2
                }
            \right),\\
            \mathcal{V}_2
            & := E \left(
                \frac{
                    F(\tau_1)
                    - F(\tau_3)
                }{
                    2
                }
                (
                    v_2
                    + v_4
                )
                ;
                1 - \frac{
                    F(\tau_1)
                    + F(\tau_3)
                }{
                    2
                }
            \right),\\
            \mathcal{W}_1
            & := \partial_t \mathcal{V}_1
            - \partial_x^2 \mathcal{V}_1,\\
            \mathcal{W}_2
            & := \partial_t \mathcal{V}_2
            - \partial_x^2 \mathcal{V}_2,
        \end{align*}
        and
        \begin{align*}
            \overline{\mathcal{V}}_1
            & := V_1 - \mathcal{V}_1,\\
            \overline{\mathcal{V}}_2
            & := V_2 - \mathcal{V}_2.
        \end{align*}
        Then we find from (\ref{eq_difference_equation_for_iteration_of_full_filter_clogging_equations}) that $(\overline{\mathcal{V}}_1, \overline{\mathcal{V}}_2)$ satisfies
    \begin{equation}\label{eq_overline_V_difference_equation_for_iteration_of_full_filter_clogging_equations}
        \begin{aligned}
            &\partial_t \overline{\mathcal{V}}_1
            - \partial_x^2 \overline{\mathcal{V}}_1
            = - \mathcal{N}_{V,0}(v_1, v_3, \tau_1, \tau_3)
            - \mathcal{N}_{V,1}(v_1, v_2, v_3, v_4)
            - \mathcal{W}_1,
            & t>0, \,
            & x \in I,\\
            &\partial_t \overline{\mathcal{V}}_2
            - \partial_x^2 \overline{\mathcal{V}}_2
            = - \mathcal{N}_{V,0}(v_2, v_4, \tau_1, \tau_3)
            + \mathcal{N}_{V,2}(v_1, v_2, v_3, v_4)
            - \mathcal{W}_2,
            & t>0, \,
            & x \in I,\\
            &B_1 \left(
                \overline{\mathcal{V}}_1;
                \frac{
                    F(\tau_1)
                    + F(\tau_3)
                }{
                    2
                }
            \right)
            = 0, \quad
            %& t>0,
            %& \,\\
            %&
            B_2\left(
                \overline{\mathcal{V}}_1;
                \frac{
                    F(\tau_1)
                    + F(\tau_3)
                }{
                    2
                }
            \right)
            = 0,
            & t>0,
            & \,\\
            &B_1 \left(
                \overline{\mathcal{V}}_2;
                \frac{
                    F(\tau_1)
                    + F(\tau_3)
                }{
                    2
                }
            \right)
            = 0, \quad
            %& t>0,
            %& \,\\
            %&
            B_2\left(
                \overline{\mathcal{V}}_2;
                \frac{
                    F(\tau_1)
                    + F(\tau_3)
                }{
                    2
                }
            \right)
            = 0,
            & t>0,
            & \,\\
            & \overline{\mathcal{V}}_1
            =0, \quad
            \overline{\mathcal{V}}_2
            = 0,
            & t=0,
            & \,
        \end{aligned}
    \end{equation}
    By the definition, we have
    \begin{align*}
        V_1
        = \overline{\mathcal{V}}_1
        + \mathcal{V}_1, \quad
        V_2
        = \overline{\mathcal{V}}_2
        + \mathcal{V}_2.
    \end{align*}
    Note that by the definition of $\tau_j$ and the condition $\mathcal{V}_j(\cdot, 0) = 0$, we have $\overline{\mathcal{V}}_1(\cdot, 0) = \overline{\mathcal{V}}_2(\cdot, 0) = 0$.
    We show the $L^\infty_tL^2_x$-estimates for $V_j$ and $\partial_t V_j$.
    We begin by showing some estimates for $\mathcal{V}_j$. 
    \begin{proposition}\label{prop_estimates_for_mathcal_V_j}
        \begin{enumerate}[label=(\roman*)]
            \item 
            It follows that
            \begin{align} \label{eq_L2H2_estimate_for_mathcal_V}
                \begin{split}
                    & \sum_{j=1,2}
                        \Vert
                            \mathcal{V}_j(t)
                        \Vert_{L^2(I)}^2
                    \leq C \Vert
                        \tau_1
                        - \tau_3
                    \Vert_{X_{\sigma,T}}^2,\\
                    %%%%%%%%%%%%%%%%%%%%%%%%%%%%
                    & \sum_{j=1,2}
                        \Vert
                            \mathcal{V}_j(t)
                        \Vert_{H^1(I)}^2
                    \leq C \Vert
                        \tau_1
                        - \tau_3
                    \Vert_{X_{\sigma,T}}^2
                    (
                        1
                        + \sum_{j=1,3}
                            \Vert
                                \tau_j
                            \Vert_{X_{\sigma,T}}^2
                    ),\\
                    %%%%%%%%%%%%%%%%%%%%%%%%%%%%
                    & \sum_{j=1,2} \int_0^t
                        \Vert
                            \mathcal{V}_j(s)
                        \Vert_{H^2(I)}^2
                    ds
                    \leq C
                    \Vert
                        \tau_1
                        - \tau_3
                    \Vert_{X_{\sigma,T}}^2
                    (
                        1
                        + \sum_{j=1,3}
                            \Vert
                                \tau_j
                            \Vert_{X_{\sigma,T}}^2
                    )
                \end{split}
            \end{align}
            for some constant $C>0$, which is bounded for $T \in (0,1]$.
            \item It follows that
            \begin{align}\label{eq_estimate_for_dt_mathcal_V_j}
                \begin{split}
                    & \sum_{j=1,2}
                        \Vert
                            \partial_t \mathcal{V}_j(t)
                        \Vert_{L^2(I)}^2\\
                    & \leq C (
                        1
                        + \sum_{j=1,3}
                            \Vert
                                \tau_j
                            \Vert_{X_{\sigma,T}}^2
                    )^4
                    \Vert
                        \tau_1
                        - \tau_3
                    \Vert_{X_{\sigma,T}}^2\\
                    & \quad\quad\quad + C T^{-\frac{1}{2}+2\delta} \Vert
                        \tau_1
                        - \tau_3
                    \Vert_{X_{\sigma,T}}^2
                    \sum_{j=1,3}
                        \Vert
                            \tau_j
                        \Vert_{X_{\sigma,T}}^2,\\
                    %%%%%%%%%%%%%%%%%%%%%%%%%%%
                    & \sum_{j=1,2} \int_0^t
                        \Vert
                            \partial_t \mathcal{V}_j(s)
                        \Vert_{L^2(I)}^2
                    ds
                    \leq C
                    \Vert
                        \tau_1
                        - \tau_3
                    \Vert_{X_{\sigma,T}}^2
                    (
                        1
                        + \sum_{j=1,3}
                            \Vert
                                \tau_j
                            \Vert_{X_{\sigma,T}}^2
                    ).
                \end{split}
            \end{align}
            for some constant $C>0$, which is bounded for $T \in (0,1]$.
    
            \item It follows that
            \begin{align}\label{eq_estimate_for_dt_mathcal_W_j}
                \begin{split}
                    & \sum_{j=1,2}
                        \int_0^t
                            \Vert
                                \partial_s^2 \mathcal{V}_j
                            \Vert_{L^2(I)}^2
                            + \Vert
                                \partial_s \mathcal{V}_j
                            \Vert_{H^2(I)}^2
                        ds\\
                    & \leq C (
                        1
                        + \sum_{j=1,3}
                            \Vert
                                \tau_j
                            \Vert_{X_{\sigma,T}}^2
                    )^4
                    \Vert
                        \tau_1
                        - \tau_3
                    \Vert_{X_{\sigma,T}}^2\\
                    & + C T^{-\frac{1}{2}+2\delta} \Vert
                        \tau_1
                        - \tau_3
                    \Vert_{X_{\sigma,T}}^2
                    \sum_{j=1,3}
                        \Vert
                            \tau_j
                        \Vert_{X_{\sigma,T}}^2
                    \end{split}
            \end{align}
            for some constant $C>0$ which is bounded for $T \in (0,1]$.
        \end{enumerate}
    \end{proposition}
    \begin{proof}
        We deduce (\ref{eq_L2H2_estimate_for_mathcal_V}) from Proposition \ref{prop_estimates_for_difference_between_boundary_extension_terms}(i) and (\ref{eq_H1_estimate_for_vj_to_Leray_Schauder}).
        We will prove (ii).
        Since $\sup_{0<t<1} \Vert v_j(t) \Vert_{L^2(I)}$ is uniformly bounded for $\tau_j$ and it holds that
        \begin{align*}
            \sup_{0<t<1} \Vert
                v_j
            \Vert_{L^2(I)}
            \leq C \sum_{j=1,3}
            [
                1
                + T^{-\frac{1}{2}+2\delta}
                \Vert
                    \sigma_j
                \Vert_{X_{\sigma,T}}^2
                + \Vert
                    \sigma_j
                \Vert_{X_{\sigma,T}}^6
            ],
        \end{align*}
        using Proposition \ref{prop_estimates_for_difference_between_boundary_extension_terms}(ii) and the estimates (\ref{eq_trivial_inequality}), (\ref{eq_H2_estimate_for_vj_to_Leray_Schauder}) and (\ref{eq_estimates_for_W1_infty_H1_H14_H2_by_X_sigma_T}), we find that
        \begin{align*}
            & \Vert
                \partial_t \mathcal{V}(t)
            \Vert_{L^2(I)}^2\\
            & \leq C \left(
                \sum_{j=1,2}
                    \vert
                        \tau_j^\prime(t)
                    \vert^2
                \vert
                    \tau_1(t)
                    - \tau_2(t)
                \vert^2
                +
                \vert
                    \tau_1^\prime(t)
                    - \tau_2^\prime(t)
                \vert^2
            \right)
            \Vert
                v(t)
            \Vert_{L^2(I)}^2\\
            & + C
            \vert
                \tau_1(t)
                - \tau_2(t)
            \vert^2
            \Vert
                \partial_t v(t)
            \Vert_{L^2(I)}^2\\
            & \leq C T^{-\frac{1}{2}+2\delta}
            \sum_{j=1,3}
                \Vert
                    \tau_j
                \Vert_{X_{\sigma,T}}^2
            \Vert
                \tau_1
                - \tau_3
            \Vert_{X_{\sigma,T}}^2\\
            & + C 
            \Vert
                \tau_1
                - \tau_3
            \Vert_{X_{\sigma,T}}^2
            \sum_{j=1,3}
            [
                1
                + T^{-\frac{1}{2}+2\delta}
                \Vert
                    \tau_j
                \Vert_{X_{\sigma,T}}^2
                + \Vert
                    \tau_j
                \Vert_{X_{\sigma,T}}^6
            ].
        \end{align*}
        for some constant $C>0$.
        Similarly, using (\ref{eq_H1_estimate_for_vj_to_Leray_Schauder}), we observe that
        \begin{align*}% \label{eq_H1L2_estimate_for_mathcal_V} 
            \begin{split}
                & \sum_{j=1,2} \int_0^t
                    \Vert
                        \partial_t \mathcal{V}_j(s)
                    \Vert_{L^2(I)}^2
                ds\\
                & \leq C
                \int_0^t
                    \left(
                        \sum_{k=1,3}
                            \vert
                                \tau_k^\prime(s)
                            \vert^2
                        \vert
                            \tau_1(s)
                            - \tau_3(s)
                        \vert^2
                        +
                        \vert
                            \tau_1^\prime(s)
                            - \tau_3^\prime(s)
                        \vert^2
                    \right)\\
                & \quad\quad\quad\quad\quad\quad \times
                    \sum_{k=1, \ldots, 4}
                        \Vert
                            v_k(s)
                        \Vert_{L^2(I)}^2
                ds\\
                & + C \int_0^t
                    \vert
                        \tau_1(s)
                        - \tau_3(s)
                    \vert^2
                    \sum_{k=1, \ldots, 4}
                        \Vert
                            \partial_t v_k(s)
                        \Vert_{L^2(I)}^2
                ds\\
                %& \leq C T^{\frac{3}{4}+\delta}
                %(
                %    \sum_{j=1,3}
                %        \Vert
                %            \tau_j
                %        \Vert_{X_{\sigma,T}}
                %    + 
                %    1
                %)
                %\Vert
                %    \tau_1
                %    - \tau_3
                %\Vert_{X_{\sigma,T}}
                %\sum_{j=1,3}
                %    \sup_{0<t<T}
                %        \Vert
                %            v_j(t)
                %        \Vert_{L^2(I)}\\
                %& + C
                %\Vert
                %    \tau_1
                %    - \tau_3
                %\Vert_{X_{\sigma,T}}
                %\left[
                %    C_{L^\infty_tH^1_x, 1}(t)
                %    + C_{L^\infty_tH^1_x, 2}(t)
                %    \Psi^\ast(\tau_1;t)
                %\right].
                %& \leq C T^{\frac{1}{2}+2\delta}
                %(
                %    \sum_{j=1,3}
                %        \Vert
                %            \tau_j
                %        \Vert_{X_{\sigma,T}}^2
                %    + 
                %    1
                %)
                %\Vert
                %    \tau_1
                %    - \tau_3
                %\Vert_{X_{\sigma,T}}^2\\
                %& + C
                %\Vert
                %    \tau_1
                %    - \tau_3
                %\Vert_{X_{\sigma,T}}^2
                %(
                %    1
                %    + \sum_{j=1,3}
                %        \Vert
                %            \tau_j
                %        \Vert_{X_{\sigma,T}}^2
                %)\\
                & \leq C
                \Vert
                    \tau_1
                    - \tau_3
                \Vert_{X_{\sigma,T}}^2
                (
                    1
                    + \sum_{j=1,3}
                        \Vert
                            \tau_j
                        \Vert_{X_{\sigma,T}}^2
                )
            \end{split}
        \end{align*}
        for some constant $C>0$.
        We will prove (iii).
        Similarly, we deduce that
        \begin{align*}
            & \sum_{j=1,2}
                \int_0^t
                    \Vert
                        \partial_s^2 \mathcal{V}_j
                    \Vert_{L^2(I)}^2
                    + \Vert
                        \partial_s \mathcal{V}_j
                    \Vert_{H^2(I)}^2
                ds\\
            & \leq C \int_0^t
                (
                    \sum_{k=1,3}
                        \vert
                            \tau_k^\prime
                        \vert^4
                    \vert
                        \tau_1
                        - \tau_3
                    \vert^2
                    + 
                    \vert
                        \tau_k^{\prime\prime}
                    \vert^2
                    \vert
                        \tau_1
                        - \tau_3
                    \vert^2
                \\
                &
                    \quad\quad\quad\quad\quad
                    + \vert
                        \tau_k^\prime
                    \vert^2
                    \vert
                        \tau_1^\prime
                        - \tau_3^\prime
                    \vert^2
                    +
                    \vert
                        \tau_1^{\prime\prime}
                        - \tau_3^{\prime\prime}
                    \vert^2
                ) \sum_{j=1,\ldots,4}
                    \Vert
                        v_j
                    \Vert_{L^2(I)}\\
                & \quad\quad\quad
                + (
                    \sum_{k=1,3}
                        \vert
                            \tau_k^\prime
                        \vert^2
                    \vert
                        \tau_1
                        - \tau_3
                    \vert^2
                    +
                    \vert
                        \tau_1^\prime
                        - \tau_3^\prime
                    \vert^2
                )
                \sum_{j=1,\ldots,4}
                    \Vert
                        \partial_s v_j
                    \Vert_{L^2(I)}^2\\
                & \quad\quad\quad
                + \vert
                    \tau_1
                    - \tau_3
                \vert^2
                \sum_{j=1,\ldots,4}
                    \Vert
                        \partial_s^2 v_j
                    \Vert_{L^2(I)}^2
            ds\\
            & + C \int_0^t
                (
                    \sum_{k=1,3}
                        \vert
                            \tau_k^\prime
                        \vert^2
                    \vert
                        \tau_1
                        - \tau_3
                    \vert^2
                    +
                    \vert
                        \tau_1^\prime
                        - \tau_3^\prime
                    \vert^2
                )
                \sum_{j=1,\ldots, 4}
                    \Vert
                        v_j
                    \Vert_{H^2(I)}^2\\
                & \quad\quad\quad
                +
                \vert
                    \tau_1
                    - \tau_2
                \vert^2
                \sum_{j=1,\ldots, 4}
                    \Vert
                        \partial_s v_j
                    \Vert_{H^2(I)}^2
            ds\\
            & =: \int_0^t
                I_1
                \sum_{j=1,\ldots,4}
                    \Vert
                        v_j
                    \Vert_{L^2(I)}^2
                + I_2
                \sum_{j=1,\ldots,4}
                    \Vert
                        \partial_s v_j
                    \Vert_{L^2(I)}^2
                + I_3
                \sum_{j=1,\ldots,4}
                    \Vert
                        \partial_s^2 v_j
                    \Vert_{L^2(I)}^2\\
            & \quad\quad\quad 
                + I_4
                \sum_{j=1,\ldots,4}
                    \Vert
                        v_j
                    \Vert_{H^2(I)}^2
                + I_5
                \sum_{j=1,\ldots,4}
                    \Vert
                        \partial_s v_j
                    \Vert_{H^2(I)}^2
            ds
        \end{align*}
        We observe that
        \begin{align*}
            & \int_0^t
                I_1
                \sum_{j=1,\ldots,4}
                    \Vert
                        v_j
                    \Vert_{L^2(I)}^2
            ds\\
            & \leq C
                T^{4\delta}
                \sum_{j=1,3}
                    \Vert
                        \tau_j
                    \Vert_{X_{\sigma,T}}^4
                \sum_{j=1,3}
                    \Vert
                        \tau_1
                        - \tau_3
                    \Vert_{X_{\sigma,T}}^2\\
            & + C T^{2\delta}
                \sum_{j=1,3}
                    \Vert
                        \tau_j
                    \Vert_{X_{\sigma,T}}^2
                \sum_{j=1,3}
                    \Vert
                        \tau_1
                        - \tau_3
                    \Vert_{X_{\sigma,T}}^2\\
            %&
            %    + C T^{2\delta}
            %    \sum_{j=1,3}
            %        \Vert
            %            \tau_j
            %        \Vert_{X_{\sigma,T}}^2
            %    \sum_{j=1,3}
            %        \Vert
            %            \tau_1
            %            - \tau_3
            %        \Vert_{X_{\sigma,T}}^2\\
            & 
                + C T^{2\delta}
                \sum_{j=1,3}
                    \Vert
                        \tau_1
                        - \tau_3
                    \Vert_{X_{\sigma,T}}^2\\
            & \leq C T^{2\delta}
            (
                1
                + \sum_{j=1,3}
                    \Vert
                        \tau_j
                    \Vert_{X_{\sigma,T}}^4
            )
            \sum_{j=1,3}
                \Vert
                    \tau_1
                    - \tau_3
                \Vert_{X_{\sigma,T}}^2.
        \end{align*}
        Using (\ref{eq_H1_estimate_for_vj_to_Leray_Schauder}) and (\ref{eq_estimates_for_W1_infty_H1_H14_H2_by_X_sigma_T}), we deduce that
        \begin{align*}
            & \sum_{j=1,\ldots,4} \int_0^t
                I_2
                    \Vert
                        \partial_s v_j
                    \Vert_{L^2(I)}^2
                + I_4
                    \Vert
                        v_j
                    \Vert_{H^2(I)}^2
            ds\\
            & \leq C T^{\frac{1}{2}+2\delta}
            (
                \sum_{j=1,3}
                    \Vert
                        \tau_j
                    \Vert_{X_{\sigma,T}}^2
                + 1
            )
            \Vert
                \tau_1
                - \tau_3
            \Vert_{X_{\sigma,T}}^2
            \sum_{j=1,3}
            [
                1
                + T^{-\frac{1}{2}+2\delta}
                \Vert
                    \tau_j
                \Vert_{X_{\sigma,T}}^2
                + \Vert
                    \tau_j
                \Vert_{X_{\sigma,T}}^6
            ]\\
            & \leq C (
                1
                + \sum_{j=1,3}
                    \Vert
                        \tau_j
                    \Vert_{X_{\sigma,T}}^2
            )^4
            \Vert
                \tau_1
                - \tau_3
            \Vert_{X_{\sigma,T}}^2.
        \end{align*}
        Similarly, we find from (\ref{eq_H3_estimate_for_vj_to_Leray_Schauder}) that
        \begin{align*}
            & \sum_{j=1,\ldots,4}\int_0^t
                I_3
                \Vert
                    \partial_s^2 v_j
                \Vert_{L^2(I)}^2
                + I_5
                \Vert
                    \partial_s v_j
                \Vert_{H^2(I)}^2
            ds\\
            & \leq C \Vert
                \tau_1
                - \tau_3
            \Vert_{X_{\sigma,T}}^2
            \sum_{j=1,3}
            [
                1
                + T^{-\frac{1}{2}+2\delta}
                \Vert
                    \tau_j
                \Vert_{X_{\sigma,T}}^2
                + \Vert
                    \tau_j
                \Vert_{X_{\sigma,T}}^6
            ].
        \end{align*}
        Therefore, we obtain (\ref{eq_estimate_for_dt_mathcal_W_j})
    \end{proof}

    Proposition \ref{prop_estimates_for_mathcal_V_j} yields
    \begin{align} \label{eq_estimate_for_mathcal_W_in_L2L2}
        \sum_{j=1,2}
            \int_0^t
                \Vert
                    \mathcal{W}_j(s)
                \Vert_{L^2(I)}^2
            ds
        \leq C (
            1
            + \sum_{j=1,3}
            \Vert
                \tau_j
            \Vert_{X_{\sigma,T}}^2
        )
        \Vert
            \tau_1
            - \tau_3
        \Vert_{X_{\sigma,T}}^2
    \end{align}
    for some constant $C>0$ and $T \in [0, 1]$.
    We next show some estimates for $\overline{\mathcal{V}}_j$.
    \begin{proposition} \label{prop_L2_estimate_for_V_j}
        Let $T \in [0,1]$ and $(\tau_j, \tau_{j+1}) \in B_{X_{\sigma, T}^2, y^\ast_1}$ for $j=1,3$.
        Assume that $T$ satisfies (\ref{eq_assumption_for_tau1_tau2_for_exponential_deal_with_term}).
        Then
        \begin{align*}% \label{eq_L2_estimate_for_V_j}
            \begin{split}
                \sum_{j=1,2}
                \left(
                    \Vert
                        V_j(t)
                    \Vert_{L^2(I)}^2
                    + \int_0^t
                        \Vert
                            \partial_x V_j(s)
                        \Vert_{L^2(I)}^2
                    ds
                \right)
                \leq C
                \Vert
                    \tau_1
                    - \tau_3
                \Vert_{X_{\sigma,T}}^2.
            \end{split}
        \end{align*}
        for some constant $C>0$, which depends continuously on $\Vert \tau_j \Vert_{X_{\sigma,T}} \in [0, y^\ast_1]$.
    \end{proposition}
    \begin{proof}
    Propositions \ref{prop_L2_estimate_for_linear_heat_eq}, \ref{prop_nonlinear_estimates_for_Nv}(i)-(ii), \ref{prop_nonlinear_estimates_for_cdv}(i) and the Young inequality yield
    \begin{align} \label{eq_differential_ineq_for_mathcal_V_in_L2}
        \begin{split}
            & \sum_{j=1,2}
            \left(
                \frac{\partial_t}{2} \Vert
                    \overline{\mathcal{V}}_j(t)
                \Vert_{L^2(I)}^2
                + \Vert
                    \partial_x \overline{\mathcal{V}}_j(t)
                \Vert_{L^2(I)}^2
            \right)\\
            & \leq C
                \sum_{j=1,2}
                    \Vert
                        \mathcal{N}_{V,0}(v_j, v_{j+2}, \tau_1, \tau_3)
                    \Vert_{L^2(I)}
                    \Vert
                        \overline{\mathcal{V}}_j(t)
                    \Vert_{L^2(I)}
            \\
            & + C \sum_{k=1,2}
                (
                    \Vert
                        \mathcal{N}_{V,j}
                    \Vert_{L^2(I)}
                    + \Vert
                        \mathcal{W}_j
                    \Vert_{L^2(I)}
                )
                \Vert
                    \overline{\mathcal{V}}_j(t)
                \Vert_{L^2(I)}\\
            & \leq C
            \vert
                \tau_1(t)
                - \tau_3(t)
            \vert^2
            \sum_{k=1, \ldots,4}
                \Vert
                    v_k(t)
                \Vert_{H^1(I)}^2\\
            & + \frac{1}{2} \sum_{j=1,2}
            (
                \Vert
                    \partial_x \mathcal{V}_j(t)
                \Vert_{L^2(I)}^2
                + \Vert
                    \partial_x \overline{\mathcal{V}}_j(t)
                \Vert_{L^2(I)}^2
            )\\
            & + C \sum_{j=1,2}
                \Vert
                    \overline{\mathcal{V}}_j(t)
                \Vert_{L^2(I)}^2\\
            & + C \sum_{k=2,4}
            (
                1
                + \Vert
                    v_k
                \Vert_{H^1(I)}
            )
            \sum_{k=1,2}
            (
                \Vert
                    \mathcal{V}_j
                \Vert_{L^2(I)}^2
                + \Vert
                    \overline{\mathcal{V}}_j
                \Vert_{L^2(I)}^2
            )\\
            & + C \sum_{j=1,2}
                \Vert
                    \mathcal{W}_j
                \Vert_{L^2(I)}^2.
        \end{split}
    \end{align}
    Therefore, we find from the Gronwall inequality, the estimate (\ref{eq_H1_estimate_for_vj_to_Leray_Schauder}), and Proposition \ref{prop_estimates_for_mathcal_V_j} that
    \begin{align} \label{eq_L2_estimate_for_mathcal_overline_V}
        \begin{split}
            & \sum_{j=1,2}
            \left(
                \Vert
                    \overline{\mathcal{V}}_j(t)
                \Vert_{L^2(I)}^2
                + \int_0^t
                    \Vert
                        \partial_x \overline{\mathcal{V}}_j(s)
                    \Vert_{L^2(I)}^2
                ds
            \right)\\
            & \leq C
            \Vert
                \tau_1
                - \tau_3
            \Vert_{X_{\sigma,T}}^2
            + C
            \Vert
                \tau_1
                - \tau_3
            \Vert_{X_{\sigma,T}}^2
            (
                1
                + \sum_{j=1,3}
                    \Vert
                        \tau_j
                    \Vert_{X_{\sigma,T}}^2
            )\\
            & \leq C
            \Vert
                \tau_1
                - \tau_3
            \Vert_{X_{\sigma,T}}^2
            (
                1
                + \sum_{k=1,3}
                    \Vert
                        \tau_k
                    \Vert_{X_{\sigma,T}}^2
            )
        \end{split}
    \end{align}
    for some constant $C>0$.
    Using Proposition \ref{prop_estimates_for_mathcal_V_j}, (\ref{eq_L2H2_estimate_for_mathcal_V}), and the definition $V_j = \overline{\mathcal{V}}_j + \mathcal{V}_j$, we conclude that
    \begin{align}\label{eq_L2_estimate_for_V_j}
        \begin{split}
            & \sum_{j=1,2}
            \left(
                \Vert
                    V_j(t)
                \Vert_{L^2(I)}^2
                + \int_0^t
                    \Vert
                        \partial_x V_j(s)
                    \Vert_{L^2(I)}^2
                ds
            \right)\\
            & \leq C
            \Vert
                \tau_1
                - \tau_3
            \Vert_{X_{\sigma,T}}^2
            (
                1
                + \sum_{k=1,3}
                    \Vert
                        \tau_k
                    \Vert_{X_{\sigma,T}}^2
            ).
        \end{split}
    \end{align}
\end{proof}

\begin{proposition} \label{prop_H1_estimate_for_V_j}
    Let $T \in [0,1]$ and $(\tau_j, \tau_{j+1}) \in B_{X_{\sigma, T}^2, y^\ast_1}$ for $j=1,3$.
    Assume that $T$ satisfies (\ref{eq_assumption_for_tau1_tau2_for_exponential_deal_with_term}).
    Then
    \begin{align*}% \label{eq_L2_estimate_for_V_j}
        \begin{split}
            & \sum_{j=1,2} \left[
                    \int_0^t
                        \Vert
                            \partial_s V_j(s)
                        \Vert_{L^2(I)}^2
                    ds
                    + \Vert
                        \partial_x V_j(t)
                    \Vert_{L^2(I)}^2
                    + \int_0^t
                        \Vert
                            \partial_x^2 V_j(s)
                        \Vert_{L^2(I)}^2
                    ds
                \right]\\
            & \leq C
            \Vert
                \tau_1
                - \tau_3
            \Vert_{X_{\sigma,T}}^2.
        \end{split}
    \end{align*}
    for some constant $C>0$, which depend continuously on $\Vert \tau_j \Vert_{X_{\sigma,T}} \in [0, y^\ast_1]$.
\end{proposition}
\begin{proof}
    Therefore, in the same way as (\ref{eq_differential_ineq_for_mathcal_V_in_L2}), we use Propositions \ref{prop_estimates_for_difference_between_nonlinear_terms}, \ref{prop_estimates_for_difference_between_convection_terms}, and (\ref{prop_estimates_for_difference_between_boundary_extension_terms}) to see that
    \begin{align*}
        & \sum_{j=1,2}
            \int_0^t
                \Vert
                    \mathcal{N}_{V,0}(v_j, v_{j+2}, \tau_1, \tau_3)
                    - \mathcal{N}_{V,j}(v_1, v_2, v_3, v_4)
                    - \mathcal{W}_j
                \Vert_{L^2(I)}^2
            ds\\
        & \leq C \int_0^t
            \Vert
                \tau_1
                - \tau_3
            \Vert_{X_{\sigma,T}}^2
            \sum_{k=2,4}
                \Vert
                    v_k(t)
                \Vert_{H^1(I)}^2
            + \sum_{j=1,2}
                \Vert
                    \partial_x V_j(t)
                \Vert_{L^2(I)}^2
        ds\\
        & + C \sum_{k=2,4} \int_0^t
            (
                1
                + \Vert
                    v_k
                \Vert_{H^1(I)}
            )
            \sum_{j=1,2}
                \Vert
                    V_j
                \Vert_{L^2(I)}^2
        ds\\
        & + C \sum_{j=1,2}
            \int_0^t
                \Vert
                    \mathcal{W}_j
                \Vert_{L^2(I)}^2
            ds\\
        & \leq C \Vert
            \tau_1
            - \tau_3
        \Vert_{X_{\sigma,T}}^2
        (
            1
            + \sum_{j=1,3}
                \Vert
                    \tau_j
                \Vert_{X_{\sigma,T}}^2
        )\\
        & + C (T+1)
        \Vert
            \tau_1
            - \tau_3
        \Vert_{X_{\sigma,T}}^2
        (
            1
            + \sum_{k=1,3}
                \Vert
                    \tau_k
                \Vert_{X_{\sigma,T}}^2
        )\\ 
        %& + C (
        %    1
        %    + \sum_{j=1,3}
        %    \Vert
        %        \tau_j
        %    \Vert_{X_{\sigma,T}}^2
        %)
        %\Vert
        %    \tau_1
        %    - \tau_3
        %\Vert_{X_{\sigma,T}}^2\\
        & \leq C (
            1
            + \sum_{j=1,3}
                \Vert
                    \tau_j
                \Vert_{X_{\sigma,T}}^2
        )
        \Vert
            \tau_1
            - \tau_3
        \Vert_{X_{\sigma,T}}^2
    \end{align*}
    for some constant $C>0$.
    We find from Proposition \ref{prop_L_infty_H1_estimate_for_v} that
    \begin{align}\label{eq_H1_estimate_for_mathcal_overline_V}
        \begin{split}
            & \sum_{j=1,2} \left[
                \int_0^t
                    \Vert
                        \partial_s \overline{\mathcal{V}}_j(s)
                    \Vert_{L^2(I)}^2
                ds
                + \Vert
                    \partial_x \overline{\mathcal{V}}_j(t)
                \Vert_{L^2(I)}^2
                + \int_0^t
                    \Vert
                        \partial_x^2 \overline{\mathcal{V}}_j(s)
                    \Vert_{L^2(I)}^2
                ds
            \right]\\
            & \leq C
            \sup_{0<t<T} \Vert
                \overline{\mathcal{V}}_j(s)
            \Vert_{L^2(I)}^2
            \int_0^t
                (
                    1
                    + \sum_{k=1,3}
                        \vert
                            \tau_k^\prime(s)
                        \vert^2
                )
            ds\\
            & + C
            \sum_{j=1,2}
                \int_0^t
                    \Vert
                        \mathcal{N}_{V,0}(v_1, v_3, \tau_1, \tau_3)
                        - \mathcal{N}_{V,j}(v_1, v_2, v_3, v_4)
                        - \mathcal{W}_j
                    \Vert_{L^2(I)}^2
                ds\\
            & \leq C
            \Vert
                \tau_1(t)
                - \tau_3(t)
            \Vert_{X_{\sigma,T}}^2
            (
                1
                + \sum_{k=1,3}
                    \Vert
                        \tau_k
                    \Vert_{X_{\sigma,T}}^2
            )
            (
                T
                + T^{\frac{1}{2}+2\delta} \Vert
                    \tau_j
                \Vert_{X_{\sigma,T}}^2
            )
            \\
            & +  C (
                1
                + \sum_{j=1,3}
                    \Vert
                        \tau_j
                    \Vert_{X_{\sigma,T}}^2
            )
            \Vert
                \tau_1
                - \tau_3
            \Vert_{X_{\sigma,T}}^2\\
            & \leq C (
                1
                + \sum_{j=1,3}
                    \Vert
                        \tau_j
                    \Vert_{X_{\sigma,T}}^2
            )^2
            \Vert
                \tau_1
                - \tau_3
            \Vert_{X_{\sigma,T}}^2.
        \end{split}
    \end{align}
    Using Proposition \ref{prop_estimates_for_mathcal_V_j} and the definition $V_j = \overline{\mathcal{V}}_j + \mathcal{V}_j$, we obtain
    \begin{align} \label{eq_H1_estimate_for_V_j}
        \begin{split}
            & \sum_{j=1,2} \left[
                    \int_0^t
                        \Vert
                            \partial_s V_j(s)
                        \Vert_{L^2(I)}^2
                    ds
                    + \Vert
                        \partial_x V_j(t)
                    \Vert_{H^1(I)}^2
                    + \int_0^t
                        \Vert
                            \partial_x^2 V_j(s)
                        \Vert_{L^2(I)}^2
                    ds
                \right]\\
            & \leq C (
                1
                + \sum_{j=1,3}
                \Vert
                    \tau_j
                \Vert_{X_{\sigma,T}}^2
            )^2
            \Vert
                \tau_1
                - \tau_3
            \Vert_{X_{\sigma,T}}^2.
        \end{split}
    \end{align}
\end{proof}

    \begin{proposition}\label{prop_estimate_for_dt_mathcal_N_j}
        %It follows that
        %\begin{align*}
        %    \sum_{j=1,2}
        %    \int_0^t
        %        \Vert
        %            \partial_t \mathcal{N}_{V,j}(v_1, v_2, v_3, v_4)
        %        \Vert_{L^2(I)}^2
        %    ds
        %    \leq C (
        %        1
        %        + \sum_{j=1,3}
        %        \Vert
        %            \tau_j
        %        \Vert_{X_{\sigma,T}}^2
        %    )^4
        %    \Vert
        %        \tau_1
        %        - \tau_3
        %    \Vert_{X_{\sigma,T}}^2
        %\end{align*}
        %for some constant $C$ which is bounded for $T \in [0,1]$.
        Let $T \in (0,1]$ and $(\tau_j, \tau_{j+1}) \in B_{X_{\sigma, T}^2, y^\ast_1}$ for $j=1,3$.
        Assume that $T$ satisfies (\ref{eq_assumption_for_tau1_tau2_for_exponential_deal_with_term}).
        Then
        \begin{align*}
            \sum_{j=1,2}
            \int_0^t
                \Vert
                    \partial_t \mathcal{N}_{V,j}(v_1, v_2, v_3, v_4)
                \Vert_{L^2(I)}^2
            ds
            \leq C
            \Vert
                \tau_1
                - \tau_3
            \Vert_{X_{\sigma,T}}^2
        \end{align*}
        for some constant $C>0$ which depends continuously on $\Vert \tau_j \Vert_{X_{\sigma,T}} \in [0, y^\ast_1]$.
    \end{proposition}
    \begin{proof}
    We find from Proposition \ref{prop_nonlinear_estimates_for_Nv} that
    \begin{align*}
        & \sum_{j=1,2}
        \int_0^t
            \Vert
                \partial_t \mathcal{N}_{V,j}(v_1, v_2, v_3, v_4)
            \Vert_{L^2(I)}^2
        ds\\
        & \leq C
        \int_0^t
            (
                \sum_{j=1,3}
                    \Vert
                        \partial_t v_j
                    \Vert_{L^2(I)}^2
                \Vert
                    v_2
                \Vert_{H^1(I)}^2
                + \Vert
                    \partial_t v_2
                \Vert_{L^2(I)}^2
            )
            \Vert
                V_1
            \Vert_{H^1(I)}^2\\
        & \quad\quad\quad
            + \Vert
                v_2
            \Vert_{H^1(I)}^2
            \Vert
                \partial_t V_1
            \Vert_{L^2(I)}^2
            + \Vert
                \partial_t v_3
            \Vert_{L^2(I)}^2
            \Vert
                V_2
            \Vert_{H^1(I)}^2
            + \Vert
                \partial_t V_2
            \Vert_{L^2(I)}^2
        ds\\
        &+ C \sum_{j=2,4}
            \int_0^t
                \Vert
                    \partial_t v_j
                \Vert_{L^2(I)}^2
                \Vert
                    V_2
                \Vert_{H^1(I)}^2
                + \Vert
                    v_j
                \Vert_{H^1(I)}^2
                \Vert
                    \partial_t V_2
                \Vert_{L^2(I)}^2
            ds\\
        & =: I_1 + I_2 + I_3 + I_4 + I_5 + I_6.
    \end{align*}
    We find from (\ref{eq_H1_estimate_for_vj_to_Leray_Schauder}) and (\ref{eq_H1_estimate_for_V_j}) that
    \begin{align*}
        I_1
        & \leq C 
        (
            1
            + \sum_{j=1,3}
                \Vert
                    \tau_j
                \Vert_{X_{\sigma,T}}^2
        )
        (
            1
            + \sum_{j=1,3}
                \Vert
                    \tau_j
                \Vert_{X_{\sigma,T}}^2
        )
        (
            1
            + \sum_{j=1,3}
            \Vert
                \tau_j
            \Vert_{X_{\sigma,T}}^2
        )^2
        \Vert
            \tau_1
            - \tau_3
        \Vert_{X_{\sigma,T}}^2\\
        & \leq C (
            1
            + \sum_{j=1,3}
            \Vert
                \tau_j
            \Vert_{X_{\sigma,T}}^2
        )^4
        \Vert
            \tau_1
            - \tau_3
        \Vert_{X_{\sigma,T}}^2
    \end{align*}
    Similarly, we observe that
    \begin{align*}
        I_2 + I_3 + I_5 + I_6
        \leq C (
            1
            + \sum_{j=1,3}
            \Vert
                \tau_j
            \Vert_{X_{\sigma,T}}^2
        )^4
        \Vert
            \tau_1
            - \tau_3
        \Vert_{X_{\sigma,T}}^2,
    \end{align*}
    and
    \begin{align*}
        I_4
        \leq C (
            1
            + \sum_{j=1,3}
            \Vert
                \tau_j
            \Vert_{X_{\sigma,T}}^2
        )^2
        \Vert
            \tau_1
            - \tau_3
        \Vert_{X_{\sigma,T}}^2.
    \end{align*}
    We obtain that
    \begin{align} \label{eq_L2_estimate_for_dt_mathcal_N_V_j}
        \begin{split}
            & \sum_{j=1,2}
            \int_0^t
                \Vert
                    \partial_t \mathcal{N}_{V,j}(v_1, v_2, v_3, v_4)
                \Vert_{L^2(I)}^2
            ds\\
            & \leq C (
                1
                + \sum_{j=1,3}
                \Vert
                    \tau_j
                \Vert_{X_{\sigma,T}}^2
            )^4
            \Vert
                \tau_1
                - \tau_3
            \Vert_{X_{\sigma,T}}^2.
        \end{split}
    \end{align}
\end{proof}
%We next estimate the drift-terms.
\begin{proposition} \label{prop_integration_by_parts_for_dt_mathcal_V0_dt_mathcal_Vj}
    Let $T \in (0,1]$ and $(\tau_j, \tau_{j+1}) \in B_{X_{\sigma, T}^2, y^\ast_1}$ for $j=1,3$.
    Assume that $T$ satisfies (\ref{eq_assumption_for_tau1_tau2_for_exponential_deal_with_term}).
    Then
    \begin{align*}
        \begin{split}
            & \left \vert
                \int_I
                    \partial_t \mathcal{N}_{V,0}(v_j, v_{j+2}, \tau_1, \tau_3) \partial_t \overline{\mathcal{V}}_j
                dx
            \right \vert\\
            & \leq C
            (
                1
                + \sum_{j=1,3}
                \Vert
                    \tau_j
                \Vert_{X_{\sigma,T}}^2
            )^4
            \Vert
                \tau_1
                - \tau_3
            \Vert_{X_{\sigma,T}}^2\\
            & + C T^{-\frac{1}{2}+2\delta} \Vert
                \tau_1
                - \tau_3
            \Vert_{X_{\sigma,T}}^2
            \sum_{j=1,3}
                \Vert
                    \tau_j
                \Vert_{X_{\sigma,T}}^2\\
            & + C \Vert
                \partial_t \overline{\mathcal{V}}_j
            \Vert_{L^2(I)}^2
            + \frac{1}{2} \sum_{j=1,2}
            \Vert
                \partial_x \partial_t  \overline{\mathcal{V}}_j
            \Vert_{L^2(I)}^2.
        \end{split}
    \end{align*}
    for some constant $C>0$ which depends continuously on $\Vert \tau_j \Vert_{X_{\sigma,T}} \in [0, y^\ast_1]$.
\end{proposition}
\begin{proof}
    We find from Proposition \ref{prop_estimates_for_difference_between_convection_terms} and the estimate (\ref{eq_H3_estimate_for_vj_to_Leray_Schauder}) that
    \begin{align}\label{eq_integration_by_parts_for_dt_mathcal_V0_dt_mathcal_Vj}
        \begin{split}
            & \left \vert
                \int_I
                    \partial_t \mathcal{N}_{V,0}(v_j, v_{j+2}, \tau_1, \tau_3) \partial_t \overline{\mathcal{V}}_j
                dx
            \right \vert\\
            & \leq C \left[
                (
                \sum_{k=1,3}
                    \vert
                        \tau_k^\prime
                    \vert
                    \vert
                        \tau_1
                        - \tau_3
                    \vert
                    + \vert
                        \tau_1^\prime
                        - \tau_3^\prime
                    \vert
                )
                \sum_{j=1,\ldots,4}
                    \Vert
                        \partial_x v_j
                    \Vert_{L^2(I)}
            \right.\\
            & \quad\quad\quad
            + \vert
                \tau_1
                - \tau_3
            \vert
            \sum_{k=1,\ldots,4}
                \Vert
                    \partial_t \partial_x v_k
                \Vert_{L^2(I)}\\
            & \left. \quad\quad\quad
                + \sum_{k=1,3}
                    \vert
                        \tau_k^\prime
                    \vert
                \sum_{j=1,2}
                    \Vert
                        \partial_x V_j
                    \Vert_{L^2(I)}
                + \sum_{j=1,2}
                    \Vert
                        \partial_t \partial_x V_j
                    \Vert_{L^2(I)}
            \right]
            \Vert
                \partial_t \overline{\mathcal{V}}_j
            \Vert_{L^2(I)}\\
            & =: (
                I_1 + I_2 + I_3
            )
            \Vert
                \partial_t \overline{\mathcal{V}}_j
            \Vert_{L^2(I)}.
        \end{split}
    \end{align}
    We observe from the Young inequality that
    \begin{align} \label{eq_definition_of_tilde_I1}
        \begin{split}
            & (
                I_1
                + I_2
            )
            \Vert
                \partial_t \overline{\mathcal{V}}_j
            \Vert_{L^2(I)}\\
            & \leq C
            (
                \sum_{k=1,3}
                \vert
                    \tau_k^\prime
                \vert^2
                \vert
                    \tau_1
                    - \tau_3
                \vert^2
                + \vert
                    \tau_1^\prime
                    - \tau_3^\prime
                \vert^2
            )
            \sum_{j=1,\ldots,4}
                \Vert
                    \partial_x v_j
                \Vert_{L^2(I)}^2
            \\
            & + \vert
                \tau_1
                - \tau_3
            \vert^2 
            \sum_{k=1,\ldots,4}
                \Vert
                    \partial_t \partial_x v_k
                \Vert_{L^2(I)}^2
            + C \Vert
                \partial_t \overline{\mathcal{V}}_j
            \Vert_{L^2(I)}^2\\
            & =: \tilde{I}_1
            + C \Vert
                \partial_t \overline{\mathcal{V}}_j
            \Vert_{L^2(I)}^2
        \end{split}
    \end{align}
    and
    \begin{align}  \label{eq_definition_of_tilde_I2}
        \begin{split}
            & I_3
            \Vert
                \partial_t \overline{\mathcal{V}}_j
            \Vert_{L^2(I)}\\
            %& \leq C 
            %\sum_{k=1,3}
            %    \vert
            %        \tau_k^\prime
            %    \vert
            %\sum_{j=1,2}
            %    \Vert
            %        \partial_x V_j
            %    \Vert_{L^2(I)}
            %\Vert
            %    \partial_t \overline{\mathcal{V}}_j
            %\Vert_{L^2(I)}\\
            %& + C \sum_{j=1,2}
            %    (
            %        \Vert
            %            \partial_t \partial_x \mathcal{V}_j
            %        \Vert_{L^2(I)}
            %        + \Vert
            %            \partial_t \partial_x \overline{\mathcal{V}}_j
            %        \Vert_{L^2(I)}
            %    )
            %\Vert
            %    \partial_t \overline{\mathcal{V}}_j
            %\Vert_{L^2(I)}\\
            & \leq 
            \sum_{k=1,3}
                \vert
                    \tau_k^\prime
                \vert^2
            \sum_{j=1,2}
                \Vert
                    \partial_x V_j
                \Vert_{L^2(I)}^2
            \\
            & + C \sum_{j=1,2}
                    \Vert
                        \partial_x \partial_t \mathcal{V}_j
                    \Vert_{L^2(I)}^2
                + C \Vert
                    \partial_t \overline{\mathcal{V}}_j
                \Vert_{L^2(I)}^2
                + \frac{1}{2} \sum_{j=1,2}
                \Vert
                    \partial_x \partial_t  \overline{\mathcal{V}}_j
                \Vert_{L^2(I)}^2\\
            & =: \tilde{I}_2
            + C \Vert
                \partial_t \overline{\mathcal{V}}_j
            \Vert_{L^2(I)}^2
            + \frac{1}{2} \sum_{j=1,2}
            \Vert
                \partial_x \partial_t  \overline{\mathcal{V}}_j
            \Vert_{L^2(I)}^2
        \end{split}
    \end{align}
    for some constant $C>0$.
    %Note that the last term can be absorbed into the left-hand side when we use integration by parts.
    We deduce from (\ref{eq_H1_estimate_for_vj_to_Leray_Schauder}) that
    \begin{align} \label{eq_estimate_for_tilde_I1_to_estimate_dt_overline_mathcal_Vj}
        \begin{split}
            \int_0^t
                \tilde{I}_1
            ds
            & \leq C T^{\frac{1}{2}+ 2\delta}
            \Vert
                \tau_1
                - \tau_3
            \Vert_{X_{\sigma,T}}^2
            (
                \sum_{j=1,3}
                    \Vert
                        \tau_j
                    \Vert_{X_{\sigma,T}}^2
                + 1
            )\\
            & + C \Vert
                \tau_1
                - \tau_3
            \Vert_{X_{\sigma,T}}^2
            \sum_{j=1,3}
                (
                    1
                    + T^{-\frac{1}{2}+2\delta}
                    \Vert
                        \tau_j
                    \Vert_{X_{\sigma,T}}^2
                    + \Vert
                        \tau_j
                    \Vert_{X_{\sigma,T}}^6
                )
        \end{split}
    \end{align}
    and
    \begin{align} \label{eq_estimate_for_tilde_I2_to_estimate_dt_overline_mathcal_Vj}
        \begin{split}
            \int_0^t
                \tilde{I}_2
            ds
            & \leq C T^{\frac{1}{2}+2\delta}
            \sum_{j=1,3}
                \Vert
                    \tau_j
                \Vert_{X_{\sigma,T}}^2
            (
                1
                + \sum_{j=1,3}
                \Vert
                    \tau_j
                \Vert_{X_{\sigma,T}}^2
            )^2
            \Vert
                \tau_1
                - \tau_3
            \Vert_{X_{\sigma,T}}^2\\
            & + C (
                1
                + \sum_{j=1,3}
                    \Vert
                        \tau_j
                    \Vert_{X_{\sigma,T}}^2
            )^4
            \Vert
                \tau_1
                - \tau_3
            \Vert_{X_{\sigma,T}}^2\\
            & + C T^{-\frac{1}{2}+2\delta} \Vert
                \tau_1
                - \tau_3
            \Vert_{X_{\sigma,T}}^2
            \sum_{j=1,3}
                \Vert
                    \tau_j
                \Vert_{X_{\sigma,T}}^2\\
            & \leq 
            C
            (
                1
                + \sum_{j=1,3}
                \Vert
                    \tau_j
                \Vert_{X_{\sigma,T}}^2
            )^4
            \Vert
                \tau_1
                - \tau_3
            \Vert_{X_{\sigma,T}}^2\\
            & + C T^{-\frac{1}{2}+2\delta} \Vert
                \tau_1
                - \tau_3
            \Vert_{X_{\sigma,T}}^2
            \sum_{j=1,3}
                \Vert
                    \tau_j
                \Vert_{X_{\sigma,T}}^2.
        \end{split}
    \end{align}
    Summing up these inequalities, we obtain (\ref{eq_integration_by_parts_for_dt_mathcal_V0_dt_mathcal_Vj}).
\end{proof}

\begin{proposition}
    Let $T \in (0,1]$ and $(\tau_j, \tau_{j+1}) \in B_{X_{\sigma, T}^2, y^\ast_1}$ for $j=1,3$.
    Assume that $T$ satisfies (\ref{eq_assumption_for_tau1_tau2_for_exponential_deal_with_term}).
    Then
    \begin{align*}
        \begin{split}
            & \sum_{j=1,2} \Vert
                    \partial_t \overline{\mathcal{V}}_j(t)
                \Vert_{L^2(I)}^2
            + \int_0^t
                \Vert
                    \partial_x \partial_s \overline{\mathcal{V}}_j(s)
                \Vert_{L^2(I)}^2
            ds\\
            & \leq C_1
            \Vert
                \tau_1
                - \tau_3
            \Vert_{X_{\sigma,T}}^2
            + C_2 T^{-\frac{1}{2}+2\delta} \Vert
                \tau_1
                - \tau_3
            \Vert_{X_{\sigma,T}}^2
        \end{split}
    \end{align*}
    for some constants $C_1, C_2>0$ which depends continuously on $\Vert \tau_j \Vert_{X_{\sigma,T}} \in [0, y^\ast_1]$.
    Moreover $V$ also satisfies the same inequality for some constants $C_1, C_2 >0$.
\end{proposition}
\begin{proof}
    We see in the way as Corollary \ref{cor_L2_a_priori_estimate_for_tildev} that
    \begin{align*}
        \begin{split}
            & \sum_{j=1,2} \Vert
                    \partial_t \overline{\mathcal{V}}_j(t)
                \Vert_{L^2(I)}^2
            + \int_0^t
                \Vert
                    \partial_x \partial_s \overline{\mathcal{V}}_j(s)
                \Vert_{L^2(I)}^2
            ds\\
            & \leq C \sum_{k=1,3} \Vert
                    \tau_k^\prime
                \Vert_{L^\infty(0,T)}^2
            \sum_{j=1,2}
                \left[
                    \Vert
                        \overline{\mathcal{V}}_j(t)
                    \Vert_{L^2(I)}^2
                    + \int_0^t
                        \Vert
                            \partial_x  \overline{\mathcal{V}}_j(s)
                        \Vert_{L^2(I)}^2
                    ds
                \right]\\
            & + C \sum_{k=1,3} \Vert
                    \tau_k^\prime
                \Vert_{L^\infty(0,T)}^2
            \sum_{j=1,2}
                \int_0^t
                    \Vert
                        \overline{\mathcal{V}}_j(s)
                    \Vert_{H^2(I)}^2
                    + \Vert
                        \partial_t \overline{\mathcal{V}}_j(s)
                    \Vert_{L^2(I)}^2
                ds\\
            & + C \sum_{j=1,2}
                \int_0^t
                    \sum_{k=1,3}
                        \vert
                            \tau_k^\prime(t)
                        \vert^4
                    \Vert
                        \overline{\mathcal{V}}_j(s)
                    \Vert_{L^2(I)}^2
                ds\\
            & + C \sum_{j=1,2} \sup_{0<t<T}
                \Vert
                    \overline{\mathcal{V}}_j(t)
                \Vert_{L^2(I)}^2
            \sum_{k=1,3}
                \int_0^t
                    \vert
                        \tau_k^{\prime\prime}(s)
                    \vert^2
                ds\\
            & + C \sum_{j=1,2}
            \int_0^t
                \tilde{I}_1
                + \tilde{I}_2
            ds\\
            & + C \sum_{j=1,2}
            \int_0^t
                \Vert
                    \partial_t \mathcal{N}_{V,j}(v_1, v_2, v_3, v_4)
                \Vert_{L^2(I)}^2
                + \Vert
                    \partial_t \mathcal{W}_j
                \Vert_{L^2(I)}^2
            ds\\
            & =: J_1 + J_2 + J_3 + J_4 + J_5 + J_6,
        \end{split}
    \end{align*}
    where $\tilde{I}_1$ and $\tilde{I}_2$ are functions defined in (\ref{eq_definition_of_tilde_I1}) and (\ref{eq_definition_of_tilde_I2}), respectively.
    We see from (\ref{eq_L2_estimate_for_mathcal_overline_V}) and (\ref{eq_H1_estimate_for_mathcal_overline_V}) that
    \begin{align*}
        J_1 +J_2
        \leq C T^{-\frac{1}{2}+2\delta}(
            1
            + \sum_{j=1,3}
                \Vert
                    \tau_j
                \Vert_{X_{\sigma,T}}^2
        )^2
        \Vert
            \tau_1
            - \tau_3
        \Vert_{X_{\sigma,T}}^2.
    \end{align*}
    We use (\ref{eq_L2_estimate_for_mathcal_overline_V}) again to see that
    \begin{align*}
        J_3 + J_4
        \leq C T^{2\delta}
        (
            1
            + \sum_{j=1,3}
                \Vert
                    \tau_j
                \Vert_{X_{\sigma,T}}^2
        )
        \Vert
            \tau_1
            - \tau_3
        \Vert_{X_{\sigma,T}}^2.
    \end{align*}
    Combining the above estimates with the estimates (\ref{eq_L2_estimate_for_dt_mathcal_N_V_j}), (\ref{eq_estimate_for_tilde_I1_to_estimate_dt_overline_mathcal_Vj}), and (\ref{eq_estimate_for_tilde_I2_to_estimate_dt_overline_mathcal_Vj}), and Propositions \ref{prop_estimates_for_mathcal_V_j}, we see that
    \begin{align*}
        \begin{split}
            & \sum_{j=1,2} \Vert
                    \partial_t \overline{\mathcal{V}}_j(t)
                \Vert_{L^2(I)}^2
            + \int_0^t
                \Vert
                    \partial_x \partial_s \overline{\mathcal{V}}_j(s)
                \Vert_{L^2(I)}^2
            ds\\
            & \leq C
            (
                1
                + \sum_{j=1,3}
                    \Vert
                        \tau_j
                    \Vert_{X_{\sigma,T}}^2
            )^4
            %& \leq C
            %(
            %    1
            %    + \sum_{j=1,3}
            %        \Vert
            %            \tau_j
            %        \Vert_{X_{\sigma,T}}^2
            %)^4
            \Vert
                \tau_1
                - \tau_3
            \Vert_{X_{\sigma,T}}^2\\
            & + C T^{-\frac{1}{2}+2\delta} \Vert
                \tau_1
                - \tau_3
            \Vert_{X_{\sigma,T}}^2
            (
                1
                + \sum_{j=1,3}
                    \Vert
                        \tau_j
                    \Vert_{X_{\sigma,T}}^2
            )^2.
        \end{split}
    \end{align*}
    Therefore, we find from Proposition \ref{prop_estimates_for_mathcal_V_j} that
    \begin{align} \label{eq_L2_estimate_for_partial_t_V_j_}
        \begin{split}
            & \sum_{j=1,2} \Vert
                    \partial_t V_j(t)
                \Vert_{L^2(I)}^2
            + \int_0^t
                \Vert
                    \partial_x \partial_s V_j(s)
                \Vert_{L^2(I)}^2
            ds\\
            & \leq C
            (
                1
                + \sum_{j=1,3}
                    \Vert
                        \tau_j
                    \Vert_{X_{\sigma,T}}^2
            )^4
            \Vert
                \tau_1
                - \tau_3
            \Vert_{X_{\sigma,T}}^2\\
            & + C T^{-\frac{1}{2}+2\delta} \Vert
                \tau_1
                - \tau_3
            \Vert_{X_{\sigma,T}}^2
            (
                1
                + \sum_{j=1,3}
                    \Vert
                        \tau_j
                    \Vert_{X_{\sigma,T}}^2
            )^2.
        \end{split}
    \end{align}
\end{proof}

\begin{proposition} \label{prop_integration_by_parts_for_dt_mathcal_N_V}
    Let $T \in [0,1]$ and $(\tau_j, \tau_{j+1}) \in B_{X_{\sigma, T}^2, y^\ast_1}$ for $j=1,3$.
    Assume that $T$ satisfies (\ref{eq_assumption_for_tau1_tau2_for_exponential_deal_with_term}).
    Then
    \begin{align*}
        \begin{split}
            &\Vert
                \partial_t \mathcal{N}_{V,0}(v_j, v_{j+2}, \tau_1, \tau_3)
            \Vert_{L^2_t L^2_x(Q_T)}^2\\
            & \leq C_1
            \Vert
                \tau_1
                - \tau_3
            \Vert_{X_{\sigma,T}}^2
            + C_2 T^{-\frac{1}{2}+2\delta} \Vert
                \tau_1
                - \tau_3
            \Vert_{X_{\sigma,T}}^2,
        \end{split}
    \end{align*}
    for some constants $C_1, C_2>0$ which depend continuously on $\Vert \tau_j \Vert_{X_{\sigma,T}} \in [0, y^\ast_1]$.
\end{proposition}
\begin{proof}
    In the same way as (\ref{eq_integration_by_parts_for_dt_mathcal_V0_dt_mathcal_Vj}), we deduce that
    \begin{align*}
        \begin{split}
            &\Vert
                \partial_t \mathcal{N}_{V,0}(v_j, v_{j+2}, \tau_1, \tau_3)
            \Vert_{L^2_t L^2_x(Q_T)}^2\\
            & \leq C \int_0^t
                (
                \sum_{k=1,3}
                    \vert
                        \tau_k^\prime
                    \vert^2
                    \vert
                        \tau_1
                        - \tau_3
                    \vert^2
                    + \vert
                        \tau_1^\prime
                        - \tau_3^\prime
                    \vert^2
                )
                \sum_{j=1,\ldots,4}
                    \Vert
                        \partial_x v_j
                    \Vert_{L^2(I)}^2
            ds\\
            & + C \int_0^t
                \vert
                    \tau_1
                    - \tau_3
                \vert^2
                \sum_{k=1,\ldots,4}
                    \Vert
                        \partial_t \partial_x v_k
                    \Vert_{L^2(I)}^2
            ds\\
            & + C \int_0^t
                \sum_{k=1,3}
                    \vert
                        \tau_k^\prime
                    \vert^2
                \sum_{j=1,2}
                    \Vert
                        \partial_x V_j
                    \Vert_{L^2(I)}^2
                + \sum_{j=1,2}
                    \Vert
                        \partial_t \partial_x V_j
                    \Vert_{L^2(I)}^2
            ds\\
            & \leq C T^{\frac{1}{2} + 2\delta}
            (
                \sum_{k=1,3}
                    \Vert
                        \tau_k
                    \Vert_{X_{\sigma,T}}^2
                + 1
            )^2
            \Vert
                \tau_1
                - \tau_3
            \Vert_{X_{\sigma,T}}^2\\
            & + C
            \Vert
                \tau_1
                - \tau_3
            \Vert_{X_{\sigma,T}}^2
            \sum_{k=1,3}
                [
                    1
                    + T^{-\frac{1}{2}+2\delta}
                    \Vert
                        \tau_k
                    \Vert_{X_{\sigma,T}}^2
                    + \Vert
                        \tau_k
                    \Vert_{X_{\sigma,T}}^6
                ]\\
            & + C
            T^{\frac{1}{2} + 2 \delta}
            (
                1
                + \sum_{j=1,3}
                \Vert
                    \tau_j
                \Vert_{X_{\sigma,T}}^2
            )^2
            \Vert
                \tau_1
                - \tau_3
            \Vert_{X_{\sigma,T}}^2
            \\
            & + C
            (
                1
                + \sum_{j=1,3}
                    \Vert
                        \tau_j
                    \Vert_{X_{\sigma,T}}^2
            )^4
            \Vert
                \tau_1
                - \tau_3
            \Vert_{X_{\sigma,T}}^2\\
            & + C T^{-\frac{1}{2}+2\delta} \Vert
                \tau_1
                - \tau_3
            \Vert_{X_{\sigma,T}}^2
            (
                1
                + \sum_{j=1,3}
                    \Vert
                        \tau_j
                    \Vert_{X_{\sigma,T}}^2
            )^2\\
            & \leq C
            (
                1
                + \sum_{j=1,3}
                    \Vert
                        \tau_j
                    \Vert_{X_{\sigma,T}}^2
            )^4
            \Vert
                \tau_1
                - \tau_3
            \Vert_{X_{\sigma,T}}^2\\
            & + C T^{-\frac{1}{2}+2\delta} \Vert
                \tau_1
                - \tau_3
            \Vert_{X_{\sigma,T}}^2
            (
                1
                + \sum_{j=1,3}
                    \Vert
                        \tau_j
                    \Vert_{X_{\sigma,T}}^2
            )^2.
        \end{split}
    \end{align*}
\end{proof}

\begin{proposition} \label{prop_H3_a_pripori_estimate_for_V}
    Let $T \in [0,1]$ and $(\tau_j, \tau_{j+1}) \in B_{X_{\sigma, T}^2, y^\ast_1}$ for $j=1,3$.
    Assume that $T$ satisfies (\ref{eq_assumption_for_tau1_tau2_for_exponential_deal_with_term}).
    Then
    \begin{align} \label{eq_H3_a_pripori_estimate_for_V}
        \begin{split}
            & \sum_{j=1,2}
                \Vert
                    \partial_x \partial_t \overline{\mathcal{V}}_j(t)
                \Vert_{L^2(I)}^2\\
            & + \sum_{j=1,2}
            \left[
                \int_0^t
                    \Vert
                        \partial_s^2 \overline{\mathcal{V}}_j(s)
                    \Vert_{L^2(I)}^2
                ds
                + \int_0^t
                    \Vert
                        \partial_x^2 \partial_s \overline{\mathcal{V}}_j(s)
                    \Vert_{L^2(I)}^2
                ds
            \right]\\
            & \leq C_1
            \Vert
                \tau_1
                - \tau_3
            \Vert_{X_{\sigma,T}}^2
            + C_2 T^{-\frac{1}{2}+2\delta} \Vert
                \tau_1
                - \tau_3
            \Vert_{X_{\sigma,T}}^2
        \end{split}
    \end{align}
    for some constants $C_1, C_2>0$ which depend continuously on $\Vert \tau_j \Vert_{X_{\sigma,T}} \in [0, y^\ast_1]$.
    Moreover, $V_j$ for $j=1,2$ also satisfy the same inequality for some $C_1, C_2>0$.
\end{proposition}
\begin{proof}
    By Corollary \ref{cor_H3_estimate_for_v}, we observe that
    \begin{align*}
        \begin{split}
            & \sum_{j=1,2}
            \left[
                \int_0^t
                    \Vert
                        \partial_s^2 \overline{\mathcal{V}}_j(s)
                    \Vert_{L^2(I)}^2
                ds
                + \Vert
                    \partial_x \partial_t \overline{\mathcal{V}}_j(t)
                \Vert_{L^2(I)}^2
                + \int_0^t
                    \Vert
                        \partial_x^2 \partial_s \overline{\mathcal{V}}_j(s)
                    \Vert_{L^2(I)}^2
                ds
            \right]\\
            & \leq C
            \sum_{j=1,2}
            \left[
                \sum_{k=1,3}
                    \Vert
                        \tau_k^\prime
                    \Vert_{L^\infty(0,T)}^2
                \int_0^t
                    (
                        \Vert
                            \overline{\mathcal{V}}_j(s)
                        \Vert_{H^2(I)}^2
                        + \Vert
                            \partial_t \overline{\mathcal{V}}_j(s)
                        \Vert_{L^2(I)}^2
                    )
                ds
            \right.\\
            & \left.
                \quad\quad\quad
                +
                \int_0^t
                    \sum_{k=1,3}
                    (
                        \vert
                            \tau_k^\prime(s)
                        \vert^4
                        + \vert
                            \tau_k^{\prime\prime}(s)
                        \vert^2
                    )
                ds
                \sup_{0<t<T} \Vert
                    \overline{\mathcal{V}}_j(t)
                \Vert_{L^2(I)}^2
            \right]\\
            & + C
            \sum_{j=1,2}
                \sum_{k=1,3}
                    \Vert
                        \tau_k^\prime
                    \Vert_{L^\infty(0,T)}^2
                \Vert
                    \overline{\mathcal{V}}_j(t)
                \Vert_{H^1(I)}^2\\
            & + C
            \sum_{j=1,2}
            \left[
                \sup_{0<t<T} \Vert
                        \overline{\mathcal{V}}_j(s)
                    \Vert_{L^2(I)}
                \int_0^t
                    \left(
                        1
                        +
                        \sum_{k=1,3}
                            \vert
                                \tau_k^\prime(s)
                            \vert^2
                    \right)
                ds
            \right]\\
            & + C
            \sum_{j=1,2}
                \int_0^t
                    \Vert
                        \partial_s \mathcal{N}_{V,0}(v_1, v_3, \tau_1, \tau_3)
                        - \partial_s \mathcal{N}_{V,j}(v_1, v_2, v_3, v_4)
                        - \partial_s \mathcal{W}_j
                    \Vert_{L^2(I)}^2
                ds\\
            & =: I_1 + I_2 + I_3 + I_4.
        \end{split}
    \end{align*}
    We find from (\ref{eq_H1_estimate_for_mathcal_overline_V}) that
    \begin{align*}
        I_1
        & \leq C
        T^{-\frac{1}{2}+2\delta}
        \sum_{k=1,3}
            \Vert
                \tau_k
            \Vert_{X_{\sigma,T}}^2
        (
                1
                + \sum_{j=1,3}
                \Vert
                    \tau_j
                \Vert_{X_{\sigma,T}}^2
            )^2
            \Vert
                \tau_1
                - \tau_3
            \Vert_{X_{\sigma,T}}^2\\
        & + C
        \sum_{k=1,3}
        (
            T^{4\delta} \Vert
                \tau_k
            \Vert_{X_{\sigma,T}}^4
            + T^{2\delta} \Vert
                \tau_k
            \Vert_{X_{\sigma,T}}^2
        )
        \Vert
            \tau_1
            - \tau_3
        \Vert_{X_{\sigma,T}}^2
        (
            1
            + \sum_{j=1,3}
                \Vert
                    \tau_j
                \Vert_{X_{\sigma,T}}^2
        ).
    \end{align*}
    We find from (\ref{eq_L2_estimate_for_mathcal_overline_V}) that
    \begin{align*}
        I_2
        \leq C
        T^{-\frac{1}{2}+2\delta}
        \sum_{k=1,3}
            \Vert
                \tau_k
            \Vert_{X_{\sigma,T}}^2
        (
            1
            + \sum_{j=1,3}
                \Vert
                    \tau_j
                \Vert_{X_{\sigma,T}}^2
        )^2
        \Vert
            \tau_1
            - \tau_3
        \Vert_{X_{\sigma,T}}^2,
    \end{align*}
    and
    \begin{align*}
        I_3
        \leq C
        \Vert
            \tau_1
            - \tau_3
        \Vert_{X_{\sigma,T}}^2
        (
            T
            + T^{\frac{1}{2}+ 2 \delta} \sum_{k=1,3}
                \Vert
                    \tau_k
                \Vert_{X_{\sigma,T}}^2
        )
        (
            1
            + \sum_{k=1,3}
                \Vert
                    \tau_k
                \Vert_{X_{\sigma,T}}^2
        ).
    \end{align*}
    We see from Propositions \ref{prop_estimates_for_mathcal_V_j}, \ref{prop_estimate_for_dt_mathcal_N_j}, and \ref{prop_integration_by_parts_for_dt_mathcal_N_V} that
    \begin{align*}
        I_4
        & \leq C
        (
            1
            + \sum_{j=1,3}
                \Vert
                    \tau_j
                \Vert_{X_{\sigma,T}}^2
        )^4
        \Vert
            \tau_1
            - \tau_3
        \Vert_{X_{\sigma,T}}^2\\
        & + C T^{-\frac{1}{2}+2\delta} \Vert
            \tau_1
            - \tau_3
        \Vert_{X_{\sigma,T}}^2
        \sum_{j=1,3}
            \Vert
                \tau_j
            \Vert_{X_{\sigma,T}}^2.
    \end{align*}
    Summing up these inequalities, we obtain the conclusion.
    Moreover, we find from Proposition \ref{prop_estimates_for_mathcal_V_j} that $V_1, V_2$ also satisfy (\ref{eq_H3_a_pripori_estimate_for_V}).
\end{proof}

\begin{proof}{Proof of Theorem \ref{thm_local_in_time_existence_for_sigma1_sigma2}}
We assume that $\Vert (\tau_1, \tau_2) \Vert_{X_{\sigma,T}^2} \leq y^\ast_1$.
Now we consider the equations for the difference (\ref{eq_difference_equation_for_iteration_of_full_filter_clogging_equations}) with $\psi_j = v_j$ for $j=1,2,3,4$.
We find from Propositions \ref{prop_estimates_for_difference_between_nonlinear_terms}(i) and \ref{prop_H1_estimate_for_V_j} that
\begin{align} \label{eq_estimate_for_difference_between_sigma1_sigma3_for_fixed_point_0}
    \begin{split}
        & \sup_{0<t<T}
            \left \vert
                \sigma_1(t)
                - \sigma_3(t)
            \right \vert
        \\
        & \leq C T \sum_{j=1,2}
            \sup_{0<t<T}
                \vert
                    \mathcal{N}_{\Sigma,j}(\tau_1, \tau_2, \tau_3, \tau_4)
                \vert
        + C T \sup_{0<t<T}
            \vert
                \mathcal{N}_{\Sigma,0}(\tau_1, \tau_2, \tau_3, \tau_4)
            \vert\\
        %& \leq C T
        %    (
        %        1
        %        + \sum_{j=1,3}
        %            \Vert
        %                \tau_j
        %            \Vert_{X_{\sigma,T}}
        %    )
        %    \sum_{j=1,2}
        %        \Vert
        %            \tau_1
        %            - \tau_{j+2}
        %        \Vert_{X_{\sigma,T}}\\
        %& + C T
        %(
        %    1
        %    + \sum_{j=1,3}
        %        \Vert
        %            \tau_j
        %        \Vert_{X_{\sigma,T}}^2
        %)
        %\Vert
        %    \tau_1
        %    - \tau_3
        %\Vert_{X_{\sigma,T}}^2\\
        & \leq C T
        \Vert
            \tau_1
            - \tau_3
        \Vert_{X_{\sigma,T}},
        %& \leq C T (
        %    1
        %    + (y^\ast_1)^2
        %)
        %\Vert
        %    \tau_1
        %    - \tau_3
        %\Vert_{X_{\sigma,T}}
    \end{split} 
\end{align}
where $C>0$ depends continuously on $y^\ast_1$ and $CT$ can be taken small for small $T>0$.
Similarly, we find that
\begin{align} \label{eq_estimate_for_difference_between_sigma1_sigma3_for_fixed_point_1}
    \begin{split}
        & \sup_{0<t<T} t^{\frac{1}{4}-\delta}
            \left \vert
                \frac{d\sigma_1(t)}{dt}
                - \frac{d\sigma_3(t)}{dt}
            \right \vert\\
        %& \leq C T^{\frac{1}{4}-\delta} (
        %    1
        %    + (y^\ast_1)^2
        %)
        & \leq C T^{\frac{1}{4}-\delta}
        \Vert
            \tau_1
            - \tau_3
        \Vert_{X_{\sigma,T}}
    \end{split} 
\end{align}
In the same way as Proposition \ref{prop_estimates_for_difference_between_nonlinear_terms} (iii)-(iv), we observe that
\begin{align*}
    \begin{split}
        & 
        \left \vert
            \sum_{j=1,2}
                \frac{d}{dt}\mathcal{N}_{\Sigma,j}(\tau_1, \tau_2, \tau_3, \tau_4)
        \right \vert\\
        &\leq C \left(
            \sum_{j=1,3}
                \vert
                    \tau_j^\prime
                \vert
            \vert
                \tau_2
            \vert
            + \vert
                \tau_2^\prime
            \vert
        \right)
        \vert
            \tau_1
            - \tau_3
        \vert
        + C \vert
            \tau_2
        \vert
        \vert
            \tau_1^\prime
            - \tau_3^\prime
        \vert\\
        & + C \vert
            \tau_3^\prime
        \vert
        \vert
            \tau_2
            - \tau_4
        \vert
        + C
        \vert
            \tau_2^\prime
            - \tau_4^\prime
        \vert
        + C \sum_{j=2,4}
        (
            \vert
                \tau_j^\prime
            \vert
            \vert
                \tau_2
                - \tau_4
            \vert
            + \vert
                    \tau_j
                \vert
            \vert
                \tau_2^\prime
                - \tau_4^\prime
            \vert
        ).
    \end{split}
\end{align*}
Moreover, in the same way as Proposition \ref{prop_estimates_for_difference_between_convection_terms} (ii), we find from the trace theorem, and Propositions \ref{prop_H1_estimate_for_V_j} and \ref{prop_H3_a_pripori_estimate_for_V} that
\begin{align*}
    & \left \vert
        \frac{d}{dt}\mathcal{N}_{\Sigma,0}(\tau_1, \tau_2, \tau_3, \tau_4)
    \right \vert\\
    & \leq C (
        \sum_{j=1,3}
            \vert
                \tau_j^\prime
            \vert
        \vert
            \tau_1
            - \tau_2
        \vert
        + \vert
            \tau_1^\prime
            - \tau_3^\prime
        \vert
    )
    \sum_{j=1,\ldots,4}
        \Vert
           v_j
        \Vert_{H^1(I)}\\
    & + C \vert
        \tau_1
        - \tau_3
    \vert
    \sum_{j=1,\ldots,4}
        \Vert
           \partial_t v_j
        \Vert_{H^1(I)}\\
    & + \sum_{k=1,3}
        \vert
            \tau_k^\prime
        \vert
    \sum_{j=1,2}
        \Vert
            V_j
        \Vert_{H^1(I)}
    + \sum_{j=1,2}
        \Vert
            \partial_t V_j
        \Vert_{H^1(I)}
    )\\
    %& \leq C (
    %    \sum_{j=1,3}
    %        \Vert
    %            \tau_j
    %        \Vert_{X_{\sigma,T}}
    %    \vert
    %        \tau_1
    %        - \tau_2
    %    \vert
    %    + \vert
    %        \tau_1^\prime
    %        - \tau_3^\prime
    %    \vert
    %)
    %(
    %    1
    %    + \sum_{j=1,3}
    %        \Vert
    %            \tau_j
    %        \Vert_{X_{\sigma,T}}^2
    %)^{\frac{1}{2}}\\
    %& + C \Vert
    %    \tau_1
    %    - \tau_3
    %\Vert_{X_{\sigma,T}}
    %\sum_{j=1,\ldots,4}
    %    [
    %        1
    %        + T^{-\frac{1}{2}+2\delta}
    %        \Vert
    %            \sigma_j
    %        \Vert_{X_{\sigma,T}}^2
    %        + \Vert
    %            \sigma_j
    %        \Vert_{X_{\sigma,T}}^6
    %    ]^\frac{1}{2}\\
    %& + C \sum_{k=1,3}
    %    \vert
    %        \tau_k^\prime
    %    \vert
    %(
    %    1
    %    + \sum_{j=1,3}
    %    \Vert
    %        \tau_j
    %    \Vert_{X_{\sigma,T}}^2
    %)
    %\Vert
    %    \tau_1
    %    - \tau_3
    %\Vert_{X_{\sigma,T}}\\
    %& + C (
    %    1
    %    + T^{-\frac{1}{4}+\delta}
    %)
    %\Vert
    %    \tau_1
    %    - \tau_3
    %\Vert_{X_{\sigma,T}}
    %(
    %    1
    %    + \sum_{j=1,3}
    %        \Vert
    %            \tau_j
    %        \Vert_{X_{\sigma,T}}^2
    %).
    &\leq C (
        1
        + T^{-\frac{1}{4}+\delta}
    )
    \Vert
        \tau_1
        - \tau_3
    \Vert_{X_{\sigma,T}}.
\end{align*}
%Since
%\begin{align*}
%    T^{\frac{1}{2}-\delta}[
%        1
%        + T^{-\frac{1}{2}+2\delta}
%        \Vert
%            \sigma_j
%        \Vert_{X_{\sigma,T}}^2
%        + \Vert
%            \sigma_j
%        \Vert_{X_{\sigma,T}}^6
%    ]^\frac{1}{2}
%    \leq C T^{\frac{1}{4}}
%    (
%        1
%        + \Vert
%            \sigma_j
%        \Vert_{X_{\sigma,T}}^2
%    )^3.
%\end{align*}
Therefore, we conclude that
\begin{align} \label{eq_estimate_for_difference_between_sigma1_sigma3_for_fixed_point_2}
    \begin{split}
        & \sup_{0<t<T}
            t^{\frac{1}{2}-\delta}
            \left \vert
                \frac{d^2(
                    \sigma_1(t)
                    - \sigma_3(t)
                )
                }{dt}
            \right \vert\\
        & \leq C
        \sup_{0<t<T} t^{\frac{1}{2}-\delta} \sum_{j=1,2}
            \left \vert
                \frac{d}{dt}\mathcal{N}_{\Sigma,j}(\tau_1, \tau_2, \tau_3, \tau_4)
            \right \vert\\
        & + C \sup_{0<t<T} t^{\frac{1}{2}-\delta} \sum_{j=1,2}
            \left \vert
                \frac{d}{dt}\mathcal{N}_{\Sigma,0}(\tau_1, \tau_2, \tau_3, \tau_4)
            \right \vert\\
        & \leq C T^{\frac{1}{4}}
        \sum_{j=1,2}
            \Vert
                \tau_j
                - \tau_{j+2}
            \Vert_{X_{\sigma,T}}
        %& \leq C T^{\frac{1}{4}}
        %    \sum_{j=1,\ldots,4}
        %    (
        %        1
        %        + (y_1^\ast)^2
        %    )^3
        %\sum_{j=1,2}
        %    \Vert
        %        \tau_j
        %        - \tau_{j+2}
        %    \Vert_{X_{\sigma,T}}
    \end{split} 
\end{align}
for $C>0$ depends continuously on $y^\ast_1$ and $CT$ can be taken small for small $T>0$.
We find from (\ref{eq_estimate_for_difference_between_sigma1_sigma3_for_fixed_point_0}), (\ref{eq_estimate_for_difference_between_sigma1_sigma3_for_fixed_point_1}), (\ref{eq_estimate_for_difference_between_sigma1_sigma3_for_fixed_point_2}) that the solution operator $(\mathcal{S}_{v,1}, \mathcal{S}_{v,2})$ is a contraction mapping in $B_{X_{\sigma, T}^2, y^\ast_1}$ for small $T>0$.
Banach's fixed point theorem implies that there exists a unique solution to (\ref{eq_abstract_filter_clogging_equation_in_the_interior}) in $X_{\sigma, T}^2$ for small $T>0$.
\end{proof}

We next establish the a priori estimates.
Recall that we have already proved the solution $(v_1, v_2)$ exists globally-in-time.
\subsubsection*{Step 3: a priori estimates}
    Using the positivity of $\sigma_{j,0}>0$ and $v_1$, the definitions of $N_{\sigma, j}$, the comparison principle for ordinary differential equations, we find that $\sigma_1(t)$ is non-negative for all $t\geq0$.
    Therefore, we also see that $\sigma_2(t)$ is non-negative for $t\geq 0$.
    By the Gronwall inequality, we have
    \begin{align*}
        \begin{split}
            \sigma_1(t)
            & \leq 
            \sigma_{0,1}
            + C \int_0^t
                \gamma_+ v_1
            ds\\
            & \leq \sigma_{0,1}
            + C \int_0^t
                \Vert
                    v_1(s)
                \Vert_{H^1(I)}
            ds,\\
            \sigma_2(t)
            & \leq C e^{R_1 t}
            \left[
                \sigma_{0,2}
                + \int_0^t
                    \gamma_+ v_2(s)
                ds
            \right]\\
            & \leq C e^{R_1 t}
            \left[
                \sigma_{0,2}
                + \int_0^t
                    \Vert
                        v_2(s)
                    \Vert_{H^1(I)}
                ds
            \right].
        \end{split}
    \end{align*}
    The integrand in the right-hand side is bounded because of the embedding $L^2(0,T) \hookrightarrow L^1(0,T)$.
    Propositions \ref{prop_estimate_for_dsigma_L_infty} and \ref{prop_nonlinear_estimates_for_Nv} (ii) implies
    \begin{align*}
        \begin{split}
            & \sum_{j=1,2}
                \left \Vert
                    \frac{d\sigma_j}{dt}
                \right \Vert_{L^\infty(0,T)}\\
            & \leq C
            (
                \Vert
                    \sigma_2
                \Vert_{L^\infty(0,T)}
                + \Vert
                    \sigma_2
                \Vert_{L^\infty(0,T)}^2
                + \sum_{j=1,2}
                    \Vert
                        v_j
                    \Vert_{L^\infty_tH^1_x(Q_T)}
            )\\
            & < \infty,
        \end{split}
    \end{align*}
    and
    \begin{align*}
        \begin{split}
            & \sum_{j=1,2}
                \left \Vert
                    \frac{d^2\sigma_j}{dt^2}
                \right \Vert_{L^\infty(0,T)}\\
            & \leq C
            \sum_{j=1,2}
                \Vert
                    \sigma_j
                \Vert_{L^\infty(0,T)}
            \sum_{k=1,2}
                \Vert
                    \sigma_k^\prime
                \Vert_{L^\infty(0,T)}\\
            & + C \sum_{j=1,2} \Vert
                \sigma_j^\prime
            \Vert_{L^\infty(0,T)}
            \sum_{j=1,2}
                \Vert
                    v_j
                \Vert_{L^\infty_tH^1_x(Q_T)}
            + C \sum_{j=1,2}
                \Vert
                    \partial_t v_j
                \Vert_{L^\infty_tH^1_x(Q_T)}\\
            & < \infty.
        \end{split}
    \end{align*}
    Suppose that there exists some $T^\ast>0$ such that $\Vert(\sigma_1, \sigma_2) \Vert_{X_{\sigma,T}^2} = \infty$.
    However, the a priori estimates implies that the solution $\Vert (\sigma_1, \sigma_2)\Vert_{W^{2, \infty}(0,T^\ast)}<\infty$.
    Therefore, there exists $\varepsilon>0$ we can extend the solution $(\sigma_1, \sigma_2)$ can exist up to $(0,T^\ast+\varepsilon)$.
    By contraction argument, $(\sigma_1, \sigma_2)$ exists globally-in-time.

\begin{proof}[Proof of Theorem \ref{thm_global_well_posedness_for_eq_abstract_filter_clogging_equation}]
    This is a direct consequent from arguments in Step 1-3.
    %We invoke that since $v_1 = v_2 =0$ $\sigma_1, \sigma_2 = 0$ are a sub solution to (\ref{eq_filter_clogging_nondimensional}), the solutions $v_1, v_2, \sigma_1, \sigma_2$ are positive if as long as the data $f, v_1(0), v_2(0), \sigma_1(0), \sigma_2(0)$ are all positive by the maximal principle, see Theorem 11-12 in Section 7 in \cite{Evans2010}.
\end{proof}

\section{Numerical Experiments} \label{sec_numerical_experiments}
\subsection{Discretization}
To perform numerical experiments, we implement a discretization scheme for the equations (\ref{eq_filter_clogging_nondimensional}).
In this section, our aim is not to establish specific numerical-analytical results for the equations (\ref{eq_filter_clogging_nondimensional}), but to explain the discretization scheme for the numerical experiments in the Section \ref{subsec_numerical_experiments}.
We will first focus on the simplified toy equation:
\begin{equation} \label{eq_filter_clogging_eq_simplified_for_discretization}
    \begin{split}
        \begin{aligned}
            &\partial_t u - \nu \partial_x^2 u + c \partial_x u
            = f(u),
            & t \in (0, T),
            & \, x \in I,\\
            &B_1(u; \theta)
            = 0,
            & t \in (0, T),
            & \,\\
            &B_2(u; \theta)
            =0,
            & t \in (0, T),
            & \,\\
            &u
            =u_0,
            & t=0,
            &\,x \in I.
        \end{aligned}
    \end{split}
\end{equation}
for some sufficiently smooth function $f$, $\theta$, $u_0$, and the velocity $c = c(t)$.
We employ a spatial discretization technique based on the finite difference method.
The domain $I=(0,1)$ is decomposed into
\begin{align*}
    & I
    = \cup_{j=0,N-1} (x_j, x_{j+1}]
    \cup (x_{N-1}, x_N), \\
    & x_0
    = 0,\,
    x_N
    = 1, \,
    x_{j+1}
    = x_j
    + \Delta x, \,
    \Delta x = \frac{1}{N}.
\end{align*}
We denote the discretization for $u$ by
\begin{align*}
    u_j
    = u_j(t)
    = u(x_j, t),\,
    t>0.
\end{align*}
For boundary points $x_0$ and $x_N$ of $I$, we introduce additional synthetic points $x_{-1}$ and $x_{N+1}$ outside of them.
These points are defined by
\begin{align*}
    x_{-1}
    = - \Delta x, \,
    x_{N+1}
    = 1
    + \Delta x.
\end{align*}
We denote $f_k := f(u_k)$.
We discretize the first equation in (\ref{eq_filter_clogging_eq_simplified_for_discretization}) as follows:
\begin{align} \label{eq_discretization_for_filter_clogging_eq_simplified}
    & \frac{
        d u_j
    }{dt}
    - \nu \frac{
        u_{j+1} - 2 u_{j} + u_{j-1}
    }{
        (\Delta x)^2
    }
    + c \frac{
        u_{j+1} - u_{j-1}
    }{
        2 \Delta x
    }
    = f_j
\end{align}
for $j = 0, 1, \ldots, N$.
We also discretize the boundary conditions in (\ref{eq_filter_clogging_eq_simplified_for_discretization}) as follows:
\begin{align} \label{eq_discretization_for_boundary_conditions_for_filter_clogging_eq_simplified}
    \begin{split}
        & \frac{
            u_{N+1} + u_{N-1}
        }{
            2
        }
        - \frac{
            u_{1} + u_{-1}
        }{
            2
        }
        = \theta  \frac{
            u_{N+1} + u_{N-1}
        }{
            2
        },\\
        & \frac{
            u_{N+1} - u_{N-1}
        }{
            2 \Delta x
        }
        - \frac{
            u_{1} - u_{-1}
        }{
            2 \Delta x
        }
        = - \theta \frac{
            u_{1} - u_{-1}
        }{
            2 \Delta x
        }.
    \end{split}
\end{align}
The discretization is second-order in $\Delta x$.
We solve (\ref{eq_discretization_for_boundary_conditions_for_filter_clogging_eq_simplified}) to obtain the formula
\begin{align} \label{eq_formula_for_u_N+1_u_minus1}
    \left[
        \begin{array}{c}
            u_{N+1}\\
            u_{-1}
        \end{array}
    \right]
    = \frac{
        1
    }{
        1 + A^2
    }
    \left[
        \begin{array}{c}
            (1 - A^2) u_{N-1}
            + 2 A u_{1}\\
            2A u_{N-1}
            + (- 1 + A^2) u_{1}
        \end{array}
    \right]
\end{align}
for
\begin{align*}
    A := 1 - \theta.
\end{align*}
We numerically solve equation (\ref{eq_filter_clogging_eq_simplified_for_discretization}) using the Crank-Nicolson scheme with (\ref{eq_discretization_for_filter_clogging_eq_simplified}) and (\ref{eq_formula_for_u_N+1_u_minus1}).
At time step $t_{k}$ ($k\in \Integer_{\geq0}$), we first solve (\ref{eq_filter_clogging_eq_simplified_for_discretization}) using the Euler scheme and obtain $u^{k+1}_{j} := u_j(t_{k+1})$ ($j=0, 1,\ldots, N$).
Then, we have the approximated values
\begin{align*}
    f^{k+1}_{j} := f(u^{k+1}_{j}), \quad
    j=1,2, \ldots, N.
\end{align*}
Using $f^{k+1}_{j}$, we again solve (\ref{eq_filter_clogging_eq_simplified_for_discretization}) using the Crank-Nicolson scheme.
The order of $\Delta t$ resulting from our time discretization is $(\Delta t)^2$.
We apply the same scheme to solve the equations of $\sigma_1, \sigma_2$ in (\ref{eq_filter_clogging_nondimensional}).

We numerically solve the original equations (\ref{eq_filter_clogging_nondimensional}).
Assume that the solution $((v_1)^k_j, (v_2)^k_j, (\sigma_1)^k, (\sigma_2)^k)$ at time $t_k$ has been computed.
Our update procedures are as follows:
\begin{enumerate}
    \item Calculate $(v_1)^{k+1}_j$ ($j=0, 1,\ldots, N$) from $((v_1)^k_j, (v_2)^k_j, (\sigma_1)^k, (\sigma_2)^k)$.
    \item Calculate $(v_2)^{k+1}_j$ ($j=0, 1,\ldots, N$) from $((v_1)^{k+1}_j, (v_2)^k_j, (\sigma_1)^k, (\sigma_2)^k)$.
    \item Calculate $(\sigma_1)^{k+1}$ from $((v_1)^{k+1}_j, (v_2)^{k+1}_j, (\sigma_1)^k)$.
    \item Calculate $(\sigma_2)^{k+1}$ from $((v_1)^{k+1}_j, (v_2)^{k+1}_j, (\sigma_1)^{k+1}, (\sigma_2)^k)$.
\end{enumerate}

%%%%%%%%%%%%%%%%%%%%%%%%%%%%%%%%%%%%%%%%%%%%%%%%%
%%%%%%%%%%%%%%%%%%%%%%%%%%%%%%%%%%%%%%%%%%%%%%%%%

\subsection{Results of numerical experiments}\label{subsec_numerical_experiments}
We conducted some numerical experiments to demonstrate and understand the (asymptotic) behavior of the solution to (\ref{eq_filter_clogging_nondimensional}).
We did not conduct the numerical experiments to (\ref{eq_filter_clogging}) itself.
\begin{table}[h]
    \centering
    \caption{Default parameter settings.}
    \begin{tabular}{|c|c||c|c|}
    %\begin{tabular}{|c|c|c||c|c||c|c|c|}
        \hline
        Parameters    & Value               & Parameters & Value        \\ \hline
        $\Delta t$    & $0.01$              & $\Delta x$ & $1/32$       \\ \hline
        $\nu_1$       & $0.1$               & $\nu_1$    & $0.1$        \\ \hline
        $A$           & $1.0$               & $B$        & $1.0$        \\ \hline
        $C_u$         & $1.0$               & $C_\rho$   & $1.0$        \\ \hline
        $\tilde{R}_1$ & $0.5$               & $R_2$      & $0.5$        \\ \hline
        $\tilde{S}_1$ & $1.0 \times C_\rho$ & $S_2$      & $1.0$        \\ \hline
        $Q_1$         & $1.0$               & $Q_2$      & $1.0/C_\rho$ \\ \hline
        $\tilde{f}$   & $1.0$               & $\Omega$   & $0.1$        \\ \hline
    \end{tabular}
    \label{fig_parameter_settings}
\end{table}
The default parameter settings for the equations (\ref{eq_filter_clogging_nondimensional}) have been specified in Table \ref{fig_parameter_settings}.
We set domain $I=(0,1)$ and the external force $f$ is constant function in $I \times (0,T)$ for $T>0$.
We define the filter power function $F$ by
\begin{align*}
    F
    = F(\sigma_1)
    = \frac{
        1
    }{
        1 + \beta \sigma_1
    }
\end{align*}
for some $\beta>0$.
Note that $F(0)=1$ and $F(+\infty)=0$.
The parameter $\beta$ indicates the rate of decline in the filter's absorption efficiency.
As $\beta$ increases, the deterioration in the filtration performance of the filter occurs more rapidly.
In our numerical experiments, we have defined $\beta=2.0$.

%%%%%%%%%%%%%%%%%%%%%%%%%%%%%%%%%%%%%%%%%%%%%%%%%%%%%%%%%%%%
%%%%%%%%%%%%%%%%%%%%%%%%%%%%%%%%%%%%%%%%%%%%%%%%%%%%%%%%%%%%
\begin{figure*}%[ht]
    \newpage
    \begin{center}
        %\vspace{-100mm}
        %\includegraphics[keepaspectratio, bb =0 0 1700 1357, scale=0.7]{unified_time_series.png}
        %\includegraphics[keepaspectratio, scale=0.7]{unified_time_series_v1.png}
        \includegraphics[keepaspectratio, scale=0.55]{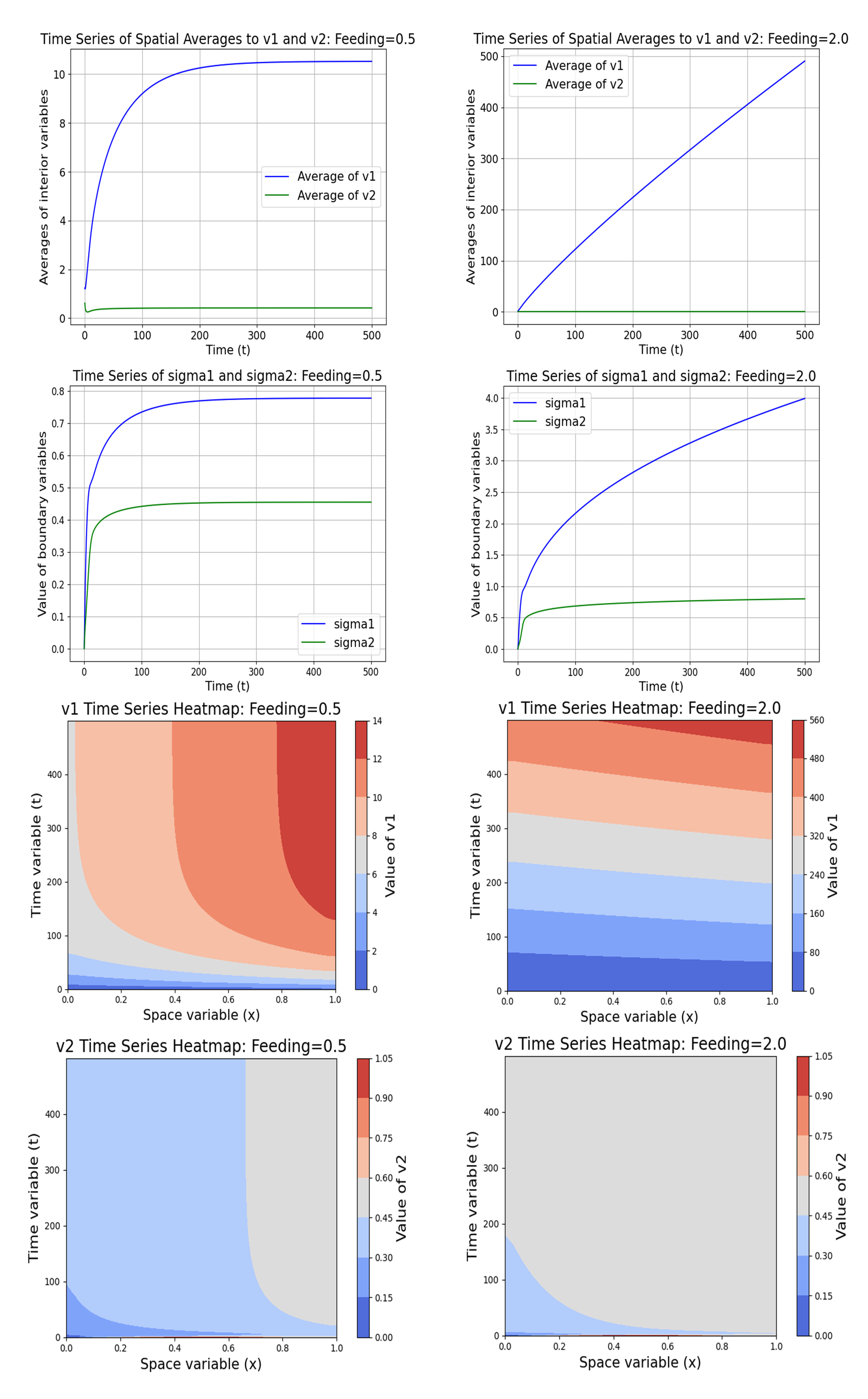}
        \caption{
            Time series of two cases. We compared with two cases: $f=0.5, 2.0$.
            The left-hand and right-hand sides four figures indicate the result when $f=0.5$ and $f=2.0$, respectively.
            The top graphs are time series of the spatial average to $v_1$ and $v_2$.
            The second graphs from the top are time series of $\sigma_1$ and $\sigma_2$.
            The second heatmap from the bottom describes the time series of $v_1$. The horizontal axis describes domain $I$.
            The vertical axis describes time interval $(0,T)$ for $T=500$.
            The second heatmap from the bottom describes the time series of $v_2$.
        }
        \label{fig_comparison_f05_f2}
    \end{center}
\end{figure*}

\begin{figure*}%[ht]
    %\newpage
    \begin{center}
        %\vspace{-100mm}
        %\includegraphics[keepaspectratio, bb =0 0 1700 1357, scale=0.7]{unified_time_series.png}
        %\includegraphics[keepaspectratio, scale=0.9]{test8_clogging_borderline_5000000_editted.png}
        \includegraphics[keepaspectratio, scale=0.9]{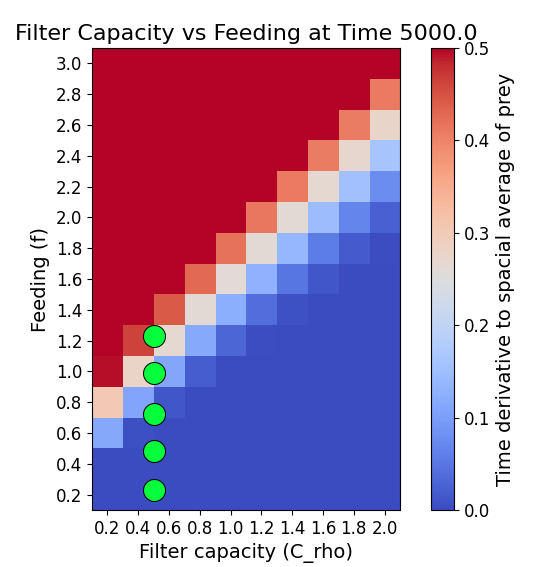}
        \caption{
            Comparison of parameters for the transition.
            The horizontal axis represents the capacity parameter $C_\rho$.
            The vertical axis represents the feeding rate $f$.
            The height corresponds to the value of the time derivative of the time-average of $v_1$ at time $5000$.
            The zero-height area indicates a non-clogging case.
            In cases where the time derivatives are non-zero, continuous growth of $v_1$ (organic matter) is observed.
            For the green points, we illustrate the time series in Fig. \ref{fig_time_series_for_transition}.
        }
        \label{fig_transition_feeding_capacity_vs_feeding}
    \end{center}
\end{figure*}

\begin{figure*}%[ht]
    %\newpage
    \begin{center}
        %\vspace{-100mm}
        %\includegraphics[keepaspectratio, bb =0 0 1700 1357, scale=0.7]{unified_time_series.png}
        \includegraphics[keepaspectratio, scale=0.55]{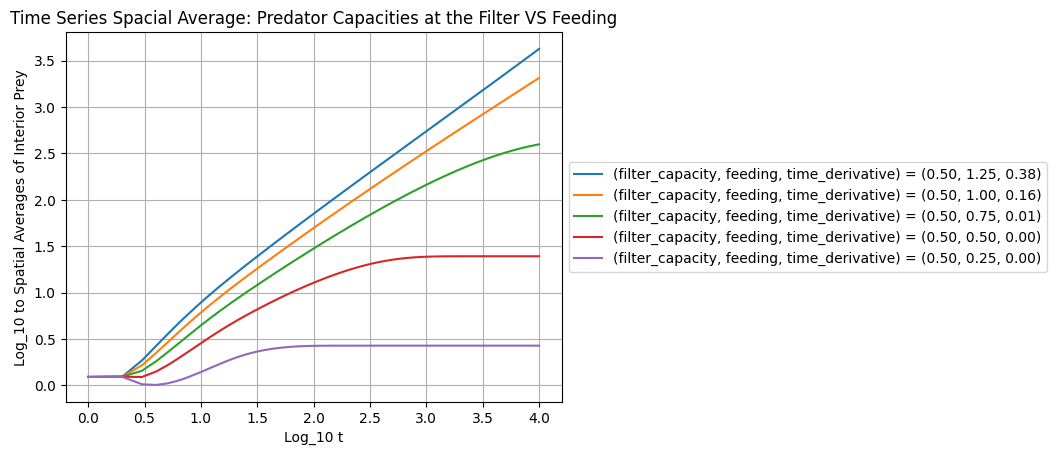}
        \caption{
            Time series graphs for points crossing the transition boundary.
            The horizontal axis represents $log_{10} t$ for time $t \in [0, 5000]$, and the vertical axis represents $log_{10} v_1$. We set $C_{\rho}=0.50$ and conducted numerical experiments for five cases of $f$ with increments of $0.25$, such that $f = 0.25, 0.50, \ldots, 1.25$.
            The time derivative of the spatial average at $t=5000$ is calculated for each case.
        }
        \label{fig_time_series_for_transition}
    \end{center}
\end{figure*}

We first demonstrate that the model (\ref{eq_filter_clogging_nondimensional}) effectively replicates a dynamic in an aquarium.
Figure \ref{fig_comparison_f05_f2} presents time series data for the cases when $f=0.5$ (left-hand side) and $f=2.0$ (right-hand side).
The case with $f=0.5$ shows $v_1, v_2, \sigma_1, \sigma_2$ monotonically converging to some steady state.
This case indicates that when a small amount of dust of organic matter is introduced into the aquarium, it suggests that filtration by bacteria within the filter is effective, and clogging does not occur.
As a result, there is no divergence of $v_1$ within the aquarium $I$.
On the other hand, when $f=2.0$, $v_2$ and $\sigma_2$ converge to some steady states, but $v_1$ and $\sigma_1$ grow up to some unbounded states.
The convergences of $v_2$ and $\sigma_2$ are reasonable because $-v_2^2$ and $\sigma_2^2$ on the right-hand side of (\ref{eq_filter_clogging_nondimensional}) suppress the growth of $v_2$ and $\sigma_2$.
As mentioned in the Introduction, these suppression terms imply that predators do not habit infinitely many in the filtration filter.
In this case, the figures on the right-hand side depict a situation where an excessive amount of dust is introduced into aquarium $I$, resulting in insufficient filtration.
In this scenario, both $v_1$ and $\sigma_1$ exhibit divergence.
The divergence of $\sigma_1$ occurs at a slower rate compared to that of $v_1$.
%Note that the clogging of the filter does not only refer to the state where $F=0$ ($\sigma=\infty$), but also to the temporal increase of $v_1$ and $\sigma_1$.
However, the first and second diagrams from the bottom of Fig. \ref{eq_filter_clogging_nondimensional}, while demonstrating differences in filtration capacity, indicate that filtration indeed works.
In these heatmaps, $v_1$ and $v_2$ absorbed from the right end $x=1$ of domain $I$ are consistently absorbed by the boundary filter.
The portion that was not absorbed flows from the right at $x=0$.
Consequently, the values of $v_1$ and $v_2$ at $x=0$ are smaller than their values at $x=1$.

%%%%%%%%%%%%%%%%%%%%%%%%%%%%%%%%%%%%%%%%%%%%%%%%%%%%%%%%%%%%
%%%%%%%%%%%%%%%%%%%%%%%%%%%%%%%%%%%%%%%%%%%%%%%%%%%%%%%%%%%%

We conducted numerical experiments to demonstrate the transition from cases without clogging to cases with clogging.
Figure \ref{fig_transition_feeding_capacity_vs_feeding} represents the numerical experiment illustrating the transition boundary.
We performed numerical simulations up to time $5000$ for $C_\rho = 0.2, 0.4, \ldots, 2.0$ in increments of 0.2, and for $f = 0.2, 0.4, \ldots, 3.0$ in increments of $0.2$.
The height in the figure represents the time derivative of spatial averages of $v_1$ at time $5000$.
The blue region (with zero height) indicates convergence to a steady state, while the red region (with non-zero height) indicates a growth up.
Cases with a non-zero time derivative imply temporal growth of $v_1$, which can be also understood as a clogging of the filter.
The transition boundary can be characterized as the neighborhood of the line segment connecting the points $(0.2, 0.4)$ and $(2.0, 1.8)$ in Fig. \ref{fig_transition_feeding_capacity_vs_feeding}.

%%%%%%%%%%%%%%%%%%%%%%%%%%%%%%%%%%%%%%%%%%%%%%%%%%%%%%%%%%%%
%%%%%%%%%%%%%%%%%%%%%%%%%%%%%%%%%%%%%%%%%%%%%%%%%%%%%%%%%%%%
We also illustrate the time series graphs for five parameter settings dotted by green points in Fig. \ref{fig_transition_feeding_capacity_vs_feeding}. These points cross the boundary of the transition.
The results are shown in Fig. \ref{fig_time_series_for_transition}.
The horizontal axis represents $\log_{10} t$ for time $t$, and the vertical axis represents $\log_{10} v_1$. We set $C_{\rho}=0.50$ and conducted numerical experiments for five cases of $f$, namely, $f = 0.25, 0.50, \ldots, 1.25$, with increments of $0.25$.
The time derivative for the spatial averages at $t=5000$ is calculated for each case.

The lower two graphs, corresponding to $f=0.25$ and $f=0.50$, represent cases without clogging and are colored in blue in Fig. \ref{fig_transition_feeding_capacity_vs_feeding}.
The upper two graphs, for $f=1.00$ and $f=1.25$, exhibit the divergence of $v_1$ and subsequently filter clogging.
The growth is linear.
The third graph from the bottom, for $f=0.75$, lies on the transition boundary and shows a slower convergence to a steady state without clogging.
The convergence can be observed from the small-time derivative.

\subsection{Conclusion}
The dynamics in an aquarium can be described by the model proposed in this paper.
The model demonstrates a transition from a convergence to a steady state to a divergence.
This transition depends on external forces for the dust $v_1$, and the filtration filter's capacities for microbes $\sigma_2$.
Roughly speaking, the numerical experiments suggest that the convergence of a steady state can be provided when the filtration filter has a sufficiently large amount of filter media compared to the amount of feedings, aligning with our physical intuition.

\section*{Acknowledgments}
The first author is grateful to the members of RIKEN Pioneering Project “Prediction for Science” and Environmental Metabolic Analysis Research Team at RIKEN for their helpful discussions and comments.
The authors are grateful to Professor Takashi Sakajo of Kyoto University for the helpful discussion of the study at this early stage.
The first author was partly supported by RIKEN Pioneering Project “Prediction for Science” and JSPS Grant-in-Aid for Young Scientists No. 22K13948.
The second author was partly supported by JSPS Grant-in-Aid for Scientific Research (B) No. JP21H01004.

\end{document}